\documentclass[11pt,a4paper]{article}
\usepackage[margin=1in]{geometry}
\usepackage[authoryear,round]{natbib}
\usepackage[T1]{fontenc}
\usepackage{amsmath,amssymb,amsthm,mathtools,bm}
\usepackage{graphicx,booktabs,tabularx,array,microtype}
\usepackage{needspace,aliascnt,placeins,float}
\usepackage{hyperref,url}
\usepackage[capitalise,nameinlink,noabbrev]{cleveref}
\newtheorem{theorem}{Theorem}[section]
\newaliascnt{lemma}{theorem}
\newtheorem{lemma}[lemma]{Lemma}
\aliascntresetthe{lemma}
\newaliascnt{proposition}{theorem}
\newtheorem{proposition}[proposition]{Proposition}
\aliascntresetthe{proposition}
\newaliascnt{corollary}{theorem}
\newtheorem{corollary}[corollary]{Corollary}
\aliascntresetthe{corollary}
\theoremstyle{definition}
\newaliascnt{definition}{theorem}

\aliascntresetthe{definition}
\newaliascnt{external}{theorem}

\aliascntresetthe{external}
\theoremstyle{remark}
\newaliascnt{remark}{theorem}
\newtheorem{remark}[remark]{Remark}
\aliascntresetthe{remark}
\crefname{external}{External theorem}{External theorems}
\Crefname{external}{External theorem}{External theorems}
\newcommand{\R}{\mathbb R}
\newcommand{\N}{\mathbb N}
\newcommand{\E}{\mathbb E}
\newcommand{\Prob}{\mathbb P}
\newcommand{\eps}{\varepsilon}
\newcommand{\Cstar}{C_\star}
\newcommand{\Tval}{\mathsf T^{\mathrm{val}}}

\DeclareMathOperator{\tr}{tr}

\DeclareMathOperator{\diag}{diag}
\DeclareMathOperator{\Var}{Var}
\DeclareMathOperator{\Sym}{Sym}
\title{Lower Bounds For Gradient-Free\\
Convex Optimization And Convex-Concave\\
Saddle-Point Problems}
\author{\small \textbf{Yuriy Dorn}\textsuperscript{1}, \textbf{Darina Dvinskikh}\textsuperscript{2}, \textbf{Timofei Loginov}\textsuperscript{3},\\
\small \textbf{Osman Osmanov}\textsuperscript{4}, \textbf{Aleksandr Shestakov}\textsuperscript{5},\\
\small \textbf{Nazarii Tupitsa}\textsuperscript{5}, \textbf{Alexander Gasnikov}\textsuperscript{4}\\[3pt]
\normalfont\small\textsuperscript{1} AI Institute MSU (msu.ru)\\
\normalfont\small\textsuperscript{2} HSE University (hse.ru)\\
\normalfont\small\textsuperscript{3} Moscow Independent Research Institute of Artificial Intelligence (miriai.org)\\
\normalfont\small\textsuperscript{4} Innopolis University (innopolis.university)\\
\normalfont\small\textsuperscript{5} Mohamed bin Zayed University of Artificial Intelligence (mbzuai.ac.ae)}

\hypersetup{hidelinks,pdfauthor={Yuriy Dorn, Darina Dvinskikh, Timofei Loginov, Osman Osmanov, Aleksandr Shestakov, Nazarii Tupitsa, Alexander Gasnikov}}
\hypersetup{pdftitle={Lower Bounds For Gradient-Free Convex Optimization And Convex-Concave Saddle-Point Problems},
 pdfsubject={Convex optimization and convex-concave saddle-point problems}}
\date{}
\begin{document}
\addtocontents{toc}{\protect\setcounter{tocdepth}{-1}}
\maketitle
\begin{abstract}
We study the query complexity of optimization with exact scalar-value information. For globally $L$-smooth convex functions on $\mathbb R^d$ with a minimizer in a Euclidean ball of radius $R$, we prove the lower bound $\Omega(d\min\{d,\sqrt{LR^2/\varepsilon}\})$ for adaptive randomized algorithms in the stated accuracy range. Queries are unrestricted, and the bound contains no logarithmic factor. It holds for quadratics and matches the minimum of two elementary upper bounds on that class. Prescribed strong convexity introduces the additional cap $\sqrt{L/\mu}$. The proof compares the mean of a matrix with the mean of its inverse on arbitrary affine slices and bounds posterior moments after adaptive measurements. We extend these information bounds to saddle-point problems with independent block dimensions. For compact bilinear games, the coupling alone yields $\Omega(r\min\{r,\sqrt{CR_xR_y/\varepsilon}\})$ queries, where $r=\min\lbrace d_x,d_y\rbrace$. A strongly convex--strongly concave construction yields a condition-dependent lower bound even with known isotropic diagonal Hessians. Separate reductions establish scalar query lower bounds under mixed value and partial-gradient access, including for outer objective accuracy. A Euclidean-image radius yields smooth and nonsmooth extensions to independently specified norms and known penalties.
\end{abstract}

\section{Introduction}
Gradient-free optimization minimizes an objective using function values.
Query complexity counts the scalar evaluations required for a prescribed
objective error. For convex problems, we seek its joint dependence on
dimension, accuracy, and geometry. Finite
differences and randomized directional estimates connect value queries to
first-order methods by approximating gradients from nearby evaluations
\citep{NesterovSpokoiny2017}.

Non-Euclidean proximal methods can improve the dimension dependence of
these upper bounds. Under Euclidean regularity, the accelerated method of
\citet{GorbunovDvurechenskyGasnikov2022} has a $\sqrt d$ dimension factor
with an $\ell_1$ setup, while the nonsmooth two-point method of
\citet{Shamir2017} has only logarithmic dimension dependence. Establishing
their optimality requires lower bounds covering all adaptive value-query
methods under the same regularity and localization assumptions.

Earlier lower bounds address complementary information models.
\citet{Yao1975} studied deterministic exact quadratic recovery, and
\citet{Protasov1996} studied approximate quadratic minimization from values.
Stochastic and noisy observations yield norm-dependent or noise-dependent
lower bounds \citep{DuchiJordanWainwrightWibisono2015,YuEtAl2023,YuEtAl2024}.
For exact values of a fixed objective,
\citet{ZhangZhangQiLin2026,Kerger2026} establish nonsmooth convex lower bounds.
Here we study smooth and nonsmooth exact-value complexity with independent
regularity and localization geometries. We allow arbitrary finite queries,
free computation and precision, and adaptive randomized algorithms with a
fixed worst-case budget and success probability at least $2/3$ on every
objective.

\paragraph{Contributions.}
\emph{Dimension and accuracy.} We prove
$d\min\{d,\sqrt{LR^2/\eps}\}/4096$, matching quadratic complexity and
full-class complexity at moderate accuracy
(\cref{thm:main,cor:quadratic,cor:general-moderate}). A distinct multiscale
construction gives a dimension-times-logarithm obstruction and matching
full-class bounds uniform for $\mu/L\in[1/2,1]$
(\cref{thm:near-quadratic,ms:thm-depth}).

\emph{Tight non-Euclidean gradient-free lower bounds.} Let $1\le p\le2$,
$q=p/(p-1)$, with $q=\infty$ at $p=1$, and let $R_p$ be the $\ell_p$
radius. For globally convex objectives, Euclidean $L_2$-smoothness with a
minimizer in this ball and Euclidean $M_2$-Lipschitz continuity with
minimization on it give, respectively,
\[
 \underbrace{\Omega\!\left(d^{1/2+1/q}\sqrt{\frac{L_2R_p^2}{\eps}}\right)}_{\text{smooth}},
 \qquad
 \underbrace{\widetilde\Omega\!\left(d^{2/q}\frac{M_2^2R_p^2}{\eps^2}\right)}_{\text{nonsmooth}}.
\]
To our knowledge, these are the first lower bounds for unrestricted
randomized exact-value algorithms matching both rates for $1\le p<2$,
up to logarithmic factors in the windows of
\cref{match:cor-smooth,match:cor-nonsmooth}. The nonsmooth result transfers
the Euclidean information bound of \citet{ZhangZhangQiLin2026}.

\emph{Arbitrary norms and known regularizers.} The Euclidean-image
profile $\mathcal R_m(B,K)$ gives lower bounds with independent
regularity and localization geometries
(\cref{gp:thm-norm,gp:thm-nonsmooth}). Actual regularizer sublevels can
improve the lower-bound formula by an unbounded factor
(\cref{gp:thm-composite,gp:cor-spectral}).

\emph{Saddle points and information.} We isolate coupling and mixed-oracle
costs for unequal block dimensions (\cref{sp:sec}). The smooth information
argument uses an affine-slice inverse comparison and an adaptive posterior
trace identity (\cref{thm:affine,sec:mechanism}).
\Cref{tab:bounds,sp:tab-main} summarize the minimization and saddle-point
bounds; \cref{app:matching,sp:app-comparison} give the detailed comparisons.

\begin{table}[t]
\centering
\caption{Comparison of upper and lower bounds for different problem setups.
See \cref{app:matching} for more details.}
\label{tab:bounds}
\begingroup
\setlength{\tabcolsep}{2pt}
\renewcommand{\arraystretch}{1.1}
\footnotesize
\begin{tabularx}{\linewidth}{@{}>{\raggedright\arraybackslash}p{.13\linewidth}
 >{\centering\arraybackslash}p{.235\linewidth}
 >{\raggedright\arraybackslash}p{.17\linewidth}
 >{\centering\arraybackslash}p{.085\linewidth}
 >{\centering\arraybackslash}p{.235\linewidth}
 >{\raggedright\arraybackslash}X@{}}
\toprule
&\multicolumn{2}{c}{Upper bound}&\multicolumn{2}{c}{Lower bound}&\\
\cmidrule(lr){2-3}\cmidrule(lr){4-5}
Class&Bound&Source&Known&Ours&Matching / results\\
\midrule
Smooth quadratic&$O(d\min\{d,S\})$&\citet{Bubeck2015}; App.~\ref{app:corollaries}.1&$\Omega(d^2)^\dagger$&$\Omega(d\min\{d,S\})$&All stated $\eps$; \ref{cor:quadratic}\\
Smooth convex&$O(dS)$&\citet{NesterovSpokoiny2017}&App.~\ref{app:wu}&$\Omega(d\min\{d,S\})$&$S\le d$; \ref{cor:general-moderate}\\
Strongly convex, $\kappa\ge1$&$O(d\sqrt\kappa\,V)$&\citet{NesterovSpokoiny2017}&---&$\Omega\!\left(d\min\{d,\sqrt\kappa,S\}\right)$&App.~\ref{match:sec-euclidean}; \ref{cor:strong}\\
Strongly convex, $\kappa\le2$&$O(dH_\kappa)$&This work, Thm.~\ref{thm:near-quadratic}&---&$\Omega(dH_\kappa)$&Uniform; \ref{thm:near-quadratic}\\
\midrule
$L_2$-smooth, $\ell_p$ localization&$\displaystyle\widetilde O\!\left(d^{1/2+1/q}\sqrt{\frac{L_2R_p^2}{\eps}}\right)$&\citet{GorbunovDvurechenskyGasnikov2022}&---&$\displaystyle\Omega\!\left(d^{1/2+1/q}\sqrt{\frac{L_2R_p^2}{\eps}}\right)$&Up to logs;\par\eqref{match:eq-main-smooth-window}, \ref{match:cor-smooth}\\
$M_2$-Lipschitz, $\ell_p$ feasible ball&$\displaystyle\widetilde O\!\left(d^{2/q}\frac{M_2^2R_p^2}{\eps^2}\right)$&\citet{Shamir2017}&$p=2$: [Z]&$\displaystyle\widetilde\Omega\!\left(d^{2/q}\frac{M_2^2R_p^2}{\eps^2}\right)$&Up to logs;\par\eqref{match:eq-main-nonsmooth-window}, \ref{match:cor-nonsmooth}\\
\midrule
Arbitrary norm, smooth&$\displaystyle O\!\left(d\mathcal D_2\sqrt{\frac L\eps}\right)$&\citet{GorbunovDvurechenskyGasnikov2022}; App.~\ref{match:sec-arbitrary}&---&\eqref{gp:eq-norm}&General gap; \ref{gp:thm-norm}\\
Arbitrary norm, nonsmooth&$\displaystyle O\!\left(\frac{dM^2\mathcal D_2^2}{\eps^2}\right)$&\citet{Shamir2017}; App.~\ref{match:sec-arbitrary}&Input [Z]&\eqref{gp:eq-nonsmooth}, $\lambda=0$&General gap; \ref{gp:thm-nonsmooth}\\
Known regularizer&---&---&---&\eqref{gp:eq-composite}, \eqref{gp:eq-nonsmooth}&Sublevels; \ref{gp:cor-spectral}\\
\bottomrule
\end{tabularx}
\vspace{4pt}
\begin{minipage}{\linewidth}
\scriptsize
$S=\sqrt{LR^2/\eps}$, $V=\log(LR^2/\eps)$, $\kappa=L/\mu$, and
$H_\kappa=1+V/\log((L+\mu)/(L-\mu))$, with $H_\kappa=1$ at $\mu=L$. The general strongly convex row uses $0<\eps\le LR^2/32$.
The $\ell_p$ rows use $1\le p\le2$, $q=p/(p-1)$ ($q=\infty$ at $p=1$),
\emph{Euclidean regularity}, and the indicated accuracy windows; the
nonsmooth row requires $d\ge m_Z^\circ$. Tildes hide logarithmic factors;
the smooth $p=2$ match is up to constants.
$^\dagger$\citet{Yao1975}: deterministic exact recovery ($\eps=0$).
[Z]: \citet{ZhangZhangQiLin2026}, the external Euclidean nonsmooth input.
The quadratic upper combines acceleration with coefficient recovery
(App.~\ref{app:corollaries}.1).
For a regularity unit ball $B$ and localization/feasible set $K$, put
$a_B=\max_{\|u\|_2=1}\|u\|$ and
$\mathcal D_2=a_B\min_{z\in K}\max_{x\in K}\|x-z\|_2$.
The general-norm rows apply the cited algorithms after norm conversion;
\cref{match:sec-arbitrary} gives the small-error regimes and proximal
alternatives. Dashes mark entries without a cited bound.
\end{minipage}
\endgroup
\end{table}

\section{Problem setting and the Euclidean lower bound}\label{sec:results}
Throughout, $\|\cdot\|_2$ denotes the Euclidean norm,
$\mathbb B_2^d(R)=\{x\in\R^d:\|x\|_2\le R\}$ the closed Euclidean ball, and
$\operatorname*{argmin}f$ the set of global minimizers of $f$. Let
$\mathcal C_{d,L,R}$ be the class of globally $L$-smooth convex functions with
a global minimizer in the ball of radius $R$, that is, of convex
differentiable functions $f:\R^d\to\R$ satisfying
\[
 \|\nabla f(x)-\nabla f(y)\|_2\le L\|x-y\|_2
 \quad\text{for all }x,y,\qquad
 \operatorname*{argmin} f\cap \mathbb B_2^d(R)\ne\varnothing .
\]
Regularity and localization are measured in the same Euclidean norm here;
\cref{geom:sec} allows the two norms to differ. The oracle
returns the exact value $f(x)$ at any queried point $x\in\R^d$. An algorithm
is equipped with an independent random seed on a standard Borel probability
space, and its query and output maps are total Borel functions of that seed
and of the finite history of real-valued answers. Queries may be adaptive,
their norms and precision are unrestricted, and internal computation is free.
A deterministic budget $T$ bounds the number of queries on every history;
a run that stops early is padded with queries at zero. An algorithm is
\emph{successful} on a class if
\[
 f(\widehat x)-\min f\le\eps\quad\text{with probability at least }2/3
\]
for each objective in the class; the output $\widehat x$ need not have been
queried and need not lie in $\mathbb B_2^d(R)$. All logarithms are natural,
and $\Cstar=4096$ throughout the paper.

\begin{theorem}[Dimension--accuracy lower bound]\label{thm:main}
For every integer $d\ge1$, $L,R>0$, and $0<\eps\le LR^2/32$, any algorithm
that is successful on $\mathcal C_{d,L,R}$ in the above sense must use
\begin{equation}\label{eq:main}
 T\ge\max\left\{\frac{5d}{16},\;
 \frac{d}{\Cstar}\min\left\{d,\sqrt{\frac{LR^2}{\eps}}\right\}\right\}.
\end{equation}
\end{theorem}

Positive-definite quadratics in the class already force this bound. The
guarantee is required for every objective, the algorithm may be randomized and
may use dense, unbounded, arbitrarily precise queries, and the bound contains
no logarithmic factor. \Cref{sec:mechanism} outlines the proof;
\cref{r12:sec-affine,r12:sec-rounding} give it in full. In terms of the
normalized accuracy $\delta=\eps/(LR^2)$, the product term has two branches:
$d/\sqrt\delta$ for $d^{-2}\le\delta\le1/32$, and $d^2$ for
$\delta\le d^{-2}$. The lower bound saturates at $d^2$ in the second branch. The next
corollary identifies the matching quadratic complexity.

\begin{corollary}[Quadratic minimax complexity]\label{cor:quadratic}
Let $\mathcal Q_{d,L,R}$ be the quadratic subclass of $\mathcal C_{d,L,R}$.
Its fixed-budget randomized exact-value minimax complexity satisfies
\begin{equation}\label{eq:quadratic}
 \Tval_\eps(\mathcal Q_{d,L,R})
 =\Theta\!\left(d\min\{d,\sqrt{LR^2/\eps}\}\right),
 \qquad 0<\eps\le LR^2/32.
\end{equation}
\end{corollary}
\begin{proof}
The lower bound is \cref{thm:main}. For the upper bound, central differences
recover the exact gradient of a quadratic from $2d$ values, so accelerated
gradient descent needs $O(d\sqrt{LR^2/\eps})$ values
\citep[Theorem~3.19]{Bubeck2015}; alternatively, $(d+1)(d+2)/2$ values identify
all coefficients, after which the minimum-norm minimizer is computed exactly.
The known parameters decide which of the two procedures is cheaper.
\Cref{app:corollaries} supplies the interpolation and measurability arguments,
including the case of singular Hessians.
\end{proof}

\begin{corollary}[Moderate accuracy on the full smooth class]\label{cor:general-moderate}
Whenever $LR^2/d^2\le\eps\le LR^2/32$,
\[
 \Tval_\eps(\mathcal C_{d,L,R})
 =\Theta\!\left(d\sqrt{LR^2/\eps}\right).
\]
\end{corollary}
\begin{proof}
In this range the minimum in \cref{thm:main} selects the accuracy term. Central
differences combined with inexact projected acceleration give the matching
upper bound on the full class; \cref{app:general-upper} specifies the finite
step size and the query budget.
\end{proof}

\begin{corollary}[Prescribed strong convexity]\label{cor:strong}
Let $0<\mu\le L$, $\kappa=L/\mu$, and $0<\eps\le LR^2/32$.
On the globally $\mu$-strongly convex subclass of $\mathcal C_{d,L,R}$,
\begin{equation}\label{eq:strong}
 T\ge\max\left\{\frac{5d}{16},\;
 \frac{d}{\Cstar}\min\left\{d,\sqrt\kappa,\sqrt{\frac{LR^2}{\eps}}\right\}\right\}.
\end{equation}
The bound again holds on the quadratic subclass. In particular, $\eps\le\mu R^2$
yields $\Omega(d\min\{d,\sqrt\kappa\})$ within the stated accuracy range.
\end{corollary}

The hard quadratic behind the product term satisfies
$LI/[50(k+1)^2]\preceq\nabla^2f\preceq LI$, where $k$ is the free parameter of the
construction; imposing $k\le\sqrt\kappa/8$ in addition to the existing
constraints therefore proves the corollary with the same $\Cstar$
(\cref{app:corollaries}). \Cref{match:sec-euclidean} compares this bound with accelerated upper bounds.
A separate multiscale family gives the logarithmic accuracy dependence
in the following theorem.

\begin{theorem}[Matching complexity near the quadratic endpoint]\label{thm:near-quadratic}
Let $\mathcal C_{d,L,R}^{\mu}$ denote the globally $\mu$-strongly convex
subclass of $\mathcal C_{d,L,R}$. If $\mu/L\in[1/2,1]$ and
$0<\eps\le LR^2/128^2$, then
\begin{equation}\label{eq:near-quadratic}
\frac{3d}{4000}\left[1+
 \frac{\log(LR^2/\eps)}{\log((L+\mu)/(L-\mu))}\right]
\le \Tval_\eps(\mathcal C_{d,L,R}^{\mu})
\le 2d\left[1+
 \frac{\log(LR^2/\eps)}{\log((L+\mu)/(L-\mu))}\right].
\end{equation}
The fractions are interpreted as zero when $\mu=L$.
\end{theorem}
\begin{proof}[Proof outline]
The nested spherical construction gives $T\ge3dm/160$, where $m$ is its
admissible depth (\cref{ms:thm-depth}). A finite-difference implementation of
the midpoint-curvature contraction gives the matching upper bound with a
budget fixed in advance (\cref{ms:thm-upper}). The complete calibration,
including $\mu=L$, is in \cref{ms:proof-matching}.
\end{proof}
For every fixed $0\le\mu/L<1$, the same lower construction gives
$\Omega_{\mu/L}(d\log(LR^2/\eps))$ at sufficiently small relative accuracy.
Thus the full-class complexity grows with accuracy even in fixed
dimension. The estimate above is uniform near $\mu=L$; the joint rate
for growing condition numbers remains open.

\section{Non-Euclidean smooth and nonsmooth lower bounds}\label{geom:sec}
Let $B$ be the unit ball of a norm $\|\cdot\|$ with dual norm $\|\cdot\|_*$,
and let $K\subset\R^d$ be nonempty, compact, and convex. The norm controlling
regularity and the set localizing a solution are independent. For a known
continuous convex penalty $\psi:K\to\R$, normalize $\min_K\psi=0$ and write
$F=g+\lambda\psi$, $\lambda\ge0$. Values and proximal computations involving
$\psi$ are free. We charge only values of the unknown $g$ and use
$\mu$ for a strong-convexity modulus.

\paragraph{A common geometric radius.}
For a nonempty compact convex set $C\subset\R^d$ and $1\le m\le d$, let
\begin{equation}\label{gp:eq-radius}
 \mathcal R_m(B,C)=\max_{\|Q\|_{B\to2}\le1}\;
 \max\{r\ge0:\ z+r\mathbb B_2^m\subseteq Q(C)\text{ for some }z\in\R^m\},
\end{equation}
where $Q$ ranges over linear maps $\R^d\to\R^m$ and
$\|Q\|_{B\to2}=\max_{x\in B}\|Qx\|_2$.
The maximum is attained, and $z=0$ is optimal when
$C=-C$ (\cref{gp:lem-attainment}). Write
$K_\tau=\{x\in K:\psi(x)\le\tau\}$, $r_m=\mathcal R_m(B,K)$,
$r_m(\tau)=\mathcal R_m(B,K_\tau)$, and
\begin{equation}\label{g:eq-profile}
 \mathcal P(m,\sigma)=\frac{m\sigma}{\log(e\sigma)}\qquad(m\ge1,\ \sigma\ge1).
\end{equation}
Each target value is simulated by at most one source value.

In the following three theorems, the lower bounds apply to algorithms
required to satisfy the success guarantee of \cref{thm:main} uniformly over
every objective in the stated class.

\begin{theorem}[Smooth transfer]\label{gp:thm-norm}
Let $f$ be globally convex and differentiable with
$\|\nabla f(x)-\nabla f(y)\|_*\le L\|x-y\|$ and
$\operatorname*{argmin}f\cap K\ne\varnothing$, under the algorithm convention of
\cref{thm:main}, including unrestricted queries and output. Then, for every
$1\le m\le d$ with $r_m>0$ and every $0<\eps\le Lr_m^2/32$,
\begin{equation}\label{gp:eq-norm}
 T\ge\max\left\{\frac{5m}{16},\;
 \frac m{\Cstar}\min\left\{m,\sqrt{\frac{Lr_m^2}{\eps}}\right\}\right\}.
\end{equation}
\end{theorem}

\begin{theorem}[Smooth transfer with a known regularizer]\label{gp:thm-composite}
Let $F=g+\lambda\psi$ be minimized over $K$, where $\psi$ is as above, the output
must lie in $K$, the unknown $g$ is globally convex with
$\|\nabla g(x)-\nabla g(y)\|_*\le L\|x-y\|$, and queries access only $g$ and are
unrestricted. Then, for every $1\le m\le d$ and $\tau\ge0$ with $r_m(\tau)>0$ and
$\eps+\lambda\tau\le Lr_m(\tau)^2/32$,
\begin{equation}\label{gp:eq-composite}
 T\ge\max\left\{\frac{5m}{16},\;
 \frac m{\Cstar}\min\left\{m,
 \sqrt{\frac{Lr_m(\tau)^2}{\eps+\lambda\tau}}\right\}\right\}.
\end{equation}
\end{theorem}

\Needspace{8\baselineskip}
\begin{theorem}[Nonsmooth transfer with a known regularizer]\label{gp:thm-nonsmooth}
In the same composite model with $|g(x)-g(y)|\le M\|x-y\|$ in place of
smoothness, there are universal constants $c_Z,\eps_Z>0$ and
$m_Z^\circ\in\N$ such that, for every integer $m_Z^\circ\le m\le d$ and every
$\tau\ge0$ with $r_m(\tau)>0$ and $\eps+\lambda\tau\le\eps_ZMr_m(\tau)$,
\begin{equation}\label{gp:eq-nonsmooth}
 T\ge c_Z\,\mathcal P\!\left(m,\;
 \min\left\{m,\left(\frac{Mr_m(\tau)}{\eps+\lambda\tau}\right)^{2}\right\}\right).
\end{equation}
\end{theorem}

Take the maximum over admissible image dimensions and sublevels, each within
its own window. Setting $\lambda=0$ in \cref{gp:thm-nonsmooth} gives the
arbitrary-norm nonsmooth lower bound. The smooth proofs use only
\cref{thm:main}; the nonsmooth proof uses
\citet[Theorem~5, version~2]{ZhangZhangQiLin2026}. The fixed pullback preserves
regularity, and comparison with a lift in $K_\tau$ gives source accuracy
$\eps+\lambda\tau$ (\cref{gp:app,gp:app-nonsmooth}).

\paragraph{Explicit norm dependence.}
For $B=\mathbb B_\beta^d$ and $K=\mathbb B_\alpha^d(R_\alpha)$, put
$\theta=(1/2-1/\beta)_++(1/\alpha-1/2)_+$ and
$\gamma=(1/\beta-1/\alpha)_+$. Then
\begin{equation}\label{geom:eq-certificate}
 \mathcal R_m(B,K)\ge
 R_\alpha\lfloor d/m\rfloor^\gamma m^{-\theta},\qquad
 \mathcal R_d(B,K)=R_\alpha d^{-\theta}.
\end{equation}
The full-dimensional image radius is exact for every pair of norms.
For $\beta=2$ and $\alpha=p\in[1,2]$, set $q=p/(p-1)$, with $q=\infty$
at $p=1$. \Cref{sec:matching} gives the resulting smooth and nonsmooth
minimax rates, their accuracy windows, and the distinction between
Euclidean regularity and non-Euclidean localization.
\Cref{g:sec} also gives complementary
constant-radius images of dimension
$\mathfrak m_{p,d}\asymp\min\{d,\min\{q,\log(ed)\}d^{2/q}\}$.

\paragraph{Geometry beyond norm equivalence.}
Matching ellipsoidal regularity and localization have
$\mathcal R_m(B,RB)=R$ without a condition-number penalty
(\cref{gp:prop-ellipsoid}). Block and Schatten examples,
strong-convexity compatibility for every $p$, and a fixed-reference
relative-smoothness extension appear in
\cref{gp:app,app:additional-geometry,app:relative-smoothness}.

\subsection{Spectral dependence of known regularization}\label{geom:sec-spectral}
A scalar upper bound on the curvature of a known regularizer can lose
information relevant to the oracle lower bound. The following specialization
of \cref{gp:thm-composite} keeps the whole spectrum. The oracle still reveals
only exact values of the unknown smooth part $g$; the penalty and its
proximal operations are available for free, and the output must be in $K$.

\begin{corollary}[Spectral profile of a quadratic regularizer]\label{gp:cor-spectral}
Let $B=\mathbb B_2^d$, $K=R\mathbb B_2^d$, and $\psi(x)=x^\top Hx/2$ with $H\succ0$
having eigenvalues $0<\nu_1\le\cdots\le\nu_d$. Then
$r_m(\tau)=\min\{R,\sqrt{2\tau/\nu_m}\}$, and \eqref{gp:eq-composite} at the
maximizing level $\tau=\nu_mR^2/2$ becomes
\begin{equation}\label{gp:eq-spectral}
 T\ge\max\left\{\frac{5m}{16},\;
 \frac m{\Cstar}\min\left\{m,
 \sqrt{\frac{LR^2}{\eps+\lambda\nu_mR^2/2}}\right\}\right\},
 \qquad \eps+\frac{\lambda\nu_mR^2}2\le\frac{LR^2}{32}.
\end{equation}
\end{corollary}
Take the maximum over all $1\le m\le d$ satisfying the displayed window.
Projection onto the $m$ least-curved eigendirections attains the radius
$r_m(\tau)$. The two containing sets, the radius-$R$ ball and the ellipsoid
$\{x:x^\top Hx\le2\tau\}$, give the reverse inequality. Moreover,
$Lr_m(\tau)^2/(\eps+\lambda\tau)$ increases up to
$\tau=\nu_mR^2/2$ and does not increase afterwards. Thus the same level
optimizes both the lower-bound expression and its validity window.
The full proof, including invariance under normalized affine tilts of
$\psi$, is in \cref{gp:app}.

\paragraph{Why the spectrum matters.}
For $H=\diag(I_n,AI_n)$, the choice $m=n$ retains $n$ directions at
regularization cost $\lambda R^2/2$, independently of the larger curvature
$A$. Replacing $H$ by the scalar majorant $AI$ charges the larger cost in
every direction. This distinction persists after optimizing both the image dimension
and the sublevel.
For the explicit family of \cref{gp:prop-gain}, let $A>7$, $\lambda>0$, and
\[
 \begin{gathered}
 D=2560,\quad b=D\sqrt{(A+1)/2},\quad n=\lceil b\rceil,\quad d=2n,\\
 R=1,\quad H=\diag(I_n,AI_n),\quad \eps=\lambda/2,\quad L=\lambda b^2.
 \end{gathered}
\]
Write $\mathcal L_{\mathrm{profile}}$ for the best actual-sublevel bound
and $\mathcal L_{\mathrm{majorant}}$ for the best bound using the balls
certified by $A\|x\|_2^2/2$, in both cases optimized over all admissible
dimensions and levels. \Cref{gp:prop-gain} gives
\[
 \mathcal L_{\mathrm{profile}}=\frac{nb}{4096},\qquad
 \mathcal L_{\mathrm{majorant}}=\frac{2nD}{4096},\qquad
 \frac{\mathcal L_{\mathrm{profile}}}{\mathcal L_{\mathrm{majorant}}}
 =\frac{\sqrt{A+1}}{2\sqrt2}.
\]
The comparison includes the dimension baseline, the saturation cap, and
both source windows. Hence the loss caused by the scalar majorant is
unbounded, even after its own dimension and sublevel are optimized.

\paragraph{The nonsmooth sublevel is different.}
For the same quadratic penalty and a globally Euclidean $M$-Lipschitz
unknown part $g$, \cref{gp:thm-nonsmooth} instead optimizes
$M^2r_m(\tau)^2/(\eps+\lambda\tau)^2$. For $\lambda>0$, its maximizing
level and the localized accuracy parameter are
\[
 \tau^*_{m,\mathrm{nonsmooth}}
 =\min\{\eps/\lambda,\nu_mR^2/2\},\qquad
 \sigma=\min\!\left\{m,\frac{M^2}{2\lambda\nu_m\eps}\right\}
 \quad\text{if }\eps/\lambda\le\nu_mR^2/2.
\]
The resulting bound is $T\ge c_Z\mathcal P(m,\sigma)$, provided
$m\ge m_Z^\circ$ and the source window in \cref{gp:thm-nonsmooth} holds.
\Cref{gp:prop-ns-gain} gives an improvement of at least $A/4$ over the
optimized scalar-majorant bound on an admissible two-level family.
The gains quantify the improvement in the lower-bound expression for
each fixed known-composite class. The smooth family varies $d$ and
$L/\lambda$ with $A$; the nonsmooth family varies $d$ and $M$.
Sharp complexity for general regularizers remains open.

\section{Why adaptive exact queries still have a cost}\label{sec:mechanism}
For the quadratic construction, one function value reveals one rank-one
linear measurement of a Wishart matrix $W$. Conditional on the seed and the
complete adaptive transcript, its law is the original density restricted to
the realized affine slice. This calculation applies to total Borel query
policies. The
sharp comparison in \cref{thm:affine} gives
$\E W^{-1}\preceq(1+(d+1)/(2\zeta))(\E W)^{-1}$ on every such slice.
Combined with an exact trace identity, it bounds the expected operator norm
of the posterior inverse in terms of the number of observations
(\cref{eq:posterior-main,eq:budget-main}). Successful optimization would
produce a sufficiently accurate estimate of $W^{-1}u$; the posterior budget
then forces the dimension--accuracy product.

The normalization is supported on valid positive-definite quadratics with
probability at least $49/50$. Three independent runs, each followed by one
validation query, give the correlation needed for the lower bound. All
$3(T+1)$ queries are counted. A separate spherical family gives the baseline
$5d/16$ and covers the small-parameter regime. The analytic comparison,
conditioning argument, and calibration are detailed in
\cref{app:mechanism-details,app:euclidean}.

The logarithmic result uses different instances: nested smooth perturbations
with independent spherical directions. Outside a selected child ball, deeper
perturbations vanish exactly. Entering a child requires recovering its
hidden direction, and a single query may enter several children. The proof in \cref{app:multiscale} accounts for these cascades to
obtain the dimension-times-depth lower bound.

\section{Saddle-point problems with unequal block dimensions}\label{sp:sec}
Consider the convex--concave saddle-point problem
$\min_{x\in X}\max_{y\in Y}F(x,y)$ with block dimensions $d_x,d_y$,
smoothness constants $L_x,L_y$, coupling constant $L_{xy}$, and strong
convexity and strong concavity moduli $\mu_x,\mu_y\ge0$. In the coupled
constructions below we write $C=L_{xy}$. A global saddle point is localized
by radii $R_x,R_y$; queries remain unrestricted. We distinguish the gap
\begin{equation}\label{sp:eq-main-gap}
 \operatorname{Gap}_{X,Y}(x,y)
 =\sup_{v\in Y}F(x,v)-\inf_{u\in X}F(u,y)
\end{equation}
on the compact balls $X=R_x\mathbb B_2^{d_x}$, $Y=R_y\mathbb B_2^{d_y}$, with
a feasible output, from the global gap, whose best responses range over the
full spaces. We state the reductions separately for these two error criteria. The success and budget
conventions are those of \cref{sec:results}; \cref{sp:app} gives the full model.

\begin{table}[!t]
\centering
\caption{Upper and lower bounds for convex--concave saddle-point problems.
See \cref{sp:app-comparison} for the oracle models and accuracy ranges.}
\label{sp:tab-main}
\begingroup
\small
\setlength{\tabcolsep}{3pt}
\renewcommand{\arraystretch}{1.13}
\begin{tabularx}{\linewidth}{@{}>{\raggedright\arraybackslash}p{.255\linewidth}
 >{\centering\arraybackslash}p{.335\linewidth}
 >{\raggedright\arraybackslash}X@{}}
\toprule
Setup and gap & Lower bound on $T$ & Upper bound / source \\
\midrule
Smooth block class; compact gap &
$\displaystyle\max_i\mathcal B_i(\eps)$ &
$O((d_x+d_y)LD^2/\eps)$\par
\citet{SadievEtAl2021} \\
Bilinear-affine; compact gap &
$\displaystyle\Omega\!\left(r\min\left\{r,\sqrt{\frac{CR_xR_y}{\eps}}\right\}\right)$ &
$d_y(d_x+1)$\par
\Cref{sp:prop-bilinear-upper} \\
Known diagonal Hessians; global gap &
$\Omega(r\min\{r,\sqrt\chi\})$ &
$1+d_x+d_y+d_xd_y$\par
\Cref{sp:prop-known-diagonal-upper} \\
\bottomrule
\end{tabularx}
\smallskip
\emph{Mixed access: scalar values and arbitrary additional $\nabla_xF$ queries.}\par
\smallskip
\begin{tabularx}{\linewidth}{@{}>{\raggedright\arraybackslash}p{.275\linewidth}
 >{\centering\arraybackslash}p{.465\linewidth}
 >{\raggedright\arraybackslash}X@{}}
\toprule
Setup and error & Lower bound on $N_0$ & Result \\
\midrule
Blockwise; compact or global pair gap &
$\mathcal B_y(\eps)$ &
\Cref{sp:thm-mixed} \\
Known coupling; outer objective error &
$\displaystyle\Omega\!\left(r\min\left\{r,\sqrt{\kappa_0},\sqrt{\frac{s\rho^2}{\eps}}\right\}\right)$ &
\Cref{sp:thm-mixed-coupled} \\
Known rectangular coupling; global pair gap &
$\displaystyle\Omega\!\left(d_y\min\left\{d_y,\sqrt{\frac{L_y}{\mu_y+s}},\sqrt{\frac{L_y\rho^2}{\eps}}\right\}\right)$ &
\Cref{sp:thm-mixed-full} \\
\bottomrule
\end{tabularx}
\smallskip
\begin{minipage}{\linewidth}
\scriptsize
$r=\min\{d_x,d_y\}$ and $\chi=C/\sqrt{\mu_x\mu_y}$. The block maximum ranges over admissible indices in \eqref{sp:eq-main-blocks}.
For the first upper bound, $L$ bounds the full saddle operator and $D$ is
its Euclidean diameter parameter. The next two upper bounds count coefficient-recovery
queries. When $d_x=d_y=d$, they match the respective lower bounds at
$\Theta(d^2)$ for $\eps\le CR_xR_y/(32d^2)$ and $\chi\ge4d^2$,
within \eqref{sp:eq-main-bilinear} and \eqref{sp:eq-main-strong}.
In the mixed rows, $s=c^2/\alpha$, $\rho=\min\{R_y,\alpha R_x/c\}$,
and $\kappa_0=(L_y+s)/(\mu_y+s)$; the last row requires $\mu_y+s\le L_y$.
Lower bounds hold with success probability $2/3$ in the cited windows.
\end{minipage}
\endgroup
\end{table}

\paragraph{Block costs and mixed information.}
Let $\mathcal B_i(\eps)$ be the minimization bound \eqref{sp:eq-block-budget}
for $(d_i,L_i,\mu_i,R_i)$, defined when $\eps\le L_iR_i^2/32$.
Embedding an unknown objective in one block and a known quadratic in
the other gives, in either gap model,
\begin{equation}\label{sp:eq-main-blocks}
 T\ge\max_{i:\,\eps\le L_iR_i^2/32}\mathcal B_i(\eps),
\end{equation}
where an empty maximum is zero. With arbitrary additional $\nabla_xF$
queries, pair-gap accuracy requires $N_0\ge\mathcal B_y(\eps)$ on its
window: for $F(x,y)=\mu_x\|x\|_2^2/2-g(y)$, these gradients are independent
of $g$ (\cref{sp:thm-blocks,sp:thm-mixed}). Every returned scalar value,
including one bundled with a gradient, counts toward $N_0$.

\paragraph{The cost of coupling.}
Write $r=\min\{d_x,d_y\}$. The first construction has vanishing diagonal
Hessians.
\begin{theorem}[Compact bilinear games]\label{sp:main-bilinear}
For $F(x,y)=y^\top Ax-b^\top y$ on the balls of radii $R_x,R_y$, with
$\|A\|_{\mathrm{op}}\le C$, every exact-value algorithm satisfying the
compact-gap guarantee requires
\begin{equation}\label{sp:eq-main-bilinear}
 T\ge\max\left\{\frac{5r}{32},\frac r{8192}
 \min\left\{r,\sqrt{\frac{CR_xR_y}{\eps}}\right\}\right\},
 \qquad 0<\eps\le CR_xR_y/32.
\end{equation}
Every hard instance has a global saddle point in the prescribed balls.
\end{theorem}
Polarization simulates one target value with two source values and
transfers compact-gap accuracy to source suboptimality
(\cref{thm:bilinear-compact}). \Cref{sp:tab-main} compares the bound with
coefficient recovery and gives the matched $d^2$ regime.

\begin{theorem}[Known isotropic diagonal Hessians]\label{sp:main-strong}
Let $\mu_x,\mu_y>0$, $\chi=C/\sqrt{\mu_x\mu_y}\ge4$, and
$\rho=\min\{R_x,4\sqrt{\mu_y/\mu_x}R_y\}$.
For quadratics with known Hessians $\nabla^2_{xx}F=\mu_xI$,
$\nabla^2_{yy}F=-\mu_yI$, coupling norm at most $C$, and the prescribed saddle
localization, the global-gap guarantee requires
\begin{equation}\label{sp:eq-main-strong}
 T\ge\max\left\{\frac{5r}{32},\frac r{8192}
              \min\{r,\tfrac12\sqrt\chi\}\right\},
 \qquad 0<\eps\le\mu_x\rho^2/8.
\end{equation}
For the compact gap, use $\rho=\min\{R_x,2\sqrt{\mu_y/\mu_x}R_y\}$ and
additionally require $\eps\le\mu_yR_y^2/8$. The statement with the blocks
interchanged holds with $R_x,\mu_x$ and $R_y,\mu_y$ exchanged throughout.
\end{theorem}
The unknown coupling moves the saddle point even when both within-block
condition numbers equal one. A source curvature floor gives the
$\tfrac12\sqrt\chi$ cap
(\cref{thm:known-diagonal-coupling,cor:known-diagonal-compact}).
Coefficient recovery matches the $d^2$ plateau in \cref{sp:tab-main}.

\paragraph{Outer accuracy under mixed access.}
Partial gradients expose the unknown coupling in the two constructions above.
For mixed access we therefore fix a known coupling of rank $r$ and vary an
inner quadratic instead. Put $\alpha=L_x=\mu_x>0$, $c>0$, $s=c^2/\alpha$,
$\kappa_0=(L_y+s)/(\mu_y+s)$, and $\rho=\min\{R_y,\alpha R_x/c\}$, with
$L_y\ge\mu_y>0$. Even with arbitrarily many $\nabla_xF$ calls, an output with
$\Phi(\widehat x)-\min\Phi\le\eps$, where $\Phi(x)=\max_{y\in\R^{d_y}}F(x,y)$,
requires
\[
 N_0=\Omega\!\left(r\min\left\{r,\sqrt{\kappa_0},
                      \sqrt{s\rho^2/\eps}\right\}\right),
 \qquad 0<\eps\le s\rho^2/32.
\]
\Cref{sp:thm-mixed-coupled} proves this statement and its pair-gap version.
A second pair-gap construction retains the full inner dimension $d_y$ when
$\mu_y+s\le L_y$, including when $d_x\ll d_y$ (\cref{sp:thm-mixed-full}).
These bounds isolate the scalar cost of recovering the unknown
inner solution. \Cref{sp:app-comparison} compares the information models
and discusses the remaining gaps.

\section{Matching non-Euclidean gradient-free rates}\label{sec:matching}
The geometric lower bounds match known gradient-free upper bounds in two
classes. Fix $1\le p\le2$ and $q=p/(p-1)$, with $q=\infty$ at $p=1$.
In each statement, $\Tval_\eps$ denotes the minimax number of exact scalar
values for the stated class, under a deterministic worst-case budget and
success probability at least $2/3$. The notation $\widetilde\Theta$ means
matching up to logarithmic factors. The regularity constants are
\emph{Euclidean}: $L_2$ and $M_2$.

\begin{corollary}[Smooth functions with non-Euclidean localization]\label{match:cor-smooth}
Consider globally convex $L_2$-smooth functions on $\R^d$ with a global
minimizer in $\mathbb B_p^d(R_p)$. Queries and output are unrestricted.
For every $\eps$ satisfying
\begin{equation}\label{match:eq-main-smooth-window}
 \frac{L_2R_p^2}{d^{1+2/p}}\le\eps
 \le\frac{L_2R_p^2}{32d^{2/p-1}}
\end{equation}
the minimax complexity satisfies
\begin{equation}\label{match:eq-main-smooth-rate}
 \Tval_\eps=\widetilde\Theta\!\left(
 d^{1/2+1/q}\sqrt{\frac{L_2R_p^2}{\eps}}\right).
\end{equation}
The lower bound has no logarithmic loss. At $p=2$, the bounds match up to absolute constants.
\end{corollary}
\begin{proof}
Use $m=d$, $\beta=2$, and $\alpha=p$ in
\cref{gp:thm-norm,geom:eq-certificate}. The left endpoint in
\eqref{match:eq-main-smooth-window} keeps the lower bound below its $d^2$
cap; the right endpoint is its validity condition. The accelerated
gradient-free upper bound of
\citet{GorbunovDvurechenskyGasnikov2022}, with both noise terms
zero, has the same polynomial order. \Cref{app:matching} gives the prox
normalization and finite-difference choice. At $p=2$, apply
\cref{cor:general-moderate}.
\end{proof}

\begin{corollary}[Nonsmooth functions on a non-Euclidean feasible ball]\label{match:cor-nonsmooth}
Consider globally convex, Euclidean $M_2$-Lipschitz functions on $\R^d$,
minimized over $\mathbb B_p^d(R_p)$. Queries are unrestricted, but the output
must be feasible. For $d\ge m_Z^\circ$ and every $\eps$ satisfying
\begin{equation}\label{match:eq-main-nonsmooth-window}
 \frac{M_2R_p}{d^{1/p}}\le\eps
 \le\frac{\eps_ZM_2R_p}{d^{1/p-1/2}},
\end{equation}
the minimax complexity satisfies
\begin{equation}\label{match:eq-main-nonsmooth-rate}
 \Tval_\eps=\widetilde\Theta\!\left(
 d^{2/q}\frac{M_2^2R_p^2}{\eps^2}\right).
\end{equation}
Here $m_Z^\circ$ and $\eps_Z$ are the constants of
\cref{gp:thm-nonsmooth}.
\end{corollary}
\begin{proof}
Set $\lambda=0$ in \cref{gp:thm-nonsmooth} and use the full-dimensional
image in \eqref{geom:eq-certificate}. In
\eqref{match:eq-main-nonsmooth-window}, the accuracy term is unsaturated.
The lower bound retains the logarithmic denominator in $\mathcal P$ and
uses the Euclidean information theorem of \citet{ZhangZhangQiLin2026}.
The two-point mirror-descent upper bound of
\citet[Theorem~1]{Shamir2017} has the same polynomial order;
\cref{app:matching} gives the norm-matched prox and query-domain details.
\end{proof}

\paragraph{Dimension dependence and accuracy regimes.}
At $p=1$, the smooth rate has a $\sqrt d$ factor, whereas the nonsmooth
rate has no polynomial dimension factor. At $p=2$, both have a $d$ factor.
These statements concern $\ell_p$ localization or feasibility under
Euclidean regularity; replacing $L_2,M_2$ by $L_p,M_p$ changes the class.
At the lower endpoints of \eqref{match:eq-main-smooth-window} and
\eqref{match:eq-main-nonsmooth-window}, the respective lower bounds
reach their dimension-dependent caps. Each two-point observation costs two scalar values.
Running an expected-error upper guarantee to $\eps/3$ gives success at
least $2/3$ by Markov's inequality with a budget fixed in advance.

\subsection{Available upper bounds for arbitrary norms}\label{match:sec-main-arbitrary}
Known gradient-free algorithms yield upper bounds for every
finite-dimensional regularity norm. For its unit ball $B$ and a known
nonempty compact convex set $K$, write
\[
 a_B=\max_{\|u\|_2=1}\|u\|,\qquad
 R_2(K)=\min_{z\in K}\max_{x\in K}\|x-z\|_2,\qquad
 \mathcal D_2=a_BR_2(K).
\]
For a nonsingleton $K$, valid Euclidean regularity constants are
$L_2=a_B^2L$ and $M_2=a_BM$. Applying smooth acceleration
\citep{GorbunovDvurechenskyGasnikov2022} and nonsmooth two-point mirror
descent \citep{Shamir2017}, respectively, gives
\[
 \begin{aligned}
 &O\!\left(d\mathcal D_2\sqrt{L/\eps}\right)
 &&(0<\eps\le L\mathcal D_2^2),\\
 &O\!\left(dM^2\mathcal D_2^2/\eps^2\right)
 &&(0<\eps\le M\mathcal D_2).
 \end{aligned}
\]
The smooth class has a global minimizer in $K$ and permits unrestricted
output. The nonsmooth class is minimized over $K$ and requires a feasible
output; the average of the mirror-descent iterates has this property.
Both objectives are globally defined, so the finite-difference queries
outside $K$ are admissible. A singleton $K$ needs no oracle calls.

\paragraph{Proximal geometry.}
Sharper bounds retain the proximal range and directional moments.
\citet[Theorem~1]{Shamir2017} allows an arbitrary proximal norm, while
\citet[Theorems~2.2 and~2.9, preprint version]{GorbunovDvurechenskyGasnikov2022}
give accelerated and non-accelerated smooth bounds with an $\ell_p$
proximal setup. \Cref{match:sec-arbitrary} states these guarantees,
including their Bregman radii and finite-difference parameters. The
objective regularity constants are $M_2$ and $L_2$, respectively.
Two-point observations cost two scalar evaluations; running to expected
error $\eps/3$ yields the same fixed-budget success probability $2/3$.

\section{Conclusion and open questions}\label{sec:discussion}\label{sec:scope}
We characterize exact-value complexity on the quadratic class throughout
the stated accuracy window, on the full smooth class at moderate accuracy,
and uniformly near $\mu=L$. Our non-Euclidean lower bounds match the
known smooth and nonsmooth rates in
\cref{match:cor-smooth,match:cor-nonsmooth}. The common image profile
separates regularity from localization, and actual regularizer sublevels
retain the low-curvature directions that determine the spectral gains.

The polynomial and multiscale constructions give complementary lower
bounds whose maximum applies to the full class. Determining the sharp
joint dependence on accuracy and a growing condition number remains open.
Further questions concern the complete image profile for arbitrary norms,
sharp complexity for known regularizers, and the coupling and dimension
dependence of saddle-point complexity outside the matched plateaus.
A joint characterization of scalar-value and partial-gradient budgets
would extend the mixed-oracle results. \Cref{app:open} formulates these
questions for the corresponding objective classes and gap criteria.

\label{main-text-end}

\bibliography{references}
\bibliographystyle{plainnat}
\clearpage
\appendix
\addtocontents{toc}{\protect\setcounter{tocdepth}{2}}
\renewcommand{\contentsname}{Appendix contents}
\tableofcontents
\section*{Guide to the appendices}
\Cref{app:mechanism-details} gives the detailed proof overview. The proof of the main minimization theorem occupies \cref{app:euclidean}: the
inverse comparison is followed by the adaptive posterior calculation, its
operator-norm budget, the radial moments, and the final calibration.
\Cref{app:corollaries} gives the upper bounds and the extension to prescribed
strong convexity. \Cref{app:high-accuracy} establishes logarithmic accuracy dependence
for the full smooth class.
\Cref{app:multiscale} strengthens that observation to a dimension-times-depth
bound and proves the matching near-endpoint theorem. \Cref{app:matching}
gives the upper--lower comparisons establishing the matching
non-Euclidean rates in \cref{sec:matching}.

\Cref{sp:app} treats blockwise and mixed oracles, and
\cref{app:coupled-saddles} isolates the contribution of the coupling through
\cref{thm:bilinear-compact,thm:known-diagonal-coupling}. \Cref{sp:app-comparison} compares saddle-point bounds according to their
gap criteria and parameter ranges.
\Cref{g:sec,gp:app,g:app,g:app-lp,gp:app-nonsmooth} contain the geometric
statements and their proofs. The remaining appendices collect the literature
comparison, the open questions, the numerical checks and protocols in
\cref{app:experiments}. \Cref{app:additional-geometry,app:relative-smoothness}
include the all-$p$ compatibility and fixed-reference extensions. The final
appendix, \cref{app:wu}, compares our results with
\citet{WuEtAl2026DeterministicExactValue}.

\paragraph{Notation.}
In the geometry sections, $\mu$ denotes a strong-convexity modulus,
$\lambda$ the coefficient of the known penalty, and $\nu_i$ the eigenvalues
of its Hessian. In the two-norm formulas, regularity is measured in
$\ell_\beta$ and localization in $\ell_\alpha$. $\mathcal R_m(B,K)$ is the
optimized image radius and $r_m(\tau)$ its specialization to the sublevel
sets of a known penalty. In the posterior proof, $M=\E W$ and $Q=\E W^{-1}$
denote conditional matrix moments. In the geometry sections, $Q$ is a
linear map between normed spaces. Saddle-point gaps are defined with
their compact or global optimization domains. We use $x^\star$ for
optimal points, $Q^*$ for adjoints, and $F^*$ for convex conjugates.

\section{Detailed information argument}\label{app:mechanism-details}
The proof tracks the conditional distribution of an unknown matrix. For
symmetric matrices, $A\preceq B$ means that $B-A$ is positive semidefinite,
and $\langle A,B\rangle=\tr(AB)$ is the Frobenius inner product. The hard
quadratic family is parameterized by a Wishart matrix $W$, and each exact
value encodes one rank-one linear measurement $\langle xx^\top,W\rangle$.
After any history, the conditional density is the Wishart density restricted
to the affine fiber determined by the observations. An inverse-moment
comparison and a trace identity then bound the posterior operator norm using
first moments alone, and an algorithm that optimizes successfully must
correlate its output with $W^{-1}u$ strongly enough to violate that bound
unless it has spent enough queries. The analytic statement holds uniformly
over \emph{every} affine fiber, which is why the query policy need not
satisfy a linear-span restriction or keep its queries in a bounded set.

\subsection{An inverse comparison on every affine slice}
\begin{theorem}[Affine-slice inverse comparison]\label{thm:affine}
Let $\zeta>0$, let $\mathcal A=W_0+V\subseteq\Sym_d$ be affine, and suppose
$\Omega=\mathcal A\cap\{W\succ0\}$ is nonempty. With respect to affine Lebesgue
measure, equip $\Omega$ with the density
$\rho(W)=Z_0^{-1}\det(W)^\zeta\exp(-\langle C,W\rangle)$ with $C$ symmetric.
If the normalizing constant and the means $M=\E W$, $Q=\E W^{-1}$ are finite, then
\begin{equation}\label{eq:affine}
 Q\preceq\left(1+\frac{d+1}{2\zeta}\right)M^{-1}.
\end{equation}
The coefficient is sharp over this class of densities and slices.
\end{theorem}

The argument is a matrix-weighted variance inequality applied to the weight
$G(W)=\rho(W)W$ on a bounded smooth convex domain inside the slice. For a
Frobenius-orthonormal tangent basis $B_j$, differentiating $G^{-1}\partial_jG$
and whitening produce the curvature energy
$\rho(W)(\zeta\sum_{j,k}\Gamma_{jk}z_j^\top z_k+\|\sum_jH_jz_j\|_2^2)$ with
$H_j=W^{-1/2}B_jW^{-1/2}$ and $\Gamma_{jk}=\tr(H_jH_k)$; both terms are
nonnegative. Orthonormalizing the whitened tangent space gives symmetric $E_i$
and $A\xi=\sum_iE_i\xi_i$ with $\mathcal K=AA^*=\sum_iE_i^2\preceq cI$,
$c=(d+1)/2$, by completion to an orthonormal basis of all symmetric matrices.
Testing with $f(W)=W^{-1}a$ turns the left-hand side into
$a^\top(Q_D-M_D^{-1})a$, while $\mathcal K(\zeta I+\mathcal K)^{-1}\preceq
c(\zeta+c)^{-1}I$ bounds the polar energy; rearranging and exhausting
$\Omega$ gives \eqref{eq:affine}. The complete proof derives the
convex-boundary term through a Neumann problem and exhausts $\Omega$ by
smooth convex sublevel sets; sharpness is attained by the full-cone Wishart
law with $2\zeta+d+1$ degrees of freedom (\cref{r12:sec-affine}).

\subsection{The adaptive posterior and its operator-norm bound}
Let $X\in\R^{n\times d}$ have independent standard Gaussian entries, where
$n=d+m$ for an integer $m>3$, put $g=m-1$, and let $W=X^\top X$. For a history
$\mathcal H_q$ of at most $q$ scalar quadratic-form observations, including the
independent seed, set $Z=gW^{-1}$, $M=\E[W\mid\mathcal H_q]$,
$J=\E[Z\mid\mathcal H_q]$, and let $r$ be the dimension of
$\operatorname{span}\{x_ix_i^\top\}_{i\le q}$. Each nonredundant query is a
linear functional fixed by the preceding history, and its coarea factor is
constant along the next fiber, so the conditional law is again the Wishart
density restricted to the realized affine fiber. \Cref{thm:affine} and a
centered divergence identity therefore give
\begin{equation}\label{eq:posterior-main}
 J\preceq nM^{-1},\qquad \tr(MJ)=nd-2r,\qquad r\le q.
\end{equation}
Put $A=M/n$ and $\Upsilon=\tfrac12\tr(A+J)-\tr(AJ)$, a scalar defect that
vanishes only when $J=I$. If $b_i>0$ are the eigenvalues of $J$ and $a_i$ the
corresponding diagonal entries of $A$, then $0<a_i\le1/b_i$, and the terms
$\Upsilon_i=(a_i+b_i)/2-a_ib_i$ are nonnegative, with
$\Upsilon_i\ge(b_i-1)^2/(2b_i)$ once $b_i\ge1/2$. Applied to the largest
eigenvalue, this gives $\|J\|_{\mathrm{op}}\le\phi(\Upsilon)$, where
$\|\cdot\|_{\mathrm{op}}$ is the $\ell_2\!\to\!\ell_2$ operator norm and
$\phi(v)=1+v+\sqrt{v^2+2v}$. The tower identities $\E A=\E J=I$ together with
\eqref{eq:posterior-main} give $\E\Upsilon=2\E r/n$, and since $\phi$ is
increasing and concave, Jensen's inequality yields
\begin{equation}\label{eq:budget-main}
 \E\|J\|_{\mathrm{op}}\le\phi(2\E r/n)\le\phi(2q/n).
\end{equation}
The posterior operator bound follows from these conditional first-moment
identities.

Fix a known unit vector $u$, and put $Y=W^{-1}u$ and $a^2=\E\|Y\|_2^2$. If a
history-measurable $\widehat h$ satisfies $\|\widehat h\|_2\le a$ and
$\E\langle\widehat h,Y\rangle\ge ca^2$, then
$cga^2\le\E\langle\widehat h,Ju\rangle\le a\sqrt{\E\|J\|_{\mathrm{op}}}$, by
$J^2\preceq\|J\|_{\mathrm{op}}J$, the Cauchy--Schwarz inequality, and
$\E[u^\top Ju]=1$. Consequently, for $B=c^2g^2a^2\ge1$, inverting
\eqref{eq:budget-main} gives
\begin{equation}\label{eq:correlation-query-main}
 q\ge\frac{n(B-1)^2}{4B}.
\end{equation}

\subsection{Prior validity, amplification, and the explicit constant}
Set $k=d/m$ and assume $m\ge64$ and $k\ge8$. Gaussian regression gives
$a^2=(n-1)/[m(m-1)(m-3)]$ and $g^2a^2\ge k$. For $r_Y=\|Y\|_2/a$, the radial
moments satisfy $\E r_Y\ge\frac{39}{40}$, $\Var(r_Y)\le\frac3{64}$,
$\E r_Y^{-1}\le\frac{13}{12}$, and $\E r_Y^{12}<16$. Together with Gaussian
norm and pseudoinverse tail bounds, these estimates show that the event
$E=\{\|Y\|_2\le7a/4,\ m^2I/(10n)\preceq W\preceq5nI\}$ has probability at
least $49/50$. Use the normalization
\begin{equation}\label{eq:prior-main}
 t=\frac{7a}{4R},\qquad \sigma=\frac{16LR^2}{245na^2},\qquad
 f_W(y)=\sigma\left[\tfrac12(ty)^\top W(ty)-u^\top(ty)\right].
\end{equation}
On $E$, the global minimizer is $y^\star=Y/t$, its Euclidean norm is at most
$R$, and $L/[50(k+1)^2]I\preceq\nabla^2f_W=(L/5n)W\preceq LI$, so $f_W$
belongs to the target class. One exact value at any point $y$ is equivalent,
in both directions, to one quadratic measurement after subtracting the known
linear term. This equivalence holds for arbitrary finite query points. For
$\eps\le LR^2/(10000k^2)$, successful optimization forces the raw output
$x=t\widehat y$ to satisfy $\|x-Y\|_2\le a/5$.

To convert the pointwise success probability $2/3$ into the required
correlation under the prior, we run three independent copies of the $T$-query algorithm on the same objective,
query each output once, and return the output with the smallest observed
value, breaking ties by index. This uses at most $q=3(T+1)$ calls and succeeds
with probability at least $26/27$ on each valid objective, so the joint event
of prior validity and successful optimization has probability at least
$(49/50)(26/27)=637/675$ under the unconditioned prior. With
$\widehat h=a\operatorname{unit}(x)$, where $\operatorname{unit}(0)=0$, the
unit-direction inequality, the radial moments, and the Cauchy--Schwarz
inequality for the covariance give
\begin{equation}\label{eq:success-main}
 \frac{\E\langle\widehat h,Y\rangle}{a^2}
 \ge\frac{39}{40}\cdot\frac{599}{675}-\frac1{10}-\frac{13}{150}
 =\frac{6107}{9000}>\frac23,
\end{equation}
where the expectation still includes failed runs and histories generated by
objectives outside the target class. The twelfth radial moment is used only
to bound the probability of prior validity.

It remains to choose the parameters. Substituting $B\ge4k/9$ into
\eqref{eq:correlation-query-main} yields $T\ge dk/40$ for $k\ge12$. Set
$\eta=\min\{d/64,\sqrt{LR^2/\eps}/100\}$ and $m=\lceil d/\eta\rceil$. If
$\min\{d,\sqrt{LR^2/\eps}\}>1280$, then $m\ge64$ and
$12<64\min\{d,\sqrt{LR^2/\eps}\}/6500\le k\le\sqrt{LR^2/\eps}/100$, so every
condition on the prior and on the accuracy holds, and
$T\ge dk/40\ge d\min\{d,\sqrt{LR^2/\eps}\}/\Cstar$. Otherwise the spherical
family $f_\theta(x)=L\|x-R\theta\|_2^2/2$ gives $T\ge5d/16$, which already
dominates the product term because
$\min\{d,\sqrt{LR^2/\eps}\}\le1280=5\Cstar/16$; its dimension bound follows
from the mean squared residual radius and an antipodal pairing on the
conditional residual sphere. Both priors are fixed before the algorithmic seed
is drawn. This proves \cref{thm:main}.

\section{Proof of the Euclidean lower bound}\label{app:euclidean}
\subsection{Inverse moments on affine positive-definite slices}
\label{app:r12-affine}
\label{r12:sec-affine}

Throughout this appendix, matrix inner products are Frobenius inner
products, $\langle A,B\rangle=\operatorname{tr}(AB)$ for symmetric
matrices, and $\operatorname{Sym}_d^{++}$ denotes the positive-definite
cone.  Integrals over affine subspaces are taken with respect to the
induced Euclidean Lebesgue measure.  We begin with the analytic result
applied to each realized adaptive observation fiber.

\begin{theorem}[Affine-slice inverse comparison]
\label{r12:affine}
Let $\zeta>0$, let $\mathcal A=W_0+V\subseteq\operatorname{Sym}_d$ be an
affine subspace, and suppose that
$\Omega=\mathcal A\cap\operatorname{Sym}_d^{++}$ is nonempty.  For a fixed
$C\in\operatorname{Sym}_d$, let the probability density on $\Omega$ be
\[
 \rho(W)=Z_0^{-1}\det(W)^\zeta\exp(-\langle C,W\rangle).
\]
Assume that the normalizing integral and the first and inverse-first
moments $M=\mathbb EW$ and $Q=\mathbb EW^{-1}$ are finite.  Then
\begin{equation}
 Q\preceq\left(1+\frac{d+1}{2\zeta}\right)M^{-1}.
 \label{r12:eq-affine}
\end{equation}
The coefficient is sharp over this class of densities.
\end{theorem}

The matrix-weighted Brascamp--Lieb inequality on the whole space is
Theorem~4 of \citet{CorderoErausquin2019}.  The bounded-domain version
below includes the boundary argument that our application requires.

\subsubsection{A matrix-weighted variance inequality on convex domains}

Let $D\subset\mathbb R^v$ be a bounded, connected, smooth, convex domain,
and let $G:\overline D\to\operatorname{Sym}_d^{++}$ be smooth up to the
boundary.  Define
\[
 A_j=G^{-1}\partial_jG,\qquad
 \Theta_{jk}=\partial_kA_j,\qquad
 B_G(z,z)=-\sum_{j,k}z_k^\top G\Theta_{jk}z_j,
\]
where $z=(z_1,\ldots,z_v)$ is a tuple of vectors in $\mathbb R^d$.
Assume that this quadratic form on vector tuples is symmetric and
positive definite at every point, and define its polar energy by
\[
 B_G^\circ(h)=\sup_z\left\{2\sum_jh_j^\top Gz_j-B_G(z,z)\right\}.
\]

\begin{lemma}[Convex-domain matrix-weighted variance]
\label{r12:weighted}
For smooth $f:\overline D\to\mathbb R^d$, set
$Z_G=\int_DG$ and $\overline f_G=Z_G^{-1}\int_DGf$.  Then
\begin{equation}
 \int_D(f-\overline f_G)^\top G(f-\overline f_G)
 \le\int_DB_G^\circ(\nabla f).
 \label{r12:eq-weighted}
\end{equation}
\end{lemma}

\begin{proof}
Consider the operator
\[
 \mathcal Lu=G^{-1}\sum_j\partial_j(G\partial_ju)
 =\Delta u+\sum_jA_j\partial_ju
\]
with the componentwise Neumann condition $\partial_\nu u=0$.
Its negative is the operator associated with the symmetric bilinear
form $\sum_j\int_D(\partial_jz)^\top G\partial_ju$.
Since the eigenvalues of $G$ are bounded above and away from zero on
$\overline D$, the ordinary Poincar\'e inequality gives coercivity
modulo the constant vector functions, which form the kernel.  The
Lax--Milgram theorem yields a solution of $\mathcal Lu=h$ whenever
$\int_DGh=0$, unique modulo constants.  Its natural conormal condition
$G\partial_\nu u=0$ is equivalent to $\partial_\nu u=0$.
The required smoothness follows from standard variational Neumann
regularity: after multiplication by $G^{-1}$, the principal part is the
componentwise Laplacian and the remaining terms are smooth first-order
matrix couplings.  Starting from the variational $H^1$ solution, Neumann
$H^2$ regularity and bootstrapping justify the differentiations below.
Explicitly, testing the weak equation with $G^{-1}\phi$ gives the
componentwise weak equation
$\Delta u=h-\sum_j(G^{-1}\partial_jG)\partial_ju$ with zero Neumann data.
Its right-hand side is initially in $L^2$, so scalar Neumann regularity
gives $H^2$ regularity, and iteration with smooth coefficients completes
the argument.

Now set $h=f-\overline f_G$ and let $u$ be a corresponding solution.
Differentiating,
\[
 \partial_k\mathcal Lu=\mathcal L(\partial_ku)
                  +\sum_j\Theta_{jk}\partial_ju.
\]
Two integrations by parts yield
\begin{align}
 \int_D(\mathcal Lu)^\top G\mathcal Lu
 &=\sum_{j,k}\int_D(\partial_{jk}u)^\top G\partial_{jk}u
   +\int_DB_G(\nabla u,\nabla u)\notag\\
 &\quad-\int_{\partial D}\sum_k(\partial_ku)^\top
                    G\partial_\nu\partial_ku.
 \label{r12:eq-bochner}
\end{align}
The first integration by parts uses $\partial_\nu u=0$; the second
integrates $-\int_D(\partial_ku)^\top G\mathcal L(\partial_ku)$.
All interior derivatives of $G$ are accounted for by $\mathcal L$ and
$\Theta$.  Let $\mathrm{II}$ denote the second fundamental form with
respect to the outward normal.  Differentiating $\partial_\nu u=0$
tangentially gives, in an orthonormal tangent frame,
\[
 -\sum_k(\partial_ku)^\top G\partial_\nu\partial_ku
 =\sum_{a,b}\mathrm{II}_{ab}
           (\partial_{\tau_a}u)^\top G\partial_{\tau_b}u\ge0.
\]
Here $\mathrm{II}\succeq0$ by convexity, and no further normal
derivative of $G$ enters this boundary flux identity.  Consequently
\[
 \int_DB_G(\nabla u,\nabla u)\le\int_Dh^\top Gh.
\]
Finally, compatibility and integration by parts give
\[
 \int_Dh^\top Gh=\int_Df^\top G\mathcal Lu
              =-\sum_j\int_D(\partial_jf)^\top G\partial_ju.
\]
The pointwise dual Cauchy--Schwarz inequality for $B_G$, followed by its
integral form, bounds the square of the last expression by
\[
 \left(\int_DB_G^\circ(\nabla f)\right)
 \left(\int_DB_G(\nabla u,\nabla u)\right)
 \le\left(\int_DB_G^\circ(\nabla f)\right)\left(\int_Dh^\top Gh\right).
\]
If the second factor is positive, dividing by it proves the claim; if it
vanishes, the desired inequality is immediate.
\end{proof}

\subsubsection{Curvature of the determinant-power matrix weight}

Fix a Frobenius-orthonormal basis $B_1,\ldots,B_v$ of $V$ and write
$W(t)=W_0+\sum_jt_jB_j$.  We first restrict to a smooth bounded convex
domain $D$ whose closure lies in $\Omega$, normalize the density there
to $\rho_D$, and take $G(t)=\rho_D(W(t))W(t)$.  Direct differentiation
gives
\begin{align*}
 G^{-1}\partial_jG
 &=\bigl[\zeta\operatorname{tr}(W^{-1}B_j)-\langle C,B_j\rangle\bigr]I
      +W^{-1}B_j,\\
 -\Theta_{jk}
 &=\zeta\operatorname{tr}(W^{-1}B_kW^{-1}B_j)I
      +W^{-1}B_kW^{-1}B_j.
\end{align*}
The product order in the second term gives the squared norm in the
curvature identity below. At a fixed $W$, define
\[
 H_j=W^{-1/2}B_jW^{-1/2},\qquad
 \Gamma_{jk}=\operatorname{tr}(H_jH_k),\qquad z_j=W^{1/2}u_j.
\]
The $H_j$ are linearly independent, so $\Gamma\succ0$, and
\begin{equation}
 B_G(u,u)=\rho_D\left[
 \zeta\sum_{j,k}\Gamma_{jk}z_j\cdot z_k
 +\left\|\sum_jH_jz_j\right\|_2^2\right].
 \label{r12:eq-curvature}
\end{equation}
Indeed, before completing the square, the second summand reads
$\sum_{j,k}z_k^\top H_kH_jz_j$.
This establishes that the curvature form is symmetric and positive
definite.

Next, set
\[
 \xi_i=\sum_j(\Gamma^{1/2})_{ij}z_j,
 \qquad E_i=\sum_j(\Gamma^{-1/2})_{ij}H_j.
\]
The $E_i$ are Frobenius-orthonormal in the whitened symmetric tangent
space.  Define the linear map $\mathcal B\xi=\sum_iE_i\xi_i$ and put
\[
 \mathcal K=\mathcal B\mathcal B^*=\sum_iE_i^2,
 \qquad c=\frac{d+1}{2}.
\]
Then \eqref{r12:eq-curvature}, divided by $\rho_D$, equals
$\zeta\|\xi\|_2^2+\|\mathcal B\xi\|_2^2$, and
\begin{equation}
 0\preceq\mathcal K\preceq cI.
 \label{r12:eq-symmetric-basis}
\end{equation}
To prove this bound, complete the $E_i$ to an orthonormal basis of
$\operatorname{Sym}_d$.  For the coordinate basis $E_{ii}$ and
$(E_{ij}+E_{ji})/\sqrt2$, $i<j$, the sum of squares equals
$I+(d-1)I/2=cI$.  This sum is invariant under orthogonal changes of
basis, and every omitted square is positive semidefinite.
All changes of frame here are pointwise algebraic changes of variables;
no derivative of a varying frame is taken.

\subsubsection{The inverse test and exhaustion}

Fix $a\in\mathbb R^d$ and apply Lemma~\ref{r12:weighted} to
$f(W)=W^{-1}a$.  Write $M_D=\mathbb E_DW$ and $Q_D=\mathbb E_DW^{-1}$.
Since $\int_DG=M_D$ and $\int_DGf=a$, the left-hand side of
\eqref{r12:eq-weighted} equals $a^\top(Q_D-M_D^{-1})a$.  With
$v_a=W^{-1/2}a$, the gradient pairing is
\[
 \sum_j(\partial_jf)^\top Gu_j=-\rho_Dv_a^\top\mathcal B\xi.
\]
Completing the square in the definition of the polar yields
\begin{align*}
 B_G^\circ(\nabla f)
 &=\rho_Dv_a^\top\mathcal B
       (\zeta I+\mathcal B^*\mathcal B)^{-1}\mathcal B^*v_a\\
 &=\rho_Dv_a^\top\mathcal K(\zeta I+\mathcal K)^{-1}v_a
 \le\rho_D\frac c{\zeta+c}\,a^\top W^{-1}a.
\end{align*}
The equality uses
$(\zeta I+\mathcal B^*\mathcal B)\mathcal B^*
=\mathcal B^*(\zeta I+\mathcal B\mathcal B^*)$.
For the inequality, diagonalize $\mathcal K$ and use
$s/(\zeta+s)\le c/(\zeta+c)$ for $0\le s\le c$.
Hence
\[
 Q_D-M_D^{-1}\preceq\frac c{\zeta+c}Q_D,
 \qquad Q_D\preceq\frac{\zeta+c}{\zeta}M_D^{-1}.
\]

To pass to $\Omega$, consider the smooth strictly convex function
\[
 \phi(t)=\|t\|_2^2+\operatorname{tr}(W(t)^{-1}).
\]
It diverges at infinity and on approach to any finite boundary point
of $\Omega$, it has a unique minimizer, and every level above the
minimum is regular.  The sublevel sets $D_b=\{\phi<b\}$ are therefore smooth
bounded convex domains with closure in $\Omega$, and they exhaust
$\Omega$.  By dominated convergence, the assumed moments give
$M_{D_b}\to M$ and $Q_{D_b}\to Q$, normalizing constants included, since
$\operatorname{tr}W$ and $\operatorname{tr}W^{-1}$ dominate all relevant
matrix entries.  Since $M\succ0$, inversion is continuous at $M$, and
the positive-semidefinite order is closed.  This proves
\eqref{r12:eq-affine}.  A zero-dimensional slice is a point mass and
satisfies the assertion directly.  Uniform ellipticity is needed only
on each individual compact domain.

For sharpness, take the full cone, $C=I/2$, and $n=2\zeta+d+1$.  The
determinant-power (Wishart) density then has $M=nI$ and
$Q=I/(2\zeta)$, also for the corresponding real degrees-of-freedom
parameter, so equality holds in \eqref{r12:eq-affine}.  This completes
the proof of Theorem~\ref{r12:affine}.

\subsection{The posterior on an adaptive observation fiber}
\label{app:r12-posterior}

Let $X\in\mathbb R^{n\times d}$ have independent standard Gaussian
entries, where $n=d+m$ and $m>3$ are integers, and put
\[
 W=X^\top X,\qquad g=m-1,\qquad Z=gW^{-1},\qquad s(W)=\frac{I-Z}{2}.
\]
The density of $W$ on $\operatorname{Sym}_d^{++}$ is proportional to
$\det(W)^{g/2}\exp(-\operatorname{tr}W/2)$, and $s$ is its negative
score with respect to the Frobenius metric.
Let $\mathcal H$ include the independent algorithmic seed and the
complete history of at most $q$ adaptive observations $x_i^\top Wx_i$,
and define
\[
 M=\mathbb E[W\mid\mathcal H],\qquad
 J=\mathbb E[Z\mid\mathcal H].
\]
The prior moments are $\mathbb EW=nI$ and $\mathbb EW^{-1}=I/g$, so
that $\mathbb EM=nI$ and $\mathbb EJ=I$.
The inverse mean also follows from the Schur-complement law derived
below: for any fixed unit vector $u$, $u^\top W^{-1}u$ is distributed
as $1/\chi^2_{m+1}$.

\begin{lemma}[Conditional law on the realized fiber]
\label{r12:disintegration}
Let
$S=\operatorname{span}\{x_ix_i^\top:1\le i\le q\}$ for the realized
history, and let $r=\dim S$.  Almost surely, the conditional law of $W$
is its original density normalized on the realized affine slice
$W_0+S^\perp$ intersected with $\operatorname{Sym}_d^{++}$.
\end{lemma}

\begin{proof}
First condition on the independent seed; by independence, the prior
density is unchanged.  Given a partial history, the next query $x$ is
fixed, and its answer is the linear functional
$\ell_x(W)=\langle xx^\top,W\rangle$ on the current affine slice.
If the projection of $xx^\top$ onto the tangent space of the slice is
nonzero, the linear coarea formula disintegrates affine Lebesgue
measure into Lebesgue measures on the codimension-one level sets.  The
coarea factor is the reciprocal norm of this projection; it is constant
on each resulting fiber and cancels when the conditional density is
normalized.  If the projection is zero, the answer is already
determined by the current history, and neither the conditional law nor
the fiber changes.

Induction therefore gives the stated conditional density.  Explicitly,
on each positive-dimensional resulting fiber $\mathcal F$ the kernel is
\[
 \mathbb P(W\in B\mid\mathcal H)
 =\frac{\displaystyle\int_{B\cap\mathcal F\cap\operatorname{Sym}_d^{++}}
          \det(U)^{g/2}e^{-\operatorname{tr}U/2}\,dU}
        {\displaystyle\int_{\mathcal F\cap\operatorname{Sym}_d^{++}}
          \det(U)^{g/2}e^{-\operatorname{tr}U/2}\,dU}
\]
at almost every realized history, and a singleton fiber carries a point
mass.  Previous query choices are determined by previous answers, so
they introduce no additional weight varying within the fiber.
Fixed-order Gram--Schmidt with zero pivots skipped gives measurable
tangent bases and separates the measurable rank strata, and applying
the Moore--Penrose inverse to the realized finite linear system gives
measurable affine base points.  Together, these constructions yield
measurable versions of the successive kernels for Borel policies and
random ranks.  The argument applies the coarea formula to each fixed
linear measurement after conditioning on its past; it requires neither
differentiability nor continuity of the policy or of the full
transcript map.
\end{proof}

\begin{remark}[General linear measurements]\label{r12:rem-linear}
The conditional-law proof and all posterior identities below hold
unchanged for observations $\langle B_i,W\rangle$ with arbitrary
symmetric $B_i$ chosen measurably from the preceding history.  Replace
the measured span by $\operatorname{span}\{B_i\}$; each informative
scalar observation still reduces the affine dimension by exactly one.
Neither positivity nor the rank-one structure of $x_ix_i^\top$ is used in
these arguments.  The rank-one restriction enters the optimization
reduction only through the quadratic objective representation.
\end{remark}

\begin{lemma}[Inverse and trace identities on the conditional fiber]
\label{r12:fiber}
\label{r12:posterior}
Almost surely,
\begin{equation}
 J\preceq nM^{-1},\qquad
 \operatorname{tr}(MJ)=nd-2r,\qquad r\le q.
 \label{r12:eq-fiber}
\end{equation}
\end{lemma}

\begin{proof}
By the tower property, the conditional first and inverse-first moments
are finite almost surely.  Applying Theorem~\ref{r12:affine} with
$\zeta=g/2$ to the conditional law of Lemma~\ref{r12:disintegration}
and multiplying by $g$ gives $J\preceq(g+d+1)M^{-1}=nM^{-1}$.

The conditional mean $M$ lies in the same affine slice.  In its $v=d(d+1)/2-r$ orthonormal tangent
coordinates, the centered vector field $W-M$ has divergence $v$, so
integration by parts gives
\[
 v=\mathbb E[\langle s(W),W-M\rangle\mid\mathcal H]
  =\frac{\operatorname{tr}(MJ)-gd}{2},
\]
where we used $\mathbb E[W-M\mid\mathcal H]=0$ and $W^{-1}W=I$.
Since $gd+2v=nd-2r$, the trace identity follows.

It remains to justify the boundary terms and the integrability in this
calculation.  Set $\zeta=g/2>1$ and extend
$p(W)=\det(W)^\zeta e^{-\operatorname{tr}W/2}$ by zero outside the
positive-definite cone.  Inside the cone,
\[
 \nabla p=e^{-\operatorname{tr}W/2}
 \left[\zeta\det(W)^{\zeta-1}\operatorname{adj}W
              -\frac12\det(W)^\zeta I\right].
\]
Both terms tend to zero at every finite boundary point, so the zero
extension is $C^1$.  For positive-semidefinite $W$,
\[
 \|W\|_F\le\operatorname{tr}W,\qquad
 \det W\le\|W\|_F^d,\qquad
 \|\operatorname{adj}W\|_{\mathrm{op}}\le\|W\|_F^{d-1}.
\]
Hence $p$, $(1+\|W\|_F)p$, $(1+\|W\|_F)\nabla p$, and
$W^{-1}p=\operatorname{adj}W\det(W)^{\zeta-1}e^{-\operatorname{tr}W/2}$
are bounded on every fixed affine slice by integrable functions of the
form polynomial times decaying exponential.  Choose the affine base
point orthogonal to the tangent space; then $\|W\|_F$ dominates the
tangent-coordinate norm.  Compact cutoffs therefore remove the boundary
contribution at infinity, while the zero extension removes the flux
through the finite boundary.  This also gives finite forward and inverse
means on every slice meeting the positive-definite cone.  For a
zero-dimensional slice, $M=W$ and $J=gW^{-1}$, and the identity follows
directly.  Finally, each scalar query reveals at most one new
dimension, so $r\le q$.
\end{proof}

\subsection{A direct posterior operator-norm budget}
\label{app:r12-budget}
\label{r12:sec-budget}

For $v\ge0$, define $\varphi(v)=1+v+\sqrt{v^2+2v}$.

\begin{lemma}[Deterministic operator bound]
\label{r12:operator}
If $A,J\succ0$ and $A\preceq J^{-1}$, then
\[
 \Upsilon=\tfrac12\operatorname{tr}(A+J)-\operatorname{tr}(AJ)\ge0,
 \qquad \|J\|_{\mathrm{op}}\le\varphi(\Upsilon)\le2+2\Upsilon.
\]
\end{lemma}
\begin{proof}
In an orthonormal eigenbasis of $J$, write $b_i>0$ for its eigenvalues
and $a_i$ for the diagonal entries of $A$, so that $0<a_i\le1/b_i$.
Put $\Upsilon_i=(a_i+b_i)/2-a_ib_i$.
If $b_i\le1/2$, both terms in $\Upsilon_i=b_i/2+(1/2-b_i)a_i$ are
nonnegative.  If $b_i\ge1/2$, minimizing over $a_i\le1/b_i$ gives
\[
 \Upsilon_i\ge\tfrac12(b_i+b_i^{-1}-2)=\frac{(b_i-1)^2}{2b_i}\ge0.
\]
Thus $\Upsilon=\sum_i\Upsilon_i\ge0$.  If $b=\|J\|_{\mathrm{op}}>1$,
its contribution gives $(b-1)^2\le2b\Upsilon$, and the larger root of
the corresponding quadratic equation is $\varphi(\Upsilon)$.
If $b\le1$, use $\varphi(\Upsilon)\ge1$.  Finally,
$\sqrt{\Upsilon^2+2\Upsilon}\le\Upsilon+1$.
\end{proof}

\begin{proposition}[Adaptive operator budget]
\label{r12:budget}
For the realized history of at most $q$ scalar quadratic measurements,
with $J=(m-1)\mathbb E[W^{-1}\mid\mathcal H]$, $n=d+m$, and integer $m>3$,
\begin{equation}
 \mathbb E\|J\|_{\mathrm{op}}
 \le\varphi(2\mathbb Er/n)\le\varphi(2q/n).
 \label{r12:eq-budget}
\end{equation}
Here $r$ is the dimension of the measured matrix span.
\end{proposition}
\begin{proof}
Set $A=M/n$.  Lemma~\ref{r12:fiber} gives $A\preceq J^{-1}$ and
$\operatorname{tr}(AJ)=d-2r/n$, and the tower identities
$\mathbb EA=\mathbb EJ=I$ give $\mathbb E\Upsilon=2\mathbb Er/n$.
The function $\varphi$ is increasing and concave: for $v>0$,
$\varphi''(v)=-(v^2+2v)^{-3/2}$, and concavity extends continuously to
zero.  Applying Jensen's inequality to the bound of
Lemma~\ref{r12:operator} and using $r\le q$ gives the claim.
Integrability follows from $\varphi(\Upsilon)\le2+2\Upsilon$ and finite
first moments.  The argument requires neither higher inverse moments
nor spectral-tail estimates.
\end{proof}

\begin{lemma}[Correlation-to-query bound]
\label{r12:correlation-query}
Let $u$ be a fixed unit vector, $Y=W^{-1}u$, and $a^2=\mathbb E\|Y\|_2^2$.
Suppose that an observable $\widehat h$ satisfies $\|\widehat h\|_2\le a$
and $\mathbb E[\widehat h^\top Y]\ge ca^2$ for $c>0$.
Put $g=m-1$ and $B=c^2g^2a^2$. If $B\ge1$, then
\begin{equation}
 q\ge\frac{n(B-1)^2}{4B}.
 \label{r12:eq-cost-inversion}
\end{equation}
\end{lemma}
\begin{proof}
Conditional expectation and boundedness give
\[
 cga^2\le\mathbb E[\widehat h^\top Ju]\le a\mathbb E\|Ju\|_2.
\]
Since $J^2\preceq\|J\|_{\mathrm{op}}J$ and $\mathbb E[u^\top Ju]=1$,
the Cauchy--Schwarz inequality yields
\[
 \mathbb E\|Ju\|_2
 \le\mathbb E\sqrt{\|J\|_{\mathrm{op}}u^\top Ju}
 \le\sqrt{\mathbb E\|J\|_{\mathrm{op}}}.
\]
Hence $B\le\mathbb E\|J\|_{\mathrm{op}}\le\varphi(2q/n)$.
On $[1,\infty)$, the inverse of $\varphi$ is
$B\mapsto(B-1)^2/(2B)$, so $(B-1)^2/(2B)\le2q/n$.
\end{proof}

\subsection{Exact radial moments}
\label{app:r12-moments}
\label{r12:sec-moments}

Let $m\ge64$ and $d\ge8m$ be integers, and define $k=d/m$, $n=d+m$,
$g=m-1$.  Fix a known deterministic unit vector $u$ and set
$Y=W^{-1}u$.  Gaussian regression gives
\begin{equation}
 \|Y\|_2^2\ \stackrel{\mathrm{law}}=\ \frac{1+U/V}{S_0^2},
 \quad U\sim\chi^2_{d-1},\quad V\sim\chi^2_{m+2},\quad
 S_0\sim\chi^2_{m+1},
 \label{r12:eq-radial-law}
\end{equation}
where the three variables on the right-hand side are independent.
To see this, rotate $u$ to $e_1$, partition $X=[x,Z_0]$, and let
\[
 C_0=Z_0^\top Z_0,\quad h=C_0^{-1}Z_0^\top x,\quad
 S_0=x^\top(I-P_{Z_0})x.
\]
Then $Y=S_0^{-1}(1,-h)$, $S_0$ is independent of $(C_0,h)$, and
$h\mid C_0\sim N(0,C_0^{-1})$.
Integrating this normal density against the Wishart density of $C_0$
gives a density proportional to $(1+\|h\|_2^2)^{-(n+1)/2}$: the matrix
rate is $I+hh^\top$, and the congruence change of variables contributes
its determinant raised to the power $-(n+1)/2$.  The same density
arises from $h=z/\sqrt V$ with $z\sim N(0,I_{d-1})$ independent of $V$,
and taking $U=\|z\|_2^2$ proves \eqref{r12:eq-radial-law}.
Moreover, $u^\top W^{-1}u=S_0^{-1}$, which gives the inverse mean used
earlier.  Gamma integration gives
\[
 \mathbb E(\chi^2_\ell)^{-j}=\prod_{i=1}^j(\ell-2i)^{-1},\quad\ell>2j.
\]
In particular,
\begin{equation}
 a^2=\frac{n-1}{m(m-1)(m-3)},\qquad g^2a^2\ge k.
 \label{r12:eq-second-moment}
\end{equation}
The inequality holds because the numerator difference is
$(m-1)(d+m-1)-d(m-3)=2d+(m-1)^2>0$.

\begin{lemma}[Uniform radial bounds]
\label{r12:moments}
For $R_Y=\|Y\|_2/a$, $\mu_Y=\mathbb ER_Y$, $\nu_Y=\mathbb ER_Y^{-1}$,
\begin{equation}
 \mu_Y\ge\frac{39}{40},\quad\operatorname{Var}(R_Y)\le\frac3{64},
 \quad\nu_Y\le\frac{13}{12},\quad\mathbb ER_Y^{12}<16.
 \label{r12:eq-radial-bounds}
\end{equation}
\end{lemma}
\begin{proof}
The ratio $B_0=V/(U+V)$ has a beta distribution with parameters
$((m+2)/2,(d-1)/2)$ and is independent of $S_0$.  Its moment of order
minus one-half gives
\[
 \mu_Y=\frac1{ga}
 \frac{\Gamma((m+1)/2)\Gamma((n+1)/2)}{\Gamma((m+2)/2)\Gamma(n/2)}
 =\frac{\mathbb E\chi_n}{ga\mathbb E\chi_{m+1}},
\]
where $\chi_j$ denotes the norm of a standard Gaussian vector in
dimension $j$.  The Cauchy--Schwarz inequality in the Gamma integral
gives
\[
 \Gamma(x+1)\le\sqrt{\Gamma(x+1/2)\Gamma(x+3/2)}
 =\sqrt{x+1/2}\,\Gamma(x+1/2).
\]
Thus $\mathbb E\chi_n\ge n/\sqrt{n+1}\ge\sqrt{n-1}$, whereas
$\mathbb E\chi_{m+1}\le\sqrt{m+1}$.  Using \eqref{r12:eq-second-moment},
we obtain
\[
 \mu_Y^2\ge\frac{m(m-3)}{m^2-1},\qquad
 \operatorname{Var}(R_Y)=1-\mu_Y^2\le\frac{3m-1}{m^2-1}\le\frac3m.
\]
Hence $\mu_Y\ge\sqrt{61/64}>39/40$ and $\operatorname{Var}(R_Y)\le3/64$.
Next, the second moment of $S_0$ and the first beta moment give
\[
 \mathbb ER_Y^{-2}=\frac{n-1}{n+1}
 \frac{(m+1)(m+2)(m+3)}{m(m-1)(m-3)}.
\]
Each of the three ratios in the last fraction is decreasing in $m$, so
the Cauchy--Schwarz inequality gives
\[
 \nu_Y\le\sqrt{\frac{65\cdot66\cdot67}{64\cdot63\cdot61}}<\frac{13}{12}.
\]
Finally, for an integer $p$ with $4p<m+1$, beta and Gamma integration
yield
\[
 \mathbb E(1+U/V)^p=\prod_{j=1}^p\frac{n+1-2j}{m+2-2j},\qquad
 \mathbb ES_0^{-2p}=\prod_{j=1}^{2p}\frac1{m+1-2j}.
\]
Dividing their product by $a^{2p}$ and bounding $n+1-2j\le n-1$ gives
\[
 \mathbb ER_Y^{2p}\le C_p(m):=
 \prod_{j=1}^p\frac m{m-(2j-2)}
 \frac{m-1}{m-(4j-3)}\frac{m-3}{m-(4j-1)}.
\]
Every ratio is $1+c/(m-b)$ with $c\ge0$, so $C_p$ is decreasing on its
domain, and exact rational arithmetic gives
\[
 C_6(64)=\frac{1562239696858251264}{103452970118876125}<16.
\]
\end{proof}

\subsection{A valid quadratic prior with offset 64}
\label{app:r12-prior}
\label{r12:sec-prior}

For an $n\times d$ standard Gaussian matrix with $n=d+m$, the
pseudoinverse bound of \citet[Proposition~10.4, equation~(10.6)]
{HalkoMartinssonTropp2011}, applied to $X^\top$, reads
\begin{equation}
 \mathbb P\left\{\|X^\dagger\|_{\mathrm{op}}
              \ge\frac{e\sqrt n}{m+1}t\right\}
 \le t^{-(m+1)},\qquad t\ge1.
 \label{r12:eq-pseudoinverse}
\end{equation}
Their expectation and concentration estimates for the Gaussian norm
(Propositions~10.1 and~10.3) give
\begin{equation}
 \mathbb P\{\|X\|_{\mathrm{op}}>\sqrt n+\sqrt d+t\}\le e^{-t^2/2}.
 \label{r12:eq-gaussian-norm}
\end{equation}

\begin{lemma}[Supported normalization]
\label{r12:lem-prior}
For integers $m\ge64$ and $k=d/m\ge8$, the event
\[
 E=\left\{\|Y\|_2\le7a/4,\quad
             \frac{m^2}{10n}I\preceq W\preceq5nI\right\}
\]
has probability at least $49/50$. Set
\begin{equation}
 t=\frac{7a}{4R},\qquad\sigma=\frac{16LR^2}{245na^2},\qquad
 f_W(y)=\sigma\left[\tfrac12(ty)^\top W(ty)-u^\top(ty)\right].
 \label{r12:eq-normalization}
\end{equation}
On $E$, $f_W\in\mathcal C_{d,L,R}$. In fact,
\begin{equation}
 y^\star=Y/t,\quad\|y^\star\|_2\le R,\qquad
 \frac{Lm^2}{50n^2}I\preceq\nabla^2f_W=\frac{L}{5n}W\preceq LI.
 \label{r12:eq-supported-hessian}
\end{equation}
On the same event, if an algorithm returns an $\varepsilon$-accurate
output and $\varepsilon\le LR^2/(10000k^2)$, the corresponding raw output
$x=t\widehat y$ satisfies $\|x-Y\|_2\le a/5$.
\end{lemma}
\begin{proof}
By Lemma~\ref{r12:moments} and Markov's inequality,
$\mathbb P\{\|Y\|_2>7a/4\}\le16(4/7)^{12}$.
Applying \eqref{r12:eq-pseudoinverse} with parameter
$\sqrt{10}(m+1)/(em)>1$ gives
\[
 \mathbb P\left\{\lambda_{\min}(W)<\frac{m^2}{10n}\right\}
 \le\left(\frac{e^2m^2}{10(m+1)^2}\right)^{(m+1)/2}
 \le(37/50)^{32}<\frac1{10000}.
\]
Here $e^2<37/5$.  Next, $\sqrt5-2>1/5$ and $n\ge576$, so using
$\sqrt n+\sqrt d\le2\sqrt n$ in \eqref{r12:eq-gaussian-norm} gives
\[
 \mathbb P\{\lambda_{\max}(W)>5n\}
 \le e^{-n/50}\le e^{-576/50}<\frac1{10000}.
\]
The exponential series gives $8/3<e<68/25$, which implies $e^2<37/5$
and the exponential comparisons above.  Since
\[
 16(4/7)^{12}+\frac1{5000}<\frac1{50},
\]
the union bound gives $\mathbb P(E)\ge49/50$.
The minimizer and the Hessian follow from \eqref{r12:eq-normalization},
since $\sigma t^2=L/(5n)$.  For an $\varepsilon$-accurate output, the
exact quadratic gap gives
\begin{align*}
 \frac{\|x-Y\|_2^2}{a^2}
 &\le\frac{20n\varepsilon}{\sigma m^2a^2}
 =\frac{1225}{4}(k+1)^2\frac{\varepsilon}{LR^2}\\
 &\le\frac{99225}{256}k^2\frac{\varepsilon}{LR^2}
 <\frac{388}{10000}<\frac1{25}.
\end{align*}
Here we used $k+1\le9k/8$.  All scaling factors are deterministic and
known.  A single value at an arbitrary point $y$ determines the
measurement
\[
 (ty)^\top W(ty)=2\left[f_W(y)/\sigma+u^\top(ty)\right],
\]
and conversely.  This gives an exact one-to-one correspondence between
arbitrary adaptive histories, including those with queries outside the
ball.  The law conditioned on $E$ is therefore a fixed prior supported
on globally valid positive-definite quadratics.
\end{proof}

\subsection{Three validated runs give an unconditional correlation}
\label{app:r12-success}
\label{r12:sec-success}

Run three independent copies of the original $T$-query algorithm on the
same objective, query each returned point once, and select the output
with the smallest observed value, breaking ties by index.  The total
number of calls is at most
\begin{equation}
 q=3(T+1).
 \label{r12:eq-amplification-cost}
\end{equation}
This is an admissible adaptive protocol with a tuple of independent
seeds.  Each run uses its own transcript and seed, so, conditional on
a fixed objective, its failure is independent of the other runs.  The
three validation calls are part of the combined history.

\begin{proposition}[Success correlation]
\label{r12:success}
Under Lemma~\ref{r12:lem-prior}, suppose that the original algorithm
succeeds with probability at least $2/3$ on every valid objective at
$\varepsilon\le LR^2/(10000k^2)$. Then the combined history admits an
observable $\widehat h$ with $\|\widehat h\|_2\le a$ and
\begin{equation}
 \mathbb E[\widehat h^\top Y]\ge\frac23a^2.
 \label{r12:eq-correlation}
\end{equation}
The expectation is taken over the unconditioned Wishart prior and all seeds.
\end{proposition}
\begin{proof}
For each fixed valid instance, all three runs fail with probability at
most $1/27$, and the selected output succeeds whenever some run
succeeds.  Let $A$ be the intersection of the validity event with the
event that the selected output succeeds.  Then
\[
 \mathbb P(A)\ge\frac{49}{50}\frac{26}{27}=\frac{637}{675},
 \qquad p:=\mathbb P(A^c)\le p_0:=\frac{38}{675}.
\]
Write $x$ for the selected raw output and set
$\widehat h=a\operatorname{unit}(x)$, with $\operatorname{unit}(0)=0$.
For $x\ne0$, the distance between its direction and that of $Y$ is at
most $2\|x-Y\|_2/\|Y\|_2$.  Squaring and using the unit-vector
inner-product identity gives
\[
 \frac{\widehat h^\top Y}{a^2}
 \ge R_Y-\frac{2\|x-Y\|_2^2}{a^2R_Y}.
\]
This remains true for $x=0$, since the right-hand side is $-R_Y$.
On $A$, the raw error is at most $a/5$, and for every history the
scalar product is at least $-a^2R_Y$.  Thus
\[
 \frac{\mathbb E[\widehat h^\top Y]}{a^2}
 \ge\mu_Y-2\mathbb E[R_Y\mathbf1_{A^c}]-\frac2{25}\nu_Y.
\]
Applying the Cauchy--Schwarz inequality to the covariance bounds the
middle expectation by
\[
 \mathbb E[R_Y\mathbf1_{A^c}]
 \le p\mu_Y+\sqrt{\operatorname{Var}(R_Y)p(1-p)}.
\]
Since $p\le p_0<1/2$, Lemma~\ref{r12:moments} implies
\begin{align*}
 \frac{\mathbb E[\widehat h^\top Y]}{a^2}
 &\ge\frac{39}{40}(1-2p_0)
     -2\sqrt{\frac3{64}p_0(1-p_0)}-\frac{13}{150}\\
 &\ge\frac{39}{40}\frac{599}{675}-\frac1{10}-\frac{13}{150}
 =\frac{6107}{9000}>\frac23.
\end{align*}
The second line holds because the square of the middle subtracted
term is
\[
 4\frac3{64}\frac{38}{675}\frac{637}{675}
 =\frac{12103}{1215000}<\frac1{100}.
\]
Failed runs and invalid instances remain in the expectation.  Total
Borel policies make the combined experiment well-defined under the
original prior, so the information calculation is carried out without
conditioning on $E$.
\end{proof}

By \eqref{r12:eq-second-moment}, the parameter $B$ of
Lemma~\ref{r12:correlation-query} is at least $4k/9$.  Since
$(B-1)^2/B$ is increasing for $B\ge1$, we obtain, for $k\ge12$,
\begin{equation}
 3(T+1)\ge\frac{n(4k-9)^2}{144k},\qquad
 m\ge64,\quad k\ge12,\quad
 \varepsilon\le\frac{LR^2}{10000k^2}.
 \label{r12:eq-calibrated}
\end{equation}
This complexity bound accounts for all three runs and their three
validation queries.

\subsection{Dimension baseline and the constant 4096}
\label{app:r12-rounding}
\label{r12:sec-rounding}

\begin{lemma}[Adaptive dimension baseline]
\label{r12:baseline}
\label{r12:lem-baseline}
Under the hypotheses of Theorem~\ref{thm:main}, $T\ge5d/16$.
This bound holds even for the positive-definite quadratics
$f_\theta(x)=L\|x-R\theta\|_2^2/2$ with $\|\theta\|_2=1$.
\end{lemma}
\begin{proof}
Let $\theta$ be uniform on the unit sphere and independent of the
seed.  After subtracting known terms, each exact value reveals one
linear measurement.  If $T\ge d$, the claim is immediate; otherwise
condition on the seed.  Conditional on the history, the unobserved
component is uniform on a residual sphere in the orthogonal complement
of the query span.  An informative query reveals one coordinate in a
direction fixed by the past, whereas a redundant query reveals nothing.
If necessary, append orthogonal measurements after the original
experiment until there are exactly $T$ informative measurements.  This
additional information can only increase the success probability, to
which we now apply an upper bound.

Let $r_j$ be the residual radius after $j$ informative measurements.
Rotational symmetry gives
\[
 \mathbb E[r_{j+1}^2\mid\mathcal H_j]
 =r_j^2\frac{d-j-1}{d-j},\qquad
 \mathbb E(1-r_T^2)=\frac Td.
\]
Put $\delta=\sqrt{2\varepsilon/(LR^2)}\le1/4$.
When $r_T>\delta$, a ball of radius $\delta$ contains at most one point
of each antipodal pair on the residual sphere: if both $v$ and $-v$
belonged to such a ball centered at $z$, averaging their squared
distances would give $r_T^2+\|z\|_2^2\le\delta^2$, a contradiction.
Hence, on this event, the conditional success probability is at most
$1/2$ for any output; a component of the output outside the residual
subspace only increases these distances.  Markov's inequality gives
\[
 \frac23\le\frac12+\frac12\mathbb P(r_T\le\delta)
 \le\frac12+\frac{T}{2d(1-\delta^2)}.
\]
Therefore $T\ge d(1-\delta^2)/3\ge5d/16$.
The same pairing argument applies to a two-point residual sphere and to
$d=1,T=0$.
\end{proof}

\begin{proof}[Proof of Theorem~\ref{thm:main}]
Let $s=\sqrt{LR^2/\varepsilon}$ and $K=\min\{d,s\}$.
If $K\le1280$, Lemma~\ref{r12:baseline} gives
$T\ge5d/16\ge dK/4096$.  Otherwise, set
\[
 \eta=\min\{d/64,s/100\},\qquad
 m=\lceil d/\eta\rceil,\qquad k=d/m.
\]
Then $m\ge64$, $k\le\eta\le s/100$, and
\[
 k\ge\frac{\eta}{1+\eta/d}\ge\frac{64}{65}\eta
 \ge\frac{64K}{6500}>12.
\]
Thus all conditions on the prior and on the accuracy are satisfied.
Since $n=d(1+1/k)$, \eqref{r12:eq-calibrated} gives
\[
 T\ge\frac{dk}{432}(1+1/k)(4-9/k)^2-1.
\]
For $k\ge12$,
\[
 (1+1/k)(4-9/k)^2
 =16-\frac{56}{k}+\frac9{k^2}+\frac{81}{k^3}\ge\frac{34}{3}.
\]
Moreover, $dk=mk^2\ge64k^2\ge9216>810$.  Since $17/648-1/40=1/810$,
\[
 T\ge\frac{17}{648}dk-1\ge\frac{dk}{40}
 \ge\frac{64}{260000}dK\ge\frac{dK}{4096}.
\]
Both hard families consist of globally valid positive-definite
quadratics, with priors fixed before the independent seeds.
Lemma~\ref{r12:baseline} supplies the other term in the maximum.
\end{proof}

\begin{remark}[Moment requirements and proof dependencies]
The operator budget uses only conditional first moments.  Together with
the spectral estimates, the twelfth radial moment gives a validity event
of probability at least $49/50$; conditioning on that event yields a
fixed prior supported on the required class.  The probabilistic reduction
combines elliptic regularity, conditional disintegration, and Wishart
integration.
Optimizing the explicit constant $4096$ remains an open question.
\end{remark}

\section{Upper bounds, moderate-accuracy optimality, and strong convexity}\label{app:corollaries}
\subsection{The two quadratic upper bounds}
Write $q(x)=\tfrac12x^\top Hx+b^\top x+c$ with $H$ symmetric.  The value
at zero determines $c$, and the values at $\pm e_i$ give
\[
 b_i=\frac{q(e_i)-q(-e_i)}2,\qquad
 H_{ii}=q(e_i)+q(-e_i)-2q(0).
\]
For $i<j$,
\[
 H_{ij}=q(e_i+e_j)-q(e_i)-q(e_j)+q(0).
\]
This procedure uses exactly $1+2d+\binom d2=(d+1)(d+2)/2$ queries.
Convexity implies $H\succeq0$, and the existence of a global minimizer
implies $b\in\operatorname{range}(H)$.  Hence $-H^\dagger b$ is the
minimum-norm global minimizer, and its norm is at most $R$.
Pseudoinversion is a Borel map, since it is continuous on each of the
finitely many fixed-rank strata.  On inconsistent transcripts the policy
may return zero.  This defines a total Borel rule.

For any fixed $h\ne0$, expanding the quadratic also gives
\[
 \frac{q(x+he_i)-q(x-he_i)}{2h}=(Hx+b)_i.
\]
Therefore $2d$ values recover the gradient exactly at any iterate.  The
accelerated method with known $L$ and starting point zero satisfies
$q(y_N)-\min q\le2LR^2/N^2$ \citep[Theorem~3.19]{Bubeck2015}, and
$N=\lceil\sqrt{2LR^2/\eps}\rceil$ gives the claimed budget of values.
Both are fixed-budget deterministic algorithms and hence admissible in
the randomized model.  Selecting the cheaper procedure on the basis of
$d,L,R,\eps$ proves the upper bound in Corollary~\ref{cor:quadratic};
Theorem~\ref{thm:main} supplies the lower bound.

\subsection{Upper bounds for the general smooth class}\label{app:general-upper}
The two procedures above rely on the fact that finitely many values
identify a quadratic or recover its gradient exactly.  For general
$f\in\mathcal C_{d,L,R}$, neither property is guaranteed.  The model
imposes no precision limit, however, so central differences with an
arbitrarily small step are available: for $h>0$, the vector
$g_h(x)_i=[f(x+he_i)-f(x-he_i)]/(2h)$ requires $2d$ values and, since
$|\partial_if(x+te_i)-\partial_if(x)|\le L|t|$, satisfies
\[
 \|g_h(x)-\nabla f(x)\|_2\le Lh\sqrt d/2=:\Delta.
\]
The acceleration bound below proves the upper bound in
\cref{cor:general-moderate}; together with \cref{thm:main}, it is tight
throughout $LR^2/d^2\le\eps\le LR^2/32$.  The cutting-plane bound gives
a smaller upper bound for the general class at sufficiently high
accuracy, where a matching dependence on dimension and accuracy remains
open.

\emph{Finite-difference acceleration, $O(d\sqrt{LR^2/\eps})$ values.}
On $Q=\mathbb B_2^d(R)$, which contains a minimizer, the pair
$(f(x)-2R\Delta,\,g_h(x))$ is a $(\delta,L)$-oracle in the sense of
\citet{DevolderGlineurNesterov2014} with $\delta=4R\Delta$, because
$|\langle g_h(x)-\nabla f(x),y-x\rangle|\le2R\Delta$ for $x,y\in Q$.
With the Euclidean prox-function $\|x\|_2^2/2$, their fast gradient
method satisfies
\[
 f(y_N)-\min f\le\frac{2LR^2}{N(N+1)}+\frac{N+2}{3}\delta.
\]
Choose $N=\lceil2\sqrt{LR^2/\eps}\rceil$ and $h=\eps/(4NLR\sqrt d)$.
Then $\delta=\eps/(2N)$, each displayed error term is at most $\eps/2$,
and $(2d+1)N=O(d\sqrt{LR^2/\eps})$ values suffice, including the value
at each current iterate.  These fixed finite choices define an
admissible deterministic policy.

\emph{Finite-difference cutting planes, $O(d^2\log(LR^2/\eps))$ values.}
Run the center-of-gravity method \citep[Section~2.1]{Bubeck2015} on
$X_0=\mathbb B_2^d(R)$: at step $k$, compute the centroid $x_k$ of
$X_{k-1}$, query its value, and compute $g_h(x_k)$.  If $g_h(x_k)=0$,
then $\|\nabla f(x_k)\|_2\le\Delta$, so convexity and
$\|x_k-x^\star\|_2\le2R$ give $f(x_k)-\min f\le2R\Delta\le\eps/2$, and
the method stops and returns $x_k$; otherwise, set
$X_k=X_{k-1}\cap\{y:\langle g_h(x_k),y-x_k\rangle\le0\}$.
The halfspace used in this intersection has a nonzero normal, and its
boundary passes through the centroid, so Gr\"unbaum's inequality gives
$\operatorname{vol}(X_N)\le(1-1/e)^N\operatorname{vol}(X_0)$.  Let
$x^\star$ be a minimizer in $X_0$ and set $\alpha=\eps/(4LR^2)$, which
is at most $1/128$ because $\eps\le LR^2/32$.  Since
$\max_{X_0}f-\min f\le2LR^2$, convexity on
$X_\alpha=x^\star+\alpha(X_0-x^\star)\subseteq X_0$ gives
$f\le\min f+2\alpha LR^2=\min f+\eps/2$ there, and
$\operatorname{vol}(X_\alpha)=\alpha^d\operatorname{vol}(X_0)$.  Once
$(1-1/e)^N<\alpha^d$, some $y\in X_\alpha$ was removed at a step $k$,
so $\langle g_h(x_k),y-x_k\rangle>0$ and
\[
 f(x_k)\le f(y)-\langle\nabla f(x_k),y-x_k\rangle
 \le f(y)+\|g_h(x_k)-\nabla f(x_k)\|_2\|y-x_k\|_2
 \le\min f+\eps/2+2R\Delta.
\]
Since $2R\Delta\le\eps/2$, the best queried point is an
$\eps$-minimizer after $N=O(d\log(1/\alpha))=O(d\log(LR^2/\eps))$ steps
of $2d+1$ values each.  The same guarantee holds if the method stops
early.  Computing centroids is allowed, since the model does not
restrict computation.

\subsection{A uniform corollary for fixed \texorpdfstring{$\mu$}{mu}}
\begin{proof}[Proof of Corollary~\ref{cor:strong}]
In this proof, let $\sigma=\sqrt{LR^2/\eps}$ and
$\Pi=\min\{d,\sigma,\sqrt\kappa\}$.
The spherical family of Lemma~\ref{r12:lem-baseline} has Hessian $LI$
and therefore belongs to the prescribed strongly convex class for every
$\mu\le L$; it gives $T\ge5d/16$.  If $\Pi\le1280$, this already
implies $T\ge d\Pi/\Cstar$.

Suppose now that $\Pi>1280$, and choose
\[
 \eta=\min\{d/64,\ \sigma/100,\ \sqrt\kappa/8\},\qquad
 m=\lceil d/\eta\rceil,\qquad k=d/m.
\]
All three denominators are at most $100$, so $\eta\ge\Pi/100$.  Since
$\eta\le d/64$, the ceiling in the definition of the integer parameter
gives
\[
 m\ge64,\qquad
 \frac{64\Pi}{6500}\le\frac{64\eta}{65}
 \le\frac{\eta}{1+\eta/d}\le k\le\eta,
\]
and hence $k>12$.  Every quadratic in the supported prior of
Lemma~\ref{r12:lem-prior} satisfies
\[
 \nabla^2 f_W\succeq\frac{L}{50(k+1)^2}I
 \succeq\frac{L}{64k^2}I\succeq\mu I.
\]
Indeed, for $k\ge8$ we have $50(k+1)^2\le(2025/32)k^2<64k^2$, and
$k\le\sqrt\kappa/8$ gives $64k^2\le\kappa=L/\mu$.
Thus every member of the supported prior is globally $\mu$-strongly
convex.  The accuracy cutoff $k\le\sigma/100$ gives
$\eps\le LR^2/(10000k^2)$, so the success amplification and the
operator-budget argument apply unchanged.  In particular,
\eqref{r12:eq-calibrated}, which includes all three validation queries,
gives
\[
 T\ge\frac{dk}{432}(1+1/k)(4-9/k)^2-1
 \ge\frac{17}{648}dk-1\ge\frac{dk}{40}
 \ge\frac{64d\Pi}{260000}\ge\frac{d\Pi}{\Cstar}.
\]
The middle inequality uses $dk=mk^2\ge64k^2>9216>810$ and
$17/648-1/40=1/810$, and the factor in the first expression is at least
$34/3$ for $k\ge12$, as in the proof of Theorem~\ref{thm:main}.
Together with the dimension baseline, this proves the claimed maximum.
Finally, when $\eps\le\mu R^2$ we have $\sigma\ge\sqrt\kappa$, which
gives the stated condition-number bound.
\end{proof}

Here $\mu$-strong convexity means that the modulus is at least $\mu$.
At $\kappa=1$, the globally smooth strongly convex class has known
quadratic curvature, and $O(d)$ values suffice for exact reconstruction.
For larger $\kappa$, this lower bound combines with the multiscale bound
by taking their maximum.

\section{A nonquadratic obstruction at high accuracy}\label{app:high-accuracy}
The quadratic reconstruction bound of \cref{cor:quadratic} relies on a
finite parameterization of the objective.  The following construction
makes the distinction from the full smooth class explicit.  Its logarithmic term combines with \cref{cor:strong} by taking the
maximum of the two lower bounds.  Classical information-based complexity
already studies logarithmic dependence on the accuracy
\citep{NemirovskiYudin1983}; our purpose here is an explicit
construction in the unrestricted exact-value model.

\begin{theorem}[Fixed-dimensional high-accuracy bound]\label{thm:nonquadratic-log}
Let $d\ge1$, $0<\mu<L$, $R>0$, and
$\rho=(L-\mu)/(128(L+\mu))$.
For every integer $m\ge1$ and every
$0<\eps\le2\mu R^2\rho^{2m}$, an algorithm with the success convention of
\cref{thm:main}, on the globally $\mu$-strongly convex subclass of
$\mathcal C_{d,L,R}$, requires $T\ge m/6$ queries.
Thus, for fixed $\kappa=L/\mu>1$, its complexity is
$\Omega_\kappa(\log(\mu R^2/\eps))$ as $\eps\downarrow0$.
\end{theorem}

\begin{proof}
First let $d=1$.  Put $a=(L+\mu)/2$ and
$\alpha=(L-\mu)/(32(L+\mu))=4\rho$.  For $s\in[0,1]$, let
$q(s)=1-3s^2+2s^3$, and define
\[
 \eta(t)=\begin{cases}
 1,&|t|\le1/2,\\
 q(2|t|-1),&1/2<|t|<1,\\
 0,&|t|\ge1,
 \end{cases}
 \qquad \psi(t)=t\eta(t).
\]
Since $q'(0)=q'(1)=0$, the first derivative of $\psi$ is continuous and
Lipschitz.  Almost everywhere, $|\eta'|\le3$, $|\eta''|\le24$, and
$|\psi''|\le30$; the last bound follows from $\psi''=2\eta'+t\eta''$ on
the transition intervals, and outside them the second derivative
vanishes almost everywhere.

For $\sigma=(\sigma_1,\ldots,\sigma_m)\in\{-1,1\}^m$, set
\[
 c_0=0,\quad r_j=R\rho^j,\quad
 c_j=c_{j-1}-\sigma_j\alpha r_{j-1},\quad
 I_j=[c_j-r_j,c_j+r_j],
\]
and define
\begin{equation}\label{eq:nested-hard-family}
 f_\sigma(x)=\frac a2x^2+
 \sum_{j=1}^m\sigma_j a\alpha r_{j-1}^2
       \psi\!\left(\frac{x-c_{j-1}}{r_{j-1}}\right).
\end{equation}
Each modification is supported in its parent interval $I_{j-1}$.
The two possible child intervals are disjoint because $\rho<\alpha$,
and they lie in the central half of their parent because
$\alpha+\rho<1/2$.  On $I_j$, the base quadratic plus the first $j$
modifications equals $a(x-c_j)^2/2$ plus a constant: the $j$th
modification is linear there and shifts the zero of the derivative from
$c_{j-1}$ to $c_j$.

At almost every point, at most one modification has a nonzero second
derivative: the transition region of a modification lies outside the
central half of its parent, whereas every later modification is
supported inside that half, and earlier modifications are linear on the
support of a later one.  Hence
\[
 |f_\sigma''-a|\le30a\alpha
 =\frac{30}{64}(L-\mu)\le\frac{L-\mu}{2}
 \quad\text{almost everywhere}.
\]
Since the derivative is locally absolutely continuous, integration
proves global $\mu$-strong convexity and $L$-smoothness.  The unique
minimizer is $c_m$, with $|c_m|\le\alpha R/(1-\rho)<R$, and all
modifications vanish outside $[-R,R]$.

If two sign sequences first differ at position $k$, their minimizers
satisfy
\[
 |c_m(\sigma)-c_m(\sigma')|
 \ge2\alpha r_{k-1}\frac{1-2\rho}{1-\rho}
 >\alpha r_{m-1}.
\]
By strong convexity, every successful output lies within
$\sqrt{2\eps/\mu}\le2R\rho^m=\alpha r_{m-1}/2$ of its minimizer.
The sets of successful outputs for different sign sequences are
therefore disjoint, including at the endpoint of the stated accuracy
window.

Now draw the signs independently and uniformly before drawing the
algorithm's seed, and consider an enhanced oracle that reveals a prefix
of the sign sequence and can simulate every exact value.  If $j<m$
signs are known and a query $x$ lies outside $I_j$, all remaining
modifications vanish at $x$.  If $x\in I_j$, reveal $\sigma_{j+1}$ and
continue until $x$ leaves the next child interval or all signs are
known.  The revealed modifications can then be evaluated, and the
unrevealed ones vanish at $x$, so this rule determines $f_\sigma(x)$
exactly.  Using closed intervals only makes the enhanced oracle more
informative at endpoints.

Conditional on the seed and the enhanced history, the unrevealed signs
remain independent and uniform: at every disclosure, both the decision
to continue and the numerical answer depend only on the prefix already
disclosed, so no condition is imposed on the remaining suffix.
A query can belong to at most one of the two disjoint children, so its
conditional probability of reaching each further level is at most
$1/2$, and its expected number of new signs is at most
$\sum_{k\ge0}2^{-k}=2$.

Let $J_T$ be the number of disclosed signs after $T$ queries.  Then
$\E J_T\le2T$ and $\Prob\{J_T=m\}\le2T/m$.
If $J_T<m$, at least two equally likely sign sequences remain, and
their sets of successful outputs are disjoint, so any
history-measurable output succeeds with conditional probability at most
$1/2$.  Therefore,
\[
 \Prob\{\text{success}\}
 \le\frac12+\frac12\Prob\{J_T=m\}
 \le\frac12+\frac Tm.
\]
The required instancewise success probability implies average success
probability at least $2/3$ under this prior, proving $T\ge m/6$.
All query comparisons and polynomial evaluations are Borel, and
arbitrary queries outside the support disclose nothing further.

For $d>1$, use $f_\sigma(x_1)+(\mu/2)\sum_{i=2}^d x_i^2$.  One source
value simulates a target value, the minimizer remains in the ball of
radius $R$, and target suboptimality bounds source suboptimality.
Finally, for $\eps\le2\mu R^2\rho^2$, choose
$m=\lfloor\log(2\mu R^2/\eps)/(2\log(1/\rho))\rfloor$.  It satisfies
the stated window and gives the asymptotic conclusion.
\end{proof}

\section{Multiscale lower bounds and matching fine-accuracy complexity}\label{app:multiscale}
We prove \cref{thm:near-quadratic} and the dimension-times-logarithm
lower bound by a nested smooth construction. The class, exact-value oracle, success probability,
and deterministic worst-case budget are those of \cref{sec:results}.
A nonnegative modulus $\mu<L$ is prescribed; at $\mu=0$ only convexity is
required. The hidden directions are sampled before the independent
algorithmic seed. Every query and the output may be any finite vector.
The parameters $a,\alpha,\rho$ specify the curvature, perturbation
amplitude, and scale contraction in this construction.

\subsection{The smooth hard family}

\begin{theorem}[Smooth nested perturbations]\label{ms:thm-construction}
Suppose $0\le\mu<L$, and set

\[
a=\frac{L+\mu}{2},\qquad
\alpha=\frac{L-\mu}{32(L+\mu)}.
\]

Choose $0<\rho\le\alpha/4$, an integer $m\ge1$, an initial radius $r_0>0$, and vectors $\xi_1,\ldots,\xi_m$ of norm at most one. Define

\[
r_j=r_0\rho^j,\qquad c_0=0,\qquad
c_{j+1}=c_j-\alpha r_j\xi_{j+1}.
\]

There is an explicit fixed $C^\infty$ cutoff $\phi$ such that

\[
f(x)=\frac a2\|x\|^2+
\sum_{j=0}^{m-1}a\alpha r_j
\langle\xi_{j+1},x-c_j\rangle
\phi\left(\frac{\|x-c_j\|}{r_j}\right)
\]

is globally $C^\infty$ and satisfies

\[
\frac{L+63\mu}{64}I\preceq\nabla^2f(x)
\preceq\frac{63L+\mu}{64}I.
\]

Its unique minimizer is $c_m$. On the last ball,

\[
f(x)-f(c_m)=\frac a2\|x-c_m\|^2,
\qquad x\in B(c_m,r_m).
\]

If $\alpha\sum_{j=0}^{m-1}r_j\le R$, this is an admissible instance. Moreover,

\[
0<\eps\le\frac{a r_m^2}{4}
\]

implies that every successful output belongs to the open last ball and to every preceding selected child ball.

\end{theorem}
\begin{proof}
 Start with the even $C^{1,1}$ function

\[
\eta(t)=
\begin{cases}
1,&|t|\le1/2,\\
1-3(2|t|-1)^2+2(2|t|-1)^3,&1/2<|t|<1,\\
0,&|t|\ge1.
\end{cases}
\]

On the transition interval, writing $u=2|t|-1$, its first derivative has magnitude $12u(1-u)\le3$, and its almost-everywhere second derivative has magnitude at most $24$. The values and first derivatives agree at the two endpoints. Thus

\[
\|\eta'\|_\infty\le3,\qquad
\|\eta''\|_\infty\le24.
\]

For definiteness, take the even nonnegative density

\[
\theta(s)=C\exp\left(-\frac1{1-s^2}\right)\mathbf{1}_{\{|s|<1\}},
\qquad \int_{\R}\theta(s)\,ds=1.
\]

Here $C>0$ is the normalizing constant, $\theta\in C_c^\infty(\R)$, and its support is $[-1,1]$. Define

\[
\phi(t)=\int_{-1}^{1}\eta\left(\frac{101t-s}{100}\right)\theta(s)\,ds.
\]

This is a smooth convolution followed by a dilation. It is even, equals one when $|t|\le49/101$, and vanishes when $|t|\ge1$. Differentiating the convolution, with the second derivative of $\eta$ interpreted almost everywhere, gives

\[
\|\phi'\|_\infty\le\frac{303}{100},\qquad
\|\phi''\|_\infty\le24\left(\frac{101}{100}\right)^2
=\frac{30603}{1250}.
\]

The map $z\mapsto\phi(\|z\|)$ is smooth at zero because it is constant near zero. For a unit vector $\xi$, put $h(z)=\langle\xi,z\rangle\phi(\|z\|)$. At $t=\|z\|>0$, with $e=z/t$, direct differentiation yields

\[
\nabla^2h(z)=\phi'(t)(\xi e^\top+e\xi^\top)
+\langle\xi,z\rangle
\left[\phi''(t)ee^\top+
\frac{\phi'(t)}t(I-ee^\top)\right].
\]

The two projectors in brackets have orthogonal ranges. Since $|\langle\xi,z\rangle|\le t$, for $0<t\le1$ we obtain

\[
\begin{aligned}
\|\nabla^2h(z)\|_{\mathrm{op}}
&\le2|\phi'(t)|+\max\{t|\phi''(t)|,|\phi'(t)|\}\\
&\le2\frac{303}{100}+\frac{30603}{1250}
=\frac{19089}{625}<31.
\end{aligned}
\]

For a vector of norm at most one the same bound holds by linearity in $\xi$. The Hessian vanishes on $t\le49/101$ and on $t\ge1$. Each perturbation in $f$ equals $a\alpha r_j^2h((x-c_j)/r_j)$, so its Hessian has norm less than $31a\alpha$ and is supported only in its transition annulus.

Every child ball lies strictly inside the parent's linear plateau, because

\[
B(c_{j+1},r_{j+1})\subset
B(c_j,(\alpha+\rho)r_j),\qquad
\alpha+\rho\le\frac{5\alpha}{4}\le\frac5{128}<\frac{49}{101}.
\]

Consequently all later supports lie inside the current plateau, and the transition annuli are pairwise disjoint. At any point at most one perturbation has a nonzero Hessian. Therefore

\[
(a-31a\alpha)I\preceq\nabla^2f(x)\preceq(a+31a\alpha)I.
\]

Substitution gives the stated bounds. In particular, the constructed functions remain strictly strongly convex even if the allowed class has $\mu=0$.

Inside $B(c_j,r_j)$ all earlier cutoffs equal one. The gradient of the base quadratic plus its first $j$ perturbations is

\[
ax+a\alpha\sum_{i=0}^{j-1}r_i\xi_{i+1}=a(x-c_j).
\]

For $j=m$ this proves stationarity at $c_m$ and the exact quadratic identity throughout the last ball. Global strong convexity proves uniqueness. Also

\[
\|c_m\|\le\alpha\sum_{j=0}^{m-1}r_j.
\]

At the boundary of the last ball the gap is $a r_m^2/2$, and the outward radial derivative is $a r_m>0$. Along every ray from $c_m$, convexity makes this derivative nondecreasing. Thus outside that ball the gap is at least $a r_m^2/2$. An output with gap at most $a r_m^2/4$ must be inside it, where its distance to $c_m$ is at most $r_m/\sqrt2$. Strict nesting gives all preceding memberships. This proves the theorem.
\end{proof}

\begin{corollary}[Radius normalization and admissible depth]\label{ms:cor-radius}
All subsequent information arguments are valid with either of the following radius normalizations, using the corresponding window:

\[
\begin{array}{c|c}
r_0&\text{sufficient upper bound on }\eps\\
R&(L+\mu)R^2\rho^{2m}/8\\
R(1-\rho)/\alpha&
\displaystyle\frac{(L+\mu)R^2(1-\rho)^2}{8\alpha^2}\rho^{2m}.
\end{array}
\]

For the first choice, $\|c_m\|\le\alpha R/(1-\rho)<R$. For the second,

\[
\|c_m\|\le R(1-\rho^m)<R.
\]

In the second normalization, two useful spherical choices are

\[
\begin{array}{c|c|c}
\rho& s=(\rho/\alpha)^2&\text{sufficient upper bound on }\eps\\
\alpha/16&1/256&
\displaystyle\frac{(L+\mu)R^2(1-\rho)^2}{2048}\rho^{2(m-1)}\\
\alpha/4&1/16&
\displaystyle\frac{(L+\mu)R^2(1-\rho)^2}{128}\rho^{2(m-1)}.
\end{array}
\]

These formulas follow by substituting $r_m$ into \cref{ms:thm-construction}. In a window $\eps\le C\rho^{2(m-1)}$, one may take

\[
m=1+\left\lfloor\frac{\log(C/\eps)}{2\log(1/\rho)}\right\rfloor
\quad\text{whenever }0<\eps\le C.
\]

The initial support may exceed $B_2^d(R)$, while the minimizer lies in
the prescribed ball. All radii are finite for $\mu<L$, the parameter
range of this construction. The endpoint $\mu=L$ is treated separately.
\end{corollary}

\subsection{A first-hit argument that counts cascades}
The following spherical recovery estimate bounds the probability of
a first entry into a hidden child.

\begin{lemma}[Adaptive spherical recovery]\label{ms:lem-recovery}
Let $\xi$ be uniform on $S^{d-1}$. After at most $t<d$ adaptive scalar
linear measurements, any Borel estimator, with independent randomization
allowed, has probability at most
\[
 \frac12+\frac{t}{2d(1-s)}
\]
of recovering $\xi$ with squared error at most $s\in(0,1)$.
\end{lemma}
\begin{proof}
Orthogonalize the informative queries and supplement dependent queries by a
measurable orthogonal completion. Conditional on the transcript, the
unobserved component is uniform on the residual sphere. If $U_i$ is its
squared radius after $i$ informative measurements, rotational invariance gives
$\E[U_{i+1}\mid\mathcal H_i]=U_i(d-i-1)/(d-i)$. Hence
$\E U_t=(d-t)/d$. When $U_t>s$, a radius-$\sqrt s$ ball contains at most one
point from each antipodal pair of the residual sphere, and conditional
success is at most $1/2$. Otherwise the observed squared norm is at least
$1-s$. Markov's inequality bounds the probability of the latter event by
$t/[d(1-s)]$. Averaging proves the claim. Orthogonal completion can test the
standard coordinate vectors in their fixed order, so all maps are Borel.
The antipodal argument includes the two-point sphere.
\end{proof}

\begin{lemma}[First entry into a hidden child]\label{ms:lem-hit}
In \cref{ms:thm-construction}, let the directions be independent and uniform
on $S^{d-1}$, and put $s=(\rho/\alpha)^2\le1/16$. There is an oracle more
informative than exact values for which the number $N_j$ of queries
processed at level $j$, including its first child hit, obeys
\[
 \Prob\{N_j\le k\mid\text{causal history before level }j\}
 \le\frac12+\frac{k-1}{2d(1-s)},\qquad 1\le k\le d.
\]
\end{lemma}
\begin{proof}
At the first unrevealed level $j$, a query $x$ receives
$\langle\xi_j,x-c_{j-1}\rangle$. If $x$ is outside the child
$B(c_j,r_j)$, processing stops. If it is inside, the oracle reveals the
whole direction and processes the same point at the next level. The query
may therefore cascade through several levels. Child membership is already
determined by the supplied linear measurement, since
\[
 \|x-c_j\|_2^2=\|x-c_{j-1}\|_2^2+\alpha^2r_{j-1}^2
       +2\alpha r_{j-1}\langle\xi_j,x-c_{j-1}\rangle.
\]
The response determines the original function value: revealed perturbations
are known, the current perturbation is computed from its projection, and
all deeper perturbations vanish where processing stops. The smooth cutoff
also vanishes on a support boundary. Thus no extra nonlinear measurement
is hidden in the child indicator.

At the instant a new level is first processed, its direction is independent
of the causal past. The causal past ends immediately before the level's first internal
oracle operation, including when that operation occurs within a cascade. Simulate the
first $k-1$ measurements at that level. If an earlier child entry occurs,
its measurement identifies the entry; return
$-(x-c_{j-1})/(\alpha r_{j-1})$ and stop before the full-direction disclosure.
If no entry occurs, return the same candidate formed from the $k$th query
point without requesting its answer. On $\{N_j\le k\}$ this recovers $\xi_j$
with squared error less than $s$. \Cref{ms:lem-recovery} applies with $t=k-1$.
The recovery estimate applies directly to this stopped experiment
under the direction's original conditional law.
\end{proof}

\begin{theorem}[Dimension times admissible depth]\label{ms:thm-depth}
Use the spherical family above at a depth $m$ allowed by
\cref{ms:cor-radius}. Every algorithm successful on each member of the
function class must satisfy
\begin{equation}\label{ms:eq-depth}
 T\ge\frac{3m}{20}\left(1+\left\lfloor\frac d8\right\rfloor\right)
 \ge\frac{3dm}{160}.
\end{equation}
\end{theorem}
\begin{proof}
Give the algorithm access to the enhanced oracle and add its output as one
final query. Success implies entry into the last ball, so all $m$ levels
are finished within $T+1$ queries. For analysis, continue after this budget
until every level is finished, using coordinate measurements and a query at
the recovered child center. This makes every $N_j$ finite without altering
success within the original budget.

Let $Q$ count queries until completion. A query is counted more than once
in $\sum_jN_j$ only when it crosses a level boundary, and at most $m-1$
boundaries can be crossed. Consequently
$Q\ge1+\sum_{j=1}^m(N_j-1)$. On success, $Q\le T+1$, and hence
\[
 T\ge k\sum_{j=1}^m\mathbf1_{\{N_j>k\}},
 \qquad k=1+\lfloor d/8\rfloor.
\]
By \cref{ms:lem-hit}, conditionally on the causal past of each level,
\[
 \Prob\{N_j\le k\}\le\frac12+\frac{8(k-1)}{15d}\le\frac{17}{30}.
\]
Let $B$ count levels with $N_j\le k$. Then $\E B\le17m/30$. On success,
$B\ge m-T/k$. If $T<km$, the pointwise success guarantee gives average
success at least $2/3$ under this fixed prior, so
\[
 \frac{17m}{30}\ge\frac23(m-T/k),\qquad T\ge3km/20.
\]
For $T\ge km$ the same conclusion is immediate. Finally $k\ge d/8$.
The argument permits cascades and arbitrary adaptive randomized policies.
\end{proof}

For $\rho=\alpha/4=(L-\mu)/(128(L+\mu))$, the rescaled construction permits
\begin{equation}\label{ms:eq-admissible-depth}
 m=1+\left\lfloor
 \frac{\log((L+\mu)R^2(1-\rho)^2/(128\eps))}
 {2\log(128(L+\mu)/(L-\mu))}\right\rfloor
\end{equation}
whenever $0<\eps\le(L+\mu)R^2(1-\rho)^2/128$.
Thus $T=\Omega_{\mu/L}(d\log(LR^2/\eps))$ for each fixed
$0\le\mu/L<1$ at sufficiently small relative accuracy. A single continuous multiscale family with independently sampled
spherical directions yields this lower bound.

\subsection{A matching value-only upper bound}

\begin{theorem}[Midpoint-curvature finite differences]\label{ms:thm-upper}
For $0<\mu<L$ and $0<\eps<LR^2/2$, every function in the full $C^{1,1}$ class can be minimized to gap at most $\eps$ deterministically using at most

\[
(d+1)\left\lceil
\frac{\log(2LR^2/\eps)}
{2\log((L+\mu)/(L-\mu))}
\right\rceil
\]

values. The budget is fixed in advance, and the policies are total Borel maps.

\end{theorem}
\begin{proof}
 In this proof set

\[
a=\frac{L+\mu}{2},\qquad b=\frac{L-\mu}{2},\qquad
\delta=\frac ba,\qquad g(x)=f(x)-\frac a2\|x\|^2.
\]

First, $\nabla g$ is globally $b$-Lipschitz even though $f$ need not be $C^2$. Convolve $f$ with a compactly supported nonnegative smooth kernel of mass one and scale tending to zero. The convolutions preserve strong convexity by averaging its defining inequality, and preserve the Lipschitz gradient bound by averaging gradient differences. Their Hessians therefore have eigenvalues in $[\mu,L]$. Subtracting $aI$ gives operator norm at most $b$. Integrating along a segment and passing to the locally uniform limit of the gradients proves the assertion for $f$. The convolution step justifies the segment argument for every pair of points.

For any fixed $h>0$, obtain $d+1$ values at $x,x+he_1,\ldots,x+he_d$ and set

\[
D_i(x)=\frac{g(x+he_i)-g(x)}h
=\frac{f(x+he_i)-f(x)}h-ax_i-\frac{ah}{2}.
\]

The unknown additive constant cancels. The fundamental theorem of calculus yields

\[
|D_i(x)-\partial_i g(x)|\le\frac1h\int_0^hbt\,dt
=\frac{bh}{2},\qquad
\|D(x)-\nabla g(x)\|\le\frac{bh\sqrt d}{2}.
\]

The map $T(x)=-\nabla g(x)/a$ is a $\delta$-contraction and fixes $x^\star$. Thus the update $x_{j+1}=-D(x_j)/a$ satisfies

\[
\|x_{j+1}-x^\star\|\le\delta\|x_j-x^\star\|+\frac{\delta h\sqrt d}{2}.
\]

Choose

\[
r=\sqrt{\frac{2\eps}{L}},\qquad
h=\frac{(1-\delta)r}{\delta\sqrt d},\qquad x_0=0,
\]

\[
n=\left\lceil\frac{\log(2R/r)}{\log(1/\delta)}\right\rceil
=\left\lceil\frac{\log(2LR^2/\eps)}{2\log(1/\delta)}\right\rceil.
\]

After exactly $n$ iterations,

\[
\|x_n-x^\star\|\le\delta^nR+
\frac{\delta h\sqrt d}{2}\sum_{j=0}^{n-1}\delta^j
=\delta^nR+\frac r2(1-\delta^n)\le r.
\]

Smoothness then gives $f(x_n)-f(x^\star)\le L\|x_n-x^\star\|^2/2\le\eps$. No value at the final output is needed.

All constants and the iteration count are fixed before the first answer. Arithmetic operations divide only by fixed positive quantities, so every finite real history produces finite query points and an output. The maps are Borel even on inconsistent histories. A large finite-difference step when $\delta$ is very small is allowed because the smoothness bound is global and query norms are unrestricted. A small step does not introduce rounding error in the exact-value model.

\end{proof}

If $\eps\ge LR^2/2$, zero is an adequate output without queries.
If $\mu=L$, the upper and lower quadratic increment inequalities coincide:
\[
 f(x)=\tfrac L2\|x\|_2^2+\langle\nabla f(0),x\rangle+f(0).
\]
The $d+1$ values at $0,e_1,\ldots,e_d$ therefore recover the minimizer exactly.
The dimension lower bound must be treated separately at this endpoint;
the nested construction is defined only for $\mu<L$.

\subsection{Proof of the matching endpoint estimate}\label{ms:proof-matching}
\begin{proof}[Proof of \cref{thm:near-quadratic}]
 For $\mu<L$, write $v=\log(LR^2/\eps)$ and $b=\log((L+\mu)/(L-\mu))\ge\log3$. In the rescaled spherical construction with $\rho=\alpha/4$, we have $\rho=e^{-b}/128\le1/384$ and

\[
(L+\mu)(1-\rho)^2\ge L.
\]

The admissible depth from \cref{ms:cor-radius} satisfies

\[
m\ge\max\left\{1,\frac{v}{24b}\right\}
\ge\frac{1+v/b}{25}.
\]

Indeed, the numerator defining the depth is at least $v-\log128\ge v/2$, and $\log128<5b$, so its denominator is at most $12b$. Also $1+\lfloor t\rfloor\ge t$. The last displayed inequality follows by splitting at $v/b=24$. \cref{ms:thm-depth} gives the claimed lower bound. For the upper bound, \cref{ms:thm-upper} and $v\ge\log2$ imply

\[
\Tval_\eps(\mathcal C_{d,L,R}^{\mu})\le(d+1)\left\lceil\frac{v+\log2}{2b}\right\rceil
\le2d(1+v/b).
\]

At $\mu=L$, the $d+1$-value reconstruction stated after \cref{ms:thm-upper} applies. For the dimension lower bound use the fixed spherical family

\[
f_\xi(x)=\frac L2\|x\|^2-LR\langle\xi,x\rangle,
\qquad \xi\sim\operatorname{Unif}(S^{d-1}).
\]

Each value is one linear measurement of $\xi$, and success implies recovery with squared error $s=2\eps/(LR^2)\le1/16$. If fewer than $d$ queries are used, \cref{ms:lem-recovery} and success $2/3$ require $T\ge d(1-s)/3\ge5d/16$; with at least $d$ queries this bound is immediate. This exceeds $3d/4000$ and completes the endpoint argument without substituting zero into the multiscale construction.
\end{proof}

\section{Saddle-point models, blockwise and mixed information bounds}
\label{sp:app}

\subsection{Models and blockwise information bounds}
Let $d_x,d_y\ge1$, $L_i,R_i>0$, and $0\le\mu_i\le L_i$ for
$i\in\{x,y\}$. The objective $F:\R^{d_x}\times\R^{d_y}\to\R$ is globally
differentiable, $\mu_x$-strongly convex in $x$ and $\mu_y$-strongly
concave in $y$; a zero modulus means ordinary convexity in $x$,
respectively ordinary concavity in $y$. The partial gradients satisfy
\begin{align}
 \|\nabla_xF(x,y)-\nabla_xF(x',y)\|_2&\le L_x\|x-x'\|_2,\notag\\
 \|\nabla_yF(x,y)-\nabla_yF(x,y')\|_2&\le L_y\|y-y'\|_2,\label{sp:eq-block-smooth}\\
 \|\nabla_xF(x,y)-\nabla_xF(x,y')\|_2&\le L_{xy}\|y-y'\|_2,\notag\\
 \|\nabla_yF(x,y)-\nabla_yF(x',y)\|_2&\le L_{xy}\|x-x'\|_2.\notag
\end{align}
The smoothness and coupling constants are upper bounds, whereas the
strong-convexity and strong-concavity moduli are lower bounds; in
particular, the class permits zero coupling even when $L_{xy}>0$. We
require a global saddle point $(x^\star,y^\star)$ with
$\|x^\star\|_2\le R_x$ and $\|y^\star\|_2\le R_y$.

The localized model takes $X=R_x\mathbb B_2^{d_x}$ and
$Y=R_y\mathbb B_2^{d_y}$, requires an output in $X\times Y$, and measures
the gap
\begin{equation}\label{sp:eq-gap}
 \operatorname{Gap}_{X,Y}(\widehat x,\widehat y)
 =\sup_{y\in Y}F(\widehat x,y)-\inf_{x\in X}F(x,\widehat y).
\end{equation}
The global-gap model takes $X=\R^{d_x}$, $Y=\R^{d_y}$ and places no
restriction on the output; the radii then serve only to localize a saddle
point. The two models impose different output guarantees on the same class
of objectives. We give the corresponding gap comparisons separately. Queries are unrestricted in both
models. As in \cref{sec:results}, policies are total Borel maps, seeds are
independent, and an $\eps$-accurate output must be returned with
probability at least $2/3$ on every instance. The budget of scalar-value
queries must be respected on every history.

For $0<\eps\le L_iR_i^2/32$, put
\begin{equation}\label{sp:eq-block-budget}
 \mathcal B_i(\eps)=\max\left\{\frac{5d_i}{16},\;
 \frac{d_i}{\Cstar}\min\left\{d_i,
 \sqrt{\frac{L_i}{\mu_i}},\sqrt{\frac{L_iR_i^2}{\eps}}\right\}\right\},
\end{equation}
with the convention $\sqrt{L_i/\mu_i}=+\infty$ when $\mu_i=0$.
An index is called admissible when the required accuracy lies in the
corresponding window.

\begin{theorem}[Blockwise scalar-value lower bound]\label{sp:thm-blocks}
In either gap model, an algorithm that receives only exact scalar values
of $F$ must satisfy $T\ge\mathcal B_i(\eps)$ for each admissible index $i$.
If both indices are admissible, then
\[
 T\ge\max\{\mathcal B_x(\eps),\mathcal B_y(\eps)\}
 \ge\tfrac12\bigl(\mathcal B_x(\eps)+\mathcal B_y(\eps)\bigr).
\]
The result holds for every $L_{xy}\ge0$, even on the subclass of separable
quadratics.
\end{theorem}
\begin{proof}
We first prove the $x$ bound. Let $f$ range over the hard quadratic
subclass used in \cref{thm:main,cor:strong} with parameters
$(d_x,L_x,R_x,\mu_x)$, and set
\[
 F_f(x,y)=f(x)-\tfrac{\mu_y}{2}\|y\|_2^2.
\]
When $\mu_y=0$, the second term vanishes. This function belongs to the
saddle-point class and has the saddle point $(x^\star,0)$. One value of
$F_f$ is simulated by one value of $f$ together with the known quadratic
term. In the localized model, $0\in Y$ and, by construction,
$\operatorname*{argmin}f$ meets $X$, so the inner supremum and the inner
infimum are attained at the same points as in the global model. Hence, for
every feasible output in the localized model and every output in the
global model,
\[
  \operatorname{Gap}_{X,Y}(\widehat x,\widehat y)
 =f(\widehat x)-\min f+\tfrac{\mu_y}{2}\|\widehat y\|_2^2
 \ge f(\widehat x)-\min f.
\]
Thus the decoded output $\widehat x$ solves the source problem with the
same budget and success probability, and the corresponding minimization
theorem applies. For the $y$ bound, use
$F_g(x,y)=\mu_x\|x\|_2^2/2-g(y)$ and decode $\widehat y$.
The two arguments need not use the same hard instance: each of them
constrains the same worst-case query budget.
\end{proof}

\begin{theorem}[A scalar-value lower bound in the mixed model]\label{sp:thm-mixed}
Suppose that arbitrary exact partial-gradient queries $\nabla_xF(x,y)$ are
available in addition to exact scalar values. Let $N_0$ count every
returned scalar value, including the values bundled with a partial-gradient
call. Then $N_0\ge\mathcal B_y(\eps)$ whenever the $y$ index is admissible,
regardless of the number of pure partial-gradient calls. With $x$ and $y$
interchanged, $N_0\ge\mathcal B_x(\eps)$ holds in the other mixed model.
\end{theorem}
\begin{proof}
Take $F_g(x,y)=\mu_x\|x\|_2^2/2-g(y)$. Every partial-gradient response
$\nabla_xF_g(x,y)=\mu_xx$ is independent of $g$ and is computed without any
source query, whereas each scalar response requires one value of $g$. The
same gap identity as before proves the claim. The simulation is adaptive
in both types of response and reuses the algorithm's own seed, so the
induced source policy is again a total Borel map with the same budget.
\end{proof}

\begin{remark}[Outer accuracy and partial-gradient information]\label{sp:rem-models}
The mixed-oracle theorem concerns a returned pair with a small gap. On its
hard family, $\Phi_g(x)=\sup_yF_g(x,y)=\mu_x\|x\|_2^2/2-\min g$, so $x=0$
minimizes $\Phi_g$ without any observation. The coupled construction below makes the outer solution depend on
the unknown inner solution. In a bilinear model, an observed partial
gradient contains $A^\top y$ or $Ax$ and thus reveals coupling information. In the mixed-oracle construction, the coupling is known
and the unknown inner objective contains the hidden information.
\end{remark}

\subsection{A coupled construction for mixed information and outer accuracy}
The hard family of the preceding mixed-oracle reduction is separable. The
next construction uses a nonzero known coupling, so that the unknown inner
solution determines the outer solution while the partial-gradient response
remains independent of the unknown objective.

\begin{theorem}[Known coupling with an unknown inner quadratic]
\label{sp:thm-mixed-coupled}
Let $\alpha,c,\mu_y,R_x,R_y>0$ and $L_y\ge\mu_y$. Set
\[
 r=\min\{d_x,d_y\},\qquad s=c^2/\alpha,\qquad
 L_0=L_y+s,\qquad \mu_0=\mu_y+s,\qquad
 \rho=\min\{R_y,\alpha R_x/c\}.
\]
Consider globally defined strongly convex--strongly concave quadratics
with block constants $L_x=\mu_x=\alpha$, $L_y,\mu_y$, coupling norm at
most $c$, and a saddle point localized by $R_x,R_y$. The algorithm may
query $\nabla_xF$ arbitrarily, and $N_0$ counts all returned scalar
values. For the global pair-gap criterion and $0<\eps\le L_0\rho^2/32$,
\begin{equation}\label{sp:eq-mixed-coupled-gap}
 N_0\ge\max\left\{\frac{5r}{16},\frac r{\Cstar}
 \min\left\{r,\sqrt{\frac{L_0}{\mu_0}},
                   \sqrt{\frac{L_0\rho^2}{\eps}}\right\}\right\}.
\end{equation}
For an outer output satisfying $\Phi(\widehat x)-\min\Phi\le\eps$, where
$\Phi(x)=\max_{y\in\R^{d_y}}F(x,y)$, the bound is instead
\begin{equation}\label{sp:eq-mixed-coupled-outer}
 N_0\ge\max\left\{\frac{5r}{16},\frac r{\Cstar}
 \min\left\{r,\sqrt{\frac{L_0}{\mu_0}},
             \sqrt{\frac{c^2\rho^2}{\alpha\eps}}\right\}\right\},
 \qquad 0<\eps\le\frac{c^2\rho^2}{32\alpha}.
\end{equation}
Both statements require a success probability of at least $2/3$ on every
instance and hold even for a fixed known coupling matrix of rank $r$.
\end{theorem}
\begin{proof}
Fix isometries $U:\R^r\to\R^{d_x}$ and $V:\R^r\to\R^{d_y}$. Let
$q(w)=w^\top Hw/2-h^\top w+c_0$ be a source quadratic with
$\mu_0 I\preceq H\preceq L_0 I$ whose minimizer $w^\star=H^{-1}h$ has norm
at most $\rho$. Put $z=U^\top x$, $w=V^\top y$ and
$y_\perp=(I-VV^\top)y$, and define
\[
 F_q(x,y)=\frac\alpha2\|x\|_2^2+c\langle z,w\rangle-q(w)
             +\frac s2\|w\|_2^2-\frac{\mu_y}{2}\|y_\perp\|_2^2.
\]
The Hessian with respect to $y$ equals $-(H-sI)$ on
$\operatorname{range}V$ and $-\mu_y I$ on its complement, so it lies
between $-L_y I$ and $-\mu_y I$; the Hessian with respect to $x$ is
$\alpha I$, and the coupling is $cUV^\top$. The saddle point is
\[
 y^\star=Vw^\star,\qquad x^\star=-\frac c\alpha Uw^\star,
\]
so that $\|x^\star\|_2=(c/\alpha)\|w^\star\|_2\le R_x$ and
$\|y^\star\|_2=\|w^\star\|_2\le R_y$. Each scalar response requires exactly
one source value, whereas $\nabla_xF_q(x,y)=\alpha x+cUw$ is known without
querying $q$.

Eliminating $x$ gives
$\inf_xF_q(x,y)=-q(w)-\mu_y\|y_\perp\|_2^2/2$.
Since the saddle-point value is $-\min q$, the global gap dominates
$q(w)-\min q$, and \eqref{sp:eq-mixed-coupled-gap} follows from
\cref{cor:strong} with source parameters $(r,L_0,\mu_0,\rho)$.

For the outer criterion, $H-sI\succeq\mu_yI$ on $\operatorname{range}V$
and $-\nabla^2_{yy}F_q=\mu_yI$ on its orthogonal complement, so
$F_q(x,\cdot)$ is $\mu_y$-strongly concave on $\R^{d_y}$ with $\mu_y>0$.
Hence $\Phi$ is finite, its maximizer is unique, and $\Phi$ is
differentiable everywhere by Danskin's theorem. Moreover, $\Phi$ is
$\alpha$-strongly convex, being $\alpha\|x\|_2^2/2$ plus a supremum of
affine functions. Consequently,
\[
 \frac\alpha2\|\widehat x-x^\star\|_2^2
 \le\Phi(\widehat x)-\Phi(x^\star).
\]
Decode $\widehat w=-(\alpha/c)U^\top\widehat x$. Source smoothness and
$\nabla q(w^\star)=0$ give
\[
 q(\widehat w)-q(w^\star)
 \le\frac{L_0}{2}\|\widehat w-w^\star\|_2^2
 \le\frac{L_0\alpha}{c^2}
       \bigl[\Phi(\widehat x)-\Phi(x^\star)\bigr].
\]
It remains to apply the source theorem at accuracy
$\delta=L_0\alpha\eps/c^2$: its window $\delta\le L_0\rho^2/32$ is exactly
$\eps\le c^2\rho^2/(32\alpha)$, and
$\sqrt{L_0\rho^2/\delta}=\sqrt{c^2\rho^2/(\alpha\eps)}$, which gives
\eqref{sp:eq-mixed-coupled-outer}.
\end{proof}

This construction yields a scalar-value lower bound even when arbitrarily
many partial-gradient queries are available and only an outer solution is
required. The shifted condition number
$(L_y+c^2/\alpha)/(\mu_y+c^2/\alpha)$ decreases as the coupling grows. Determining a sharp coupling-dependent lower bound and a joint
trade-off between scalar-value and partial-gradient budgets remains open.

\begin{theorem}[A full-dimensional mixed lower bound with rectangular coupling]
\label{sp:thm-mixed-full}
Let $\alpha,c,\mu_y,R_x,R_y>0$, assume
$\mu_y+c^2/\alpha\le L_y$, and fix any known matrix
$K:\mathbb R^{d_y}\to\mathbb R^{d_x}$ with
$\|K\|_{\mathrm{op}}=c$. Set
\[
 s=c^2/\alpha,\qquad \mu_0=\mu_y+s,\qquad
 \rho=\min\{R_y,\alpha R_x/c\}.
\]
In the global-gap mixed model with arbitrary partial-gradient queries
$\nabla_xF$, block constants $L_x=\mu_x=\alpha$ and $L_y,\mu_y$,
known coupling $K$, and saddle-point radii $R_x,R_y$, the total number
$N_0$ of returned scalar values must satisfy
\[
 N_0\ge\max\left\{\frac{5d_y}{16},\;
 \frac{d_y}{\Cstar}\min\left\{d_y,
          \sqrt{\frac{L_y}{\mu_y+c^2/\alpha}},
          \sqrt{\frac{L_y\rho^2}{\eps}}\right\}\right\},
 \qquad 0<\eps\le L_y\rho^2/32.
\]
The partial-gradient budget is unrestricted, and no relation between
$d_x$ and $d_y$ is required. The success probability is at least $2/3$
on each instance, as in the preceding statements.
\end{theorem}
\begin{proof}
Let $q(y)=y^\top Hy/2-h^\top y+c_0$ range over the source quadratics on
$\mathbb R^{d_y}$ with $\mu_0 I\preceq H\preceq L_y I$ whose minimizer
$y^\star$ has norm at most $\rho$, and define
\[
 F_q(x,y)=\frac\alpha2\|x\|_2^2+x^\top Ky-q(y)
                         +\frac1{2\alpha}\|Ky\|_2^2.
\]
Its negative $y$-Hessian satisfies
\[
 \mu_y I\preceq H-\frac1\alpha K^\top K\preceq L_y I,
\]
and its remaining derivatives satisfy the stated block bounds.
Eliminating $x$ gives $\inf_xF_q(x,y)=-q(y)$ exactly. Hence the unique
saddle point is $(x^\star,y^\star)$ with $x^\star=-Ky^\star/\alpha$, where
$\|x^\star\|_2\le c\rho/\alpha\le R_x$ and $\|y^\star\|_2\le R_y$, and the
global gap of a returned pair dominates $q(\widehat y)-\min q$.

Each scalar response is simulated by one query to $q$, whereas
$\nabla_xF_q(x,y)=\alpha x+Ky$ is independent of $q$ and requires no
source query. The mixed procedure therefore induces an $N_0$-query source
algorithm with the same accuracy and probability of success. It remains
to apply \cref{cor:strong} in dimension $d_y$ with smoothness $L_y$,
strong convexity $\mu_0$, and radius $\rho$.
\end{proof}

This construction retains the full inner dimension even when
$d_x\ll d_y$. Its parameter regime is complementary to that of
\cref{sp:thm-mixed-coupled}: it requires $\mu_y+c^2/\alpha\le L_y$ and
uses the source smoothness $L_y$, whereas the shared-rank construction
permits any coupling and uses the source smoothness $L_y+c^2/\alpha$. The
full-dimensional assertion concerns a returned pair: an outer output alone
does not in general identify the inner solution when the known rectangular
coupling has a nontrivial kernel.

\begin{corollary}[Compact gaps when the minimizing responses are feasible]
\label{sp:cor-mixed-compact}
The pair-gap conclusions of \cref{sp:thm-mixed-coupled,sp:thm-mixed-full}
also hold for feasible outputs and the gap on the balls $X,Y$ if
$R_x\ge cR_y/\alpha$. Their source accuracy windows are unchanged.
\end{corollary}
\begin{proof}
For every feasible $y$, the unconstrained minimizing response
$x=-Ky/\alpha$ lies in $X$, where $K=cUV^\top$ for the shared-rank
construction. Hence the restricted dual function coincides with the global
dual function used in the proofs above. Since the saddle point is
feasible, the restricted primal value at any feasible $x$ is at least the
saddle-point value, and subtracting the dual value gives the same
source-gap inequality as before.
\end{proof}

\subsection{Block geometry and known regularizers}
Let $K_i$ be nonempty compact convex sets, let $\|\cdot\|_i$ be
independently chosen regularity norms with unit balls $B_i$, and let
$\psi_i:K_i\to\R$ be known continuous convex functions that are
nonnegative with minimum value zero. In the constrained composite model,
let $g_x,g_y$ be globally convex and consider
\[
 F(x,y)=g_x(x)+\lambda_x\psi_x(x)-g_y(y)-\lambda_y\psi_y(y),
 \qquad \lambda_x,\lambda_y\ge0.
\]
The regularity bounds apply only to the unknown terms; a known regularizer
need not be smooth. The oracle returns $g_x(x)-g_y(y)$ at unrestricted
queries; equivalently, it may return $F$ wherever both known penalty
values are defined, and unrestricted values of $F$ require specified
finite extensions of the penalties. Outputs lie in $K_x\times K_y$, and
the gap is taken over these sets; it equals the sum of the two composite
suboptimalities. For $i\in\{x,y\}$, $1\le m\le d_i$ and $\tau_i\ge0$,
write
\[
 r_{i,m}(\tau_i)=\mathcal R_m\bigl(B_i,
 \{z\in K_i:\psi_i(z)\le\tau_i\}\bigr),\qquad
 \delta_i=\eps+\lambda_i\tau_i.
\]

\begin{corollary}[Independent block-image transfers]\label{sp:cor-geometry}
If $g_i$ is globally $L_i$-smooth in $\|\cdot\|_i$, then each pair
$(m,\tau_i)$ with $r_{i,m}(\tau_i)>0$ and
$0<\delta_i\le L_ir_{i,m}(\tau_i)^2/32$ gives
\[
 T\ge\max\left\{\frac{5m}{16},\frac{m}{\Cstar}
 \min\left\{m,\sqrt{\frac{L_ir_{i,m}(\tau_i)^2}{\delta_i}}\right\}\right\}.
\]
If instead $g_i$ is globally $M_i$-Lipschitz, then every admissible pair
in \cref{gp:thm-nonsmooth} gives
\[
 T\ge c_Z\mathcal P\left(m,\min\left\{m,
 \left(\frac{M_ir_{i,m}(\tau_i)}{\delta_i}\right)^2\right\}\right).
\]
Take the supremum over the admissible pairs $(m,\tau_i)$ for each block,
and then the maximum of the two resulting bounds. Both conclusions also
bound $N_0$ in the mixed model when the partial-gradient queries concern
the other, known block and those derivatives are defined and computable
from the known data.
\end{corollary}
\begin{proof}
Retain the unknown term in one block and set the unknown term of the
other block to zero. The composite term of the other block is then known
and has minimum value zero. The gap dominates the active composite
suboptimality. Subtracting the known term simulates one scalar query to
the target by one source value, and
\cref{gp:thm-composite,gp:thm-nonsmooth} apply. In the mixed model, the
partial derivatives in the other block are independent of the unknown
source.
\end{proof}

A strongly convex $\psi_i$ supplies the corresponding strong
convexity or concavity of the composite objective. Smoothness of the full
objective additionally requires a smoothness bound on the known penalty.
An injective pullback preserves strong convexity after norm scaling;
a lower-dimensional pullback requires a regularizer that also controls
its kernel to obtain ambient strong convexity. In \cref{sp:cor-geometry},
the stated regularity constants concern the unknown terms.

\section{Coupled saddle-point reductions with unequal block dimensions}
\label{app:coupled-saddles}

Throughout this section, $r=\min\{d_x,d_y\}$, and
$U:\mathbb R^r\to\mathbb R^{d_x}$ and $V:\mathbb R^r\to\mathbb R^{d_y}$
are fixed linear isometries. All algorithms follow the conventions of
\cref{thm:main} on Borel measurability, independent randomization, fixed
query budgets, and unrestricted queries. Each exact scalar query is
simulated by a fixed number of evaluations of a source quadratic. This
simulation cost determines the constants in the resulting lower bounds.

\begin{theorem}[A compact bilinear lower bound]\label{thm:bilinear-compact}
Let $d_x,d_y\ge1$ and $C,R_x,R_y>0$. Consider bilinear-affine functions
$F(x,y)=y^\top A x-b^\top y$, with $\|A\|_{\mathrm{op}}\le C$, on
$X=R_x\mathbb B_2^{d_x}$ and $Y=R_y\mathbb B_2^{d_y}$.
Any scalar-value algorithm that returns a feasible pair with saddle-point
gap at most $\varepsilon$, with probability at least $2/3$ on every such
instance, must satisfy, for $0<\varepsilon\le CR_xR_y/32$,
\[
 T\ge\max\left\{\frac{5r}{32},\;
 \frac{r}{8192}\min\left\{r,
                    \sqrt{\frac{CR_xR_y}{\varepsilon}}\right\}\right\}.
\]
The assertion holds even on instances with a saddle point $(x^\star,0)$,
where $\|x^\star\|_2\le R_x$. In particular, the hard instances have
vanishing within-block curvature, so $L_x=L_y=0$ is admissible, and the
block dimensions need not be equal.
\end{theorem}
\begin{proof}
Let $q(z)=\tfrac12z^\top Hz-h^\top z+c_0$ be any positive-definite
quadratic with $H\preceq I$ and $z^\star=H^{-1}h\in\mathbb B_2^r(R_x)$.
Put $z=U^\top x$, $w=V^\top y$, and define
\[
 F_q(x,y)=\frac C2\bigl[q(z+w)-q(z-w)\bigr]
        =C\bigl(z^\top Hw-h^\top w\bigr).
\]
Thus $A=CVHU^\top$ has operator norm at most $C$, and $(Uz^\star,0)$ is a
saddle point with value zero. Each target value equals
$\tfrac C2[q(z+w)-q(z-w)]$ and hence costs two source values; the constant
$c_0$ cancels, so it need not be known. The compact gap is exactly
\[
 \operatorname{Gap}_{X,Y}(x,y)
 =C R_y\|Hz-h\|_2+C R_x\|Hw\|_2+C h^\top w
 \ge C R_y\|H(z-z^\star)\|_2,
\]
because $h^\top w=(z^\star)^\top Hw\ge-R_x\|Hw\|_2$.
Feasibility gives $\|z-z^\star\|_2\le2R_x$, so
\[
 q(z)-q(z^\star)
 =\tfrac12(z-z^\star)^\top H(z-z^\star)
 \le R_x\|H(z-z^\star)\|_2
 \le\frac{R_x}{CR_y}\operatorname{Gap}_{X,Y}(x,y).
\]
Consequently, the assumed saddle-point algorithm yields a $2T$-query
source algorithm of accuracy $\delta=R_x\varepsilon/(CR_y)$. It remains to
apply \cref{thm:main} with dimension $r$, smoothness $1$, radius $R_x$,
and $\delta\le R_x^2/32$.
\end{proof}

\begin{proposition}[Two upper bounds for compact bilinear games]
\label{sp:prop-bilinear-upper}
For the entire bilinear-affine class of \cref{thm:bilinear-compact},
there is a deterministic exact-value algorithm with query complexity
\[
 O\left(\min\left\{d_y(d_x+1),\;
 (d_x+d_y)\left(1+\frac{CR_xR_y}{\varepsilon}\right)\right\}\right).
\]
In particular, if $d_x=d_y=d$ and
$0<\varepsilon\le CR_xR_y/(32d^2)$, the randomized minimax
query complexity is $\Theta(d^2)$.
\end{proposition}
\begin{proof}
The values at $(0,e_j)$ recover $b_j$, and the values at $(e_i,e_j)$
then recover $A_{ji}$; this uses exactly $d_y(d_x+1)$ scalar queries.
Once the coefficients are known, the compact gap is the explicit
continuous function
\[
 R_y\|Ax-b\|_2+R_x\|A^\top y\|_2+b^\top y.
\]
A saddle point exists by compact convex minimax duality, and the displayed
gap vanishes there. Fix an enumeration of the countable dense set of
rational points of $\operatorname{int}X\times\operatorname{int}Y$ and
return the first pair whose displayed gap is smaller than $\varepsilon$.
Since the gap is continuous and
$\operatorname{int}X\times\operatorname{int}Y$ is dense in $X\times Y$,
some enumerated pair has gap below $\varepsilon$, so the search
terminates. This is a Borel rule, because each stopping condition is a
strict inequality in the observed coefficients; inconsistent transcripts
are assigned the output zero, which makes the rule total. Internal
computation is unrestricted in the present oracle model.

Alternatively, rescale $x=R_xu$, $y=R_yv$ to the unit balls. The monotone
saddle-point operator then has Lipschitz constant at most $CR_xR_y$,
independently of the affine term, and Euclidean Mirror-Prox
\citep{Nemirovski2004Prox} attains gap $O(CR_xR_y/N)$ after $N$
iterations on this bounded domain. Since the payoff is affine in each
coordinate, forward differences recover both partial gradients exactly
from at most $d_x+d_y+1$ values per operator evaluation. Each Mirror-Prox
iteration uses two operator evaluations, so $N=O(CR_xR_y/\eps)$
iterations cost $O((d_x+d_y)(1+CR_xR_y/\eps))$ values in total. Choosing
the cheaper of the two procedures on the basis of the known parameters
proves the displayed bound. For equal dimensions the first branch is
$O(d^2)$, and \cref{thm:bilinear-compact} gives the matching order in the
stated window.
\end{proof}

\begin{proposition}[Recovery under the global-saddle promise]
\label{sp:prop-bilinear-compressed}
On the subclass in \cref{thm:bilinear-compact} with a global saddle
$(x^\star,0)$ and $\|x^\star\|_2\le R_x$, a randomized algorithm recovers a
zero-gap pair from $d_x(d_x+1)$ values, almost surely on every fixed
instance. This query count is independent of $d_y$.
\end{proposition}
\begin{proof}
The promise implies $b=Ax^\star$. Sample a matrix
$Q\in\R^{d_x\times d_y}$ with independent standard Gaussian entries and
write $q_j^\top$ for its rows. The values at $(0,q_j)$ and $(e_i,q_j)$
recover $Qb$ and $QA$ in $d_x(d_x+1)$ queries. For each fixed $A$, the
restriction of $Q$ to $\operatorname{range}A$ is injective almost surely:
in any fixed orthonormal basis of this space it is a Gaussian matrix with
at most $d_x$ columns, which has full column rank almost surely.
Consequently,
\[
 \{x:QAx=Qb\}=\{x:Ax=b\}
\]
almost surely. The minimum-norm solution, computed by the Moore--Penrose
inverse, then has norm at most $R_x$ and solves $Ax=b$; paired with $y=0$,
it has zero gap. The pseudoinverse is Borel, and on histories whose
computed solution is infeasible or violates the compressed equations we
return zero. This defines a total Borel policy with the same budget and
the stated instancewise probability-one guarantee.
\end{proof}

The promise is essential for this compressed recovery rule. The upper
bound of \cref{sp:prop-bilinear-upper} applies to the full compact
bilinear-affine class, including functions without a global saddle, and
the lower bound of \cref{thm:bilinear-compact} holds on both classes. For
example, when $d_x=1$ the promised subclass is solvable with two values
regardless of $d_y$, so its complexity admits no uniform lower bound
proportional to the dual dimension.

The theorem isolates the dependence on the interaction;
\cref{sp:app-comparison} compares the lower bound with the available
upper bounds for compact bilinear games.

\begin{theorem}[Shared-curvature coupling and the global gap]
\label{thm:shared-curvature}
Let $a,b>0$, $c\ge0$, $R_x,R_y>0$, and $0\le\mu_x\le a$,
$0\le\mu_y\le b$. Write
\[
 \theta=\max\{\mu_x/a,\mu_y/b\},\qquad
 E_x=a+c^2/b,\qquad E_y=b+c^2/a.
\]
Consider globally defined convex--concave quadratics with block
smoothness at most $a,b$, mixed derivative norm at most $c$,
strong-convexity and strong-concavity moduli at least $\mu_x,\mu_y$, and
a saddle point with component norms at most $R_x,R_y$. The error
criterion is the \emph{global} saddle-point gap, where both best responses
range over the full Euclidean spaces. Any scalar-value algorithm that
attains error $\varepsilon$ with probability at least $2/3$ on every
instance satisfies
\[
 T\ge\max\left\{\frac r{16},\;
 \frac r{20480}\min\left\{r,\theta^{-1/2},
                   \sqrt{\frac{E_xR_x^2}{\varepsilon}}\right\}\right\},
 \qquad 0<\varepsilon\le E_xR_x^2/32,
\]
with $\theta^{-1/2}=+\infty$ when $\theta=0$.
The symmetric bound with $E_xR_x^2$ replaced by $E_yR_y^2$ also holds in
its corresponding accuracy window. Each assertion holds even on a
family for which the other component of the saddle point is zero.
\end{theorem}
\begin{proof}
Take a source quadratic $q(z)=\tfrac12z^\top Hz-h^\top z+c_0$ with
$\theta I\preceq H\preceq I$ and $H\succ0$, whose minimizer
$z^\star=H^{-1}h$ has norm at most $R_x$. Put $z=U^\top x$, $w=V^\top y$,
$x_\perp=(I-UU^\top)x$, and $y_\perp=(I-VV^\top)y$, and define
\begin{align*}
 F_q(x,y)
 &=a q(z)-\frac b2\bigl[q(w)+q(-w)\bigr]
       +\frac c2\bigl[q(z+w)-q(z-w)\bigr]\\
 &\hspace{1cm}+\frac a2\|x_\perp\|_2^2
                -\frac b2\|y_\perp\|_2^2\\
 &=\frac a2z^\top Hz+c z^\top Hw-\frac b2w^\top Hw
       -a h^\top z-c h^\top w+(a-b)c_0\\
 &\hspace{1cm}+\frac a2\|x_\perp\|_2^2
                -\frac b2\|y_\perp\|_2^2.
\end{align*}
Five source calls simulate each query to the saddle-point objective.
The block Hessians are $a[UHU^\top+I-UU^\top]$ and
$-b[VHV^\top+I-VV^\top]$, and the mixed derivative is $cUHV^\top$; all
the prescribed inequalities are satisfied. The unique saddle point is
$(Uz^\star,0)$. A direct completion of squares gives
\begin{align*}
 \operatorname{Gap}_{\mathbb R^{d_x},\mathbb R^{d_y}}(x,y)
 &=\frac{E_x}{2}(z-z^\star)^\top H(z-z^\star)
   +\frac{E_y}{2}w^\top Hw\\
 &\hspace{1cm}+\frac a2\|x_\perp\|_2^2
                +\frac b2\|y_\perp\|_2^2\\
 &\ge E_x\bigl[q(z)-q(z^\star)\bigr].
\end{align*}
The induced source algorithm therefore uses at most $5T$ calls and has
accuracy $\varepsilon/E_x$. Apply \cref{cor:strong} with smoothness $1$
and strong convexity $\theta$ when $\theta>0$, and \cref{thm:main} when
$\theta=0$. Interchanging the roles of the blocks and negating the payoff
proves the symmetric assertion.
\end{proof}

The radii in \cref{thm:shared-curvature} localize the saddle point,
while both best responses range over the full spaces. The factor $c^2/b$
arises because the maximizing response can lie at a distance of order
$c/b$ times the primal distance error from the dual solution. A
compact-domain comparison therefore requires a feasible truncated
response, which we construct next.

\begin{corollary}[Compact shared-curvature transfer on arbitrary balls]
\label{cor:shared-truncated}
Under the assumptions of \cref{thm:shared-curvature}, let
$X=R_x\mathbb B_2^{d_x}$, $Y=R_y\mathbb B_2^{d_y}$, and put
\[
 \beta=\min\left\{\frac cb,\frac{R_y}{2R_x}\right\},\qquad
 S_x=a+2c\beta-b\beta^2.
\]
For compact gaps and feasible outputs, with no further restriction on
the radii,
\[
 T\ge\max\left\{\frac r{16},\frac r{20480}
 \min\left\{r,\theta^{-1/2},
       \sqrt{\frac{S_xR_x^2}{\varepsilon}}\right\}\right\},
 \qquad 0<\varepsilon\le S_xR_x^2/32.
\]
The symmetric statement holds with $\beta'=\min\{c/a,R_x/(2R_y)\}$,
$S_y=b+2c\beta'-a\beta'^2$, and $0<\varepsilon\le S_yR_y^2/32$.
In particular, when $R_y\ge2cR_x/b$, we have $S_x=E_x$ and recover
the full $x$-oriented global-gap bound on compact domains; the symmetric
sufficient condition is $R_x\ge2cR_y/a$.
\end{corollary}
\begin{proof}
Use the source family of \cref{thm:shared-curvature} and write
$e=z-z^\star$. Feasibility gives $\|e\|_2\le2R_x$, so $V\beta e$ is a
feasible maximizing response. Evaluating the compact maximum at this
point and the compact minimum at $Uz^\star$ gives
\begin{align*}
 \operatorname{Gap}_{X,Y}(x,y)
 &\ge F_q(x,V\beta e)-F_q(Uz^\star,y)\\
 &=\frac{S_x}{2}e^\top He+\frac a2\|x_\perp\|_2^2
   +\frac b2w^\top Hw+\frac b2\|y_\perp\|_2^2\\
 &\ge S_x[q(z)-q(z^\star)].
\end{align*}
Here $S_x\ge a>0$ because $b\beta\le c$. Five source values still
simulate each target value, so the source theorem applies at accuracy
$\varepsilon/S_x$. The unknown minimizer enters only this gap inequality;
neither the oracle simulator nor the output decoder uses it.
\end{proof}

The coefficient agrees with $E_x$ when $R_y\ge2cR_x/b$; otherwise
\[
 S_xR_x^2=aR_x^2+cR_xR_y-bR_y^2/4.
\]
The choice of $\beta$ maximizes $2c\beta-b\beta^2$ over
$0\le\beta\le R_y/(2R_x)$, which is the interval on which this response
is feasible for every source instance and every feasible output. This
choice is optimal only among responses proportional to $e$; a different
family of feasible responses could give a larger coefficient.

\begin{corollary}[Compact domains in a local accuracy window]
\label{cor:shared-compact-local}
Assume $c>0$. If $\theta>0$, the $x$-oriented bound of
\cref{thm:shared-curvature}, with its full coefficient $E_x$, holds for
compact gaps and feasible outputs whenever
\[
 0<\varepsilon\le
 \min\left\{\frac{E_xR_x^2}{32},\;
             \frac{\theta E_x b^2R_y^2}{2c^2}\right\}.
\]
More generally, allowing also $\theta=0$, for any
$0<\varepsilon\le E_xR_x^2/32$ for which
\[
 \theta_\varepsilon
 =\max\left\{\frac{\mu_x}{a},\frac{\mu_y}{b},
             \frac{2c^2\varepsilon}{E_xb^2R_y^2}\right\}\le1,
\]
the same compact-domain bound holds with $\theta$ replaced by
$\theta_\varepsilon$.
\end{corollary}
\begin{proof}
Every source Hessian in \cref{thm:shared-curvature} satisfies
$H\succeq\theta I$. Fix a feasible output $(x,y)$ and put $e=z-z^\star$
and $s=\|e\|_2$. If $s\le bR_y/c$, the unconstrained maximizing response
is feasible and the preceding transfer applies. Otherwise, take the
feasible response $w'=R_y e/s$, which has no perpendicular component, and
evaluate the minimization term at $Uz^\star$. Dropping nonnegative terms
gives
\begin{align*}
 \operatorname{Gap}_{X,Y}(x,y)
 &\ge \left(\frac a2+\frac{cR_y}{s}
                    -\frac{bR_y^2}{2s^2}\right)e^\top He\\
 &\ge\theta\left(\frac a2s^2+cR_ys-\frac b2R_y^2\right)
 >\frac{\theta E_xb^2R_y^2}{2c^2}.
\end{align*}
Indeed, the map $s\mapsto\frac a2s^2+cR_ys-\frac b2R_y^2$ is strictly
increasing on $s\ge0$, and at $s=bR_y/c$ it equals
\[
 \frac{ab^2R_y^2}{2c^2}+\frac b2R_y^2=\frac{E_xb^2R_y^2}{2c^2},
 \qquad E_x=a+\frac{c^2}{b},
\]
so in the present case $s>bR_y/c$ the last inequality is strict. Hence an
output meeting the stated accuracy necessarily falls into the first case.
For the last assertion, choose a source Hessian with
$H\succeq\theta_\varepsilon I$; the resulting family remains in the
prescribed saddle-point class, and the required accuracy inequality holds
by construction.
\end{proof}

\paragraph{Recovery with partial-gradient information.}
For the shared-curvature family, the additional response $\nabla_xF$
reveals $h$ at $(x,y)=(0,0)$ and every column of $H$ at the queries
$(x,y)=(Ue_i,0)$; hence $r+1$ partial-gradient queries recover its saddle
point exactly. For the bilinear family, $r$ bundled queries returning
both $F$ and $\nabla_xF$ at $(0,Ve_i)$ recover the columns of $H$ and all
coordinates of $h$ simultaneously. Thus partial-gradient access identifies the unknown coefficients of
these hard families. The mixed-oracle constructions in \cref{sp:app}
instead keep the coupling known and place the unknown information in
the inner objective.

\begin{theorem}[Coupling dependence with known diagonal curvatures]
\label{thm:known-diagonal-coupling}
Let $\mu_x,\mu_y,C,R_x,R_y>0$, and define
\[
 \chi=\frac C{\sqrt{\mu_x\mu_y}}\ge4,\qquad
 \rho=\min\left\{R_x,\;4\sqrt{\frac{\mu_y}{\mu_x}}R_y\right\}.
\]
Consider globally defined strongly convex--strongly concave quadratics
with known diagonal Hessians $\nabla^2_{xx}F=\mu_x I$ and
$\nabla^2_{yy}F=-\mu_y I$, mixed derivative norm at most $C$, and a
saddle point whose block norms are at most $R_x,R_y$. Every scalar-value
algorithm attaining global gap at most $\varepsilon$ with probability
at least $2/3$ on this class satisfies
\begin{equation}\label{sp:eq-known-diagonal-extended}
 T\ge\max\left\{\frac{5r}{32},\;
 \frac r{8192}\min\left\{r,\frac{\sqrt\chi}{2},
       \left(\frac{\mu_x\chi^2\rho^2}{128\varepsilon}\right)^{1/4}\right\}\right\},
 \qquad 0<\varepsilon\le\frac{\mu_x\chi^2\rho^2}{128}.
\end{equation}
In particular, throughout $0<\varepsilon\le\mu_x\rho^2/8$,
\[
 T\ge\max\left\{\frac{5r}{32},\;
 \frac r{8192}\min\left\{r,\frac{\sqrt\chi}{2}\right\}\right\}.
\]
The corresponding assertions with the blocks interchanged also hold, with
$R_x,\mu_x$ and $R_y,\mu_y$ exchanged throughout. Since
$L_x/\mu_x=L_y/\mu_y=1$ on this class, the bound is driven by the coupling
alone. Sharpness of the square-root coupling exponent remains open.
\end{theorem}
\begin{proof}
Set $\tau=\chi^{-1}$ and choose the source curvature according to the
requested accuracy:
\[
 \eta=\max\left\{\frac4\chi,
       \sqrt{\frac{128\varepsilon}{\mu_x\chi^2\rho^2}}\right\}.
\]
The hypotheses ensure that $0<\eta\le1$. Take a source quadratic
$q(z)=\tfrac12z^\top Hz-h^\top z+c_0$ with
$\eta I\preceq H\preceq I$ and
$z^\star=H^{-1}h\in\mathbb B_2^r(\rho)$, and define
\begin{align*}
 F_q(x,y)
 &=\frac{\mu_x}{2}\|x\|_2^2
   +\frac C2\bigl[q(U^\top x+V^\top y)
                  -q(U^\top x-V^\top y)\bigr]
   -\frac{\mu_y}{2}\|y\|_2^2\\
 &=\frac{\mu_x}{2}\|x\|_2^2
   +C\bigl[(U^\top x)^\top H(V^\top y)-h^\top V^\top y\bigr]
   -\frac{\mu_y}{2}\|y\|_2^2.
\end{align*}
Two source values simulate each target value. The diagonal Hessians are
as prescribed, and the coupling norm is at most $C$.
The unique saddle point has zero perpendicular components and shared
coordinates
\[
 s=(I+\tau^2H^{-2})^{-1}z^\star,\qquad
 t=-\frac{\mu_x}{C}H^{-1}s.
\]
These expressions solve both stationarity equations. The inverse matrix in
the first expression has eigenvalues in $(0,1)$, and $\chi\eta\ge4$, so
\[
 \|s\|_2\le\rho\le R_x,\qquad
 \|t\|_2\le\frac{\mu_x\rho}{C\eta}
 \le\frac14\sqrt{\frac{\mu_x}{\mu_y}}\rho\le R_y.
\]

For an output $(x,y)$, put $z=U^\top x$. Completing squares in the global
primal value function gives
\begin{equation}\label{sp:eq-known-diagonal-primal}
  \operatorname{Gap}_{\mathbb R^{d_x},\mathbb R^{d_y}}(x,y)
 \ge\frac12(z-s)^\top
       \left(\mu_xI+\frac{C^2}{\mu_y}H^2\right)(z-s).
\end{equation}
Since $H^2\succeq\eta H$, a gap of at most $\varepsilon$ implies
\[
 \|z-s\|_H^2\le\frac{2\varepsilon}{\mu_x\chi^2\eta}.
\]
At an eigenvalue $\lambda$ of $H$, the multiplier taking $z^\star$ to
$s-z^\star$ is $-\tau^2/(\lambda^2+\tau^2)$. Hence
\[
 \|s-z^\star\|_H^2
 \le\frac{\tau^4}{\eta^3}\|z^\star\|_2^2
 \le\frac{\rho^2}{\chi^4\eta^3}.
\]
The squared triangle inequality and the two lower bounds in the
definition of $\eta$ yield
\begin{align*}
 q(z)-q(z^\star)
 &\le\frac{2\varepsilon}{\mu_x\chi^2\eta}
       +\frac{\rho^2}{\chi^4\eta^3}\\
 &\le\frac{\eta\rho^2}{64}+\frac{\eta\rho^2}{256}
  =\frac{5\eta\rho^2}{256}
  <\delta:=\frac{\eta\rho^2}{32}.
\end{align*}
The fixed linear decoder $x\mapsto U^\top x$ therefore produces a source
algorithm with budget $2T$ and accuracy $\delta$. Apply
\cref{cor:strong} with smoothness $1$, strong convexity $\eta$, and
radius $\rho$. Since $\sqrt{\rho^2/\delta}=\sqrt{32/\eta}$ exceeds
$\eta^{-1/2}$, substituting the definition of $\eta$ gives
\eqref{sp:eq-known-diagonal-extended}. When
$\varepsilon\le\mu_x\rho^2/8$, the chosen curvature is $\eta=4/\chi$,
which proves the simpler bound. Interchanging the blocks and negating
$F$ proves the symmetric assertions.
\end{proof}

\begin{lemma}[A compact gap forces a feasible maximizing response]
\label{sp:lem-compact-response}
Let $F$ be quadratic, with diagonal Hessians $\mu_xI,-\mu_yI$,
$\mu_x,\mu_y>0$, and coupling matrix $A=\nabla^2_{yx}F$.
Suppose its global saddle satisfies $x^\star\in X=R_x\mathbb B_2^{d_x}$ and
$\|y^\star\|_2\le R_y/2$. If a feasible pair satisfies
\[
 \operatorname{Gap}_{X,Y}(x,y)\le\frac{\mu_yR_y^2}{8},
 \qquad Y=R_y\mathbb B_2^{d_y},
\]
then the unrestricted maximizing response $y^\star+A(x-x^\star)/\mu_y$
belongs to $Y$, and
\[
 \operatorname{Gap}_{X,Y}(x,y)
 \ge\sup_{v\in\R^{d_y}}F(x,v)-F(x^\star,y^\star).
\]
\end{lemma}
\begin{proof}
Put $v=A(x-x^\star)$. Stationarity at the saddle point gives, for every
$u$,
\[
 F(x,y^\star+u)-F(x^\star,y)
 =\frac{\mu_x}{2}\|x-x^\star\|_2^2
  +\frac{\mu_y}{2}\|y-y^\star\|_2^2
  +u^\top v-\frac{\mu_y}{2}\|u\|_2^2.
\]
Suppose that the unrestricted response lies outside $Y$; then $v\ne0$.
The ray from $y^\star$ in the direction $v$ first reaches the boundary at
$y^\star+a v/\|v\|_2$, where $R_y/2\le a<\|v\|_2/\mu_y$. Evaluate the
compact maximum at this point and the compact minimum at $x^\star$. After
dropping the two nonnegative squared terms, the gap is at least
\[
 a\|v\|_2-\frac{\mu_y a^2}{2}
 >\frac{\mu_y a^2}{2}\ge\frac{\mu_yR_y^2}{8},
\]
which is a contradiction. The unrestricted response is therefore
feasible. Finally, $F(x^\star,y)\le F(x^\star,y^\star)$, so evaluating the
minimum at $x^\star$ proves the remaining inequality.
\end{proof}

\begin{corollary}[A compact accuracy window for known diagonal curvatures]
\label{cor:known-diagonal-compact}
The bound \eqref{sp:eq-known-diagonal-extended} holds for the compact
gap and feasible outputs if the source radius and accuracy window are
replaced by
\[
 \rho=\min\left\{R_x,\;2\sqrt{\frac{\mu_y}{\mu_x}}R_y\right\},
 \qquad
 0<\varepsilon\le
 \min\left\{\frac{\mu_x\chi^2\rho^2}{128},\frac{\mu_yR_y^2}{8}\right\}.
\]
In particular, its simpler $\sqrt\chi$ bound holds whenever
$0<\varepsilon\le\min\{\mu_x\rho^2/8,\mu_yR_y^2/8\}$.
The assertions with the blocks interchanged also hold, with
$\rho=\min\{R_y,2\sqrt{\mu_x/\mu_y}R_x\}$ and
$0<\varepsilon\le\min\{\mu_y\chi^2\rho^2/128,\ \mu_xR_x^2/8\}$.
\end{corollary}
\begin{proof}
With this smaller radius, the same source family satisfies
$\|s\|_2\le R_x$ and $\|t\|_2\le R_y/2$. On every successful output,
\cref{sp:lem-compact-response} bounds the global primal error by the
compact gap, so \eqref{sp:eq-known-diagonal-primal} remains valid with
the compact gap on its left-hand side. The source simulation and the rest
of the decoding argument are unchanged.
\end{proof}

For the last family, a mixed oracle returning the value together with
$\nabla_xF$ again determines $H$ and $h$ in $r$ calls at $(0,Ve_i)$: the
gradient is $CUHe_i$ and the value equals $-Ch_i-\mu_y/2$.

\begin{proposition}[Reconstruction with known diagonal curvatures]
\label{sp:prop-known-diagonal-upper}
An arbitrary quadratic with known diagonal Hessians $\mu_xI,-\mu_yI$,
an unknown bilinear coupling, affine terms, and a constant can be
reconstructed from $1+d_x+d_y+d_xd_y$ exact scalar values. Its global
saddle point can then be determined exactly. Consequently, when
$d_x=d_y=d$, $C\ge4d^2\sqrt{\mu_x\mu_y}$, and $0<\eps\le\mu_x\rho^2/8$
with $\rho$ as in \cref{thm:known-diagonal-coupling}, the global-gap
minimax complexity of that class, with the specified saddle-point
localization, is $\Theta(d^2)$.
\end{proposition}
\begin{proof}
The value at $(0,0)$ determines the constant. The values at $(e_i,0)$ and
$(0,e_j)$ determine the two affine terms after the known diagonal
quadratics are subtracted. Finally,
\[
 A_{ji}=F(e_i,e_j)-F(e_i,0)-F(0,e_j)+F(0,0)
\]
determines the coupling. The stationarity equations have a unique
solution, since eliminating $y$ gives the positive-definite matrix
$\mu_xI+A^\top A/\mu_y$; solving them is a continuous operation in the
coefficients and yields the unique global saddle point. The query count
is as claimed. Under the displayed coupling condition,
$\tfrac12\sqrt{C/\sqrt{\mu_x\mu_y}}\ge d$, so
\cref{thm:known-diagonal-coupling} gives a lower bound of the same order.
\end{proof}

\begin{proposition}[Limitation of the primal-coordinate decoder]
\label{sp:prop-decoder-limit}
Consider the known-diagonal construction in
\cref{thm:known-diagonal-coupling}, allowing source matrices
$\eta I\preceq H\preceq I$, $0<\eta\le1$, and a positive source radius
$\rho\le\min\{R_x,2\sqrt{\mu_y/\mu_x}R_y\}$.
A uniform source-error transfer that decodes only $U^\top x$ cannot
produce a source condition/accuracy factor larger than order
$\sqrt\chi$, even if the target solver returns an exact saddle.
\end{proposition}
\begin{proof}
For $H=\lambda I$ and $\|z^\star\|_2=\rho$, the exact saddle coordinate
$s=(I+\chi^{-2}H^{-2})^{-1}z^\star$ has source error
\[
 q(s)-q(z^\star)=\frac{\lambda\rho^2}{2(1+\chi^2\lambda^2)^2}.
\]
If $\eta\le1/(\sqrt3\chi)$, choose $\lambda=1/(\sqrt3\chi)$. The error
then equals $3\sqrt3\rho^2/(32\chi)$, so any uniform source tolerance
$\delta$ for this decoder satisfies
$\sqrt{\rho^2/\delta}\le\sqrt{32/(3\sqrt3)}\sqrt\chi$.
If $\eta>1/(\sqrt3\chi)$, the source condition factor already satisfies
$\eta^{-1/2}<3^{1/4}\sqrt\chi$.
In either case, the minimum of the two factors is $O(\sqrt\chi)$.
The test instances satisfy the localization: the exact formulas give
$\|s\|_2\le\rho$ and
$\|t\|_2\le\tfrac12\sqrt{\mu_x/\mu_y}\rho$, using
$u/(1+u^2)\le1/2$ for $u\ge0$.
\end{proof}

The proposition identifies the limit of the primal-coordinate decoder
in this reduction. Decoders using both output blocks, further oracle
values, or the observed history may yield stronger lower bounds.

\section{Saddle-point complexity: comparisons and remaining gaps}
\label{sp:app-comparison}

Our saddle-point results distinguish the dimensions $d_x$ and $d_y$, the
information returned by each query, and the set over which each best
response is taken. We compare upper and lower bounds with these three specifications
fixed. In this section, $N_0$ counts
scalar values and $N_x$ counts partial gradients $\nabla_xF$; a value
returned together with a gradient contributes to both counts. For
$\mu_x,\mu_y>0$, write
\[
 \kappa_x=\frac{L_x}{\mu_x},\qquad
 \kappa_y=\frac{L_y}{\mu_y},\qquad
 \chi=\frac{L_{xy}}{\sqrt{\mu_x\mu_y}}.
\]
Here $L_x,L_y$ are the within-block constants denoted by $L_{xx},L_{yy}$
in several of the cited papers. The notation $\widetilde O$ suppresses
logarithmic factors in the accuracy, the confidence, and the geometric and
conditioning parameters of the cited result. Polynomial dependence on
dimensions and condition numbers is displayed explicitly.

\begin{table}[t]
\caption{Selected saddle-point bounds in their original information models.
Each row specifies its query responses, error criterion, and parameter
conditions.}
\label{sp:tab-comparison}
\small
\setlength{\tabcolsep}{4pt}
\renewcommand{\arraystretch}{1.12}
\begin{tabularx}{\textwidth}{@{}>{\raggedright\arraybackslash}p{.22\textwidth}
 >{\raggedright\arraybackslash}p{.28\textwidth}
 >{\raggedright\arraybackslash}X@{}}
\toprule
Work / result & Information and complexity & Accuracy and assumptions \\
\midrule
\citet{SadievEtAl2021,BeznosikovNovitskiiGasnikov2021} &
Values: $O((d_x+d_y)LD^2/\eps)$ &
Expected compact-domain gap; noiseless specialization of the coordinate
extragradient bound. One-point feedback changes the sampling model again
\citep{BeznosikovNovitskiiGasnikov2021}. \\
\citet{GladinEtAl2021}, small outer block &
$N_x=\widetilde O(d_x)$;
$N_0=\widetilde O(d_xd_y\sqrt{\kappa_y})$ &
Outer objective error; compact $X$ with interior, unrestricted $y$,
strong concavity in $y$. \\
\citet{GladinEtAl2021}, accelerated regime &
$N_x=\widetilde O(\sqrt{\kappa_x+2\chi^2})$;
$N_0=\widetilde O\bigl(d_y\sqrt{\kappa_x\kappa_y+2\chi^2}\bigr)$ &
Strong convexity--concavity; high-probability nested objective accuracy,
pair-gap conditions are given below. \\
\citet{AlkousaEtAl2023} &
Partial gradients and regularizer proximal maps; low-dimensional outer
cutting planes &
Distance to the saddle point; separates the oracle costs of the composite
terms. \\
\citet{OuyangXu2021,ZhangHongZhang2022} &
First-order and proximal lower bounds &
Conditioning and coupling dependence; deterministic algorithms and
sufficiently large hard-instance dimensions. \Citet{OuyangXu2021} give the
bilinear first-order benchmark. \\
\citet{BorodichGasnikovKovalev2025} &
Separate gradients and matrix--vector products &
Matching bounds in the stated first-order model; weighted distance to
the solution set of a bilinearly coupled problem. \\
\Cref{sp:thm-blocks,sp:thm-mixed} &
$T\ge\max_i\mathcal B_i(\eps)$, see \eqref{sp:eq-block-budget};
$N_0\ge\mathcal B_y(\eps)$ with $\nabla_xF$ &
Gap accuracy with probability $2/3$; independently prescribed dimensions;
the mixed lower bound allows arbitrary pure partial-gradient calls. \\
\Cref{thm:bilinear-compact,sp:prop-bilinear-upper} &
Values: lower $\Omega(r\min\{r,\sqrt{L_{xy}R_xR_y/\eps}\})$;
recovery in $d_y(d_x+1)$ &
Compact bilinear gap; balanced blocks have a matching $\Theta(d^2)$
plateau in the stated accuracy window. \\
\Cref{thm:known-diagonal-coupling,sp:prop-known-diagonal-upper} &
Values: $\Omega(r\min\{r,\sqrt\chi\})$,
$r=\min\{d_x,d_y\}$ &
Known isotropic diagonal Hessians; global gap in a stated accuracy window;
balanced-block recovery matches when $\chi\ge4d^2$. \\
\Cref{sp:thm-mixed-coupled,sp:thm-mixed-full} &
Scalar-value lower bound with arbitrary $\nabla_xF$ calls and known coupling &
Pair gap, or outer error for the shared-rank construction; the rectangular
variant retains the full inner dimension under its curvature condition. \\
\bottomrule
\end{tabularx}
\end{table}

\paragraph{Scalar-value and mixed extragradient bounds.}
In the exact deterministic specialization of \citet[Theorem~2 and
Corollary~2]{SadievEtAl2021}, coordinate differences give the first row of
\cref{sp:tab-comparison}; here $D$ is the Euclidean diameter parameter and
$L$ bounds the full saddle-point operator. Their strongly convex--strongly
concave bound is $O((d_x+d_y)(L/\mu)\log(1/\eps))$ scalar calls, with the
initial-distance factors of the cited result inside the logarithm; its
criterion is the expected squared distance to a saddle point. Their
mixed result \citep[Theorem~3 and Corollary~3]{SadievEtAl2021} uses values
in $x$ and gradients in $y$, with $O(d_xLD^2/\eps)$ scalar calls and
$O(LD^2/\eps)$ partial-gradient calls for the expected gap. These
estimates allow unequal block dimensions. Finite differences require
function values in a neighborhood of the feasible sets, which is
consistent with our unrestricted query model, but the compact-gap and
global-gap output criteria remain distinct. A bound on the expectation of
a nonnegative error yields a fixed-confidence guarantee by Markov's
inequality at a constant-factor loss in the target accuracy; the
controlled quantity is unchanged.

\paragraph{Accelerated mixed methods and their accuracy criteria.}
\citet[Theorem~5]{GladinEtAl2021} obtain the small-outer-block rates in
\cref{sp:tab-comparison} for $\Phi(x)=\max_yF(x,y)$. Their error criterion is $\Phi(\widehat x)-\min\Phi$. Theorem~7 supplies the accelerated rates
together with an explicit target failure probability. Its Catalyst
construction gives the scalar-value budget
$\widetilde O(d_y\sqrt{\kappa_x\kappa_y+2\chi^2})$ in
\cref{sp:tab-comparison}. Its analysis controls the nested objectives and
converts their errors into the dual objective error by means of Lemma~4.
A pair-gap comparison would in addition require control of the primal
objective error: writing $X$ for the outer feasible set of the cited
problem,
\[
 \operatorname{Gap}_{X,\R^{d_y}}(\widehat x,\widehat y)
 =\Phi(\widehat x)-\Phi(x^\star)+\Psi(y^\star)-\Psi(\widehat y),
 \qquad \Psi(y)=\min_{x\in X}F(x,y).
\]
The separable family in \cref{sp:thm-mixed} gives a pair-gap lower
bound with a known outer minimizer (\cref{sp:rem-models}). For the outer
objective error, \cref{sp:thm-mixed-coupled} uses a coupled construction,
with the parameter dependence detailed below.

\paragraph{Small blocks and composite information.}
\citet[Definition~1 and Theorem~2]{AlkousaEtAl2023} study
$r(x)+F(x,y)-h(y)$ and use the distance to the saddle point. For a
low-dimensional $x$ block, their Vaidya-based construction uses
$\widetilde O(d_x)$ evaluations of $\nabla_xF$ and $\nabla r$, and
$\widetilde O(d_x\sqrt{L_y/\mu_y})$ evaluations of $\nabla_yF$ and of the
proximal map of $h$. When $h$ is smooth instead,
the budget of its own gradients is separated using its smoothness
constant. The roles of the two blocks can be interchanged. These results
motivate retaining both dimensions and separate oracle budgets, but
proximal and vector responses provide more information than a scalar
value. Our known-regularizer transfer (\cref{sp:cor-geometry}) therefore
addresses a different information model for the unknown smooth or
Lipschitz term.

\paragraph{First-order lower bounds.}
For smooth strongly convex--strongly concave problems,
\citet{ZhangHongZhang2022} establish an iteration lower bound of order
\[
 \sqrt{\kappa_x+\chi^2+\kappa_y}\,\log(1/\eps)
\]
in their first-order model, together with a proximal bilinear analogue.
Their orthogonal-invariance argument extends beyond the initial
linear-span restriction for deterministic algorithms. Their dimension requirement grows with the target iteration count;
our scalar-value bounds keep $(d_x,d_y)$ explicit.
For bilinearly coupled objectives,
\citet[Theorem~3.3 and Corollary~4.8]{BorodichGasnikovKovalev2025}
characterize gradient evaluations and matrix--vector products separately
in a first-order memory model. In the strongly convex--strongly concave
specialization, the respective orders are
$\sqrt{\kappa_x}\log(1/\eps)$,
$\sqrt{\kappa_y}\log(1/\eps)$, and
$\chi\log(1/\eps)$, and the error criterion is a weighted squared
distance. Finite differences turn these gradient-based algorithms into
scalar-value upper bounds. The corresponding scalar lower bounds require
an information argument for each fixed pair of block dimensions.

\paragraph{Coupling dependence proved here.}
\Cref{thm:bilinear-compact} gives a scalar-value lower bound even when both
within-block Hessians vanish:
\[
\Omega\!\left(r\min\left\{r,
       \sqrt{\frac{L_{xy}R_xR_y}{\eps}}\right\}\right),
 \qquad r=\min\{d_x,d_y\}.
\]
For the entire bilinear-affine class,
\cref{sp:prop-bilinear-upper} supplies the upper bound
\[
 O\!\left(\min\left\{d_y(d_x+1),\;
 (d_x+d_y)\left(1+\frac{L_{xy}R_xR_y}{\eps}\right)\right\}\right).
\]
The first branch identifies the payoff coefficients, and the second uses
exact coordinate differences and Mirror-Prox. Consequently, the
compact-gap minimax complexity is $\Theta(d^2)$ when $d_x=d_y=d$ and
$0<\eps\le L_{xy}R_xR_y/(32d^2)$.
For shared-curvature quadratics, \cref{thm:shared-curvature} instead
introduces the effective scale $a+c^2/b$ in the global gap, together with
the curvature parameter $\theta=\max\{\mu_x/a,\mu_y/b\}$.
For arbitrary radii, \cref{cor:shared-truncated} uses the feasible
truncated response and the scale $S_x=a+2c\beta-b\beta^2$ with
$\beta=\min\{c/b,R_y/(2R_x)\}$; this scale equals $E_x$ when
$R_y\ge2cR_x/b$, and \cref{cor:shared-compact-local} also recovers $E_x$
within a local accuracy window. These results give interaction-dependent lower bounds for exact
scalar values. Outside the matching regime, a gap remains between the
square-root accuracy dependence of the compact bilinear lower bound and
the inverse-linear dependence of the extragradient upper bound. The sharp
dimension dependence for highly asymmetric blocks also remains open. The distinction between compact existence and a
global-saddle promise also matters: on the latter subclass,
\cref{sp:prop-bilinear-compressed} gives almost-sure exact recovery from
$d_x(d_x+1)$ values, independently of $d_y$.

The known-diagonal construction of \cref{thm:known-diagonal-coupling}
keeps $L_x/\mu_x=L_y/\mu_y=1$ and establishes a lower bound of
$\Omega(r\min\{r,\sqrt\chi\})$ scalar calls for $\chi\ge4$ and
$\eps\le\mu_x\rho^2/8$, where
$\rho=\min\{R_x,4\sqrt{\mu_y/\mu_x}R_y\}$.
The theorem also states a wider global accuracy window, and
\cref{cor:known-diagonal-compact} gives the compact-gap conditions.
The reconstruction in \cref{sp:prop-known-diagonal-upper} matches the
lower bound at $\Theta(d^2)$ for equal blocks, $\chi\ge4d^2$, and
$0<\eps\le\mu_x\rho^2/8$. Outside this plateau, the dependence on
conditioning and accuracy remains open. \Cref{sp:prop-decoder-limit} identifies the accuracy limitation of
the primal-coordinate decoder used in this construction.

\paragraph{Mixed information with known coupling.}
Unlike the separable mixed family, \cref{sp:thm-mixed-coupled} makes the
outer solution depend on an unknown inner quadratic. Its partial-gradient
responses are nevertheless known, so the pair-gap and outer-error lower
bounds \eqref{sp:eq-mixed-coupled-gap}--\eqref{sp:eq-mixed-coupled-outer}
allow arbitrary pure $\nabla_xF$ calls. The source condition number
$(L_y+c^2/\alpha)/(\mu_y+c^2/\alpha)$ decreases as the coupling grows.
Thus stronger coupling improves the conditioning of this inner problem.
The complementary \cref{sp:thm-mixed-full} retains the full dimension
$d_y$ for a fixed rectangular coupling when $\mu_y+c^2/\alpha\le L_y$,
for the pair-gap criterion. This criterion controls the inner
components in the kernel of the coupling, which the outer solution alone
leaves undetermined.
\Cref{sp:cor-mixed-compact} supplies a common sufficient condition for
compact pair gaps. Matching the nested upper bounds in \cref{sp:tab-comparison} under
each accuracy criterion, and determining the joint trade-off between
scalar-value and partial-gradient budgets, remain open.

\paragraph{Accuracy saturation and unresolved comparisons.}
\Cref{cor:quadratic} characterizes scalar \emph{minimization} on
quadratics, and the saddle-point comparisons above give matching bounds
in two plateau regimes with equal block dimensions. The necessity of the logarithmic accuracy factors in the cited
general-class upper bounds remains open in the unmatched regimes. A quadratic in $d_x+d_y$ variables can be interpolated from
$O(d_x^2+d_xd_y+d_y^2)$ values, after which its saddle-point problem can
be solved without further oracle information; this explains why
finite-dimensional quadratic lower bounds saturate. Obtaining the optimal
coupling dependence, determining the role of the two dimensions beyond
the shared subspace, and proving sharp mixed-budget trade-offs require
further arguments.

\section{Independent regularity norms, localization sets, and known regularizers}\label{g:sec}
A non-Euclidean formulation has three ingredients that may be chosen
\emph{independently}. We first fix the norm or set attached to each
regularity constant and radius.

\begin{itemize}\itemsep1pt
\item \textbf{The regularity norm} $\|\cdot\|$, with unit ball $B$ and dual norm
$\|\cdot\|_*$. It measures the regularity of the unknown function:
$\|\nabla f(x)-\nabla f(y)\|_*\le L\|x-y\|$ for all $x,y$ in the smooth case,
and $|g(x)-g(y)|\le M\|x-y\|$ in the nonsmooth case.
\item \textbf{The localization set} $K\subset\R^d$, a nonempty compact convex set.
In the smooth model, a global minimizer lies in $K$ and queries are
unrestricted; in the nonsmooth model, the minimum is taken over $K$ and the
returned point must lie in $K$.
\item \textbf{The known regularizer} $\psi:K\to\R$, continuous and convex with
$\min_K\psi=0$, which enters the objective as $\lambda\psi$ with $\lambda\ge0$.
Its values and proximal operations are free. Only the unknown summand is
accessed through exact values, and its regularity is bounded by $L$ or $M$.
Setting $\lambda=0$ recovers the unregularized models. Throughout, the symbol
$\mu$ is reserved for a strong-convexity modulus.
\end{itemize}
The accuracy $\eps$ is measured in objective \emph{values}. The
regularity norm $B$ and localization set $K$ can vary independently: the example below combines a feasible $\ell_1$ ball with a
gradient that is Lipschitz with respect to $\ell_2$.

\paragraph{The common transfer framework.}
The definition of $\mathcal R_m(B,C)$ and the three transfer theorems are
stated in the main text, \cref{gp:eq-radius,gp:thm-norm,gp:thm-composite,gp:thm-nonsmooth}.
The smooth proofs are in \cref{gp:app}, and the nonsmooth regularized proof
is in \cref{gp:app-nonsmooth}. We specialize these results to
independently chosen regularity and localization norms.

\paragraph{Balls: an exact full-dimensional radius.}
Take $B=\mathbb B_\beta^d$ and $K=\mathbb B_\alpha^d(R_\alpha)$ with
$1\le\alpha,\beta\le\infty$ chosen independently. The subscripts below
identify the norm attached to each constant:
\begin{equation}\label{g:eq-constants}
 \|\nabla f(x)-\nabla f(y)\|_{\beta^*}\le L_\beta\|x-y\|_\beta,\quad
 |g(x)-g(y)|\le M_\beta\|x-y\|_\beta,\quad
 \|x^\star\|_\alpha\le R_\alpha .
\end{equation}
For the comparison that follows, we use the minimal valid smoothness moduli
of a fixed objective and take each radius to be the corresponding norm of one
chosen minimizer. The complexity theorems allow any independently
specified admissible upper bounds on these quantities. The constants can
differ for the same objective: $L_\beta$ is the largest curvature
over the unit $\ell_\beta$ ball, so $L_1\le L_2\le L_\infty$ with each successive
ratio at most $d$, whereas $R_1\ge R_2\ge R_\infty$ with $R_1\le\sqrt d\,R_2$
(\cref{g:rem-constants}). The accuracy $\eps$ remains independent of the norm.

\begin{proposition}[Image radius of an $\ell_\alpha$ ball]\label{g:prop-lp-radius}
Let $1\le\alpha,\beta\le\infty$, $1\le m\le d$, and put
\begin{equation}\label{g:eq-theta}
 \theta=\theta(\alpha,\beta)=\left(\frac12-\frac1\beta\right)_+
  +\left(\frac1\alpha-\frac12\right)_+,
 \qquad
 \gamma=\gamma(\alpha,\beta)=\left(\frac1\beta-\frac1\alpha\right)_+ .
\end{equation}
Then $\theta\in[0,1]$, $\theta+\gamma\le1$, and
\begin{equation}\label{g:eq-lp-radius}
 \mathcal R_m\!\left(\mathbb B_\beta^d,\mathbb B_\alpha^d(R_\alpha)\right)
 \;\ge\; R_\alpha\,\Big\lfloor \frac dm\Big\rfloor^{\gamma} m^{-\theta}.
\end{equation}
An explicit block map attains this value. At $m=d$ the bound is exact for
every $\alpha,\beta$:
$\mathcal R_d(\mathbb B_\beta^d,\mathbb B_\alpha^d(R_\alpha))
=R_\alpha d^{-\theta}$ (\cref{gp:prop-full-lp}).
For $\alpha=1$ and every $\beta$,
$\mathcal R_m(\mathbb B_\beta^d,\mathbb B_1^d(R_1))\le R_1\min\{1,\sqrt{2\log(2ed)/m}\}$;
hence for $\alpha=1$, $\beta\le2$, where $\theta=1/2$, the block certificate
is optimal up to $\sqrt{2\log(2ed)}$ at every image dimension.
For general $m<d$, the displayed certificate need not maximize the image
radius.
\end{proposition}

The exponent vanishes exactly when $\beta\le2\le\alpha$, that is, when the
feasible unit ball contains the Euclidean unit ball and the regularity unit
ball is contained in the Euclidean unit ball; otherwise, an increase of
$1/\alpha$ above $1/2$ and a decrease of $1/\beta$ below $1/2$ contribute
equally to the exponent. The block factor $\lfloor d/m\rfloor^\gamma$ is
nontrivial only for $\beta<\alpha$, and it acts differently on the smooth and
on the nonsmooth bound. Write $\Psi=\gamma+\theta\le1$. In the smooth case
the product term is largest at $m=d$ (\cref{g:cor-lp-blocks}): a block map
can widen the accuracy window but cannot improve the bound when the
full-dimensional certificate is admissible. The nonsmooth profile changes
behavior at $\Psi=1/2$, so for $\Psi>1/2$ and
$(M_\beta R_\alpha/\eps)^2<d^{1+2\theta}$ an interior block dimension can
improve the bound relative to $m=d$ (\cref{g:cor-lp}). No single certificate
is optimal in every parameter regime. Euclidean sections give radius
$\asymp R_\alpha$ at a smaller dimension, which at $\alpha=\beta=\infty$
improves \eqref{g:eq-lp-radius} by $\asymp\sqrt d$ (\cref{gp:app}). The next
result records the bounds obtained from a \emph{single} certificate,
uniformly in $(\alpha,\beta)$.

\begin{corollary}[Explicit bounds on $\ell_\alpha$ balls]\label{g:cor-lp}
Let $\theta=\theta(\alpha,\beta)$ and let the query and success conventions be
those of \cref{gp:thm-norm,gp:thm-nonsmooth} with $\lambda=0$.
\emph{(a) Smooth.} For $0<\eps\le L_\beta R_\alpha^2/(32d^{2\theta})$,
\begin{equation}\label{g:eq-lp-smooth}
 T\ge\max\left\{\frac{5d}{16},\;\frac d{\Cstar}
 \min\left\{d,\sqrt{\frac{L_\beta R_\alpha^2}{d^{2\theta}\eps}}\right\}\right\},
\end{equation}
and, for every $\theta>0$, on the accuracy window
$0<\eps\le L_\beta R_\alpha^2/32$, which does not involve $d$,
\begin{equation}\label{g:eq-lp-lowacc}
 T\ge\frac5{16}\min\left\{d,\;
 \left\lfloor\left(\frac{L_\beta R_\alpha^2}{32\eps}\right)^{1/(2\theta)}
 \right\rfloor\right\}.
\end{equation}
\emph{(b) Nonsmooth.} If $d\ge m_Z^\circ$ and
$0<\eps\le\eps_ZM_\beta R_\alpha d^{-\theta}$,
\begin{equation}\label{g:eq-lp-nonsmooth}
 T\ge c_Z\,\mathcal P\!\left(d,\;
 \min\left\{d,\ \frac{M_\beta^2R_\alpha^2}{d^{2\theta}\eps^{2}}\right\}\right).
\end{equation}
The dimension $m=d$ need not give the strongest bound. Put $N=(M_\beta R_\alpha/\eps)^2$,
$\Psi=\gamma+\theta$ and $c=(Nd^{2\gamma})^{1/(1+2\Psi)}$. If $\Psi>1/2$ and
$m^\dagger=\min\{d,\lfloor(Nd^{2\gamma}/4)^{1/(1+2\Psi)}\rfloor\}\ge m_Z^\circ$,
then
\begin{equation}\label{g:eq-lp-peak}
 T\ge c_Z\,\frac{(m^\dagger)^2}{\log(e\,m^\dagger)} ,
\end{equation}
which can improve on \eqref{g:eq-lp-nonsmooth} only when $c<d$, that is, when
$\eps>M_\beta R_\alpha d^{-(1/2+\theta)}$; for $c\ge d$ the dimension $m=d$
already attains the cap stated below. Over the block family, the quantity
$m\sigma_m$, where
$\sigma_m=\min\{m,N\lfloor d/m\rfloor^{2\gamma}m^{-2\theta}\}$,
determines the right-hand side of \eqref{gp:eq-nonsmooth} up to the factor
$\log(ed)$ and satisfies
\begin{equation}\label{g:eq-ns-cap}
 \max_{1\le m\le d}m\sigma_m\le
 \begin{cases}
  d\min\{d,Nd^{-2\theta}\},&\Psi\le1/2,\\
  \min\{d,c\}^2,&\Psi>1/2.
 \end{cases}
\end{equation}
The first bound is attained at $m=d$. Here $\Psi=\theta$ exactly when
$\beta\ge\alpha$, and in particular whenever $\theta>1/2$
(\cref{g:lem-exponents}).
\end{corollary}

\begin{table}[tb]
\centering\small
\caption{Independent choices of the feasible ball and the regularity norm,
with constants indexed by their associated norms as in
\eqref{g:eq-constants}: $R_\alpha$ is an $\ell_\alpha$ radius, and
$L_\beta,M_\beta$ use $\ell_\beta$ and its dual. These constants can differ
for one objective, so rows sharing $\theta$ need not agree
(\cref{g:rem-constants}). Entries are the product branches of
\eqref{g:eq-lp-smooth} and \eqref{g:eq-lp-nonsmooth}, active for
$L_\beta R_\alpha^2d^{-2\theta-2}\le\eps\le L_\beta R_\alpha^2d^{-2\theta}/1280^2$
and for
$M_\beta R_\alpha d^{-\theta-1/2}\le\eps\le\eps_ZM_\beta R_\alpha d^{-\theta}$,
$d\ge m_Z^\circ$; outside these regimes the relevant branches are $d^2/\Cstar$, $5d/16$, or
\eqref{g:eq-lp-lowacc}. The nonsmooth column omits the denominator
$\log(e\sigma)\le\log(ed)$ of \eqref{g:eq-lp-nonsmooth}, as its header indicates;
$c_Z,\eps_Z,m_Z^\circ$ are universal (\cref{g:app}). All rows use the
certificate \eqref{g:eq-lp-radius}; in the $p$ row, $q$ is the conjugate
exponent of $p$. The full-dimensional radii are exact. The matching complexity
regimes are specified in \cref{sec:matching}.}
\label{tab:lp}
\begin{tabularx}{\textwidth}{@{}l l c >{\raggedright\arraybackslash}X
 >{\raggedright\arraybackslash}X@{}}
\toprule
Minimizer in & Regularity constant & $\theta$
 & Smooth: $T\ge$ & Nonsmooth: $\log(ed)\cdot T\gtrsim$\\
\midrule
$\mathbb B_\alpha^d(R_\alpha)$, $\alpha\ge2$ & $L_2,M_2:\ \ell_2\!\to\!\ell_2$ & $0$
 & $dR_\alpha\sqrt{L_2/\eps}/\Cstar$ & $c_Z\,dM_2^2R_\alpha^2/\eps^2$\\
$\mathbb B_1^d(R_1)$ & $L_2,M_2:\ \ell_2\!\to\!\ell_2$ & $1/2$
 & $\sqrt d\,R_1\sqrt{L_2/\eps}/\Cstar$ & $c_Z\,M_2^2R_1^2/\eps^2$\\
$\mathbb B_1^d(R_1)$ & $L_1,M_1:\ \ell_1\!\to\!\ell_\infty$ & $1/2$
 & $\sqrt d\,R_1\sqrt{L_1/\eps}/\Cstar$ & $c_Z\,M_1^2R_1^2/\eps^2$\\
$\mathbb B_1^d(R_1)$ & $L_\infty,M_\infty:\ \ell_\infty\!\to\!\ell_1$ & $1$
 & $R_1\sqrt{L_\infty/\eps}/\Cstar$ & $c_Z\,M_\infty^2R_1^2/(d\eps^2)$\\
$\mathbb B_p^d(R_p)$, $p\le2$ & $L_\beta,M_\beta$, $\beta\in\{p,2\}$ & $\tfrac1p-\tfrac12$
 & $d^{\frac12+\frac1q}R_p\sqrt{L_\beta/\eps}/\Cstar$
 & $c_Z\,d^{2/q}M_\beta^2R_p^2/\eps^2$\\
Any $\alpha,\beta$ & $L_\beta,M_\beta:\ \ell_\beta\!\to\!\ell_{\beta^*}$ & \eqref{g:eq-theta}
 & $d^{1-\theta}R_\alpha\sqrt{L_\beta/\eps}/\Cstar$
 & $c_Z\,d^{1-2\theta}M_\beta^2R_\alpha^2/\eps^2$\\
\bottomrule
\end{tabularx}
\end{table}

\begin{figure}[tb]
\centering
\includegraphics[width=0.97\linewidth]{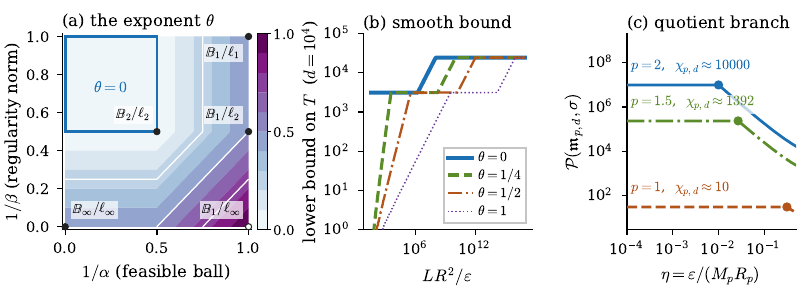}
\caption{\textbf{(a)} The exponent $\theta(\alpha,\beta)$ of \eqref{g:eq-theta},
with the rows of \cref{tab:lp} marked: it is zero on $\beta\le2\le\alpha$ and
rises to $1$ when the feasible set is an $\ell_1$ ball and regularity is
measured in $\ell_\infty$. \textbf{(b)} The proved smooth bound at
$d=10^{4}$: for $\theta>0$ the branch \eqref{g:eq-lp-lowacc} has log--log
slope $1/(2\theta)$, which delays the product branch and the $d^2/\Cstar$
plateau by $d^{2\theta}$. \textbf{(c)} The quotient branch
\eqref{g:eq-quotient-ns} at $d=10^4$: its plateau is
$\mathfrak m_{p,d}^2/\log(e\mathfrak m_{p,d})$ up to the marked regime change
$\eta=\mathfrak m_{p,d}^{-1/2}$; the other branch is
$\mathfrak m_{p,d}\eta^{-2}/\log(e\eta^{-2})$, while the effective dimension
decreases from $d$ at $p=2$ to $\log(ed)$ at $p=1$. The panel plots
$\chi_{p,d}$ and sets $r_G=1$; $\mathfrak m_{p,d}=\Theta(\chi_{p,d})$.
All three panels plot analytical rate expressions with universal
constants suppressed.}
\label{fig:geometry}
\end{figure}

\paragraph{Example: an $\ell_1$ ball with gradient regularity in $\ell_2$.}
Here $\alpha=1$, $\beta=2$, $\gamma=0$ and $\theta=1/2$: the minimizer
satisfies $\|x^\star\|_1\le R_1$, while the gradient is Lipschitz in $\ell_2$
with constant $L_2$. The certificate is the identity at $m=d$: it does not
expand $\ell_2$, and $\mathbb B_1^d(R_1)\supseteq(R_1/\sqrt d)\mathbb B_2^d$,
so in fact $r_d=R_1/\sqrt d$ exactly (\cref{gp:prop-full-lp}); at image
dimensions $m<d$, every $Q$ has image radius at most
$R_1\sqrt{2\log(2ed)/m}$. Substitution gives
$T\ge\max\{5d/16,\ (d/\Cstar)\min\{d,\sqrt{L_2R_1^2/(d\eps)}\}\}$ for
$\eps\le L_2R_1^2/(32d)$, with the $d^2$ plateau for
$\eps\le L_2R_1^2/d^{3}$. For every $\eps\le L_2R_1^2/32$, the bound with a
dimension-free window in \eqref{g:eq-lp-lowacc} gives
$T\ge\frac5{16}\min\{d,\lfloor L_2R_1^2/(32\eps)\rfloor\}$. On the subrange
$L_2R_1^2/(32d)\le\eps\le L_2R_1^2/32$, where this minimum is not capped, the
bound has the same accuracy dependence as the $O(L_2R_1^2/\eps)$ iteration
bound of conditional-gradient methods in this setting, and it is smaller
than their $O(dL_2R_1^2/\eps)$ value bound by a factor of order $d$.
For an $M_2$-Lipschitz $g$ on the same ball, \eqref{g:eq-lp-nonsmooth} gives
$T\ge c_Z\mathcal P(d,\min\{d,M_2^2R_1^2/(d\eps^2)\})$, provided
$d\ge m_Z^\circ$ and $\eps\le\eps_ZM_2R_1d^{-1/2}$.

\paragraph{The nonsmooth case on a non-Euclidean ball: a second certificate.}
At full dimension, the block map of \eqref{g:eq-lp-radius} incurs a
reduction in radius. On the diagonal $\alpha=\beta=p\le2$, where
$\theta=1/p-1/2$ and $\gamma=0$, one can instead reduce the dimension and
obtain an \emph{almost Euclidean} image whose inradius is a constant fraction
of $R_p$. With $q=p/(p-1)$, interpreted as $q=\infty$ at $p=1$,
$1/\infty=0$, and $\bar q=\min\{q,\log(ed)\}$, put
\begin{equation}\label{g:eq-eff-dim}
 \chi_{p,d}=\min\{d,\ \bar q\,d^{2/q}\},\qquad
 \mathfrak m_{p,d}=\max\{1,\lfloor c_D\chi_{p,d}\rfloor\}\quad(c_D\text{ universal}),
\end{equation}
the \emph{effective Euclidean dimension} of $\mathbb B_p^d$: it is of order $d$
at $p=2$ but only of order $\log(ed)$ at $p=1$. The next proposition constructs a constant-radius image of this
dimension.

\begin{proposition}[Almost-Euclidean quotient]\label{g:prop-quotient}
There are universal $c_D,r_G\in(0,1)$ such that
$\mathcal R_m(\mathbb B_p^d,\mathbb B_p^d(R_p))\ge r_GR_p$ at
$m=\mathfrak m_{p,d}$, for all $1\le p\le2$ and $d\ge2$; a Gaussian realization
also has Euclidean spectral scale
$\rho_{p,d}=\|Q^*\|_{2\to2}\asymp d^{1/2-1/q}/\sqrt{\bar q}$
(\cref{g:prop-flm,g:prop-gauss-rho}).
\end{proposition}

Write $t_{p,d}=\max\{p,1+1/\log(ed)\}$ and
$a_{p,d}=d^{2(1/p-1/t_{p,d})}/(t_{p,d}-1)$ for the Lipschitz budget of the prox
function $\omega=\tfrac{a_{p,d}}2\|x\|_{t_{p,d}}^2$, which is one-strongly convex
relative to $\ell_p$; one has $a_{p,d}\asymp\bar q$ (\cref{g:lem-prox}).

\begin{corollary}[Nonsmooth $\ell_p$ bounds in a dimension-free window]
\label{g:cor-quotient-ns}
Let $m=\mathfrak m_{p,d}\ge m_Z^\circ$ and $N=(M_pR_p/\eps)^2$. For every
$0<\eps\le\eps_ZM_pR_pr_G$, a condition involving \emph{no} power of $d$,
\begin{equation}\label{g:eq-quotient-ns}
 T\ \ge\ c_Zr_G^2\,\mathcal P\big(m,\ \min\{m,N\}\big),
\end{equation}
so that, up to universal constants and the displayed logarithm,
\begin{equation}\label{g:eq-endpoints}
 p=2:\ \ T\gtrsim\frac{d\min\{d,N\}}{\log(e\min\{d,N\})},
 \qquad
 p=1:\ \ T\gtrsim\frac{\ell\min\{\ell,N\}}{\log(e\min\{\ell,N\})},
 \quad \ell=\log(ed).
\end{equation}
If the function must also be $M_2$-Lipschitz in the Euclidean norm, the same
bound holds with $M_pR_p$ replaced by $\Lambda=R_p\min\{M_p,M_2/\rho_{p,d}\}$,
for the Gaussian realization of \cref{g:prop-gauss-rho}
(\cref{g:cor-two-scale}). If instead the objective is $\mu$-strongly convex
relative to $\ell_p$ and $\mu a_{p,d}R_p\le\theta_0M_p$ for a fixed
$\theta_0<1$, then for $\eps\le c_{\theta_0}\min\{M_pR_p,M_p^2/(\mu a_{p,d})\}$,
\begin{equation}\label{g:eq-quotient-sc}
 T\ \gtrsim_{\theta_0}\ \mathcal P(m,\Xi),\qquad
 \Xi=\min\left\{m,\ N,\ \frac{M_p^2}{\mu a_{p,d}\eps}\right\},
\end{equation}
(\cref{g:lem-prox,g:cor-sc-readable}).
When both Lipschitz bounds and strong convexity are required simultaneously,
the regularizer contribution must be included in both budgets, as in
\cref{g:cor-two-scale}; substituting $\Lambda$ alone is insufficient.
\end{corollary}

The two certificates trade dimension against radius differently, and
neither dominates the other, so the bound we state is the envelope of the
two. Write $\eta=\eps/(M_pR_p)$. For $p<2$, a coordinate certificate of
dimension $m$ has radius $R_pm^{-\theta}$, so at accuracy $\eta$ it is
admissible only for $m\le(\eps_Z/\eta)^{1/\theta}$, whereas the quotient
certificate keeps a constant radius at dimension $\mathfrak m_{p,d}$. At the
coarse end $\eta=\eps_Zr_G$, put
$t=r_G^{-1/\theta}$, $J=\min\{d,\lfloor t\rfloor\}$,
$s=\eps_Z^{-2}$, and $m_Q=\mathfrak m_{p,d}\ge m_Z^\circ$.
The exact certificate expressions of \cref{gp:thm-nonsmooth}, before the
constants are simplified in \eqref{g:eq-quotient-ns}, have the capped
parameters
\[
 \sigma_j=\min\{j,s(t/j)^{2\theta}\}\quad\text{for coordinates},\qquad
 \sigma_Q=\min\{m_Q,s\}\quad\text{for the quotient}.
\]
Only coordinate dimensions $m_Z^\circ\le j\le J$ are admissible.
The quotient certificate is strictly larger than the envelope of the
coordinate certificates exactly when $m_Q>J$. To prove this comparison, write $a=2\theta\in(0,1]$ and consider
$F(v)=\mathcal P(v,\min\{v,s(t/v)^a\})$ for real $1\le v\le t$.
On its capped branch, $F(v)=v^2/\log(ev)$ is increasing. On its other branch,
\[
 F(v)=\frac{st^av^{1-a}}{\log(e\,st^av^{-a})}
\]
is increasing as well, since the numerator is nondecreasing and the positive
denominator decreases. Hence $F(J)$ is the largest coordinate expression
whenever $J\ge m_Z^\circ$. If $m_Q>J$, then $J=\lfloor t\rfloor$ and $m_Q>t$,
so
\[
 F(J)\le F(t)=\mathcal P(t,\min\{t,s\})
 <\mathcal P(m_Q,\min\{m_Q,s\}),
\]
where the last inequality uses the strict increase in the displayed
dimension. If $J<m_Z^\circ$, there is no admissible coordinate certificate to
compare with. Conversely, when $m_Q\le J$, the coordinate certificate of
dimension $m_Q$ has radius at least $r_G$ and therefore gives at least the
quotient bound. For each fixed $p<2$, the inequality $m_Q>J$ holds for all
sufficiently large $d$, although the threshold grows as $p\uparrow2$.
At sufficiently high accuracy the comparison reverses: when
$\eta\le\eps_Zd^{-\theta}$, the identity is admissible and
\eqref{g:eq-lp-nonsmooth} can reach $d^2/\log(ed)$ as the accuracy increases,
which the quotient certificate cannot attain when $\mathfrak m_{p,d}<d$ ---
that is, for fixed $p<2$ at large $d$, and at every $d\ge2$ when $p=1$. The
quotient profile has the plateau $\mathfrak m_{p,d}^2/\log(e\mathfrak m_{p,d})$
until $\eta\asymp\mathfrak m_{p,d}^{-1/2}$; its other branch is proportional
to $\mathfrak m_{p,d}\eta^{-2}/\log(e\eta^{-2})$ (\cref{fig:geometry}c). The
Euclidean scale also depends on the geometry: at $p=1$, the quotient has
logarithmic dimension and spectral scale $\rho_{1,d}\asymp\sqrt{d/\log(ed)}$.
The Euclidean subgradient bound enters through this spectral scale
(\cref{g:rem-normalization}).

\paragraph{A regularizer with nonuniform curvature.}
For $B=\mathbb B_2^d$, $K=R\mathbb B_2^d$ and $\psi(x)=x^\top Hx/2$ with eigenvalues
$0<\nu_1\le\cdots\le\nu_d$, the actual sublevel sets give
$r_m(\tau)=\min\{R,\sqrt{2\tau/\nu_m}\}$ on the $m$ eigendirections of lowest
curvature, and \eqref{gp:eq-composite} is maximized at $\tau=\nu_mR^2/2$. For
$H=\diag(I_n,AI_n)$, with the remaining parameters chosen as in
\cref{gp:prop-gain,gp:prop-ns-gain}, this improves the bound certified by the
scalar majorant $A\|x\|_2^2/2$ by a factor $\sqrt{A+1}/(2\sqrt2)$ when $A>7$,
and by a factor $A/4$ in the nonsmooth case when $A>4$; the two families use
different dimensions and regularity constants
(\cref{gp:cor-spectral,gp:prop-gain,gp:prop-ns-gain}). Schatten and endpoint
image radii are computed in \cref{gp:app}. Taking the maximum over the image profiles combines their lower
bounds on the same complexity.

\section{Euclidean images and sublevels of known regularizers}\label{gp:app}

This appendix proves the norm and known-regularizer transfers of
\cref{gp:thm-norm,gp:thm-composite}, the spectral corollary in
\cref{gp:cor-spectral}, and the geometric examples. All smooth information
bounds below rest on \cref{thm:main} alone, with $\Cstar=4096$. The geometric statements compute or bound the image radius
$\mathcal R_m(B,C)$ of \eqref{gp:eq-radius}. \Cref{sec:matching}
identifies the resulting matching complexity regimes.

\subsection{Attainment, centers, and smooth oracle transfers}

Let $B$ be the unit ball of a norm $\|\cdot\|$ on $\R^d$, and let
$\|\cdot\|_*$ be its dual norm. For $Q:\R^d\to\R^m$, write
$\|Q\|_{B\to2}=\sup_{x\in B}\|Qx\|_2$. Its Euclidean adjoint satisfies
\begin{equation}\label{gp:eq-adjoint}
 \|Q^*u\|_*\le\|Q\|_{B\to2}\|u\|_2.
\end{equation}
Indeed, the left-hand side equals $\sup_{x\in B}\langle u,Qx\rangle$.

\begin{lemma}[Attainment and elementary invariances]\label{gp:lem-attainment}
Let $C\subset\R^d$ be nonempty, compact, and convex, and let $1\le m\le d$.
The maximum defining $\mathcal R_m(B,C)$ is attained. For $c\in\R^d$ and
$t>0$,
\[
 \mathcal R_m(B,c+tC)=t\mathcal R_m(B,C).
\]
The radius is monotone in $C$ under inclusion. If $C=-C$, optimizing over
translated inscribed balls gives the same result as requiring their centers
to be zero. A positive radius requires $\operatorname{rank}Q=m$ for every
certificate attaining the maximum, and $\mathcal R_m(B,C)=0$ when
$m>\dim\operatorname{aff}C$.
\end{lemma}
\begin{proof}
First, the contractions form a compact subset of the finite-dimensional
space of linear maps: the operator norm is continuous, and by norm
equivalence the set is bounded in any matrix norm. Set
$M=\max_{x\in C}\|x\|$. Whenever $z+r\mathbb B_2^m\subseteq Q(C)$ and
$\|Q\|_{B\to2}\le1$, we have $\|z\|_2\le M$, because $z\in Q(C)$, and
$0\le r\le M$, because the diameter of $Q(C)$ is at most $2M$. Hence the
feasible triples $(Q,z,r)$ lie in a fixed compact set. They also form a
closed set: with $h_C(v)=\max_{x\in C}\langle v,x\rangle$, the containment
is equivalent to
\[
 h_C(Q^*u)-\langle u,z\rangle\ge r\|u\|_2
 \quad\text{for every }u\in\R^m.
\]
Each inequality defines a closed set because $h_C$ is continuous. The
feasible set is nonempty, since it contains $Q=0,z=0,r=0$, so the maximum
of $r$ is attained.

Next, translating or positively scaling $C$ translates or scales every
image accordingly, and inclusion gives monotonicity. If $C=-C$ and
$z+r\mathbb B_2^m\subseteq Q(C)$, then by reflection also
$-z+r\mathbb B_2^m\subseteq Q(C)$, and taking midpoints, by convexity,
gives $r\mathbb B_2^m\subseteq Q(C)$. Finally, a positive radius forces
$Q(C)$ to be full-dimensional in $\R^m$, whereas its affine dimension is at
most both $\operatorname{rank}Q$ and $\dim\operatorname{aff}C$.
\end{proof}

\begin{proof}[Proof of \cref{gp:thm-norm}]
Fix an admissible $m$ and set $r=\mathcal R_m(B,K)>0$.
By \cref{gp:lem-attainment}, choose a fixed contraction $Q$ and a fixed
$z\in\R^m$ with $z+r\mathbb B_2^m\subseteq Q(K)$. For every
convex, Euclidean $L$-smooth source objective $q:\R^m\to\R$ with a
global minimizer $u_*$ in $r\mathbb B_2^m$, define
\[
 f_q(x)=q(Qx-z),\qquad x\in\R^d.
\]
This function is convex and differentiable, and by \eqref{gp:eq-adjoint},
\begin{align*}
 \|\nabla f_q(x)-\nabla f_q(y)\|_*
 &\le\|\nabla q(Qx-z)-\nabla q(Qy-z)\|_2\\
 &\le L\|Q(x-y)\|_2\le L\|x-y\|.
\end{align*}
The image inclusion provides $x_*\in K$ with $Qx_*-z=u_*$, so $x_*$ is a
global minimizer of $f_q$ and $\min f_q=\min q$.

Now suppose a target algorithm with budget $T$ satisfies the required
guarantee for every such $f_q$. A source algorithm runs it with the same
random seed. Whenever the target requests $f_q(x)$, the simulator requests
the single source value $q(Qx-z)$ and returns it unchanged; given the
target output $\widehat x$, it returns $Q\widehat x-z$. The objective gaps
agree exactly. The fixed affine query and output maps are continuous, so
composing them with the target's total Borel policies produces total Borel
source policies with the same independent seed, budget, and success
probability. All source queries and the decoded output are permitted, since
the source model is unrestricted. The lift establishes membership of $f_q$ in the target class; the
query simulator and output decoder use the fixed affine map $x\mapsto Qx-z$.

Finally, apply \cref{thm:main} in dimension $m$, at radius $r$ and accuracy
$\eps$, under the stated accuracy condition $0<\eps\le Lr^2/32$. This gives
\eqref{gp:eq-norm}. Since every admissible dimension bounds the same target
complexity, their maximum is also a lower bound. Positive-definite source
quadratics suffice; their pullbacks are convex quadratics, which can have a
kernel when $m<d$.
\end{proof}

\begin{proof}[Proof of \cref{gp:thm-composite}]
Fix admissible $m,\tau$, put $r=\mathcal R_m(B,K_\tau)>0$ and
$\delta=\eps+\lambda\tau$, and choose a fixed certificate
$z+r\mathbb B_2^m\subseteq Q(K_\tau)$. For a Euclidean source objective
$q$ as in the preceding proof, define
\[
 g_q(x)=q(Qx-z),\qquad F_q(x)=g_q(x)+\lambda\psi(x),\quad x\in K.
\]
The same calculation as before shows that $g_q$ is globally convex and
smooth in the specified norm. A source minimizer $u_*\in r\mathbb B_2^m$ has
a lift $x_0\in K_\tau$, so $g_q(x_0)=\min q$ and $0\le\psi(x_0)\le\tau$.
For every feasible $\eps$-optimal output, the nonnegativity of $\psi$ on $K$
gives
\begin{align*}
 q(Q\widehat x-z)
 &\le F_q(\widehat x)
 \le\min_{x\in K}F_q(x)+\eps\\
 &\le F_q(x_0)+\eps
 \le\min q+\lambda\tau+\eps.
\end{align*}
Hence the same affine decoder is $\delta$-accurate for the source.

Each query to the globally defined unknown term $g_q$ is exactly one
unrestricted source query. If the interface returns $F_q$ on its query
domain, the simulator adds the known value $\lambda\psi(x)$; this requires a
specified finite extension of the penalty wherever queries outside $K$ are
allowed. Conversely, subtracting that known value recovers $g_q(x)$ on this
domain. Any free operations that depend only on the known penalty, set, and
parameters are independent of $q$ and can be built into the simulator;
within the stated Borel algorithm model they do not affect the reduction.
The measurability and fixed-budget argument of the preceding proof applies.
Finally, $\delta>0$ and the assumed condition $\delta\le Lr^2/32$ give
exactly the source accuracy window, and applying \cref{thm:main} at
accuracy $\delta$ proves \eqref{gp:eq-composite}. Taking the supremum over
all admissible $m,\tau$ preserves this bound.
\end{proof}

\begin{proposition}[Affine normalization of a known penalty]\label{gp:prop-affine}
For a known affine function $\ell:\R^d\to\R$, define
\[
 c_\ell=\min_{x\in K}(\psi(x)-\ell(x)),\qquad
 \phi_\ell(x)=\psi(x)-\ell(x)-c_\ell.
\]
The smooth composite class with known penalty $\psi$ is equivalent, up to
known additive constants and one-for-one value simulation, to that with
penalty $\phi_\ell$. Consequently, \eqref{gp:eq-composite} may also be
optimized over these normalized penalties, subject to the admissibility
condition for each sublevel.
\end{proposition}
\begin{proof}
By compactness and continuity, $c_\ell$ is finite, and $\phi_\ell$ is
convex, continuous, and nonnegative on $K$ with minimum zero. Set
$\widetilde g=g+\lambda\ell$. Adding an affine function preserves convexity
and the gradient Lipschitz constant in any norm, and
\[
 g+\lambda\psi=\widetilde g+\lambda\phi_\ell+\lambda c_\ell.
\]
The correspondence is a bijection on the unknown smooth class, with
inverse $g=\widetilde g-\lambda\ell$. Oracle values differ by the known
affine term, while the two composite objectives have identical
suboptimality gaps on $K$. It remains to apply \cref{gp:thm-composite}
after this reparameterization. The same conclusion does not follow for a
fixed global Lipschitz budget, since adding a nonconstant affine function
can change that budget.
\end{proof}

\subsection{Exact ellipsoidal and quadratic-sublevel profiles}

\begin{proposition}[Ellipsoidal geometries]\label{gp:prop-ellipsoid}
Let $P,S$ be positive-definite $d\times d$ matrices, let
$B_P=\{x:x^\top Px\le1\}$, and let $C=c+S^{1/2}\mathbb B_2^d$.
If $\nu_m^\downarrow$ denotes the $m$th largest eigenvalue, then
\begin{equation}\label{gp:eq-ellipsoid}
 \mathcal R_m(B_P,C)^2
 =\nu_m^\downarrow(P^{1/2}SP^{1/2}),\qquad 1\le m\le d.
\end{equation}
In particular, Euclidean smoothness and a feasible ellipsoid with semiaxes
$a_1\ge\cdots\ge a_d>0$ give $\mathcal R_m=a_m$; matching geometries
$C=c+RB_P$ give $\mathcal R_m=R$ for every $m$.
\end{proposition}
\begin{proof}
By translation invariance we may drop $c$. The contraction condition reads
$\|QP^{-1/2}\|_{2\to2}\le1$, so write $Q=AP^{1/2}$ with
$\|A\|_{2\to2}\le1$, and put $D=P^{1/2}SP^{1/2}$. The image ellipsoid
has shape matrix $ADA^*$ and Euclidean inradius
$\sqrt{\nu_{\min}(ADA^*)}$, which is zero for a singular image.
Choosing the rows of $A$ to be orthonormal eigenvectors associated with the
largest eigenvalues of $D$ gives the lower bound in \eqref{gp:eq-ellipsoid}.
For the reverse bound, take any such $A$ and choose a unit vector
$u\in\R^m$ such that $A^*u$ is orthogonal to the top $m-1$ eigenvectors of
$D$; such a vector exists because these are at most $m-1$ homogeneous
linear constraints. Then
\[
 \nu_{\min}(ADA^*)\le u^\top ADA^*u
 \le\nu_m^\downarrow(D)\|A^*u\|_2^2
 \le\nu_m^\downarrow(D).
\]
The stated examples follow by taking $P=I$ or $S=R^2P^{-1}$.
\end{proof}

\begin{proposition}[Quadratic sublevels, including affine tilts]\label{gp:prop-sublevel}
Let $K=R\mathbb B_2^d$, $R>0$, and
$\psi(x)=x^\top Hx/2$, where $H$ is positive definite with eigenvalues
$0<\nu_1\le\cdots\le\nu_d$. For every $\tau\ge0$,
\begin{equation}\label{gp:eq-sublevel-radius}
 \mathcal R_m(\mathbb B_2^d,K_\tau)
 =\min\left\{R,\sqrt{\frac{2\tau}{\nu_m}}\right\}.
\end{equation}
The right-hand side also equals the supremum of the image radii of
$\{x\in K:\phi_\ell(x)\le\tau\}$ over all affine $\ell$, with the
normalization in \cref{gp:prop-affine}.
\end{proposition}
\begin{proof}
For $\tau>0$, the set $K_\tau$ is contained both in $R\mathbb B_2^d$
and in the ellipsoid with shape matrix $2\tau H^{-1}$. By monotonicity and
\cref{gp:prop-ellipsoid}, its image radius is therefore at most the
right-hand side of \eqref{gp:eq-sublevel-radius}. For the reverse
inequality, let $Q_m$ be the orthogonal projection, in eigenvector
coordinates, onto the $m$ directions of lowest curvature; it is a Euclidean
contraction. Every $u\in\R^m$ with norm at most the displayed radius has
the lift $x=Q_m^*u$, which satisfies
\[
 \|x\|_2=\|u\|_2\le R,\qquad
 \tfrac12x^\top Hx\le\tfrac12\nu_m\|u\|_2^2\le\tau.
\]
Thus this projection attains the upper bound. At $\tau=0$, the centered
sublevel set is $\{0\}$ and both sides of the identity vanish.

For an affine tilt, the constant term cancels in the normalization, so
write $\ell(x)=a^\top x$. Let $x_*\in K$ minimize
$h(x)=x^\top Hx/2-a^\top x$ over $K$. First-order optimality on the convex
set $K$ gives $\langle Hx_*-a,x-x_*\rangle\ge0$ for every $x\in K$, and
hence
\[
 \phi_\ell(x)=h(x)-h(x_*)
 =\tfrac12(x-x_*)^\top H(x-x_*)
   +\langle Hx_*-a,x-x_*\rangle
 \ge\tfrac12(x-x_*)^\top H(x-x_*).
\]
Consequently, every tilted sublevel set is contained in $K$ and in the
ellipsoid $x_*+\sqrt{2\tau}\,H^{-1/2}\mathbb B_2^d$, and translation
invariance, monotonicity, and \cref{gp:prop-ellipsoid} give the same upper
bound on its image radius. The argument also covers a minimizer $x_*$ on
the boundary of $K$, and at $\tau=0$ it yields a singleton. The zero tilt
attains the bound by the first part of the proof.
\end{proof}

\begin{proof}[Proof of \cref{gp:cor-spectral}]
Fix $m$ and write $t_m=\nu_mR^2/2$. By
\cref{gp:prop-sublevel}, the squared accuracy factor in
\eqref{gp:eq-composite} is
\[
 \eta_m(\tau)
 =\frac{L\min\{R^2,2\tau/\nu_m\}}
        {\eps+\lambda\tau}.
\]
On $0\le\tau\le t_m$ its derivative is
$2L\eps/[\nu_m(\eps+\lambda\tau)^2]>0$, and on $\tau\ge t_m$ it is
$-\lambda LR^2/(\eps+\lambda\tau)^2\le0$.
Thus $t_m$ is optimal; when $\lambda=0$, every $\tau\ge t_m$ is also
optimal. The source window is exactly $\eta_m(\tau)\ge32$, so an
admissible level exists if and only if
\[
 \eps+\lambda\nu_mR^2/2\le LR^2/32.
\]
Both the capped product term and the full maximum with the dimension
baseline are nondecreasing in $\eta_m$. Substituting $\tau=t_m$ into
\eqref{gp:eq-composite} therefore proves \eqref{gp:eq-spectral} with its
stated window, and it optimizes the choice of sublevel for this $m$.
By \cref{gp:prop-sublevel}, allowing the affine normalizations of
\cref{gp:prop-affine} cannot improve this conclusion.
\end{proof}

\subsection{An unbounded gain over the scalar curvature majorant}

We compare two lower-bound expressions for the same known-composite class.
A bound based on the scalar majorant $\psi(x)\le\nu_d\|x\|_2^2/2$, where
$\nu_d$ is the largest eigenvalue of $H$, certifies only the centered balls
\[
 \left\{x\in K:\nu_d\|x\|_2^2/2\le\tau\right\}
 =\rho(\tau)\mathbb B_2^d,\qquad
 \rho(\tau)=\min\{R,\sqrt{2\tau/\nu_d}\}.
\]
Their exact image radius is $\rho(\tau)$ in every dimension $m\le d$.
Optimizing over all linear images of these certified balls, all
dimensions, and all admissible levels therefore gives a precise
scalar-majorant benchmark. The calculation below keeps the dimension
baseline, the saturation cap, and the validity window in both expressions.

\begin{proposition}[A fully specified strict-gain family]\label{gp:prop-gain}
Let $A>7$, $\lambda>0$, and define
\[
\begin{gathered}
 D=2560,\quad b=D\sqrt{\frac{A+1}{2}},\quad n=\lceil b\rceil,
 \quad d=2n,\\
 R=1,\quad
 H=\operatorname{diag}(I_n,AI_n),\quad
 \eps=\frac\lambda2,\quad L=\lambda b^2.
\end{gathered}
\]
Let $\mathcal L_{\mathrm{profile}}$ be the largest right side of
\eqref{gp:eq-composite}, optimized over dimensions, admissible actual
sublevels, and affine normalizations. Let
$\mathcal L_{\mathrm{majorant}}$ be the corresponding optimized expression
using the certified balls from the scalar majorant $A\|x\|_2^2/2$.
Then
\begin{equation}\label{gp:eq-strict-gain}
 \mathcal L_{\mathrm{profile}}=\frac{nb}{4096},\qquad
 \mathcal L_{\mathrm{majorant}}=\frac{2nD}{4096},\qquad
 \frac{\mathcal L_{\mathrm{profile}}}
      {\mathcal L_{\mathrm{majorant}}}
 =\frac{\sqrt{A+1}}{2\sqrt2}.
\end{equation}
In particular, the improvement is unbounded as $A\to\infty$.
\end{proposition}
\begin{proof}
For a dimension $m$ and an accuracy factor $s$, denote the full source
expression by
\[
 \mathcal B(m,s)=\max\left\{\frac{5m}{16},
                         \frac{m\min\{m,s\}}{4096}\right\}.
\]
It is nondecreasing in $m$ for fixed $s$. The spectral optimization above
gives
\[
 s_m=\sqrt{\frac{L}{\eps+\lambda\nu_m/2}}
 =\begin{cases}b,&1\le m\le n,\\ D,&n<m\le2n.\end{cases}
\]
Both branches satisfy their source accuracy conditions, since
$s_m^2\ge D^2>32$. On the first branch, the maximum is attained at $m=n$:
because $n\ge b>D=2560>1280=5\cdot4096/16$, the cap is inactive at $b$ and
the product term dominates the baseline, so $\mathcal B(n,b)=nb/4096$. On
the second branch, the maximum is attained at $m=2n$, and since
$2n>D>1280$ its value is $\mathcal B(2n,D)=2nD/4096$. The assumption $A>7$
gives $b>2D$, so the first value is larger. The corresponding
optimal levels are $1/2$ and $A/2$, respectively. By
\cref{gp:prop-sublevel}, affine normalization cannot increase any of these
radii. This proves the first equality.

For the balls certified by the scalar majorant, the same one-variable
optimization gives the optimal level $\tau=A/2$ in every dimension. Its
accuracy factor is exactly
\[
 \sqrt{\frac{L}{\eps+\lambda A/2}}
 =\sqrt{\frac{\lambda D^2(A+1)/2}{\lambda(A+1)/2}}=D.
\]
This is admissible because $D^2>32$. Optimizing $\mathcal B(m,D)$ over all
$1\le m\le2n$ gives $m=2n$ and the value $2nD/4096$, by the cap and
baseline comparisons established above. No other level can increase the
accuracy factor, and the radius of each scalar-majorant ball is already
optimal over all contraction images. Dividing the two expressions yields
\eqref{gp:eq-strict-gain}, with no loss from the ceiling in $n$.
\end{proof}
These equalities compare the real-valued lower-bound expressions
before rounding the query budget to an integer. The family varies the
dimension and $L/\lambda$ together with $A$, and the ratio quantifies the
gain from using actual sublevels for the same class. Determining the
fine-accuracy complexity of the smooth-composite class beyond the plateau
in \eqref{gp:eq-spectral} remains open.

\subsection{Block norms, Schatten norms, and endpoint images}

\begin{proposition}[Euclidean blocks with an outer norm]\label{gp:prop-block}
Let $\R^d=\bigoplus_{i=1}^k\R^{d_i}$, with $d_i\ge1$, and equip it
with the norm
$\|x\|_{2,p}=(\sum_i\|x_i\|_2^p)^{1/p}$ for $1\le p<\infty$, or
$\|x\|_{2,\infty}=\max_i\|x_i\|_2$. For a set $I$ of $b\ge1$
blocks, put $m=\sum_{i\in I}d_i$ and $a=|1/p-1/2|$. There is a
contraction image of the unit ball $B$ in $\R^m$ whose Euclidean
inradius is exactly $b^{-a}$. Consequently,
$\mathcal R_m(B,B)\ge b^{-a}$ and
$\mathcal R_m(B,RB)\ge Rb^{-a}$.
\end{proposition}
\begin{proof}
Put $a_-=(1/p-1/2)_+$ and $a_+=(1/2-1/p)_+$, so that
$a=a_-+a_+$, and let $P_I$ retain the selected blocks. Applying the scalar
norm comparisons to the $b$ block lengths gives
\[
 \|P_Ix\|_2\le b^{a_+}\|P_Ix\|_{2,p}\le b^{a_+}\|x\|_{2,p},
 \qquad \|u\|_{2,p}\le b^{a_-}\|u\|_2.
\]
Thus $Q=b^{-a_+}P_I$ is a contraction. If
$\|u\|_2\le b^{-(a_-+a_+)}$, then placing $b^{a_+}u$ in the
selected blocks and zero elsewhere gives a lift in $B$, which establishes
the claimed lower bound on the inradius. In fact, the image is exactly
$b^{-a_+}$ times the unit ball of the selected blocks, and its inradius is
the reciprocal of
$b^{a_+}\max_{\|u\|_2=1}\|u\|_{2,p}=b^{a_-+a_+}$;
the maximum in this identity is attained by equal block lengths when
$p\le2$ and by a vector in a single block when $p\ge2$.
Finally, apply the scaling identity of \cref{gp:lem-attainment}.
\end{proof}

\begin{proposition}[Full-dimensional Schatten image]\label{gp:prop-schatten}
On the real matrix space $\R^{n\times k}$, write $\|X\|_{S_p}$ for the Schatten $p$ norm
and $B_{S_p}$ for its unit ball. Put $N=nk$, $h=\min(n,k)$, and
$a=|1/p-1/2|$, with $1/\infty=0$. Then
\begin{equation}\label{gp:eq-schatten}
 \mathcal R_N(B_{S_p},B_{S_p})=h^{-a},\qquad 1\le p\le\infty.
\end{equation}
For square $n\times n$ matrices, retaining any $j\le n$ columns also
gives a certificate of dimension $nj$ and radius $j^{-a}$.
\end{proposition}
\begin{proof}
Write $\|X\|_F$ for the Frobenius norm. Comparing the $\ell_p$ and $\ell_2$
norms of the at most $h$ nonzero singular values shows that the extrema of
$\|X\|_{S_p}^2/\|X\|_F^2$ over nonzero matrices are
\[
 q_{\min}=\min\{1,h^{2/p-1}\},\qquad
 q_{\max}=\max\{1,h^{2/p-1}\}.
\]
Rank-one matrices and matrices with $h$ equal nonzero singular values
attain the two extrema. The scaled Frobenius identification
$QX=\sqrt{q_{\min}}\,X$ is a contraction with image inradius
$\sqrt{q_{\min}/q_{\max}}=h^{-a}$.

For the upper bound, suppose a contraction $Q$ has an image containing
$r\mathbb B_2^N$ for some $r>0$. By symmetry we may center the ball at
zero, and $Q$ is invertible. The two inclusions are equivalent to
\begin{equation}\label{gp:eq-matrix-sandwich}
 r^2\|X\|_{S_p}^2\le\|QX\|_2^2\le\|X\|_{S_p}^2
 \quad\text{for every }X.
\end{equation}
Average this inequality after replacing $X$ by $UXV$, where $U$ and
$V$ range independently over the signed permutation matrices of sizes
$n$ and $k$; these transformations preserve the Schatten norm. Row and
column sign changes cancel all cross terms of the averaged quadratic
form, and row and column permutations make all its diagonal
coefficients equal, so its value is $c\|X\|_F^2$ for some
$c>0$. The averaged version of \eqref{gp:eq-matrix-sandwich} gives
$r^2q_{\max}\le c\le q_{\min}$, which proves
$r\le h^{-a}$ uniformly over all $Q$.

For the final statement, let $J:\R^j\to\R^n$ be the coordinate
inclusion and compress $X$ to $XJ$. Since
$XJJ^\top X^\top \preceq XX^\top $, the ordered singular values do not increase,
and hence $\|XJ\|_{S_p}\le\|X\|_{S_p}$, also for $p=\infty$.
Zero-padding provides a norm-preserving lift of every $n\times j$
matrix. Composing the compression with the scaled Frobenius
identification for that space gives a map whose image is exactly the
corresponding scaled Schatten ball, of radius $j^{-a}$ by the first part.
\end{proof}

\begin{proposition}[Exact full-dimensional $\ell_p$ images]
\label{gp:prop-full-lp}
For every $1\le\alpha,\beta\le\infty$ and $R>0$,
\begin{equation}\label{gp:eq-full-lp}
 \mathcal R_d(\mathbb B_\beta^d,\mathbb B_\alpha^d(R))
 =R\,d^{-\theta(\alpha,\beta)},\qquad
 \theta(\alpha,\beta)=\left(\frac12-\frac1\beta\right)_+
                    +\left(\frac1\alpha-\frac12\right)_+.
\end{equation}
\end{proposition}
\begin{proof}
By scaling, it suffices to consider $R=1$. Put
$u=(1/2-1/\beta)_+$ and $v=(1/\alpha-1/2)_+$.
The map $Q=d^{-u}I$ is a contraction from $\ell_\beta^d$ to
$\ell_2^d$, and by the elementary norm comparisons its image of
$\mathbb B_\alpha^d$ has Euclidean inradius $d^{-(u+v)}$.

For the upper bound, suppose a contraction $Q$ has an image containing
a ball of positive radius $r$. By symmetry we may center that ball at
zero, and $Q$ must be invertible. The two inclusions imply
\[
 r^2\|x\|_\alpha^2\le\|Qx\|_2^2\le\|x\|_\beta^2
 \qquad(x\in\R^d).
\]
Average this inequality after replacing $x$ by every signed coordinate
permutation of $x$; both norms in the outer terms are invariant under these
transformations. Sign changes remove the cross terms of the middle
quadratic form and permutations equalize its diagonal coefficients, so its
average is $c\|x\|_2^2$ for some $c>0$. The extremal norm ratios therefore
give
\[
 r^2d^{(2/\alpha-1)_+}\le c
 \le d^{\min\{2/\beta-1,0\}}.
\]
Consequently $r\le d^{-(u+v)}$. Singular maps have zero
full-dimensional inradius, so the bound applies to every certificate, and
the equality follows. This establishes the full-dimensional radius. Smaller image
dimensions are treated by the certificates below.
\end{proof}

\begin{proposition}[The $\ell_1$ endpoint]\label{gp:prop-l1}
For $1\le m\le d$,
\begin{equation}\label{gp:eq-l1}
 \mathcal R_m(\mathbb B_1^d,\mathbb B_1^d)^2
 \le\min\left\{1,\frac{2\log(2ed)}{m}\right\}.
\end{equation}
\end{proposition}
\begin{proof}
Let $Q$ be a contraction from $\ell_1^d$ to $\ell_2^m$. Its columns
$v_i=Qe_i$ have norm at most one, and
$Q\mathbb B_1^d=\operatorname{conv}\{\pm v_1,\ldots,\pm v_d\}$.
If this image contains a ball of radius $r$, by symmetry we may center it
at zero, and comparing support functions gives
\[
 r\|u\|_2\le\max_{1\le i\le d}|\langle u,v_i\rangle|
 \quad\text{for every }u\in\R^m.
\]
Let $G$ be a standard Gaussian vector in $\R^m$ and put
$Z=\max_i|\langle G,v_i\rangle|^2$. The Gaussian tail bound and a
union bound give $\Prob\{Z>t\}\le\min\{1,2d e^{-t/2}\}$ for
$t\ge0$, and integrating the tail bound, split at $t_0=2\log(2d)$, yields
\[
 mr^2\le\E Z
 \le 2\log(2d)+\int_{2\log(2d)}^\infty2d e^{-t/2}\,dt
 =2\log(2ed).
\]
Moreover, $r\le1$, since the entire image lies in $\mathbb B_2^m$.
Taking the maximum over contractions proves \eqref{gp:eq-l1}.
\end{proof}

For comparison, the standard Euclidean-section theorem recalled in
\cref{g:app} gives a universal $c>0$ such that, for
$m=\max\{1,\lfloor cd\rfloor\}$,
\begin{equation}\label{gp:eq-linfty-image}
 \mathcal R_m(\mathbb B_\infty^d,\mathbb B_\infty^d)\ge\tfrac12.
\end{equation}
To prove this specialization, observe that on $\ell_1^d$ one has
$\|x\|_2\le\|x\|_1\le\sqrt d\|x\|_2$, and the spherical mean
of $\|\vartheta\|_1$ is comparable to $\sqrt d$: write a standard
Gaussian vector as $G=\|G\|_2\vartheta$ and use
$\E\|G\|_1=d\sqrt{2/\pi}$ and $\E\|G\|_2\asymp\sqrt d$.
The mean--median comparison and the Euclidean-section theorem of
\citet{FigielLindenstraussMilman1977}, in the form stated in
\cref{g:app}, therefore give an injection $U:\R^m\to\R^d$ with
$\|u\|_2\le\|Uu\|_1\le2\|u\|_2$, after scaling and decreasing
$c$. Set $Q=U^*/2$. Duality gives
$\|Q\|_{\infty\to2}\le1$, while
$h_{Q\mathbb B_\infty^d}(u)=\|Q^*u\|_1\ge\|u\|_2/2$ establishes the
inner inclusion. If the section dimension rounds to zero, a coordinate
projection supplies the displayed one-dimensional conclusion instead.
Thus the two endpoints have different lower-dimensional image profiles
despite their equal full-dimensional radii $d^{-1/2}$ from
\cref{gp:prop-full-lp}: \eqref{gp:eq-linfty-image} keeps radius
$1/2$ at $m\asymp d$, whereas \eqref{gp:eq-l1} forces $m=O(\log d)$ for any
fixed positive image radius of $\mathbb B_1^d$. Thus a constant-radius image of $\mathbb B_1^d$ has at most
logarithmic dimension. Other hard-instance constructions may yield
stronger optimization lower bounds.

\section{Geometric transfers and nonsmooth strong convexity}\label{g:app}
\subsection{Specializations to \texorpdfstring{$\ell_p$}{ell-p} balls}
Fix $1\le p\le2$, let $q$ be its conjugate exponent, and put
$\Gamma=d^{2/p-1}$ and $\mathbb B_p^d(R)=\{x:\|x\|_p\le R\}$.
We distinguish global smooth optimization, with unrestricted queries and a
global minimizer in this ball, from nonsmooth optimization, with queries and
output restricted to the ball. All complexities use the fixed-budget and
success-probability conventions of \cref{thm:main}. Both transfers are
query-preserving: one source value answers one target value. The smooth
transfer is a direct corollary of \cref{thm:main}; the nonsmooth transfer
carries an \emph{external} Euclidean bound, stated explicitly as
\cref{g:ext-zzql}, over to non-Euclidean geometry. In the notation of
\eqref{g:eq-constants}, this appendix treats the diagonal case
$\alpha=\beta=p$: the constants $L$, $M$ and $R$ below stand for $L_p$,
$M_p$ and $R_p$, respectively, where $L_p$ acts from $\ell_p$ to $\ell_q$,
$M_p$ is the Lipschitz modulus with respect to $\ell_p$, and $R_p$ is an
$\ell_p$ radius (\cref{g:rem-constants}).

\paragraph{Smooth geometry and strong convexity.}
Let $\mathcal G_p(d,L,R)$ be the class of globally convex functions that
satisfy $\|\nabla f(x)-\nabla f(y)\|_q\le L\|x-y\|_p$ and have a global
minimizer in $\mathbb B_p^d(R)$. Its $\mu$-strongly convex subclass, with
strong convexity relative to $\ell_p$, is denoted by
$\mathcal G_{p,\mu}(d,L,R)$. Write $\mathsf S_\eps$ for the unrestricted
exact-value complexity. Without strong convexity, $\mathsf S_\eps(\mathcal G_p)$
is bounded below by \eqref{g:eq-lp-smooth} at $\alpha=\beta=p$, where
$d^{2\theta}=\Gamma$. Adding a prescribed modulus gives the following.
\begin{corollary}[Smooth norm transfer under strong convexity]\label{g:thm-smooth}
For $L\ge\mu\Gamma>0$ and $0<\eps\le LR^2/(32\Gamma)$,
\begin{equation}\label{g:eq-smooth-sc}
 \mathsf S_\eps(\mathcal G_{p,\mu})
 \ge\max\left\{\frac{5d}{16},\,
 \frac d{\Cstar}\min\left\{d,\sqrt{\frac L{\mu\Gamma}},
                  \sqrt{\frac{LR^2}{\Gamma\eps}}\right\}\right\}.
\end{equation}
\end{corollary}
This is an elementary consequence of \cref{cor:strong} with Euclidean radius
$R/\sqrt\Gamma$ and strong-convexity modulus $\mu\Gamma$. Sharp
dependence on $p$ in the strongly convex class remains open. The condition $L\ge\mu\Gamma$ is
exactly the threshold for the class to be nonempty. At equality, every
objective has the form $L\|x\|_2^2/2+b^\top x+c$, and $d+1$ values suffice
for exact optimization (\cref{g:thm-rigidity}).

\paragraph{Nonsmooth certificates, including strong convexity.}
Let $\mathcal F_p^0(d,R,M)$ be the class of convex $M$-Lipschitz functions on
$\mathbb B_p^d(R)$, and let $\mathcal F_{p,\mu}^0$ be its $\mu$-strongly
convex subclass. Write $\mathsf T_\eps$ for their exact-value complexity
with feasible queries and outputs. A certificate is a fixed linear map
$Q:\mathbb R^d\to\mathbb R^m$ such that
\begin{equation}\label{g:eq-certificate}
 r\mathbb B_2^m\subseteq Q(\mathbb B_p^d)\subseteq\mathbb B_2^m,
 \qquad 1\le m\le d,\quad 0<r\le1.
\end{equation}
Set $t=\max\{p,1+1/\log(ed)\}$ and $a_{p,d}=d^{2(1/p-1/t)}/(t-1)$.
\begin{theorem}[Nonsmooth dimension--radius transfer]\label{g:thm-nonsmooth}
There are universal $c_Z,\eps_Z>0$ and $m_Z^\circ\in\mathbb N$
with the following properties for $M,R,\eps>0$ and every certificate
with $m\ge m_Z^\circ$.
If $0<\eps\le\eps_ZMRr$, then
\begin{equation}\label{g:eq-ns}
 \mathsf T_\eps(\mathcal F_p^0)
 \ge c_Z\,\mathcal P\big(m,\min\{m,(MRr/\eps)^2\}\big).
\end{equation}
For $\mu>0$, put $a=a_{p,d}$ and
$R_\eps=\min\{R,\sqrt{\eps/(\mu a)}\}$, and assume the margin
condition $M_0:=M-\mu aR>0$, whose order is shown to be correct in
\cref{g:thm-budget}.
If $3\eps/2\le\eps_ZM_0R_\eps r$, then
\begin{equation}\label{g:eq-ns-sc}
 \mathsf T_\eps(\mathcal F_{p,\mu}^0)
 \ge c_Z\,\mathcal P(m,\sigma_{\rm sc}),\qquad
 \sigma_{\rm sc}=\min\left\{m,\frac{4r^2}{9}
 \min\left\{\frac{M_0^2R^2}{\eps^2},
                  \frac{M_0^2}{\mu a\eps}\right\}\right\}.
\end{equation}
\end{theorem}
The Euclidean nonsmooth information bound is due to
\citet[Theorem~5, version~2]{ZhangZhangQiLin2026}; our contributions are the
geometric transfer with exact preservation of query counts, the envelope
over certificates, and the localization under strong convexity
(\cref{fig:dimradius}). The supremum is taken only over triples $(m,r,Q)$
that satisfy their respective radius-dependent accuracy conditions.
Coordinate certificates have $(m,r)=(j,j^{1/2-1/p})$; the identity gives
$\Omega(d^2/\log(ed))$ for every $p$ at high accuracy. Constant-distortion
certificates from \citet{FigielLindenstraussMilman1977} have dimension of
order $\min\{d,\min(q,\log(ed))d^{2/q}\}$ and complement this high-accuracy
branch. The regularizer coefficient satisfies
$a_{p,d}\asymp\min\{q,\log(ed)\}$, and nonemptiness of the class requires
$M\gtrsim\mu R a_{p,d}$ for $d\ge2$ (\cref{g:thm-budget}). All reductions and
geometric proofs are given in the remainder of this appendix; the
$\ell_\alpha/\ell_\beta$ profiles are in \cref{g:app-lp}.

\begin{figure}[t]
\centering
\includegraphics[width=0.70\linewidth]{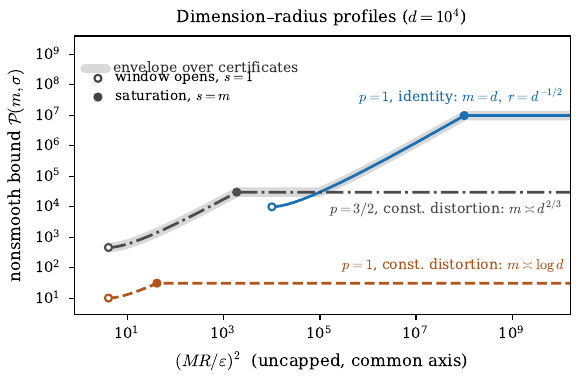}
\caption{Schematic nonsmooth dimension--radius profiles from \eqref{g:eq-ns}
against the common uncapped variable $(MR/\eps)^2$, for $d=10^4$. Each
certificate uses its own inscribed radius: the coordinate identity for $p=1$
has $(m,r)=(d,d^{-1/2})$; the constant-distortion certificates ($p=1$,
$m\asymp\log d$; $p=3/2$, $m\asymp d^{2/3}$) are drawn with a representative
$r=1/2$. Each curve begins where its capped variable reaches one,
$r^2(MR/\eps)^2\ge1$ (hollow markers); the window $\eps\le\eps_ZMRr$ of
\eqref{g:eq-ns} and its dimension requirement $m\ge m_Z^\circ$ are stricter by
universal constants, which are suppressed here. Each curve increases with
$s=r^2(MR/\eps)^2$ and saturates at $(MR/\eps)^2=m/r^2$ (filled markers) at
the level $m^2/\log(em)$; the gray band is the supremum over the three
certificates. The curves show the analytical rate expressions with universal
constants suppressed.}
\label{fig:dimradius}
\end{figure}

\subsection{Classes, oracle conventions, and elementary reductions}
Throughout this appendix, $1\le p\le2$, $q=p/(p-1)$ with $q=\infty$ at
$p=1$, and $1/\infty=0$. We write $\mathbb B_p^d(R)$ for the closed $\ell_p$
ball of radius $R$ and omit $R$ for the unit ball. A convex function on this
ball is $\mu$-strongly convex if, for all $x,y\in\mathbb B_p^d(R)$ and
$0\le u\le1$,
\begin{equation}\label{g:eq-strong-definition}
 f((1-u)x+uy)\le(1-u)f(x)+uf(y)
 -\frac\mu2u(1-u)\|x-y\|_p^2.
\end{equation}
The class $\mathcal F_p^0(d,R,M)$ consists of the convex functions satisfying
$|f(x)-f(y)|\le M\|x-y\|_p$ on the ball; $\mathcal F_{p,\mu}^0$ is the
subclass that also satisfies \eqref{g:eq-strong-definition}. The constants
$M,L$ are upper bounds, whereas $\mu$ is a lower bound. For a class
$\mathcal F$ on a known compact set $X$, define
\[
 \mathsf T_\eps(\mathcal F)
 =\inf\left\{T\in\mathbb N_0:\ \exists A,
 \sup_{f\in\mathcal F}\mathbb P_A\{f(\widehat x)-\min_X f>\eps\}
 \le\tfrac13\right\}.
\]
Here the algorithm $A$ makes at most $T$ exact-value queries on every
transcript, its queries and output lie in $X$, and its policies are total
Borel maps of an independent standard Borel seed and the previous replies.
The infimum over an empty set of admissible budgets is $+\infty$. Since
$f(0)-\min f\le MR$, the accuracy windows below are essential: when
$\eps\ge MR$, zero queries suffice.

For the global class $\mathcal G_p(d,L,R)$, the gradient condition of
\cref{g:thm-smooth} holds on all of $\mathbb R^d$, and at least one global
minimizer lies in $\mathbb B_p^d(R)$. For $\mathcal G_{p,\mu}$, strong
convexity is global as well. Its complexity $\mathsf S_\eps$ uses the same
policies and success criterion, but admits arbitrary points as queries and
outputs and uses the global minimum. This model differs from
$\mathsf T_\eps$.

\begin{lemma}[Exact online reduction]\label{g:lem-reduction}
Fix a map from source objectives to objectives in a target class. Suppose
that each target query can be answered using at most $b\ge1$ feasible source
queries, and that a Borel decoder maps every target output of error at most
$\eps$ to a source output of error at most $a\eps+c$, where $a,c\ge0$.
The simulator and the decoder are independent of the hidden source. Then
\[
 T^{\rm target}_\eps\ge
 \left\lceil T^{\rm source}_{a\eps+c}/b\right\rceil.
\]
\end{lemma}
\begin{proof}
Run the source algorithm with the seed of the target algorithm. Induction
over the queries identifies the simulated transcript with the true target
transcript. Decoding preserves the success event, and the simulation uses at
most $b$ times the target budget. The claim is the contrapositive of this
implication between integer budgets. The same argument applies without
change to a distributional source obstruction.
\end{proof}

We write $\operatorname{QC}_p(d;m,r;Q)$ for the certificate condition
\eqref{g:eq-certificate}. The image condition implies that $Q$ is onto and
that $\|Q\|_{p\to2}\le1$. A certificate is fixed before the source objective
is chosen.

\begin{lemma}[Coordinate certificates]\label{g:lem-coordinate}
For every integer $1\le j\le d$, the projection $P_j$ onto the first $j$
coordinates satisfies
$\operatorname{QC}_p(d;j,j^{1/2-1/p};P_j)$.
\end{lemma}
\begin{proof}
The outer inclusion follows from the inequality $\|P_jx\|_2\le\|x\|_p$.
For the inner inclusion, let $\|z\|_2\le j^{1/2-1/p}$; the zero-padded lift
of $z$ has $\ell_p$ norm at most $j^{1/p-1/2}\|z\|_2\le1$.
\end{proof}

\begin{lemma}[Radial forcing]\label{g:lem-radial}
Let $h:\mathbb R^m\to\mathbb R$ be convex, positively one-homogeneous,
and $H$-Lipschitz in the Euclidean norm. Let $K$ be a compact set containing
$r\mathbb B_2^m$, where $r>0$. Define
\[
 \Psi_{r,h}(z)=h(z)+H(\|z\|_2-r)_+,
 \qquad
 \pi_r(z)=\begin{cases}0,&z=0,\\
 z\min\{1,r/\|z\|_2\},&z\ne0.
 \end{cases}
\]
Then $\Psi_{r,h}$ is convex and $2H$-Lipschitz, and for every $z\in K$,
\begin{align*}
 \min_K\Psi_{r,h}&=r\min_{\mathbb B_2^m}h,\\
 h(\pi_r(z)/r)-\min_{\mathbb B_2^m}h
 &\le\frac{\Psi_{r,h}(z)-\min_K\Psi_{r,h}}r.
\end{align*}
Moreover, any value $h(z)$ can be obtained with at most one exact query in
$\mathbb B_2^m$.
\end{lemma}
\begin{proof}
The penalty term is $H$ times the distance to the convex ball
$r\mathbb B_2^m$. Lipschitz continuity gives $h(\pi_r(z))\le\Psi_{r,h}(z)$.
Since the penalty vanishes on that ball, the two minima agree there.
Positive homogeneity yields the factor $r$ and the error inequality.
Finally, $h(0)=0$ and $h(z)=\|z\|_2h(z/\|z\|_2)$ for $z\ne0$. This identity
holds for every $z\ne0$, however large its norm, and each value is obtained
from a single query at the feasible unit vector $z/\|z\|_2$.
\end{proof}

\subsection{The external Euclidean nonsmooth lower bound}
The only external nonsmooth scalar-information result used here is the
following normalized form of
\citet[Theorem~5, version~2]{ZhangZhangQiLin2026}. Their hard functions
have Lipschitz constant $2$; dividing by $2$ gives the normalization used
here.

\begin{proposition}[External ZZQL input]\label{g:ext-zzql}
There exist universal constants $c_{\rm dim},c_{\rm prob},c_{\rm err}>0$
and $m_Z\in\mathbb N$ such that for all integers $m\ge m_Z$ and
$1\le k\le c_{\rm dim}m$ there is a fixed Borel distribution of globally
defined convex, positively one-homogeneous, one-Lipschitz functions $h$ with
the following property. Every deterministic adaptive algorithm making at
most $T$ exact queries in $\mathbb B_2^m$ and returning a point there satisfies
\[
 \mathbb P_h\left\{h(\widehat u)-\min_{\mathbb B_2^m}h>
                 \frac{c_{\rm err}}{2\sqrt k}\right\}\ge\frac34
 \quad\text{if}\quad
 1\le T\le c_{\rm prob}\frac{mk}{\log(ek)}.
\]
\end{proposition}
If necessary, we increase $m_Z$ so that $c_{\rm prob}m_Z\ge1$; a zero-query
algorithm can be padded with one irrelevant feasible query. Conditioning on
the algorithmic seed and then integrating over it yields the same
distributional obstruction for randomized algorithms \citep{Yao1977}. The
restrictions of $h$ lie in $C(\mathbb B_2^m)$ equipped with the uniform norm;
evaluation is jointly continuous, and total Borel policies produce jointly
measurable transcripts and failure events. Consequently, an average failure
probability of at least $3/4$ contradicts a guarantee of failure probability
at most $1/3$ on every fixed objective.

For all nonsmooth statements we set
\begin{align*}
 a_Z&=\min\{c_{\rm dim}/2,c_{\rm err}^2/16,1/4\},&
 \eps_Z&=\sqrt{a_Z/2},& c_Z&=c_{\rm prob}a_Z/2,\\
 m_Z^\circ&=\max\{m_Z,\lceil c_{\rm prob}^{-1}\rceil,
                               \lceil2/a_Z\rceil\}.
\end{align*}
The dimension threshold $m_Z^\circ$ depends only on the constants in
\cref{g:ext-zzql}. The profile $\mathcal P$ in \eqref{g:eq-profile}
is nondecreasing in its second argument, because
the derivative of $\sigma/\log(e\sigma)$ equals
$\log(\sigma)/(1+\log\sigma)^2\ge0$ for $\sigma\ge1$.

\begin{proof}[Proof of the convex part of \cref{g:thm-nonsmooth}]
Let $\sigma=\min\{m,(MRr/\eps)^2\}$ and $k=\lfloor a_Z\sigma\rfloor$.
The dimension and accuracy hypotheses give $a_Z\sigma\ge2$, so that
\[
 a_Z\sigma/2\le k\le a_Z\sigma\le c_{\rm dim}m.
\]
For a normalized source $h$, define on the target ball
\[
 F_h(x)=\frac{MR}{2}\Psi_{r,h}(Q(x/R)).
\]
This function is convex and $M$-Lipschitz. Each target query is answered
with one feasible source query and known arithmetic. By \cref{g:lem-radial},
an output with target error at most $\eps$ yields source error at most
\[
 \frac{2\eps}{MRr}
 \le \frac{c_{\rm err}}{2\sqrt k},
\]
where the last inequality uses $k\le a_Z(MRr/\eps)^2$ and
$a_Z\le c_{\rm err}^2/16$.
The external input and \cref{g:lem-reduction} therefore give
\[
 \mathsf T_\eps(\mathcal F_p^0)
 \ge\left\lfloor c_{\rm prob}\frac{mk}{\log(ek)}\right\rfloor+1
 \ge c_Z\,\mathcal P(m,\sigma).
\]
The final inequality uses $k\ge a_Z\sigma/2$, $k\le\sigma$ and the
positivity of the logarithms. This proves \eqref{g:eq-ns} under the stated dimension and accuracy
conditions.
\end{proof}

\begin{corollary}[Admissible certificate envelope]\label{g:cor-envelope}
For any collection $\mathcal C$ of certificates with the same $p,d$, let
$\mathcal C_\eps$ consist of those with $m\ge m_Z^\circ$ and
$\eps\le\eps_ZMRr$. Then
\[
 \mathsf T_\eps(\mathcal F_p^0)
 \ge c_Z\sup\left(\{0\}\cup
 \left\{\mathcal P\!\left(m,\min\{m,(MRr/\eps)^2\}\right):
                   (m,r,Q)\in\mathcal C_\eps\right\}\right).
\]
In particular, if $d\ge m_Z^\circ$ and
$0<\eps\le\eps_ZMRd^{1/2-1/p}$, the identity gives
\begin{equation}\label{g:eq-identity-ns}
 \mathsf T_\eps(\mathcal F_p^0)
 \ge c_Z\frac{ds_d}{\log(es_d)},\qquad
 s_d=\min\{d,(MR/\eps)^2d^{2/q-1}\}.
\end{equation}
If also $\eps\le MRd^{-1/p}$, this equals
$c_Zd^2/\log(ed)$.
\end{corollary}
\begin{proof}
Each retained certificate bounds the same minimax quantity from below, so
taking the supremum gives the claim; all terms are bounded by $d^2$. For the
identity, substitute $m=d$ and $r=d^{1/2-1/p}=d^{1/q-1/2}$ from
\cref{g:lem-coordinate}. The additional accuracy condition is exactly
$(MR/\eps)^2d^{2/q-1}\ge d$.
\end{proof}
At $p=1$ the identity has radius $d^{-1/2}$, and the saturation scale is
$MR/d$. In the unsaturated branch the denominator remains $\log(es_d)$;
replacing it by $\log d$ loses information. Including all coordinate sizes
$j\in\{m_Z^\circ,\ldots,d\}$ is useful when the accuracy window of the
identity is unavailable. The supremum is taken over the branches whose accuracy conditions
hold at the prescribed error.

\subsection{Uniform constant-distortion certificates}
We use the following form of
\citet[Theorem~2.6, Lemma~2.7 and equation~(2.8)]{FigielLindenstraussMilman1977}:
let $a|x|_E\le\|x\|\le b|x|_E$ on an $n$-dimensional normed space, and
let $M_E$ be the median of the norm on the $E$-unit sphere. For a universal
$\eta>0$ there is a subspace of dimension $\lfloor\eta nM_E^2/b^2\rfloor$
whose Banach--Mazur distance from Euclidean space is at most $3/2$. When
$b/a\le\sqrt n$, the spherical mean and median are comparable up to
universal factors. These results are classical; see also
\citet{MilmanSchechtman1986}. The specialization below requires constants
uniform in $q$, which \cref{g:lem-sphere} provides.

\begin{lemma}[Spherical norm scale]\label{g:lem-sphere}
For $d\ge2$ and $2\le q\le\infty$,
\[
 \operatorname{med}_{\vartheta\in\mathbb S^{d-1}}\|\vartheta\|_q
 \asymp d^{1/q-1/2}\sqrt{\min\{q,\log(ed)\}},
\]
with constants independent of $d,q$.
\end{lemma}
\begin{proof}
Let $G$ be a standard Gaussian vector in $\mathbb R^d$, write
$\vartheta=G/\|G\|_2$ almost surely, and put $\ell=\log(ed)$.
Since the direction and the radius are independent,
$\mathbb E\|\vartheta\|_q=\mathbb E\|G\|_q/\mathbb E\|G\|_2$, where
$\sqrt{2/\pi}\sqrt d\le\mathbb E\|G\|_2\le\sqrt d$.
For a scalar standard Gaussian, the moments
$\mathfrak g_s=2^{s/2}\Gamma((s+1)/2)/\sqrt\pi$ satisfy
$\mathfrak g_s^{1/s}\asymp\sqrt s$ and $\mathfrak g_{2s}/\mathfrak g_s^2\le C2^s$ for
$s\ge2$, by the standard two-sided Stirling bounds.
First let $2\le q\le\ell$ and $Z=\sum_i|G_i|^q$. Then
\[
 \frac{\mathbb E Z^2}{(\mathbb E Z)^2}
 \le1+C2^q/d\le C_0,
\]
because $2^{\log(ed)}/d$ is uniformly bounded. The Paley--Zygmund
inequality yields $\mathbb P\{Z\ge\mathbb E Z/2\}\ge1/(4C_0)$, so
$\mathbb E\|G\|_q\gtrsim\sqrt q\,d^{1/q}$; Jensen's inequality gives the
reverse bound. Next let $q\ge\ell$. Then
$\|G\|_\infty\le\|G\|_q\le d^{1/q}\|G\|_\infty\le e\|G\|_\infty$.
If $d\ge3$, the bound just proved at $q=\ell$ implies
$\mathbb E\|G\|_\infty\gtrsim\sqrt\ell$; the case $d=2$ follows from a
single coordinate. The exponential-moment union bound
$\mathbb E e^{\xi\|G\|_\infty}\le2d e^{\xi^2/2}$, valid for $\xi>0$, gives
$\mathbb E\|G\|_\infty\le\sqrt{2\log(2d)}$.
This proves the estimate for the spherical mean. For $\ell_q^d$ we have
$a=d^{1/q-1/2}$ and $b=1$, so $b/a\le\sqrt d$, and the external
mean--median comparison completes the proof.
\end{proof}

\begin{proposition}[FLM certificate]\label{g:prop-flm}
For a universal $0<c_D\le1$ and $d\ge2$, define
\[
 \chi_{p,d}=\min\{d,\min(q,\log(ed))d^{2/q}\},\qquad
 \mathfrak m_{p,d}=\max\{1,\lfloor c_D\chi_{p,d}\rfloor\}.
\]
Then there is a certificate
$\operatorname{QC}_p(d;\mathfrak m_{p,d},1/2;Q)$.
\end{proposition}
\begin{proof}
After reducing $c_D$, the external section theorem and \cref{g:lem-sphere}
provide an injection $U:\mathbb R^m\to\mathbb R^d$ with
$\|u\|_2\le\|Uu\|_q\le2\|u\|_2$, where
$m=\lfloor c_D\chi_{p,d}\rfloor\ge1$. Put $Q=U^*/2$.
By duality, $\|Q\|_{p\to2}\le1$, and the support function of
$Q(\mathbb B_p^d)$ in the direction $u$ equals
$\|Q^*u\|_q\ge\|u\|_2/2$. Separation of compact convex sets then yields the
inner inclusion. If the integer part vanishes, a single-coordinate
projection gives the same radius $1/2$.
\end{proof}
Consequently, for $m=\mathfrak m_{p,d}\ge m_Z^\circ$ and
$0<\eps\le\eps_ZMR/2$, \cref{g:thm-nonsmooth} gives
\[
 \mathsf T_\eps(\mathcal F_p^0)
 \gtrsim
 \frac{m\min\{m,(MR/\eps)^2\}}
      {\log(e\min\{m,(MR/\eps)^2\})}.
\]
Indeed, the factor $r^2=1/4$ reduces the numerator by at most a factor of
four and cannot increase the logarithmic denominator. At $p=2$ the identity
gives $m=d,r=1$ without any external geometric theorem. At $p=1$ this
constant-distortion branch has $m\asymp\log(ed)$, whereas
\eqref{g:eq-identity-ns} reaches the full dimension on a smaller accuracy
window. Geometric lower-bound reductions have a precedent in
\citet{GuzmanNemirovski2015}; the exact scalar-information obstruction used
here is the explicitly attributed result of ZZQL.

\begin{proposition}[Gaussian realization and its Euclidean spectral scale]
\label{g:prop-gauss-rho}
Let $1\le p\le2$, $d\ge2$, and $\bar q=\min\{q,\log(ed)\}$. After reducing
$c_D$, the certificate of \cref{g:prop-flm} can be realized as
$\operatorname{QC}_p(d;\mathfrak m_{p,d},r_G;Q)$ for a universal
$r_G\in(0,1/2]$, with the additional property
\begin{equation}\label{g:eq-rho}
 \rho_{p,d}:=\|Q^*\|_{2\to2}\asymp\frac{d^{1/2-1/q}}{\sqrt{\bar q}},
\end{equation}
where the implicit constants are universal. In particular
$\rho_{2,d}\asymp1$, whereas $\rho_{1,d}\asymp\sqrt{d/\log(ed)}$.
\end{proposition}
\begin{proof}
If $c_D\chi_{p,d}<1$, a single-coordinate projection gives
$\mathfrak m_{p,d}=1$, $r_G=1/2$ and $\rho=1\asymp d^{1/2-1/q}/\sqrt{\bar q}$,
since $\chi_{p,d}\asymp\bar q\,d^{2/q}$ is then bounded. We may therefore
assume that $m=\lfloor c_D\chi_{p,d}\rfloor\ge1$.

Let $A\in\mathbb R^{d\times m}$ have independent standard Gaussian entries,
and put $b_q=\mathbb E\|G\|_q$ for $G\sim N(0,I_d)$. The estimates in the
proof of \cref{g:lem-sphere} give $b_q\asymp\sqrt{\bar q}\,d^{1/q}$
uniformly in $q\in[2,\infty]$: for $q\le\log(ed)$ this is the spherical
mean bound, and for $q\ge\log(ed)$ one has $b_q\asymp\sqrt{\log(ed)}$ while
$1\le d^{1/q}\le e$. For a fixed unit vector $u\in\mathbb R^m$, the vector
$Au$ is standard Gaussian in $\mathbb R^d$, and $A\mapsto\|Au\|_q$ is
$1$-Lipschitz in the Frobenius norm, because $\|\cdot\|_q\le\|\cdot\|_2$
for $q\ge2$. Gaussian concentration at level $b_q/4$, combined with a
$1/4$-net of the Euclidean sphere of $\mathbb R^m$ of cardinality at most
$9^m$, therefore yields universal constants $0<c_1\le C_1$ such that
\[
 c_1\|u\|_2\le\frac{\|Au\|_q}{b_q}\le C_1\|u\|_2
 \qquad\text{for every }u\in\mathbb R^m
\]
holds with probability at least $3/4$, provided that $m\le cb_q^2$ for a
small universal $c$. Since $b_q^2\asymp\bar q\,d^{2/q}$ and $m\le d$, this
condition is exactly $m\le c'\chi_{p,d}$, so it holds after reducing $c_D$.
Separately, there are universal constants $0<c_2\le C_2$ such that
$c_2\sqrt d\le\|A\|_{2\to2}\le C_2\sqrt d$ with probability at least $7/8$:
the upper bound is the standard Gaussian operator-norm estimate for
$m\le d$, and the lower bound follows from the single column $\|Ae_1\|_2$,
whose square is $\chi^2_d$. The intersection of the two events has positive
probability, so a deterministic realization exists.

Put $Q=A^*/(C_1b_q)$. Then $\|Q^*u\|_q=\|Au\|_q/(C_1b_q)\le\|u\|_2$, hence
$\|Q\|_{p\to2}=\|Q^*\|_{2\to q}\le1$, which is the outer inclusion. The
support function of the symmetric compact convex set $Q(\mathbb B_p^d)$ in
the direction $u$ equals $\|Q^*u\|_q\ge(c_1/C_1)\|u\|_2$, so separation
gives the inner inclusion with $r_G=\min\{c_1/C_1,1/2\}$. Finally,
$\|Q^*\|_{2\to2}=\|A\|_{2\to2}/(C_1b_q)\asymp\sqrt d/b_q$, which is
\eqref{g:eq-rho}.
\end{proof}

The two scales of a certificate need not coincide, and a function class may
impose subgradient bounds in both norms simultaneously.

\begin{corollary}[Simultaneous Lipschitz bounds in two norms]\label{g:cor-two-scale}
Let $\mathcal F_p^0(d,R,M_p,M_2)$ be the subclass of $\mathcal F_p^0(d,R,M_p)$
whose functions are also $M_2$-Lipschitz in the Euclidean norm on the ball.
For a certificate
$\operatorname{QC}_p(d;m,r;Q)$ with spectral scale $\rho=\|Q^*\|_{2\to2}$, put
\begin{equation}\label{g:eq-lambda}
 \Lambda_Q=R\min\{M_p,\ M_2/\rho\}.
\end{equation}
Then the convex bound \eqref{g:eq-ns} holds for this subclass with $MR$
replaced by $\Lambda_Q$, on the corresponding accuracy window
$\eps\le\eps_Z\Lambda_Qr$.

For its $\mu$-strongly convex subclass with $\mu>0$, let $\omega$ and
$a=a_{p,d}$ be the prox function and the constant of \cref{g:lem-prox}, and
define
\[
 b_2=\max_{\|x\|_p\le R}\|\nabla\omega(x)\|_2,
 \qquad
 M_0=\min\left\{M_p-\mu aR,\frac{M_2-\mu b_2}{\rho}\right\}.
\]
The maximum is finite and satisfies
$b_2\le d^{1/p-1/2}aR$. Suppose $M_0>0$, put
$R_\eps=\min\{R,\sqrt{\eps/(\mu a)}\}$, and assume
$m\ge m_Z^\circ$ and
$3\eps/2\le\eps_ZM_0R_\eps r$.
Then \eqref{g:eq-ns-sc} holds with this value of $M_0$; explicitly,
\[
 T\ge c_Z\mathcal P(m,\sigma_{\rm sc}),\qquad
 \sigma_{\rm sc}=\min\left\{m,\frac{4r^2}{9}
 \min\left\{\frac{M_0^2R^2}{\eps^2},
            \frac{M_0^2}{\mu a\eps}\right\}\right\}.
\]
Thus both Lipschitz budgets account for the known strongly convex term.
\end{corollary}
\begin{proof}
For the convex statement, apply the proof of \cref{g:thm-nonsmooth} with
$F_h(x)=(\Lambda_Q/2)\Psi_{r,h}(Q(x/R))$.
For a normalized source, $\Psi_{r,h}$ is $2$-Lipschitz in $\ell_2$, so every
subgradient of the first term at $x$ has the form $(\Lambda_Q/(2R))Q^*v$
with $\|v\|_2\le2$. Hence
$\|\partial F_h\|_q\le(\Lambda_Q/R)\|Q^*\|_{2\to q}\le\Lambda_Q/R\le M_p$ and
$\|\partial F_h\|_2\le(\Lambda_Q/R)\|Q^*\|_{2\to2}=\rho\Lambda_Q/R\le M_2$, so
$F_h$ belongs to the subclass. The query simulator, the radial decoder and
the integer choice of $k$ are unchanged.

For the strongly convex statement, $b_2$ is finite because $\nabla\omega$
is continuous and the ball is compact. Since $q\ge2$, norm comparison and
\cref{g:lem-prox} give
$\|\nabla\omega(x)\|_2\le d^{1/2-1/q}\|\nabla\omega(x)\|_q
\le d^{1/p-1/2}aR$.
Now use the construction of the strongly convex proof of
\cref{g:thm-nonsmooth} with the present $M_0$. Its unknown term has
subgradients bounded by $M_0$ in $\ell_q$ and by $\rho M_0$ in
$\ell_2$. Adding $\mu\omega$ keeps the respective bounds below
$M_0+\mu aR\le M_p$ and $\rho M_0+\mu b_2\le M_2$, so the function lies in
the required subclass. The value of the regularizer at the comparison point
is still at most $\eps/2$, so the same decoder has source error at most
$3\eps/(M_0R_\eps r)$. The stated window and the original rounding argument
yield the displayed bound.
\end{proof}

\begin{remark}[The role of the Euclidean scale]\label{g:rem-normalization}
At $p=2$ the two scales agree: $\rho_{2,d}\asymp1$ and
$\Lambda_Q\asymp R\min\{M_p,M_2\}$. At
$p=1$ the constant-distortion certificate has only
$\mathfrak m_{1,d}\asymp\log(ed)$ dimensions and spectral scale
$\rho_{1,d}\asymp\sqrt{d/\log(ed)}$, so
$\Lambda_Q\asymp R\min\{M_1,M_2\sqrt{\log(ed)/d}\}$.
The coefficient of $M_2$ is smaller than that of $M_1$ by a factor
$\sqrt{d/\log(ed)}$. The coordinate certificate of \cref{g:lem-coordinate} has
$\rho=1$, but its radius $j^{1/2-1/p}$ is correspondingly smaller.
\end{remark}

\subsection{Regularizer geometry and localized strong convexity}
Let $t=t_{p,d}=\max\{p,1+1/\log(ed)\}$ and
\[
 a=a_{p,d}=\frac{d^{2(1/p-1/t)}}{t-1},\qquad
 \omega(x)=\frac a2\|x\|_t^2.
\]
The definition covers $d=1$, where $t=2$ and $a=1$.

\begin{lemma}[A prox function with an explicit Lipschitz budget]\label{g:lem-prox}
The function $\omega$ is one-strongly convex relative to $\ell_p$, and on
$\mathbb B_p^d(R)$,
\[
 0\le\omega(x)\le aR^2/2,\qquad
 \|\nabla\omega(x)\|_q\le aR,
 \qquad
 a\le e^2\min\{(p-1)^{-1},\log(ed)\}.
\]
At $p=1$ the reciprocal is $+\infty$, and $a_{2,d}=1$.
Moreover, $a\asymp\min\{q,\log(ed)\}$ uniformly in $p,d$.
\end{lemma}
\begin{proof}
For $1<t\le2$, the function $\|x\|_t^2/2$ is $(t-1)$-strongly convex
in $\ell_t$. We recall the underlying estimate. Put $v=t/(t-1)\ge2$ and
$F(u)=\|u\|_v^2/2$. Differentiating away from zero gives a Hessian whose
first, rank-one term has coefficient $2-v\le0$ and whose remaining term is
\[
 (v-1)\|u\|_v^{2-v}\sum_i|u_i|^{v-2}h_i^2
 \le(v-1)\|h\|_v^2
\]
by H\"older's inequality. Integrating along segments gives
$(v-1)$-smoothness, and segments through zero are covered by continuity.
The conjugate is $F^*(x)=\|x\|_t^2/2$, and the smoothness/conjugacy
inequality gives strong convexity with modulus $1/(v-1)=t-1$. This standard
duality follows directly by applying the smoothness inequality for $F$ at a
maximizing dual point and optimizing the resulting quadratic bound over the
dual increment.
Next, $\|h\|_p\le d^{1/p-1/t}\|h\|_t$, which produces the factor in $a$.
The gradient of $\|x\|_t^2/2$ has norm $\|x\|_t$ in the dual of $\ell_t$.
Since $p\le t$ implies $q\ge t/(t-1)$ and $\|x\|_t\le\|x\|_p$,
the estimates for the value and the gradient follow.
If $t=p$, then $a=(p-1)^{-1}\le\log(ed)$. Otherwise
$t=1+1/\log(ed)$ and
$1\le d^{2(1/p-1/t)}\le e^2$, so that $\log(ed)\le a\le e^2\log(ed)$.
Finally, for $p>1$, $(p-1)^{-1}=q-1$ lies between $q/2$ and $q$;
the case $p=1$ follows from the second branch.
\end{proof}

\begin{proof}[Proof of the strongly convex part of \cref{g:thm-nonsmooth}]
Put $M_0=M-\mu aR>0$ and
$R_\eps=\min\{R,\sqrt{\eps/(\mu a)}\}$.
For a normalized homogeneous source $h$, define on the original target ball
of radius $R$
\[
 f_h(x)=\frac{M_0R_\eps}{2}
        \Psi_{r,h}(Q(x/R_\eps)),
 \qquad F_h(x)=f_h(x)+\mu\omega(x).
\]
The first term is convex and $M_0$-Lipschitz; the second is
$\mu$-strongly convex and $\mu aR$-Lipschitz on the entire target ball.
Hence $F_h\in\mathcal F_{p,\mu}^0(d,R,M)$.
If $u_*$ minimizes $h$ on the Euclidean unit ball, the certificate provides
a lift $v\in\mathbb B_p^d$ with $Qv=ru_*$. By \cref{g:lem-radial}, the point
$x_0=R_\eps v$ is feasible and minimizes $f_h$ on the entire target ball.
Moreover, $\mu\omega(x_0)\le\mu aR_\eps^2/2\le\eps/2$.
If $F_h(\widehat x)-\min F_h\le\eps$, comparison with $x_0$ and the
nonnegativity of $\omega$ yield
\[
 f_h(\widehat x)-\min f_h\le3\eps/2.
\]
All target queries remain in the original ball. Their normalized images
$Q(x/R_\eps)$ may lie outside the source ball, but by positive homogeneity,
as in \cref{g:lem-radial}, each value is still simulated with at most one
feasible exact source query. The regularizer and the radial penalty are
known. The radial decoder returns a feasible source point with error at most
$3\eps/(M_0R_\eps r)$.
Apply the integer choice of the convex proof with accuracy $3\eps/2$,
scale $M_0R_\eps$ and radius $r$. The stated accuracy window is precisely
the condition that argument requires. It gives
\[
 \sigma_{\rm sc}=\min\{m,(2M_0R_\eps r/(3\eps))^2\},
\]
which is \eqref{g:eq-ns-sc} after expanding $R_\eps^2$.
\end{proof}

The strongly convex envelope is the supremum of \eqref{g:eq-ns-sc} over all
certificates satisfying the dimension and accuracy conditions of that bound;
it is zero if no certificate is admissible. In particular, the identity gives
\[
 \sigma_{{\rm sc},d}=\min\left\{d,\frac49d^{2/q-1}
       \min\{(M_0R/\eps)^2,M_0^2/(\mu a\eps)\}\right\}
\]
and the lower bound $c_Z\,\mathcal P(d,\sigma_{{\rm sc},d})$, provided that
$d\ge m_Z^\circ$ and
$3\eps/2\le\eps_ZM_0R_\eps d^{1/q-1/2}$.

\begin{corollary}[Uniform strong-convexity regimes]\label{g:cor-sc-readable}
Fix $0<\theta_0<1$. Assume $\mu a_{p,d}R\le\theta_0 M$ and let
$m=\mathfrak m_{p,d}\ge m_Z^\circ$. For a sufficiently small constant
$c_{\theta_0}>0$, if
\[
 0<\eps\le c_{\theta_0}\min\{MR,M^2/(\mu a_{p,d})\},
\]
then
\[
 \mathsf T_\eps(\mathcal F_{p,\mu}^0)
 \gtrsim_{\theta_0}\frac{m\Xi}{\log(e\Xi)},\qquad
 \Xi=\min\{m,(MR/\eps)^2,M^2/(\mu a_{p,d}\eps)\}.
\]
For $p=2$, use $m=d\ge m_Z^\circ$ and $a=1$ directly. If also
$\eps\le\mu R^2$, this becomes
\[
 \mathsf T_\eps(\mathcal F_{2,\mu}^0)
 \gtrsim_{\theta_0}
 \frac{d\min\{d,M^2/(\mu\eps)\}}
 {\log(e\min\{d,M^2/(\mu\eps)\})}.
\]
\end{corollary}
\begin{proof}
The assumed slack gives $M_0\ge(1-\theta_0)M$. With the fixed radius
$r=1/2$, the two accuracy restrictions ensure, after decreasing
$c_{\theta_0}$, that $3\eps/2\le\eps_ZM_0Rr$ and
$3\eps/2\le\eps_ZM_0r\sqrt{\eps/(\mu a)}$, respectively; together they
imply the window for $R_\eps$.
The exact profile in \eqref{g:eq-ns-sc} dominates a fixed positive multiple
of $\Xi$ and is at most $\Xi$, so its logarithmic denominator cannot be
larger. At $p=2$ we use the identity. The last restriction makes
$(MR/\eps)^2\ge M^2/(\mu\eps)$.
\end{proof}
At $p=1$ both quantities entering the statement are logarithmic,
$\mathfrak m_{1,d}\asymp\log(ed)\asymp a_{1,d}$, so under
$\mu R\log(ed)\lesssim M$ the bound becomes
\[
 \mathsf T_\eps(\mathcal F_{1,\mu}^0)
 \gtrsim_{\theta_0}\mathcal P\!\left(\log(ed),\ \min\left\{\log(ed),
 \left(\frac{MR}\eps\right)^2,
 \frac{M^2}{\mu\log(ed)\,\eps}\right\}\right),
\]
so that both the effective dimension and the contribution of the
regularizer are logarithmic, the latter explicitly in the strong-convexity
branch. Letting $\mu\downarrow0$ removes that branch and recovers the convex
bound. The radius term and the strong-convexity term cross at
$\eps=\mu a_{p,d}R^2$: above this threshold the radius term is smaller,
below it the strong-convexity term. The dimension cap must still be
imposed. In particular, the strong-convexity term is active only for
$M^2/(\mu a_{p,d}m)\le\eps\le\mu a_{p,d}R^2$,
within the source accuracy window of the corollary. This interval may be
empty; a parameter choice need not exhibit all three branches.

\begin{theorem}[Necessary Lipschitz budget]\label{g:thm-budget}
If $d\ge2$ and $\mathcal F_{p,\mu}^0(d,R,M)$ is nonempty, then with
$k=\lfloor\log_2 d\rfloor$,
\begin{equation}\label{g:eq-budget}
 M\ge\mu R\,2^{2/p-4}\sum_{j=0}^{k-1}2^{-2j/q}
 \gtrsim\mu R\min\{q,\log(ed)\}
 \asymp\mu R a_{p,d}.
\end{equation}
At $p=1$, the explicit bound is $M\ge\mu Rk/4$.
In every dimension $d\ge1$ one also has $M\ge\mu R$.
\end{theorem}
\begin{proof}
Let $n=2^k\le d$. For a nonempty subset $S$ of the first $n$ coordinates,
put $u_S=(R/|S|)\sum_{i\in S}e_i$; all these points are feasible.
For disjoint sets $S,T$ of size $2^j$,
\[
 \|u_S-u_T\|_p^2=R^2 2^{2/p}2^{-2j/q}.
\]
Apply \eqref{g:eq-strong-definition} at the midpoint, average over the pairs
at each level of a balanced binary tree, and telescope:
\[
 \frac1n\sum_{i=1}^n f(Re_i)
 -f\!\left(\frac Rn\sum_{i=1}^ne_i\right)
 \ge\mu R^2 2^{2/p-3}\sum_{j=0}^{k-1}2^{-2j/q}.
\]
Each endpoint is at distance at most $2R$ from the final barycenter, so the
left-hand side is at most $2MR$. This proves the explicit inequality.
For $q\in[2,\infty]$ the sum is uniformly comparable to $\min\{q,k\}$,
interpreting the latter as $k$ at infinity; for $d\ge2$,
$k\asymp\log(ed)$. Finally, apply \cref{g:lem-prox}.
For the dimension-free statement, restrict to a coordinate diameter and set
$g(s)=(f(se_1)+f(-se_1))/2$. This function is even, $M$-Lipschitz and
$\mu$-strongly convex. For $0<h<R$, the chord between $-R$ and $R$ gives
$g(R-h)\le g(R)-\mu Rh+\mu h^2/2$, while Lipschitz continuity gives
$Mh\ge\mu Rh-\mu h^2/2$. Divide by $h$ and let $h\downarrow0$.
\end{proof}
This condition determines the order of the necessary Lipschitz budget
on the full ball. The strict margin $M>\mu aR$ is sufficient for the
construction. Sharp dependence on $a$ in the complexity term
$M^2/(\mu a\eps)$ remains open.

\subsection{Smooth geometric transfer and its compatibility endpoint}
Applying \cref{gp:thm-norm} with $B=\mathbb B_p^d$, $K=\mathbb B_p^d(R)$ and any
certificate $\operatorname{QC}_p(d;m,r;Q)$, which supplies $r_m\ge Rr$, gives
\[
 \mathsf S_\eps(\mathcal G_p(d,L,R))
 \ge\max\left\{\frac{5m}{16},\,
 \frac m{\Cstar}\min\left\{m,\sqrt{\frac{LR^2r^2}{\eps}}\right\}\right\},
 \qquad 0<\eps\le LR^2r^2/32,
\]
and the supremum over all certificates satisfying their accuracy windows is
again a lower bound (\cref{g:cor-envelope}). A rank-deficient $Q$ makes the
pullback constant along $\ker Q$, so it does not preserve positive strong
convexity.

\begin{proof}[Proof of \cref{g:thm-smooth}]
Write $\Gamma=d^{2/p-1}$. Norm comparison gives
\[
 \|x\|_2\le\|x\|_p\le\sqrt\Gamma\|x\|_2,
 \qquad \|u\|_q\le\|u\|_2.
\]
Hence every Euclidean $L$-smooth function is $L$-smooth from $\ell_p$
to $\ell_q$, and a Euclidean minimizer in
$\mathbb B_2^d(R/\sqrt\Gamma)$ belongs to $\mathbb B_p^d(R)$.
Euclidean strong convexity with modulus $\mu\Gamma$ implies
\eqref{g:eq-strong-definition} with modulus $\mu$. Applying \cref{cor:strong}
with radius $R/\sqrt\Gamma$ and Euclidean modulus $\mu\Gamma$ gives
\eqref{g:eq-smooth-sc} together with its exact accuracy window. The hard
instances are globally defined quadratics, so the resulting lower bounds
apply also to the corresponding larger ball-restricted smooth classes with
feasible queries and output.
\end{proof}
If $L\ge\mu\Gamma$ and
$0<\eps\le\min\{LR^2/(32\Gamma),\mu R^2\}$,
then \eqref{g:eq-smooth-sc} gives
$\Omega(d\min\{d,\sqrt{L/(\mu\Gamma)}\})$. This estimate is uniform at fixed $\mu$ throughout the displayed
window. Sharp joint dependence on accuracy and conditioning for the
general normed class remains open.

\begin{theorem}[Compatibility and rigidity]\label{g:thm-rigidity}
Let $R>0$ and $L,\mu>0$. The class of convex functions differentiable on a
neighborhood of $\mathbb B_p^d(R)$, $L$-smooth from $\ell_p$ to $\ell_q$
on that ball, and $\mu$-strongly convex there is nonempty if and only if
\[
 L\ge\mu\Gamma,\qquad \Gamma=d^{2/p-1}.
\]
At equality, every member agrees on the ball with
$f(x)=L\|x\|_2^2/2+b^\top x+c$, so exact constrained optimization needs at
most $d+1$ feasible values. The global class $\mathcal G_{p,\mu}$ has the
same nonemptiness threshold and the same global quadratic form at equality;
$d+1$ unrestricted values suffice there as well.
\end{theorem}
\begin{proof}
Choose $h>0$ with $hd^{1/p}<R$ and let $\sigma$ be uniform on
$\{-1,1\}^d$. Strong convexity applied at opposite vertices gives
\[
 \mathbb E f(h\sigma)\ge f(0)+\frac\mu2h^2d^{2/p}.
\]
Smoothness implies the descent inequality along every feasible coordinate
segment; in particular,
$[f(z+he_i)+f(z-he_i)]/2\le f(z)+Lh^2/2$.
Reveal the signs one coordinate at a time. Partial vertices and connecting
segments remain inside the ball, so iteration yields
$\mathbb E f(h\sigma)\le f(0)+Ldh^2/2$.
Comparing the two bounds gives $L\ge\mu d^{2/p-1}$. Conversely,
$L\|x\|_2^2/2$ has the required smoothness and strong-convexity modulus
$L/\Gamma$, and its global minimizer is zero.

At equality, the gradient is locally Euclidean Lipschitz on the open ball.
Its weak derivative is represented almost everywhere by a symmetric Hessian
$H$. At common points of differentiability, the smoothness and
strong-convexity inequalities give, for all $v$,
\[
 H_{ii}\le L,\qquad v^\top Hv\ge\mu\|v\|_p^2.
\]
Averaging the second inequality over sign vectors gives
$\operatorname{tr}H\ge\mu d^{2/p}=Ld$, so every diagonal entry equals $L$.
Every quadratic form evaluated at a sign vector is at least $Ld$, and their
average is $Ld$; thus all of them equal $Ld$. For $i\ne j$, multiplying by
$\sigma_i\sigma_j$ and averaging gives $2H_{ij}=0$. Therefore $H=LI$ almost
everywhere. The locally Lipschitz vector field $\nabla f-Lx$ has zero weak
derivative. Mollifying on relatively compact interior balls makes its
derivative identically zero, and passing to the limit gives local constancy.
Connectedness then gives $\nabla f=Lx+b$ throughout the open ball.
Integrating this continuous identity and extending to the boundary proves
the formula. The identity is established on the open ball before
integration, so no almost-everywhere statement is restricted to a line,
which could be a null set. The same local argument on $\mathbb R^d$ proves
global rigidity.

Query $0$ and $se_i$ for $i=1,\ldots,d$ with a fixed $0<s\le R$.
Then $c=f(0)$ and
$b_i=[f(se_i)-f(0)-Ls^2/2]/s$ are known. Minimize the reconstructed
quadratic over the known compact ball, or output $-b/L$ in the global model.
These computations require no further values.
\end{proof}
The estimate \eqref{g:eq-smooth-sc} is consistent with this endpoint:
at $L=\mu\Gamma$ the condition-number branch equals $1$, so the bound reduces
to $5d/16$, which matches the upper bound $d+1$ in order. The rigidity
theorem also rules out a lower bound growing logarithmically in the inverse
accuracy uniformly up to the equality case.

\section{Independent norms: explicit profiles on \texorpdfstring{$\ell_\alpha$}{ell-alpha} balls}\label{g:app-lp}
This appendix proves \cref{g:prop-lp-radius,g:cor-lp} and their block
refinements. Throughout, $1\le\alpha,\beta\le\infty$ are chosen
independently, $\alpha^*$ and $\beta^*$ denote the conjugate exponents,
$\mathbb B_\alpha^d(R_\alpha)$ is the closed $\ell_\alpha$ ball of radius
$R_\alpha$, and $\theta,\gamma$ are defined in \eqref{g:eq-theta}. All
statements concern the image radius \eqref{gp:eq-radius} and the lower bounds
it yields. \Cref{sec:matching} identifies the matching complexity regimes.

\subsection{Two elementary norm constants}
\begin{remark}[Norm dependence of the regularity constants]\label{g:rem-constants}
In \eqref{g:eq-constants}, $L_\beta$ is the minimal Lipschitz modulus of
$\nabla f$ from $\ell_\beta$ to $\ell_{\beta^*}$; for twice differentiable
convex $f$, this modulus is
\[
 L_\beta=\sup_x\ \sup_{\|h\|_\beta\le1}h^\top\nabla^2f(x)h,
\]
the largest curvature over the unit $\ell_\beta$ ball, because for a
positive-semidefinite $A$ one has
$\sup_{\|g\|_\beta,\|h\|_\beta\le1}g^\top Ah=\sup_{\|h\|_\beta\le1}h^\top Ah$.
Changing $\beta$ changes the set over which the supremum is taken. From
$\mathbb B_1^d\subseteq\mathbb B_2^d\subseteq\mathbb B_\infty^d$ and
$\mathbb B_2^d\subseteq\sqrt d\,\mathbb B_1^d$,
$\mathbb B_\infty^d\subseteq\sqrt d\,\mathbb B_2^d$, we obtain
\[
 L_1\le L_2\le L_\infty,\qquad L_2\le dL_1,\qquad L_\infty\le dL_2,
\]
and both factors $d$ are attained: $f(x)=\|x\|_2^2/2$ has
$(L_1,L_2,L_\infty)=(1,1,d)$, while $f(x)=(\mathbf 1^\top x)^2/2$ has
$(L_1,L_2,L_\infty)=(1,d,d^2)$. For the radius comparisons in this remark,
each radius is the corresponding norm of the same chosen minimizer; the
theorem parameters remain independent admissible upper bounds. These radii
satisfy the reverse ordering,
$R_1\ge R_2\ge R_\infty$ with $R_1\le\sqrt d\,R_2$. The extremes are attained
by a $1$-sparse $x^\star$, for which $R_1=R_2$, and by
$x^\star=\mathbf 1R_2/\sqrt d$, for which $R_1=\sqrt d\,R_2$.

Two rows of \cref{tab:lp} that share $\theta$ can therefore be compared only
after the constants have been specified. The $(\mathbb B_1^d,\ell_2)$ and
$(\mathbb B_2^d,\ell_2)$ rows give $\sqrt d\,R_1\sqrt{L_2/\eps}/\Cstar$ and
$dR_2\sqrt{L_2/\eps}/\Cstar$: for a $1$-sparse minimizer the first is
smaller by $\sqrt d$, and for $x^\star=\mathbf 1R_2/\sqrt d$ the two
coincide. \citet{GorbunovDvurechenskyGasnikov2022} report the same factor
$\sqrt d$ as the improvement obtained from an $\ell_1$ prox geometry for
sparse solutions.
\end{remark}

\begin{lemma}[Extremal comparisons in $\R^k$]\label{g:lem-extremal}
For every integer $k\ge1$ and $1\le\beta,\alpha\le\infty$,
\begin{equation}\label{g:eq-extremal}
 \max_{\|y\|_\beta\le1}\|y\|_2=k^{(1/2-1/\beta)_+},
 \qquad
 \min_{\|u\|_2=1}\|u\|_{\alpha^*}=k^{-(1/\alpha-1/2)_+}.
\end{equation}
\end{lemma}
\begin{proof}
If $\beta\le2$, then $\|y\|_2\le\|y\|_\beta$, with equality at $y=e_1$; if
$\beta\ge2$, then $\|y\|_2\le k^{1/2-1/\beta}\|y\|_\beta$ by H\"older's
inequality, with equality at the constant-coordinate vector
$k^{-1/\beta}\mathbf 1$. For the second constant, $\alpha\ge2$ gives
$\alpha^*\le2$ and hence $\|u\|_{\alpha^*}\ge\|u\|_2$, with equality at
$u=e_1$, while $\alpha\le2$ gives $\alpha^*\ge2$ and hence
$\|u\|_{\alpha^*}\ge k^{1/\alpha^*-1/2}\|u\|_2$, with equality at the
constant-coordinate vector; finally, $1/\alpha^*-1/2=1/2-1/\alpha$.
\end{proof}

\begin{lemma}[Range of the exponents]\label{g:lem-exponents}
We have $0\le\theta\le1$ and $\theta+\gamma\le1$. Equality $\theta=1$ holds
only for $(\alpha,\beta)=(1,\infty)$, and $\theta+\gamma=1$ only for
$(\alpha,\beta)\in\{(1,\infty),(\infty,1)\}$. Moreover, $\theta=0$ if and
only if $\beta\le2\le\alpha$, and $\theta>1/2$ forces $\gamma=0$.
\end{lemma}
\begin{proof}
Each of the two summands of $\theta$ lies in $[0,1/2]$, and both vanish
exactly when $1/\beta\ge1/2\ge1/\alpha$. As for $\theta+\gamma$, if
$\beta\ge\alpha$, then $\gamma=0$; otherwise $\gamma=1/\beta-1/\alpha$, and
there are three cases. If $\beta\le2\le\alpha$, then $\theta=0$ and
$\gamma\le1$. If $\beta<\alpha\le2$, then $\theta=1/\alpha-1/2$ and
$\theta+\gamma=1/\beta-1/2\le1/2$. If $2\le\beta<\alpha$, then
$\theta=1/2-1/\beta$ and $\theta+\gamma=1/2-1/\alpha\le1/2$.
For the final claim, $\theta>1/2$ requires both summands to be positive,
that is, $\beta>2$ and $\alpha<2$; then $1/\beta<1/2<1/\alpha$, so
$\gamma=0$.
\end{proof}

\subsection{Block certificates}
\begin{proof}[Proof of \cref{g:prop-lp-radius}]
The claims about the exponents follow from \cref{g:lem-exponents}, and by
\cref{gp:lem-attainment} it suffices to treat $R_\alpha=1$.
Fix $1\le m\le d$. If $\gamma=0$, set $b=1$; otherwise set
$b=\lfloor d/m\rfloor\ge1$. Choose disjoint index sets
$B_1,\dots,B_m\subseteq\{1,\dots,d\}$ of cardinality $b$, which is possible
because $mb\le d$, and define $Q:\R^d\to\R^m$ and the scalar $c>0$ by
\[
 (Qx)_i=c\sum_{j\in B_i}x_j,\qquad
 c=b^{-(1-1/\beta)}m^{-(1/2-1/\beta)_+}.
\]

\emph{The map is a contraction from $\ell_\beta$ to $\ell_2$.} Write
$y_i=\sum_{j\in B_i}x_j$. H\"older's inequality on each block gives
$|y_i|\le b^{1-1/\beta}\|x_{B_i}\|_\beta$, where $x_{B_i}$ denotes the
restriction of $x$ to $B_i$. Applying \eqref{g:eq-extremal} in $\R^m$ to the
vector of block norms, we obtain
\begin{align*}
 \|y\|_2&\le b^{1-1/\beta}\big\|(\|x_{B_i}\|_\beta)_{i\le m}\big\|_2
 \le b^{1-1/\beta}m^{(1/2-1/\beta)_+}
   \big\|(\|x_{B_i}\|_\beta)_{i\le m}\big\|_\beta\\
 &\le b^{1-1/\beta}m^{(1/2-1/\beta)_+}\|x\|_\beta ,
\end{align*}
where the middle step is the first identity of \eqref{g:eq-extremal} for
$\beta\ge2$ and the norm comparison $\ell_\beta\subseteq\ell_2$ for
$\beta\le2$, and the last step uses
$\sum_i\|x_{B_i}\|_\beta^\beta\le\|x\|_\beta^\beta$ for $\beta<\infty$ and the
maximum for $\beta=\infty$. Hence $\|Q\|_{\beta\to2}\le cb^{1-1/\beta}
m^{(1/2-1/\beta)_+}=1$.

\emph{A lower bound on the Euclidean inradius.} The support function of the
symmetric convex body $Q(\mathbb B_\alpha^d)$ is
\[
 h(u)=\max_{\|x\|_\alpha\le1}\langle u,Qx\rangle
     =c\max_{\|x\|_\alpha\le1}\langle w(u),x\rangle=c\|w(u)\|_{\alpha^*},
\]
where $w(u)\in\R^d$ equals $u_i$ on $B_i$ and $0$ outside $\bigcup_iB_i$.
Therefore
$\|w(u)\|_{\alpha^*}=b^{1/\alpha^*}\|u\|_{\alpha^*}=b^{1-1/\alpha}\|u\|_{\alpha^*}$,
and the second identity of \eqref{g:eq-extremal} in $\R^m$ gives
$h(u)\ge cb^{1-1/\alpha}m^{-(1/\alpha-1/2)_+}\|u\|_2$. A symmetric compact
convex set whose support function dominates $r\|\cdot\|_2$ contains
$r\mathbb B_2^m$, so
\[
 \mathcal R_m(\mathbb B_\beta^d,\mathbb B_\alpha^d)
 \ \ge\ c\,b^{1-1/\alpha}m^{-(1/\alpha-1/2)_+}
 \ =\ b^{1/\beta-1/\alpha}\,m^{-\theta}.
\]
If $\gamma>0$, this equals $\lfloor d/m\rfloor^{\gamma}m^{-\theta}$; if
$\gamma=0$, then $b=1$, and the exponent $1/\beta-1/\alpha\le0$ is applied
to $b=1$, giving $m^{-\theta}=\lfloor d/m\rfloor^{\gamma}m^{-\theta}$.
This is \eqref{g:eq-lp-radius}.

\emph{Sharpness.} For $m=d$, \cref{gp:prop-full-lp} shows that the radius
$d^{-\theta(\alpha,\beta)}$ is exact for every $\alpha,\beta$.
For $\alpha=1$ and any $\beta$, every
admissible $Q$ has columns $v_j=Qe_j$ with $\|v_j\|_2\le\|e_j\|_\beta=1$, and
$Q(\mathbb B_1^d)=\operatorname{conv}\{\pm v_1,\dots,\pm v_d\}$; the Gaussian
estimate in the proof of \cref{gp:prop-l1} then gives
$mr^2\le2\log(2ed)$, which together with $r\le1$ is the stated upper bound.
This matches $m^{-\theta}$ up to the stated logarithmic factor when
$\theta=1/2$, that is, when $\beta\le2$. For $\beta>2$, this Gaussian
estimate need not be sharp at smaller image dimensions.
\end{proof}

\begin{lemma}[A general coarse upper bound]\label{g:lem-crude}
For all $\alpha,\beta$ and $1\le m\le d$,
$\mathcal R_m(\mathbb B_\beta^d,\mathbb B_\alpha^d(R_\alpha))\le R_\alpha\,d^{\gamma}$.
\end{lemma}
\begin{proof}
If $\|Q\|_{\beta\to2}\le1$, then
$Q(\mathbb B_\alpha^d(R_\alpha))\subseteq R_\alpha\max_{\|x\|_\alpha\le1}\|x\|_\beta\cdot
\mathbb B_2^m$, and $\max_{\|x\|_\alpha\le1}\|x\|_\beta=d^{(1/\beta-1/\alpha)_+}$
by \cref{g:lem-extremal} applied in $\R^d$ (for $\beta\ge\alpha$ the maximum
equals $1$).
\end{proof}

For the cube and $1\le\beta\le2$, \cref{gp:prop-full-lp} gives
$\mathcal R_d(\mathbb B_\beta^d,\mathbb B_\infty^d(R_\infty))=R_\infty$.
Here $\theta=0$ and $\gamma=1/\beta$, so at full dimension \cref{g:lem-crude}
is loose by the factor $d^{1/\beta}$.

\subsection{Optimizing the dimension}
\begin{proof}[Proof of \cref{g:cor-lp}, part (a)]
Put $\Lambda=R_\alpha\sqrt{L_\beta/\eps}$ and, for $1\le m\le d$, let
$\rho_m=m^{-\theta}$, which is admissible by \eqref{g:eq-lp-radius} with the
choice $b=1$. \Cref{gp:thm-norm} applies with $r_m=R_\alpha\rho_m$ whenever
$\eps\le L_\beta R_\alpha^2\rho_m^2/32$, that is, whenever
\begin{equation}\label{g:eq-window}
 m\le m_{\max}:=\left(\frac{L_\beta R_\alpha^2}{32\eps}\right)^{1/(2\theta)}
 \qquad(\theta>0),
\end{equation}
with no restriction on $m$ when $\theta=0$. For $m=d$ this is the hypothesis
$\eps\le L_\beta R_\alpha^2/(32d^{2\theta})$ of \eqref{g:eq-lp-smooth}, and
substituting $r_d=R_\alpha d^{-\theta}$ into \eqref{gp:eq-norm} gives
\eqref{g:eq-lp-smooth}.

For \eqref{g:eq-lp-lowacc}, let $m^\sharp=\min\{d,\lfloor m_{\max}\rfloor\}$;
this is at least $1$, because $\eps\le L_\beta R_\alpha^2/32$ forces
$m_{\max}\ge1$. The dimension $m^\sharp$ is admissible, so \eqref{gp:eq-norm}
gives $T\ge5m^\sharp/16$, which is \eqref{g:eq-lp-lowacc}.

Two further observations describe the dependence on the image dimension.
First, along this family the product term is
$\min\{m^2,\Lambda m^{1-\theta}\}$, the minimum of two nondecreasing
functions, since $\theta\le1$. Among admissible dimensions it is therefore
largest at $m^\sharp$: choosing the largest $m$ allowed by the accuracy
window maximizes the bound. Second, in the regime $m^\sharp<d$ the product
term cannot exceed the baseline. Indeed, $m_{\max}^{-\theta}\Lambda=\sqrt{32}$
by the definition \eqref{g:eq-window}, and $m^\sharp=\lfloor m_{\max}\rfloor\ge
m_{\max}/2$, so
$\rho_{m^\sharp}\Lambda\le2^{\theta}\sqrt{32}\le2\sqrt{32}$, and the product
term is at most $2\sqrt{32}\,m^\sharp/\Cstar<m^\sharp/360<5m^\sharp/16$. This
is why \eqref{g:eq-lp-lowacc} contains no product factor.
\end{proof}

\begin{corollary}[Blocks in the smooth bound]\label{g:cor-lp-blocks}
Write $\rho_m=\lfloor d/m\rfloor^\gamma m^{-\theta}$ and $\Psi=\gamma+\theta$.
For every $\Lambda>0$ and every $1\le m\le d$,
\begin{equation}\label{g:eq-block-smooth}
 m\min\{m,\Lambda\rho_m\}\ \le\ d\min\{d,\Lambda\rho_d\},\qquad \rho_d=d^{-\theta},
\end{equation}
so that along the block family the product term of \eqref{gp:eq-norm} is
largest at $m=d$, and block maps can improve the result only by enlarging
the accuracy window. Specifically, if $\gamma>0$ and
$0<\eps\le L_\beta R_\alpha^2d^{2\gamma}/(32\cdot4^{\gamma+\theta})$, then
\[
 T\ \ge\ \frac5{16}\,
 \min\left\{d,\ \left\lfloor\frac12
 \left(\frac{L_\beta R_\alpha^2d^{2\gamma}}{32\eps}\right)^{\frac1{2(\gamma+\theta)}}
 \right\rfloor\right\}.
\]
\end{corollary}
\begin{proof}
For \eqref{g:eq-block-smooth}, the bound $\lfloor d/m\rfloor\le d/m$ gives
$\rho_m\le d^\gamma m^{-\Psi}$, hence $m\rho_m\le d^\gamma m^{1-\Psi}\le
d^\gamma d^{1-\Psi}=d^{1-\theta}=d\rho_d$, using $\Psi\le1$
(\cref{g:lem-exponents}) and $m\le d$; together with $m^2\le d^2$, this is
\eqref{g:eq-block-smooth}. The map $m\mapsto m\min\{m,\Lambda\rho_m\}$ is
\emph{not} monotone, because $\lfloor d/m\rfloor$ has jumps. The preceding
monotonicity argument covers only $\rho_m=m^{-\theta}$; for the block
family, the claim follows from the endpoint comparison
\eqref{g:eq-block-smooth}.

For the window, write $A$ for the displayed power of
$L_\beta R_\alpha^2d^{2\gamma}/(32\eps)$ and let
$1\le m\le\min\{d,A/2\}$. Since $\lfloor d/m\rfloor\ge d/(2m)$ for $1\le m\le d$,
\eqref{g:eq-lp-radius} gives $\mathcal R_m\ge R_\alpha(d/2m)^{\gamma}m^{-\theta}$,
so that
\[
 \frac{32\eps}{L_\beta\mathcal R_m^2}
 \le\frac{32\eps}{L_\beta R_\alpha^2}\,2^{2\gamma}d^{-2\gamma}m^{2(\gamma+\theta)}
 =2^{2\gamma}\left(\frac mA\right)^{2(\gamma+\theta)}\le
 2^{2\gamma}2^{-2(\gamma+\theta)}\le1 ,
\]
so $m$ is admissible in \cref{gp:thm-norm} and $T\ge5m/16$. Choose the
largest admissible $m$.
\end{proof}

\begin{proof}[Proof of \cref{g:cor-lp}, part (b)]
Put $N=(M_\beta R_\alpha/\eps)^2$, $\Psi=\gamma+\theta$ and
$c=(Nd^{2\gamma})^{1/(1+2\Psi)}$. For $1\le m\le d$, let
$\rho_m=\lfloor d/m\rfloor^\gamma m^{-\theta}$ and
$\sigma_m=\min\{m,N\rho_m^2\}$. Substituting
$r_m=R_\alpha\rho_m$ and $\lambda=0$ into \eqref{gp:eq-nonsmooth}
gives $T\ge c_Z\mathcal P(m,\sigma_m)$ whenever
$m\ge m_Z^\circ$ and $N\rho_m^2\ge\eps_Z^{-2}$.
At $m=d$ we have $\rho_d=d^{-\theta}$, and the statement becomes
\eqref{g:eq-lp-nonsmooth}, with exactly its accuracy window.

For the cap \eqref{g:eq-ns-cap},
use $\lfloor d/m\rfloor\le d/m$ to obtain $\rho_m\le d^\gamma m^{-\Psi}$ and
therefore
$m\sigma_m\le\min\{m^2,\ Nd^{2\gamma}m^{1-2\Psi}\}$. The first term increases
with $m$; the second is nondecreasing exactly when $\Psi\le1/2$, and the two
agree at $m=c$. If $\Psi>1/2$, the minimum is at most $c^2$ on either side
of $c$, including when $c<1$; together with $m^2\le d^2$, this proves
the second cap. If $\Psi\le1/2$, the maximum is at $m=d$, where
$2\gamma+1-2\Psi=1-2\theta$ turns the second term into $Nd^{1-2\theta}$, so
the value is $\min\{d^2,Nd^{1-2\theta}\}=d\min\{d,Nd^{-2\theta}\}$; at $m=d$
the floor equals $1$ and $\rho_d=d^{-\theta}$, so this value is attained.

For the rounded certificate \eqref{g:eq-lp-peak}, the bounds
$\lfloor d/m\rfloor\ge d/(2m)$ and $\gamma\le1$ give
$\rho_m\ge d^\gamma m^{-\Psi}/2$, hence
$N\rho_m^2\ge Nd^{2\gamma}m^{-2\Psi}/4$. At $m=m^\dagger$, the definition of
$m^\dagger$ makes the right-hand side at least $m^\dagger$, so
$\sigma_{m^\dagger}=m^\dagger$ and
$\mathcal P(m^\dagger,\sigma_{m^\dagger})=(m^\dagger)^2/\log(em^\dagger)$. The
accuracy window of \cref{gp:thm-nonsmooth} is
$N\rho_{m^\dagger}^2\ge\eps_Z^{-2}$, which holds because
$N\rho_{m^\dagger}^2\ge m^\dagger\ge m_Z^\circ\ge\eps_Z^{-2}$.

For admissible dimensions, $1\le\log(e\sigma_m)\le\log(ed)$, so the cap
controls $\mathcal P(m,\sigma_m)$ up to a factor $\log(ed)$.
Finally, $\theta>1/2$ forces $\gamma=0$ (\cref{g:lem-exponents}), so
in that regime $\Psi=\theta$ and $c=N^{1/(1+2\theta)}$.
\end{proof}

\subsection{Comparison of the tabulated certificates}\label{g:app-lp-claims}
Every entry of \cref{tab:lp} is the value of \eqref{gp:eq-norm} or
\eqref{gp:eq-nonsmooth} at the single certificate of \cref{g:prop-lp-radius}
with $m=d$, within its accuracy window. Optimizing over image
dimensions and alternative maps can strengthen these full-dimensional
bounds in two ways.

First, $m=d$ need not be the optimal dimension even within the block family.
For the smooth bound it is optimal by \eqref{g:eq-block-smooth}, so there a
block map affects only the accuracy window (\cref{g:cor-lp-blocks}). For the
nonsmooth bound the transition occurs at $\Psi=\gamma+\theta=1/2$, so once $\Psi>1/2$ and $c=(Nd^{2\gamma})^{1/(1+2\Psi)}<d$, a
smaller block dimension can give a larger bound (\cref{g:cor-lp}). The
strict comparison for $c<d$ also depends on integer rounding and the
profile's logarithmic denominator.
For $c\ge d$ the cap $\min\{d,c\}^2=d^2$ is already attained at $m=d$. An
example is $\alpha=\infty$, $\beta=1$, where $\theta=0$ but $\Psi=1$. Take
$d=n^3$, $m=n^2$ and $\eps=M_\beta R_\alpha/\sqrt C$ with
$C\ge\max\{1,\eps_Z^{-2}\}$ fixed, so that $N=C$ and $c\asymp C^{1/3}n^2<d$
for large $n$. Both dimensions are admissible, $\rho_{n^2}=n$ by
\eqref{g:eq-lp-radius}, and $m\sigma_m$ equals $Cn^3$ at $m=d$ but $n^4$ at
$m=n^2$, so the ratio diverges as $n$ grows; the optimal
$m^\dagger\asymp C^{1/3}n^2$ gives $\asymp C^{2/3}n^4$.

Second, certificates at smaller image dimensions can improve on the
full-dimensional block certificate in the final lower bound. At $m=d$,
\cref{gp:prop-full-lp} shows that the block radius is exact for every
$\alpha,\beta$; the gap to the coarse estimate of \cref{g:lem-crude} reflects
slack in that upper estimate. The smaller-dimensional radii remain open. The
Euclidean-section certificates of \cref{g:prop-flm} and
\eqref{gp:eq-linfty-image} give radius $\asymp1$ at dimension
$m\asymp\min\{d,\min(\beta^*,\log ed)d^{2/\beta^*}\}$ for $\alpha=\beta\le2$,
and at $m\asymp d$ for $\alpha=\beta=\infty$. In the second case this
improves the corresponding row by a factor $\asymp\sqrt d$ in the
accuracy-limited branch and widens the window from
$\eps\le L_\beta R_\alpha^2/(32d)$ to $\eps\le L_\beta R_\alpha^2/128$. In the
first case it widens the window at a smaller dimension; whether it also
improves the bound at a fixed accuracy at which $m=d$ is already admissible
depends on the constant $c_D$ of \cref{g:prop-flm}. The monotonicity argument in the proof of
\cref{g:cor-lp}(a) compares dimensions only within the family
$r_m=Rm^{-\theta}$.

The exponent $\theta$ gives the exact full-dimensional radius for every
$\alpha,\beta$. At smaller dimensions the block certificate is optimal up to
the factor $\sqrt{2\log(2ed)}$ when $\alpha=1$, $\beta\le2$, while the other smaller-dimensional profiles remain open.
\Cref{sec:matching} identifies the classes and accuracy windows in which
the resulting lower bounds match known algorithms.

\section{Nonsmooth hardness from sublevels of a known regularizer}\label{gp:app-nonsmooth}
This appendix uses the external information theorem and the constants fixed
in \cref{g:ext-zzql}; otherwise it is independent of the smooth proofs in
\cref{gp:app}.
Let $B$ be the unit ball of a norm on $\R^d$, let $K$ be nonempty, compact
and convex, and let the known continuous convex penalty $\psi:K\to\R$
satisfy $\min_K\psi=0$. Fix $M>0$ and $\lambda\ge0$, and consider
\[
 \min_{x\in K}F(x)=g(x)+\lambda\psi(x),
 \qquad |g(x)-g(y)|\le M\|x-y\|\quad(x,y\in\R^d),
\]
where $g$ is globally convex. The oracle returns only exact values of the
unknown function. Queries to $g$ may be made at arbitrary points of $\R^d$,
whereas the output must lie in $K$. All policies are total Borel maps, the
success probability is at least $2/3$ on each objective, and a deterministic
query budget applies to every history. The penalty and its proximal
operations are known and may be used freely. As in \cref{gp:thm-composite},
define $K_\tau=\{x\in K:\psi(x)\le\tau\}$ and
$r_m(\tau)=\mathcal R_m(B,K_\tau)$.

We now prove \cref{gp:thm-nonsmooth} in this model. Throughout the proof,
write $r=r_m(\tau)$ and $\delta=\eps+\lambda\tau$ for the admissible pair
$(m,\tau)$ fixed in the statement, and recall that $c_Z,\eps_Z,m_Z^\circ$
are the constants fixed after \cref{g:ext-zzql}.

\begin{proof}[Proof of \cref{gp:thm-nonsmooth}]
Fix $z+r\mathbb B_2^m\subseteq Q(K_\tau)$ with $\|Q\|_{B\to2}\le1$.
For a normalized, convex, positively one-homogeneous, one-Lipschitz source
$h$ from \cref{g:ext-zzql}, put
\[
 g_h(x)=\frac M2\Psi_{r,h}(Qx-z),\qquad
 \Psi_{r,h}(y)=h(y)+(\|y\|_2-r)_+.
\]
The radial forcing lemma, \cref{g:lem-radial}, shows that $g_h$ is globally
convex and $M$-Lipschitz. If $u_*$ minimizes $h$ on $\mathbb B_2^m$, then
$ru_*$ minimizes $\Psi_{r,h}$ globally, and it has a lift at some
$x_0\in K_\tau$. For an $\eps$-optimal feasible output,
\[
 g_h(\widehat x)\le F_h(\widehat x)
 \le F_h(x_0)+\eps\le g_h(x_0)+\delta.
\]
Consequently, the decoder $\widehat u=\pi_r(Q\widehat x-z)/r$ lies in
$\mathbb B_2^m$ and has source error at most $2\delta/(Mr)$.
A query at an arbitrary point $x$ uses at most one feasible source value:
for $y=Qx-z\ne0$ use $h(y)=\|y\|_2h(y/\|y\|_2)$, while $h(0)=0$ is known.
The radial term is known as well. The piecewise simulator and the decoder
are Borel, and these fixed maps do not depend on the hidden $h$.

Let $a_Z$ be the constant fixed after \cref{g:ext-zzql}, put
$\sigma=\min\{m,(Mr/\delta)^2\}$ and set
$k=\lfloor a_Z\sigma\rfloor$. The dimension and source windows imply
$a_Z\sigma\ge2$, so that $a_Z\sigma/2\le k\le a_Z\sigma\le c_{\rm dim}m$ and
\[
 \frac{2\delta}{Mr}\le\frac{c_{\rm err}}{2\sqrt{k}}.
\]
Apply the external obstruction conditionally on the algorithm's seed, then
integrate over the seed. The rounding argument in the proof of
\eqref{g:eq-ns} gives
\[
 T\ge\left\lfloor c_{\rm prob}\frac{mk}{\log(ek)}\right\rfloor+1
 \ge c_Z\,\mathcal P(m,\sigma),\qquad \sigma=\min\{m,(Mr/\delta)^2\}.
\]
Every valid certificate yields a lower bound on the complexity of the same
class, so the supremum of these bounds is also valid.
\end{proof}

\paragraph{Spectral optimizer.}
For $K=R\mathbb B_2^d$ and $\psi(x)=x^\top Hx/2$, we use the exact radius of
\cref{gp:cor-spectral}. Before the dimension cap is applied, the quantity to
be maximized is
\[
 \frac{M^2\min\{R^2,2\tau/\nu_m\}}{(\eps+\lambda\tau)^2}.
\]
If $\lambda>0$, the derivative of $\tau/(\eps+\lambda\tau)^2$ has the sign of
$\eps-\lambda\tau$, and the radius is constant for
$\tau\ge\nu_mR^2/2$. Hence an optimizer is
\begin{equation}\label{gp:eq-ns-optimizer}
 \tau_m^*=\min\{\eps/\lambda,\nu_mR^2/2\}.
\end{equation}
If $\lambda=0$, every $\tau\ge\nu_mR^2/2$ is optimal. This level maximizes
the ratio $Mr_m(\tau)/(\eps+\lambda\tau)$ and hence also the profile;
consequently, if $\tau_m^*$ violates the accuracy window, then no level is
admissible at that $m$.
Since $\sigma\mapsto\sigma/\log(e\sigma)$ is nondecreasing for
$\sigma\ge1$, maximizing the uncapped quantity also maximizes the lower
bound whenever the window holds. In the localized branch
$\eps/\lambda\le\nu_mR^2/2$, substitution gives the capped parameter
$\min\{m,M^2/(2\lambda\nu_m\eps)\}$. The source window and the condition
$m\ge m_Z^\circ$ are still required.

\begin{proposition}[Unbounded nonsmooth spectral gain]\label{gp:prop-ns-gain}
For every $A>4$, there are instances with a fixed known penalty for which
the actual-sublevel lower bound exceeds, by a factor of at least $A/4$, the
best bound obtained from the scalar majorant optimized over both dimension
and sublevel. The dimension $d$ and the Lipschitz constant $M$ vary with
$A$, while $R=\lambda=1$ and $\eps=1/4$ remain fixed. The comparison is between two lower-bound expressions for the same
known-composite class.
\end{proposition}
\begin{proof}
Set
\[
\begin{gathered}
 R=\lambda=1,\quad \eps=\tfrac14,\quad
 s_0=\max\{A,m_Z^\circ,\eps_Z^{-2}\},\\
 n=\lceil As_0\rceil,\quad d=2n,\quad M^2=As_0/2,
 \quad H=\diag(I_n,AI_n).
\end{gathered}
\]
For the actual penalty, take $m=n$ and $\tau=1/4$. Then $r^2=1/2$,
$\delta=1/2$, and the capped parameter is $\sigma_0=As_0\le n$.
The dimension and source accuracy conditions hold because
$n\ge m_Z^\circ$ and $As_0\ge\eps_Z^{-2}$. The certified bound is therefore
$c_ZnAs_0/\log(eAs_0)$.

Using only $\psi(x)\le A\|x\|_2^2/2$ certifies the centered balls
of radius $\min\{1,\sqrt{2\tau/A}\}$ as comparison sets with bounded
regularizer values for the same objective class. By
\eqref{gp:eq-ns-optimizer}, their best uncapped parameter is $s_0$, attained
at $\tau=1/4$ independently of $m$. The envelope over all dimensions is
therefore exactly $c_Z(2n)s_0/\log(es_0)$: for fixed $m$ the bound is a
nondecreasing function of the capped parameter, and
$m\min\{m,s_0\}/\log(e\min\{m,s_0\})$ increases with $m\ge1$. The choice
$m=2n$ is admissible, and its cap is inactive. The ratio of the
actual-sublevel certificate to this optimized scalar-majorant bound is
\[
 \frac A2\,\frac{\log(es_0)}{\log(eAs_0)}\ge\frac A4,
\]
because $A\le s_0$ implies $\log(eAs_0)\le2\log(es_0)$.
Optimizing the actual-sublevel envelope can only increase the numerator.
\end{proof}

\paragraph{Oracle and class boundaries.}
If the interface returns $F$ at points outside $K$, the known penalty must
have a specified finite extension there, so that its value can be added or
subtracted. An interface that returns $F$ only at feasible points needs no
such extension. The result allows arbitrary known proximal operations for
$\psi$, but it supplies no proximal oracle for the unknown $g$ or for $F$.
The constant $M$ bounds the Lipschitz modulus of $g$; requiring the full
objective $F$ to be Lipschitz would also require accounting for the
variation of the regularizer, as in \cref{g:thm-nonsmooth}. An affine tilt absorbed into $g$ contributes to its Lipschitz
constant. In the nonsmooth model, any such normalization must therefore
preserve the prescribed budget $M$.

\section{Upper--lower comparisons and their precise regimes}\label{app:matching}
We compare the lower bounds with known gradient-free algorithms to
prove \cref{match:cor-smooth,match:cor-nonsmooth}. An exact
scalar evaluation is the unit of cost: a two-point observation costs two
queries. All lower bounds refer to success probability at least $2/3$ and a
deterministic worst-case query budget. Known randomized upper bounds stated
in expected suboptimality are run to expected error at most $\eps/3$;
Markov's inequality then supplies the same success probability without
changing the order of a predetermined budget.

\subsection{Euclidean minimization}\label{match:sec-euclidean}
On $\mathcal Q_{d,L,R}$, exact gradient differences and coefficient
interpolation give $O(d\min\{d,\sqrt{LR^2/\eps}\})$ values
(\cref{app:corollaries}; acceleration as in \citet{Bubeck2015}).
\Cref{thm:main} supplies the matching lower bound for every
$0<\eps\le LR^2/32$. The deterministic exact-recovery lower bound of
\citet{Yao1975} concerns the deterministic endpoint $\eps=0$. Our bound applies
to positive accuracy and randomized algorithms.

For the full smooth class, finite-difference acceleration gives
$O(d\sqrt{LR^2/\eps})$ values, and the cutting-plane construction in
\cref{app:general-upper} gives $O(d^2\log(LR^2/\eps))$. Thus the best of
these two documented upper bounds is
\[
 O\!\left(\min\left\{d\sqrt{LR^2/\eps},\ d^2\log(LR^2/\eps)\right\}\right)
 \quad(0<\eps\le LR^2/32).
\]
The polynomial lower bound matches the first branch when
$LR^2/d^2\le\eps\le LR^2/32$. At finer accuracy, \cref{ms:thm-depth} gives a dimension-times-logarithm
lower bound for the full smooth class. Determining its sharp joint
dependence on dimension and accuracy remains open.

For $0<\mu\le L$, put $\kappa=L/\mu$. On the globally $\mu$-strongly
convex subclass, \cref{cor:strong} gives
$\Omega(d\min\{d,\sqrt\kappa,\sqrt{LR^2/\eps}\})$ for
$0<\eps\le LR^2/32$. The accelerated method of
\citet{NesterovSpokoiny2017} gives
$O(d\sqrt\kappa\log(LR^2/\eps))$ exact scalar values.
A finite positive difference radius controls the approximation error;
running to expected error $\eps/3$ gives success probability $2/3$
with a fixed budget. When $\kappa\le d^2$ and $\eps\le\mu R^2$
within this window, these polynomial orders match up to the accuracy
logarithm. The smooth and nonsmooth $\ell_p$ strong-convexity bounds
are given in \cref{app:additional-geometry,g:thm-nonsmooth}.

The midpoint-curvature implementation in
\cref{ms:thm-upper} has the explicit budget
\[
 (d+1)\left\lceil
 \frac{\log(2LR^2/\eps)}{2\log((L+\mu)/(L-\mu))}
 \right\rceil\qquad(0<\mu<L).
\]
This is the upper bound attributed to this work in \cref{tab:bounds}.
Together with the multiscale lower bound it yields
\cref{thm:near-quadratic}, with absolute comparison constants uniformly over
$\mu/L\in[1/2,1]$ and $0<\eps\le LR^2/128^2$. At $\mu=L$, $d+1$ values
recover the affine term exactly and the complexity is $\Theta(d)$ in this
window.

\subsection{Euclidean smoothness and non-Euclidean localization}\label{match:sec-lp-smooth}
Fix $1\le p\le2$, $q=p/(p-1)$ with $q=\infty$ at $p=1$, and consider
\emph{globally} Euclidean $L_2$-smooth convex functions with a global
minimizer in $\mathbb B_p^d(R_p)$. Queries and outputs remain unrestricted.
Using $\alpha=p$, $\beta=2$ in \cref{g:cor-lp} gives
\begin{equation}\label{match:eq-smooth}
 T\ge\max\left\{\frac{5d}{16},\frac d{\Cstar}
 \min\left\{d,\sqrt{\frac{L_2R_p^2}{d^{2/p-1}\eps}}\right\}\right\},
 \qquad 0<\eps\le\frac{L_2R_p^2}{32d^{2/p-1}}.
\end{equation}
The deterministic exact-objective specialization of the randomized
accelerated method of \citet[arXiv version~3, Theorem~2.2]{GorbunovDvurechenskyGasnikov2022}
has upper order
\begin{equation}\label{match:eq-smooth-upper}
 \widetilde O\!\left(d^{1/2+1/q}\sqrt{L_2R_p^2/\eps}\right).
\end{equation}
The gradient regularity is Euclidean; the proximal setup uses $\ell_p$.
The prox in \cref{g:lem-prox} has Bregman radius at most a logarithmic
multiple of $R_p^2$; substituting this radius in the cited general
$p$-norm theorem gives \eqref{match:eq-smooth-upper}. Both stochastic and
additive value noise are set to zero. Its finite-difference radius is chosen
finite, positive, and small enough for the prescribed accuracy. Each two-point estimate is charged two scalar evaluations.

Consequently the upper and lower orders match up to logarithmic factors
whenever the following window is nonempty
\begin{equation}\label{match:eq-smooth-window}
 \frac{L_2R_p^2}{d^{1+2/p}}\le\eps
 \le\frac{L_2R_p^2}{32d^{2/p-1}}.
\end{equation}
At $p=1$ this reads
$L_2R_1^2/d^3\le\eps\le L_2R_1^2/(32d)$ and the polynomial order is
$\sqrt{dL_2R_1^2/\eps}$. At $p=2$, \cref{cor:general-moderate} removes the
logarithmic loss. Below the left endpoint, the lower bound saturates,
and sharp full-class complexity remains open.

\subsection{Euclidean Lipschitz continuity and non-Euclidean feasibility}\label{match:sec-lp-nonsmooth}
Now let $g:\R^d\to\R$ be globally convex and $M_2$-Lipschitz in the Euclidean
norm, and minimize it over $\mathbb B_p^d(R_p)$ with a feasible output.
Queries to $g$ can be unrestricted. The unregularized instance of
\cref{gp:thm-nonsmooth}, or \cref{g:cor-lp}, gives
\begin{equation}\label{match:eq-nonsmooth}
 T\ge c_Z\mathcal P\!\left(d,\min\left\{d,
 \frac{M_2^2R_p^2}{d^{2/p-1}\eps^2}\right\}\right),\qquad
 \eps\le\frac{\eps_ZM_2R_p}{d^{1/p-1/2}},\quad d\ge m_Z^\circ.
\end{equation}
At $p=2$, this recovers the Euclidean information theorem of
\citet{ZhangZhangQiLin2026}; the geometric transfer extends it to
$1\le p<2$.
The two-point mirror-descent guarantee of \citet[Theorem~1]{Shamir2017},
with the norm-matched prox, yields
\begin{equation}\label{match:eq-nonsmooth-upper}
 \widetilde O\!\left(d^{2/q}M_2^2R_p^2/\eps^2\right)
\end{equation}
scalar evaluations. The domain is the $\ell_p$ ball, while the function's
Lipschitz constant is $M_2$. The allowed global extension makes the
small two-point perturbations legitimate even at a boundary iterate.
The dimension dependence comes from the dual-norm moment of a
Euclidean random direction and the prox radius, with the Euclidean
Lipschitz constant $M_2$ unchanged.

Thus the bounds match up to logarithmic factors whenever
\begin{equation}\label{match:eq-nonsmooth-window}
 \frac{M_2R_p}{d^{1/p}}\le\eps
 \le\frac{\eps_ZM_2R_p}{d^{1/p-1/2}},\qquad d\ge m_Z^\circ.
\end{equation}
In particular, the polynomial rate at $p=1$ is $M_2^2R_1^2/\eps^2$.
Under $M_1$-Lipschitz continuity, norm conversion supplies the
corresponding Euclidean constant $M_2$. The upper bounds of the smoothing framework
\citep{GasnikovEtAl2022} also distinguish these two constants. We use the
mirror-descent bound here without invoking the extra parameter window
needed to suppress the iteration term in an accelerated smoothing bound.

\subsection{Known upper bounds for arbitrary regularity norms}\label{match:sec-arbitrary}
Upper bounds are available for both general-norm classes in
\cref{gp:thm-norm,gp:thm-nonsmooth} (with $\lambda=0$ in the latter).
We first apply norm conversion and then keep the proximal geometry
explicit. Sharp complexity in terms of the entire image profile
$\mathcal R_m(B,K)$ for general $B,K$ remains open.

\paragraph{Norm-converted bounds valid for every $B,K$.}
Let $B$ be the unit ball of the regularity norm $\|\cdot\|$ and let
$K\subset\R^d$ be nonempty, compact, and convex. Define
\[
 a_B=\max_{\|u\|_2=1}\|u\|,\qquad
 R_2(K)=\min_{z\in K}\max_{x\in K}\|x-z\|_2,\qquad
 \mathcal D_2=\mathcal D_2(B,K)=a_BR_2(K).
\]
All quantities are known geometric parameters; the minimum exists by
compactness. A singleton $K$ needs no oracle calls, so suppose
$R_2(K)>0$. Duality gives $\|v\|_2\le a_B\|v\|_*$, hence
\[
 \|\nabla f(x)-\nabla f(y)\|_2
 \le a_B L\|x-y\|\le a_B^2L\|x-y\|_2,
 \qquad |g(x)-g(y)|\le a_BM\|x-y\|_2.
\]
Thus $L_2=a_B^2L$ and $M_2=a_BM$ are valid Euclidean regularity bounds.

For a globally convex smooth $f$ with a global minimizer in $K$, start at
a center attaining $R_2(K)$. The Euclidean specialization of the
accelerated method of \citet{GorbunovDvurechenskyGasnikov2022}, or the
finite-difference acceleration in \cref{app:general-upper}, gives
\begin{equation}\label{match:eq-arbitrary-smooth-upper}
 O\!\left(d\mathcal D_2\sqrt{\frac L\eps}\right)
 \qquad(0<\eps\le L\mathcal D_2^2)
\end{equation}
exact scalar values. Queries and output are unrestricted, as required by
\cref{gp:thm-norm}. For globally convex $M$-Lipschitz $g$ minimized over
$K$, \citet[Corollary~2]{Shamir2017}, translated to the same center and
run on $K$, gives
\begin{equation}\label{match:eq-arbitrary-nonsmooth-upper}
 O\!\left(\frac{dM^2\mathcal D_2^2}{\eps^2}\right)
 \qquad(0<\eps\le M\mathcal D_2).
\end{equation}
Its averaged output is in $K$. The global extension of $g$ permits the
two perturbation queries outside $K$. Expected-error guarantees are
converted to the fixed-budget $2/3$ guarantee as at the start of
this appendix, and each pair costs two scalar values. These are the
baselines displayed in \cref{tab:bounds}; sharper choices are possible.
For example, the smooth cutting-plane benchmark in
\cref{app:general-upper} also gives
$O(d^2\log(L\mathcal D_2^2/\eps))$ on its small-accuracy window.

\paragraph{Nonsmooth bounds with an arbitrary proximal norm.}
\citet[Theorem~1]{Shamir2017} allows an arbitrary proximal norm.
Choose an auxiliary norm $\|\cdot\|_{\mathrm{pr}}$ and a known prox
$\omega$ that is $1$-strongly convex in that norm, normalized by
$\min_K\omega=0$. Write
\[
 \Theta_\omega=\sup_{x\in K}\omega(x),\qquad
 p_{\mathrm{pr}}=\bigl(\E\|U\|_{\mathrm{pr},*}^4\bigr)^{1/4},
 \qquad U\sim\operatorname{Unif}(\mathbb S^{d-1}).
\]
For finite $\Theta_\omega$, the theorem yields the scalar-query order
\begin{equation}\label{match:eq-arbitrary-prox-upper}
 O\!\left(1+\frac{d\,p_{\mathrm{pr}}^2M_2^2\Theta_\omega}{\eps^2}\right),
 \qquad M_2\le a_BM.
\end{equation}
When only $M$ is given, use the valid bound $M_2=a_BM$.
The bound depends separately on the objective's Euclidean Lipschitz
constant and the chosen proximal norm.
With the Euclidean prox this recovers
\eqref{match:eq-arbitrary-nonsmooth-upper}; with a suitable $\ell_p$ prox
and Euclidean regularity it gives \eqref{match:eq-nonsmooth-upper}.
On a compact $K$, finite prox range is available for any chosen norm:
a sufficiently scaled squared Euclidean norm is strongly convex in it.
All proximal operations involve the known geometry; the unknown
objective is accessed through scalar values.

\paragraph{Accelerated and non-accelerated non-Euclidean smooth bounds.}
\citet[Theorems~2.2 and~2.9, preprint version]{GorbunovDvurechenskyGasnikov2022}
analyze Euclidean gradient smoothness with a non-Euclidean proximal setup.
Their equations~(2.5) and~(2.31) give the ARDFDS and RDFDS guarantees,
respectively. With our notation,
for $d\ge8$, $1\le p\le2$, and $q=p/(p-1)$, put
\[
 \rho_{d,p}=\min\{q-1,16\log d-8\}\,d^{2/q-1},
 \qquad V[x_0](x^\star)\le\Theta_p.
\]
Here $V$ is the Bregman divergence of a $1$-strongly convex $\ell_p$
prox; $\Theta_p$ bounds its value at a global minimizer. With both source
noise terms set to zero and a positive finite-difference radius chosen
for the prescribed accuracy, the two results give, respectively,
\begin{equation}\label{match:eq-ardfds-rdfds}
 \begin{aligned}
 \mathrm{ARDFDS}:&\quad
 O\!\left(1+d\sqrt{\frac{\rho_{d,p}L_2\Theta_p}{\eps}}\right),\\
 \mathrm{RDFDS}:&\quad
 O\!\left(1+\frac{d\rho_{d,p}L_2\Theta_p}{\eps}\right).
 \end{aligned}
\end{equation}
A batch size of one suffices in this exact deterministic-objective
specialization; the algorithms remain randomized through their directions.
Their finite-difference error terms tend to zero with the positive radius
for each fixed iteration budget. No derivative oracle or limiting query
is used. The factor of two for scalar evaluations is absorbed in $O(\cdot)$.
The source's Section~2.1 allows $p\in[1,2]$; its numerical parameter tables
single out the endpoints $p=1,2$.

Substituting the prox radius from \cref{g:lem-prox} gives the familiar
orders
\[
 \widetilde O\!\left(d^{1/2+1/q}\sqrt{\frac{L_2R_p^2}{\eps}}\right),
 \qquad
 \widetilde O\!\left(d^{2/q}\frac{L_2R_p^2}{\eps}\right),
\]
respectively, with integer rounding understood. One may choose the
smaller budget in advance. In particular, the non-accelerated guarantee
can be useful at coarse accuracy; \cref{match:cor-smooth} asserts matching
only on its stated window. For a function smooth in a general norm with
unit ball $B$, these same algorithms remain applicable after the explicit
substitution $L_2\le a_B^2L$ and an appropriate known $R_p$ or $\Theta_p$.
The conversion can change the smoothness scale by a dimension
factor: $(\mathbf1^\top x)^2/2$ has $L_1=1$ and $L_2=d$
(\cref{g:rem-constants}).

These upper bounds depend on norm conversion, directional moments, and
prox range. Our lower bounds retain the image-dimension/radius tradeoff,
and, for a known penalty, its actual sublevels. The two descriptions yield matching rates in the regimes of
\cref{match:cor-smooth,match:cor-nonsmooth}. Their sharp relation for
general norms and localization sets remains open.

\subsection{Geometric formulas and matching complexity regimes}
\begin{table}[ht]
\centering\small
\caption{Geometric conclusions and the corresponding complexity regimes.}\label{tab:geometric-scope}
\begin{tabularx}{\linewidth}{@{}>{\raggedright\arraybackslash}p{.25\linewidth}
 >{\raggedright\arraybackslash}p{.40\linewidth}>{\raggedright\arraybackslash}X@{}}
\toprule
Setting&Established result&Matching status\\
\midrule
Matched ellipsoids&$\mathcal R_m(B,RB)=R$&Euclidean reduction without a condition-number loss; \ref{gp:prop-ellipsoid}\\
Independent $\ell_\alpha/\ell_\beta$&$\mathcal R_d=R_\alpha d^{-\theta}$; block certificates for $m<d$&Exact full-dimensional geometry; \ref{g:prop-lp-radius}\\
Known quadratic penalty&$r_m(\tau)=\min\{R,\sqrt{2\tau/\nu_m}\}$&Exact radius and optimizer; sharp composite complexity open; \ref{gp:prop-sublevel}\\
Block and Schatten norms&Block count and matrix rank determine radius loss&Explicit transfer certificates; \ref{gp:prop-block}, \ref{gp:prop-schatten}\\
Compact bilinear games&$r\min\{r,\sqrt{CR_xR_y/\eps}\}$ lower order&Matched $d^2$ plateau for equal blocks; \ref{sp:main-bilinear}\\
\bottomrule
\end{tabularx}
\end{table}
For a fixed known regularizer, the lower bound is optimized over its
\emph{actual} sublevels. The spectral gains of
\cref{gp:prop-gain,gp:prop-ns-gain} compare valid lower-bound expressions for
the same class. Sharp upper bounds for general known penalties remain
open. The information complexity is intrinsic to the objective class and
oracle; an algorithm's auxiliary prox enters its upper-bound analysis.

\section{Strong-convexity compatibility in every coordinate norm}\label{app:additional-geometry}
The diagonal norm model requires both smoothness and strong convexity in
$\ell_p$, and localizes a global minimizer in $\mathbb B_p^d(R)$.
\Cref{g:thm-smooth,g:thm-rigidity} treats $1\le p\le2$. The complementary
range gives the following unified statement. Here $\mu$ controls the strong convexity of the entire objective.

\begin{corollary}[All coordinate norms]\label{geom:cor-all-p}
Let $1\le p\le\infty$ and $\Gamma_p=d^{|2/p-1|}$. The class of globally
convex differentiable functions that are $L$-smooth and $\mu$-strongly convex
in $\ell_p$, $\mu>0$, is nonempty precisely when $L\ge\mu\Gamma_p$.
Under this condition, localization in $\mathbb B_p^d(R)$ implies the
exact-value lower bound
\begin{equation}\label{geom:eq-all-p}
 T\ge\max\left\{\frac{5d}{16},\frac d{\Cstar}
 \min\left\{d,\sqrt{\frac{L}{\mu\Gamma_p}},
                 \sqrt{\frac{LR^2}{\Gamma_p\eps}}\right\}\right\},
 \qquad 0<\eps\le\frac{LR^2}{32\Gamma_p}.
\end{equation}
At equality $L=\mu\Gamma_p$, every objective is a known quadratic curvature
plus an unknown affine term: its Hessian is $LI$ for $p\le2$, and $\mu I$
for $p\ge2$. Hence $d+1$ values recover its minimizer exactly.
\end{corollary}
\begin{proof}
For $p\le2$, this is \cref{g:thm-smooth,g:thm-rigidity}. Let $p\ge2$ and
$q=p/(p-1)$. Euclidean $\mu$-strongly convex quadratics of smoothness
$L/\Gamma_p$ are admissible, because
\[
 \|x\|_p\le\|x\|_2\le d^{1/2-1/p}\|x\|_p,
 \qquad \|v\|_q\le d^{1/2-1/p}\|v\|_2.
\]
Their gradient Lipschitz constant from $\ell_p$ to $\ell_q$ is at most $L$;
Euclidean strong convexity implies $\ell_p$ strong convexity with modulus
$\mu$. A minimizer in $\mathbb B_2^d(R)$ lies in $\mathbb B_p^d(R)$.
Apply \cref{cor:strong} with Euclidean smoothness $L/\Gamma_p$. The oracle
and decoder are unchanged, so the budget, randomization, and unrestricted
query convention are preserved.

For necessity, first take a smooth objective and write $H=\nabla^2f(x)$.
Strong convexity on coordinate directions gives $H_{ii}\ge\mu$.
For every sign vector $s\in\{-1,1\}^d$, smoothness gives
$s^\top Hs\le Ld^{2/p}$. Averaging over signs yields
\[
 \mu d\le\tr H\le Ld^{2/p},
\]
which is equivalent to $L\ge\mu\Gamma_p$. At equality all diagonal entries
are $\mu$, and every sign-vector quadratic form attains its common upper
bound. Averaging $s^\top Hs$ against $s_i s_j$ for $i\ne j$ forces
$H_{ij}=0$. Thus $H=\mu I$. For a globally $C^{1,1}$ objective, convolution
with a smooth compactly supported kernel preserves both norm inequalities.
The preceding argument applies to each convolution; locally uniform
convergence of gradients gives $\nabla f(x)=\mu x+b$ in the equality case.
This also justifies necessity without a $C^2$ assumption.

Conversely, $f(x)=\mu\|x\|_2^2/2$ has the required strong convexity and
smoothness $\mu\Gamma_p$ when $p\ge2$. The analogous quadratic in
\cref{g:thm-rigidity} supplies sufficiency for $p\le2$.
At equality write $f(x)=h\|x\|_2^2/2+b^\top x+c$ with known $h>0$.
Then $c=f(0)$ and $b_i=f(e_i)-f(0)-h/2$, which recovers $-b/h$ from $d+1$
values. The recovered minimizer satisfies the prescribed localization
assumption.
\end{proof}

The threshold characterizes compatibility of smoothness and strong
convexity in the objective class itself. At $p=2$ both endpoint formulas
coincide. Sharp complexity in the ill-conditioned regime remains open.

\section{A fixed reference function and relative smoothness}\label{app:relative-smoothness}
We consider a fixed non-Euclidean reference function as part of the
objective class. Fix an open convex domain
$\Omega\subset\R^d$, a point $c\in\Omega$, and a known convex $C^2$ function
$h:\Omega\to\R$. Write
\[
 D_h(x,c)=h(x)-h(c)-\langle\nabla h(c),x-c\rangle.
\]
Following the relative-smoothness convention of
\citet{LuFreundNesterov2016}, require $Lh-f$ to be convex on $\Omega$.
Also require a minimizer $x^\star\in\Omega$ with
$D_h(x^\star,c)\le\Delta$. Queries and the output may be arbitrary points
of $\Omega$; computations involving the fixed known $h$ are free.
The unknown function is accessed only through its exact values.

\begin{theorem}[A global metric certificate]\label{geom:thm-relative}
Suppose $H\succ0$ and $a,r>0$ satisfy
\begin{equation}\label{geom:eq-relative-certificate}
 \nabla^2h(x)\succeq aH\quad(x\in\Omega),\qquad
 E_r=c+H^{-1/2}\mathbb B_2^d(r)\subset\Omega,\qquad
 \sup_{x\in E_r}D_h(x,c)\le\Delta.
\end{equation}
Every adaptive randomized exact-value algorithm with success probability
at least $2/3$ on every objective in the stated relative-smooth class needs
\begin{equation}\label{geom:eq-relative}
 T\ge\max\left\{\frac{5d}{16},\frac d{\Cstar}
       \min\left\{d,\sqrt{\frac{La r^2}{\eps}}\right\}\right\},
 \qquad 0<\eps\le La r^2/32.
\end{equation}
The budget is a deterministic worst-case budget.
\end{theorem}
\begin{proof}
Let $g$ be the Euclidean hard quadratic in \cref{thm:main}, with smoothness
$La$ and minimizer radius $r$, and put
$f(x)=g(H^{1/2}(x-c))$. Its Hessian obeys
\[
 \nabla^2f(x)\preceq LaH\preceq L\nabla^2h(x)\quad(x\in\Omega),
\]
so $Lh-f$ is convex. The source minimizer lifts into $E_r$ and therefore
has Bregman radius at most $\Delta$. The map
$x\mapsto H^{1/2}(x-c)$ preserves objective gaps and answers each target
query with exactly one source query. Its continuity preserves adaptive
Borel policies and their success probabilities. The source permits arbitrary
finite queries, so restricting target queries to $\Omega$ causes no
simulation difficulty. Apply \cref{thm:main}.
\end{proof}

\begin{corollary}[A fixed nonquadratic reference]\label{geom:cor-quartic}
Let $\Omega=\R^d$, $c=0$, and
$h(x)=\|x\|_2^2/2+\|x\|_2^4/4$. For $0<\Delta\le1$,
\begin{equation}\label{geom:eq-quartic}
 T\ge\max\left\{\frac{5d}{16},\frac d{\Cstar}
       \min\left\{d,\sqrt{\frac{L\Delta}{\eps}}\right\}\right\},
 \qquad 0<\eps\le L\Delta/32.
\end{equation}
Queries remain unrestricted throughout $\R^d$.
\end{corollary}
\begin{proof}
Here $\nabla^2h(x)=(1+\|x\|_2^2)I+2xx^\top\succeq I$, and
$\sup_{\|x\|_2\le r}D_h(x,0)=r^2/2+r^4/4$.
The choice $r^2=\sqrt{1+4\Delta}-1$ is admissible and satisfies
$r^2\ge\Delta$ when $0<\Delta\le1$. Apply
\cref{geom:thm-relative} with $H=I$ and $a=1$.
\end{proof}

\paragraph{Global curvature and the fixed reference.}
The Hessian lower bound in \eqref{geom:eq-relative-certificate} holds
throughout the query domain, ensuring that every queried point satisfies
the relative-smoothness condition. The reference function is fixed and
known; the hard objectives vary within the resulting class. Sharp
complexity for general fixed references remains open.

A proximal oracle for the unknown objective returns an exact minimizer of
$f(x)+t^{-1}D_h(x,c)$. When this minimizer exists, comparison with $x^\star$
gives objective error at most $\Delta/t$, so one such stronger oracle call
with $t\ge\Delta/\eps$ is sufficient. This stronger information model
therefore has a one-call guarantee for the same accuracy criterion.

\section{Comparison of oracle models and proof dependencies}\label{app:literature}
We compare the function classes, query responses and regions,
randomization, error criteria, and dimension ranges of the related
oracle models in Table~\ref{tab:models}.

\begin{table}[ht]
\caption{Model distinctions relevant to the smooth product theorem. Each row specifies the information model and class used by the cited
result.}\label{tab:models}
\begin{tabularx}{\textwidth}{@{}>{\raggedright\arraybackslash}p{.24\textwidth}>{\raggedright\arraybackslash}p{.24\textwidth}>{\raggedright\arraybackslash}X@{}}
\toprule
Work & Information / class & Relation to the present theorem \\
\midrule
\citet{Yao1975} & Exact values; exact minimizer of a quadratic form &
Classical $\Omega(d^2)$ bound at the $\eps=0$ exact-recovery endpoint;
Corollary~\ref{cor:quadratic} treats positive accuracy.
The original model has no accuracy parameter and no randomization. \\
\citet{Simchowitz2018} & Gradient queries; random smooth quadratics &
Randomized matrix-based hardness for vector responses; our theorem
quantifies the cost of scalar responses. \\
\citet{ZhangZhangQiLin2026} & Exact values; randomized nonsmooth convex &
Supplies the Euclidean nonsmooth input to our geometric transfer;
Theorem~\ref{thm:main} is independent of this input. \\
\citet{Kerger2026} & Exact values; deterministic nonsmooth convex &
Deterministic high-accuracy endpoint: $\Omega(d^2/\log(d+1))$ at the single
level $\eps\asymp d^{-1/2}$, with a near-matching upper bound at that deterministic endpoint. \\
\citet{YuEtAl2023,YuEtAl2024} & Noisy values; quadratic / strongly convex &
Sample-complexity lower bounds with noise; noiseless hardness requires a
separate information argument. \\
Theorem~\ref{thm:main} & Exact values; randomized smooth convex &
Unrestricted adaptive queries; a product bound without logarithmic loss,
even for quadratic objectives. \\
\Cref{g:sec} & Exact values; smooth \emph{and} nonsmooth, any regularity norm
and any feasible ball & A common Euclidean-image radius determines both transfers;
explicit $\ell_\alpha/\ell_\beta$ rates (\cref{tab:lp}) and the
almost-Euclidean quotient branch on $\ell_p$ balls
(\cref{g:cor-quotient-ns}). \\
\bottomrule
\end{tabularx}
\end{table}

\paragraph{Exact quadratic recovery.}
\citet{Yao1975} addresses a question posed by C.~W.~Gear: how many adaptive
evaluations of an unknown quadratic form in $d$ variables are needed to
determine its minimizer? The resulting worst-case lower bound is of order
$d^2$, while $(d+1)(d+2)/2$ evaluations suffice by interpolation;
\citet{Protasov1996} later gave an $O(d^2\log^2d)$ value-only method for
approximate minimization on this class. That setting concerns exact recovery
by adaptive evaluation, with no accuracy parameter and no randomization.
Corollary~\ref{cor:quadratic} recovers the same order in the regime
$\sqrt{LR^2/\eps}\ge d$ and extends the analysis in two respects. First, it
treats positive accuracy $\eps$ through the profile
$d\min\{d,\sqrt{LR^2/\eps}\}$, which interpolates between the quadratic
dimension scale $d^2$ and the linear dimension scale $d$. Second, it covers
randomized algorithms with success probability $2/3$. The worst-case
guarantee is obtained from a fixed prior and its adaptive posterior, with the
explicit constant $1/4096$.

\paragraph{First-order lower bounds.}
The local-oracle theory of \citet{NemirovskiYudin1983}, the non-Euclidean
smooth reductions of \citet{GuzmanNemirovski2015}, and the randomized lower
bounds of \citet{DiakonikolasGuzman2020,WoodworthSrebro2017} motivate the
accuracy scales. Comparisons with scalar measurements must account for the
oracle responses and admissible dimension regimes of those results.
Simulating a gradient with $O(d)$ values transfers a first-order upper
bound to scalar access. To quantify the unavoidable scalar cost, our
argument instead controls the information gained from each exact value.

\paragraph{Noise and smoothing.}
Gradient-free algorithms for smooth problems \citep{NesterovSpokoiny2017} and
accelerated stochastic derivative-free methods
\citep{GorbunovDvurechenskyGasnikov2022} provide relevant upper-bound
context. Two-point feedback \citep{DuchiJordanWainwrightWibisono2015,Shamir2017}
and black-box smoothing \citep{GasnikovEtAl2022} have their own sampling and
domain conventions. Analytic smoothing \citep{Nesterov2005Smoothing} and
infimal convolution \citep{Moreau1965,RockafellarWets1998} can change what an
exact value reveals. We prove the smooth lower bound directly on quadratics, for which
one exact value corresponds to one rank-one matrix measurement.
Inexact first-order oracle analysis \citep{DevolderGlineurNesterov2014}
provides the approximation guarantees used in the upper-bound
implementations.

\paragraph{Non-Euclidean exact-value complexity.}
\citet{DuchiJordanWainwrightWibisono2015} establish norm-dependent minimax
lower bounds in a stochastic observation model. Their oracle returns
values of random sample functions, and their lower-bound constructions
compare the resulting observation laws. Here each query returns the exact
value of a fixed objective. In this model,
\citet{ZhangZhangQiLin2026} prove the Euclidean nonsmooth information
bound that we transfer to non-Euclidean feasible sets.
\Cref{match:cor-smooth,match:cor-nonsmooth} identify the matching
non-Euclidean dimension--accuracy rates under Euclidean regularity.

\paragraph{Geometric information.}
The Euclidean-image condition records the operator norm, image dimension, and
inscribed radius separately. \Cref{gp:thm-norm} lets the norm that
controls smoothness and the compact set that contains a minimizer vary
independently; translated image balls also accommodate asymmetric sets.
\Cref{gp:thm-composite} applies the same scalar information theorem to actual
sublevels of a fixed known regularizer, and its spectral gain compares two
lower-bound formulas for the same known-composite class. Geometric reductions
and Schatten examples have classical precedents \citep{GuzmanNemirovski2015}.
That useful images exist follows from classical finite-dimensional
Banach-space geometry
\citep{FigielLindenstraussMilman1977,MilmanSchechtman1986,Pisier1989}, and the
elementary identity image remains available when a constant-radius image has
too few dimensions. Information-theoretic local-oracle work
\citep{BraunGuzmanPokutta2017,BraunGuzmanPokutta2024Correction} provides
further context. The nonsmooth numerical bound used here is the explicitly
stated input from \citet{ZhangZhangQiLin2026}.

\paragraph{External inputs to the smooth proof.}
\citet{CorderoErausquin2019} supplies the whole-space matrix-weighted variance
framework; we prove the bounded-domain version with its boundary term.
We apply the Gaussian norm and pseudoinverse bounds of
\citet{HalkoMartinssonTropp2011}. The other inputs are standard: Neumann
existence and regularity, Gaussian regression, and Gamma integrals. These inputs suffice for the smooth proof. The nonsmooth transfer
additionally uses the theorem of \citet{ZhangZhangQiLin2026}.

\section{Open problems}\label{app:open}
The results lead to the following questions, extending the discussion
in \cref{sec:scope}.

\begin{itemize}\itemsep2pt
\item \textbf{High accuracy and strong convexity for general smooth functions.}
\Cref{cor:general-moderate} settles the range
$LR^2/d^2\le\eps\le LR^2/32$. Below this range the quadratic class has
complexity $\Theta(d^2)$ by exact reconstruction, whereas the general-class
upper bound $O(d^2\log(LR^2/\eps))$ of \cref{app:general-upper} grows with
accuracy. \Cref{ms:thm-depth} establishes dimension-times-logarithm growth
for each fixed $0\le\mu/L<1$, and \cref{thm:near-quadratic} gives a uniform
matching estimate for $\mu/L\in[1/2,1]$. Taking the maximum with \cref{cor:strong} yields a lower bound for
growing condition numbers. What is the sharp joint dependence on dimension,
conditioning, and accuracy in this regime? For general composite
objectives, the fine-accuracy behavior beyond the saturation of
\eqref{gp:eq-spectral} is likewise open.
\item \textbf{Smaller image dimensions and the certificate envelope.}
\Cref{gp:prop-full-lp} determines the full-dimensional radius exactly, as
$R_\alpha d^{-\theta(\alpha,\beta)}$ for every $\alpha,\beta$.
What is the complete profile
$\mathcal R_m(\mathbb B_\beta^d,\mathbb B_\alpha^d(R_\alpha))$ for $m<d$?
The block and Euclidean-section certificates give complementary lower
bounds, and the Gaussian estimate controls the $\ell_1$ endpoint.
The choice of dimension raises a further question: the nonsmooth profile
changes regime at $\gamma+\theta=1/2$ (\cref{g:cor-lp}), whereas the
quotient construction lowers the dimension to retain a constant radius and
a wider accuracy window. Beyond the Euclidean case, the resulting envelope
over dimensions and certificates, and its relation to minimax complexity,
are only partially understood.
\item \textbf{Coupling on compact domains and unequal block dimensions.}
The compact bilinear reduction yields a square-root dependence on the inverse
accuracy, whereas relevant upper bounds can have an inverse-linear
dependence. Settling the sharp accuracy dependence requires a stronger
construction or a better algorithm. Our coupled families use $r=\min\{d_x,d_y\}$. The sharp dependence
on both dimensions when the blocks differ greatly remains to be determined. The global-gap coupling
bounds and their versions for local accuracy windows also leave open whether
a logarithmic accuracy factor is necessary at fixed parameters, subject to
quadratic reconstruction.
\item \textbf{Mixed partial-gradient and scalar-value budgets.}
The mixed-oracle constructions bound scalar-value queries under
unrestricted access to the known block's gradient. What is the attainable
region of scalar-value and partial-gradient query budgets? Its dependence
on whether the coupling is known is especially important: partial
gradients can reveal a hidden coupling matrix directly, whereas a fixed known
coupling can make the outer solution depend on the unknown inner solution
without making the partial-gradient responses informative about that
solution.
\item \textbf{Sharp constants and alternative priors.} How much of
\cref{thm:affine} survives for non-Wishart priors, and what replaces $4096$?
\end{itemize}

\section{Experimental protocol and numerical diagnostics}\label{app:experiments}

We report two independent checks. The first compares the two value-only
procedures behind the quadratic upper bound; the second evaluates the analytic
inequality and the trace identity on which the lower bound rests. Both checks use synthetic instances generated by the procedures
below. The lower-bound hard prior is analyzed in \cref{app:euclidean}.

\subsection{Numerical checks}\label{sec:experiments}
We evaluate the quadratic upper-bound procedures on seeded test
instances and the matrix inequalities on prescribed affine slices.
The hard distribution and minimax lower bound are analyzed in
\cref{app:euclidean}.

\begin{figure}[htbp]
\centering
\includegraphics[width=\linewidth]{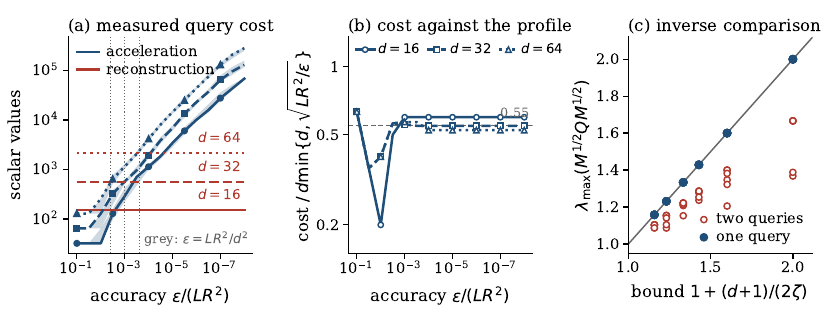}
\caption{Measured scalar query counts on 20 quadratics per dimension, and a
direct evaluation of \cref{thm:affine}. \textbf{(a)} Median and interquartile
range of the values spent by acceleration before the first iterate meeting
each accuracy, assessed after a fixed run, against the constant cost
$(d+1)(d+2)/2$ of reconstruction; accuracy increases to the right. The grey
verticals mark $\eps=LR^2/d^2$, where the two branches of the profile meet.
\textbf{(b)} The cheaper of the two procedures divided by
$d\min\{d,\sqrt{LR^2/\eps}\}$. The three coarsest levels are reached in one or
two iterations and are quantization-dominated; from $\eps=3\times10^{-3}$
downward the ratio stays in $[0.50,0.60]$ for every dimension. \textbf{(c)}
$\lambda_{\max}(M^{1/2}QM^{1/2})$ on affine slices of $\Sym_2$ against the
bound $1+(d+1)/(2\zeta)$, for six values of $\zeta$ and 36 slices of each
codimension.}
\label{fig:value-experiments}
\end{figure}
For $d\in\{16,32,64\}$, we compare finite-difference acceleration with exact
coefficient recovery on 20 seeded convex quadratics per dimension. Their
measured crossover follows the two branches in \cref{cor:quadratic}.
Reconstruction uses exactly $(d+1)(d+2)/2$ values at every requested
accuracy.

Numerical integration on affine slices of $\Sym_2$ checks the inverse
comparison and trace identity. Rank-one slices attain the inverse-comparison
bound analytically; the reported discrepancies measure numerical
agreement at the stated quadrature resolution. Checks on adaptive fibers and an exact rational
coefficient-recovery test are also included. The following subsections give the full protocols and supplementary checks.

\subsection{The two value-only procedures}

\paragraph{Instance generation.}
For each $d\in\{16,32,64\}$, each seed $s\in\{0,\ldots,19\}$ and each family, we
initialize NumPy's \texttt{default\_rng} with
\texttt{SeedSequence([20260921,$d$,$s$,$j$])}, where $j=0$ for the singular
family and $j=1$ for the strongly convex contrast. A square matrix with
independent standard Gaussian entries is orthogonalized by QR, with the signs
fixed so that the diagonal of the triangular factor is positive; the resulting
orthogonal matrix $U$ is Haar-distributed on $O(d)$. We set
$H=U\operatorname{diag}(\xi_1,\ldots,\xi_d)U^\top$ with
\[
 \text{singular family:}\quad \xi_1=0,\quad \xi_i=10^{-6+6(i-2)/(d-2)}\ \ (i\ge2),
\]
\[
 \text{strongly convex contrast:}\quad \xi_i=10^{-4+4(i-1)/(d-1)} .
\]
An independent Gaussian vector, projected off $\ker H$ in the singular family,
is normalized to give $x^\star$ with $\|x^\star\|_2=1$ and
$Hx^\star\in\operatorname{range}H$. The oracle returns
$q(x)=\tfrac12x^\top Hx-(Hx^\star)^\top x+\tfrac12(x^\star)^\top Hx^\star$.
Hence $L=R=1$ in both families, and the contrast has the prescribed condition
number $10^4$. The singular family has effective condition number $10^6$ on
$\operatorname{range}H$, the invariant subspace that contains its iterates.
The additive constant is kept so that the oracle does not silently reveal the
minimum value. All computations use double precision.

\paragraph{Scalar access and query accounting.}
Each procedure receives a callable exact-value oracle and the known
parameters $d,L$. The oracle may evaluate a batch of points and charges one
scalar query per row; it returns neither gradients nor coefficients.
Reconstruction acquires the values at $0$, $\pm e_i$, and $e_i+e_j$ for
$i<j$, exactly $(d+1)(d+2)/2$ queries, and then computes the minimum-norm
minimizer from the recovered system --- a least-squares solve, since $H$ is
singular here. Solving that system has no oracle cost in the model of the
paper.

Acceleration starts from $x_0=y_0=0$ and $t_0=1$ and, for $k=0,\ldots,3071$,
computes
\begin{align*}
 (g_k)_i&=\frac{q(y_k+e_i)-q(y_k-e_i)}2,\qquad
 x_{k+1}=y_k-g_k/L,\\
 t_{k+1}&=\frac{1+\sqrt{1+4t_k^2}}2,\qquad
 y_{k+1}=x_{k+1}+\frac{t_k-1}{t_{k+1}}(x_{k+1}-x_k).
\end{align*}
The central-difference step equals one. Over the reals, the identity
$g_k=Hy_k-Hx^\star$ is exact for quadratic objectives at any nonzero step, so
an arbitrarily small step would add cancellation without any gain.
Each iteration costs $2d$ scalar queries, and neither the condition number nor
the true minimizer is passed to the solver.

Every trajectory runs for the same fixed budget of 3072 iterations. A separate
evaluator then computes the objective gap and identifies the first iterate that
meets each of 15 targets, logarithmically spaced from $10^{-1}$ to $10^{-8}$.
Target crossings are evaluated after the fixed-budget run. All 120
acceleration trajectories reach every target within that budget. The
figure reports medians and interquartile ranges over the 20 seeds; the
bands show the middle 50\% of the observed instance-wise query counts.

\paragraph{Measured query counts.}
\Cref{tab:experiment-counts} lists selected medians for the singular family.
Fitting $\log_{10}$ of the median accelerated cost against $\log_{10}(1/\eps)$
over $\eps\le10^{-4}$ gives slopes $0.449$, $0.463$ and $0.466$ at
$d=16,32,64$, compared with the exponent $1/2$ of the accuracy-dependent
branch. The contrast-family slopes are $0.299$, $0.310$ and $0.322$.
The fitted slopes summarize the stated finite accuracy range for
the two positive spectra. Indeed, adding $10^{-6}vv^\top$ for a unit null
vector $v$ makes the first family globally strongly convex without changing
its iterates or gaps, since every iterate is orthogonal to $v$; the additional
terms also cancel in each central difference. Singularity therefore does not
explain the difference between the two families. The recurrence uses neither
a strong-convexity parameter nor a restart schedule.

\begin{table}[h]
\centering
\caption{Median scalar queries for value-only acceleration on the singular
family. The final column is the fixed reconstruction cost, which suffices for
every listed target.}
\label{tab:experiment-counts}
\begin{tabular}{rrrrrr}
\toprule
$d$ & $\eps=10^{-2}$ & $\eps=10^{-4}$ & $\eps=10^{-6}$ & $\eps=10^{-8}$ & Reconstruction\\
\midrule
16 & 32 & 1136 & 10096 & 68656 & 153\\
32 & 128 & 1920 & 23072 & 130656 & 561\\
64 & 256 & 4160 & 44480 & 284672 & 2145\\
\bottomrule
\end{tabular}
\end{table}

\paragraph{Cheaper updates and the dimension factor.}
A natural question is whether value-only methods that spend fewer than $2d$
values per update avoid the dimension factor. \Cref{fig:value-only-comparison}
reports a separate comparison on the singular and conditioned families
(30 new seeds per dimension, $d\in\{16,32,64,128\}$, $L=R=1$) and on the
isotropic family $H=I$ with a unit-norm minimizer. Five methods receive the
same instance and the same budget of $2d\cdot2048$ values: the accelerated
recurrence of the main experiment; gradient descent with the full
central-difference gradient ($2d$ values per update, step $1/L$);
random-coordinate descent (a uniformly random coordinate $e_i$, the central
difference at $\pm e_i$, two values per update, step $1/L$); random-direction
descent (a uniformly random unit direction $u$, the central difference at
$\pm u$, two values per update, step $1/L$ along $u$ without a dimension
factor); and coefficient reconstruction. All finite differences use step one,
which is exact for quadratics, and first hits are assessed after the run by
the hidden evaluator, as in the main experiment; missed targets are
right-censored at the budget and enter the medians as $+\infty$.

Three observations follow. First, on the singular and conditioned families, at
every accuracy $\eps\le10^{-3}$ at which the three non-accelerated methods
finish within the budget, the ratio of the median count of either two-value
method to that of full finite-difference gradient descent lies in
$[0.89,1.27]$: the cheaper update is paid for by a proportionally larger
number of updates, so the total number of values grows with $d$ in the same
way (\cref{fig:value-only-comparison}b). On the same cells the accelerated method
needs $2$ to $11$ times fewer values than gradient descent, and reconstruction
becomes the cheapest procedure at fine accuracy, as in
\cref{fig:value-experiments}. Second, on the isotropic family every
full-gradient method finishes in one update, that is, $2d$ values, whereas at
$\eps=10^{-6}$ random-coordinate descent needs between $6d$ and $11d$ values
and random-direction descent between $24d$ and $26d$ values
(\cref{fig:value-only-comparison}c). Third, within the common budget the
two-value methods and gradient descent reach $\eps=3\times10^{-5}$ on the
singular family at $d=64$ in every run but miss $10^{-5}$ in about a third of
the runs, whereas acceleration reaches $3\times10^{-8}$ in every run. These comparisons quantify the query costs of the five algorithms
on the specified finite instance families.

\begin{figure}[t]
\centering
\includegraphics[width=\linewidth]{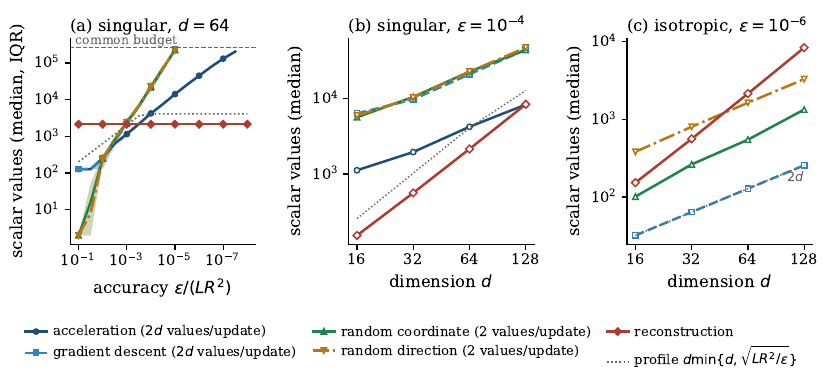}
\caption{Five value-only methods on the same quadratic instances, with
$L=R=1$ and a common budget of $2d\cdot2048$ values. \textbf{(a)} Singular
family at $d=64$: median and interquartile range of the values spent before
the first iterate meeting each accuracy; the medians of cells in which more
than half of the runs miss the target within the budget (dashed grey line) are
omitted. The dotted grey curve is the profile
$d\min\{d,\sqrt{LR^2/\eps}\}$. \textbf{(b)} Singular family at
$\eps=10^{-4}$: medians against the dimension. \textbf{(c)} Isotropic family
$H=I$ at $\eps=10^{-6}$: medians against the dimension; the accelerated
method coincides with gradient descent here (one update, $2d$ values) and is
not drawn separately. All counts are exact query totals for the stated methods and instances.}
\label{fig:value-only-comparison}
\end{figure}

\paragraph{Numerical error.}
Reconstruction attains objective error below $10^{-12}$ on every instance. The
largest entrywise error in the recovered Hessian is below $2\times10^{-16}$,
and in the recovered linear term below $3\times10^{-17}$, that is, a few units
in the last place. These magnitudes depend on the linear algebra library and
on the platform. Per-iteration and reconstruction query costs are exact
integers, whereas the first iterate that crosses a target can depend on
floating-point rounding. The tolerance-free statements are the
rational-arithmetic checks below.

\paragraph{Exact arithmetic checks.}
A separate test uses Python's rational \texttt{Fraction} arithmetic on 100
quadratics with rational coefficients: dimensions $1,2,3,5,8$, each with 20
seeds. It verifies 1220 recovered coefficients against their source values,
380 central-difference gradient coordinates at random rational points, and 100
query counts against $(d+1)(d+2)/2$, each by exact equality. Such finite
tests detect implementation mistakes; the general identities are proved in
\cref{app:corollaries}.

\subsection{The inequalities behind the lower bound}
The script \path{scripts/verify_affine_slice.py} evaluates \cref{thm:affine} and
the trace identity of \cref{r12:fiber} in dimension two, where an affine slice
of $\Sym_2$ of codimension $r$ is a $(3-r)$-dimensional domain and the
conditional moments are computed by quadrature. For
each $\zeta\in\{3/2,5/2,7/2,9/2,13/2,19/2\}$ it draws six slices of
codimension one and six of codimension two, each cut by rank-one measurements
$\langle x_ix_i^\top,W\rangle$ with Gaussian $x_i$ through $W=nI$, and
integrates $\rho(W)\propto\det(W)^\zeta e^{-\tr W/2}$ over the slice by a
composite midpoint rule on a fixed box of half-width $6n+40$ in tangent
coordinates, with $1400^2$ nodes in codimension one and $40000$ nodes in
codimension two. The outer-shell mass is recorded as a truncation diagnostic;
no adaptive widening or resolution refinement is performed.

Across the 72 slices, the ratio of $\lambda_{\max}(M^{1/2}QM^{1/2})$ to
$1+(d+1)/(2\zeta)$ lies in $[0.99989,1.00009]$ on every codimension-one slice
and in $[0.68,0.96]$ on every codimension-two slice. On the first family the
maximum absolute and relative discrepancies are below $2.2\times10^{-4}$ and
$1.1\times10^{-4}$, respectively. The identity $\tr(MJ)=nd-2r$ holds with
relative error at most $7\times10^{-5}$, attained at the smallest $\zeta$, and
below $10^{-9}$ for $\zeta\ge9/2$. These residuals measure agreement at the specified quadrature
resolution. The conditional-slice identity is proved in \cref{app:euclidean}.

\paragraph{Exact attainment on a rank-one fiber.}
For $d=2$, rotate a positive rank-one constraint to $W_{11}=a>0$ and write
\[
 W=\begin{pmatrix}a&b\\ b&t+b^2/a\end{pmatrix},\qquad t>0.
\]
The change of variables from the affine coordinates $(b,W_{22})$ has unit
Jacobian, and their constant slice-volume factor cancels in normalization.
Since $\det W=at$ and $\tr W=a+t+b^2/a$, the density makes
$b\sim N(0,a)$ and $t\sim\operatorname{Gamma}(\zeta+1,2)$ independent,
with shape $\zeta+1$ and scale $2$. For $\zeta>0$,
\[
 M=\diag(a,2\zeta+3),\qquad
 Q=\diag\!\left(\frac{1+1/(2\zeta)}a,\frac1{2\zeta}\right).
\]
Hence the larger eigenvalue of $M^{1/2}QM^{1/2}$ is $1+3/(2\zeta)$, exactly
the coefficient in \cref{thm:affine}. This proves attainment analytically on the tested rank-one fibers.

\paragraph{Higher-dimensional and adaptive fibers.}
A second run evaluates the same two identities on the realized fibers of
adaptive histories in $d\in\{2,3,5\}$ with $\zeta=9/2$, that is, with
$n=d+10$, under three query policies: independent random unit directions,
one unit direction repeated (so that the rank stays one), and a
history-dependent rule in which each new unit direction is computed from the
answers received so far. For each policy and each of 12 seeds, the
conditional moments are recorded after $q\in\{1,d,2d\}$ queries for
$d\in\{3,5\}$ and after $q\in\{1,3\}$ queries for $d=2$. On a fiber of rank
$r$ the conditional law is an integral over $d(d+1)/2-r$ dimensions. For
$d=2$ it is evaluated by quadrature on three nested grids (128, 256 and 512
nodes per axis on boxes of half-width $6n$, $10n$ and $16n$ in tangent
coordinates); for $d=3,5$ it is sampled by slice sampling with four chains,
2000 warm-up sweeps and 8192 retained draws thinned by two; a full-rank fiber
is a point mass and is evaluated exactly. \Cref{tab:posterior-diagnostics}
summarizes the results. On every rank-one fiber the ratio of
$\lambda_{\max}(M^{1/2}QM^{1/2})$ to the bound $1+(d+1)/(2\zeta)$ equals one
up to quadrature error in $d=2$ and lies in $[0.994,1.026]$ in $d=3,5$,
consistent with equality on rank-one fibers up to a Monte Carlo error of a
few percent; on fibers of higher rank the ratio decreases, down to the
point-mass value $2\zeta/(2\zeta+d+1)$ at full rank. The trace identity holds
with relative residual below $1.1\times10^{-6}$ under quadrature and below
$1.1\times10^{-2}$ under sampling. The operator norm of the posterior inverse
$J$, averaged over the twelve histories of a policy, lies between $1.09$ and
$1.98$, whereas the budget $\phi(2\E r/n)$ of \cref{r12:budget}, evaluated at
the empirical mean rank, lies between $1.67$ and $4.44$; the budget therefore
holds with a wide margin in all 24 policy--query cells. These are numerical residuals at the stated grid and sampler
settings. The exact disintegration and moment identities are established
in \cref{app:euclidean}.

\begin{table}[t]
\centering\small
\caption{Posterior diagnostics on realized fibers ($\zeta=9/2$, $n=d+10$,
12 seeds, three query policies). The ratio is
$\lambda_{\max}(M^{1/2}QM^{1/2})$ divided by the bound $1+(d+1)/(2\zeta)$
of \cref{thm:affine}; the last column is the largest relative residual of the
identity $\tr(MJ)=nd-2r$ over the listed fibers. Rank-one fibers arise from
the first query of every policy and from all queries of the repeated policy;
full-rank fibers are point masses.}
\label{tab:posterior-diagnostics}
\begin{tabular}{rrrlrrr}
\toprule
$d$ & bound & rank $r$ & evaluation & fibers & ratio: median [min, max] & trace residual\\
\midrule
2 & $4/3$ & 1 & quadrature & 48 & 1.0000 [1.0000, 1.0000] & $1.1\times10^{-6}$\\
2 & $4/3$ & 3 & point mass & 24 & 0.7500 & $<10^{-15}$\\
3 & $13/9$ & 1 & sampling & 60 & 1.003 [0.994, 1.014] & $7.5\times10^{-3}$\\
3 & $13/9$ & 3 & sampling & 24 & 0.920 [0.866, 0.951] & $2.8\times10^{-3}$\\
3 & $13/9$ & 6 & point mass & 24 & 0.6923 & $<10^{-15}$\\
5 & $5/3$ & 1 & sampling & 60 & 1.008 [0.997, 1.026] & $1.0\times10^{-2}$\\
5 & $5/3$ & 5 & sampling & 24 & 0.985 [0.926, 1.005] & $5.1\times10^{-3}$\\
5 & $5/3$ & 10 & sampling & 24 & 0.830 [0.791, 0.919] & $2.4\times10^{-3}$\\
\bottomrule
\end{tabular}
\end{table}

\paragraph{Saddle and geometry diagnostics.}
The script \path{scripts/verify_saddle_algebra.py} performs
4284 numerical checks over four seeds and the dimension pairs
$(1,8),(8,1),(3,7),(7,3),(5,5)$. They cover the two- and five-value payoff
identities, stationarity, rectangular padding, global and compact gap
transfers, compact response windows, the known-diagonal bias bound, the
algebraic identities behind the mixed-oracle reductions, and coefficient
reconstruction. The shared-curvature and known-diagonal gaps are compared with
independently computed best responses, and the compact bilinear gap is checked
against its closed form, derived in \cref{thm:bilinear-compact}. All checks pass at the
script's stated tolerance. The geometry script likewise passes its finite
scalar-grid and sampled Gaussian diagnostics. The random-direction
extrema are finite-sample estimates; the operator-norm and image-radius
bounds are proved in the geometric appendices.
The geometry figures in the appendix are deterministic illustrations of the
stated profiles, distinct from the measured data in
\cref{fig:value-experiments}.

\paragraph{Two value-only policies for known diagonal curvatures.}
\Cref{fig:saddle-value-queries} illustrates \cref{sp:prop-known-diagonal-upper}
on strongly convex--strongly concave quadratics with known diagonal Hessians
$\mu_x=\mu_y=1$, block pairs
$(d_x,d_y)\in\{(8,32),(32,8),(16,16),(16,64),(64,16),(32,32)\}$, coupling
operator norm $1$ or $10$, and ten seeds per pair. Two value-only policies
receive the same instances and the same budget of $4(d_x+d_y)$ values per
iteration for 512 iterations: an extragradient method whose two partial
gradients are recovered exactly by central differences at each of its two
points, and the reconstruction of \cref{sp:prop-known-diagonal-upper}, which
uses $1+d_x+d_y+d_xd_y$ values and then solves the stationarity equations. The
error is the global gap, evaluated in closed form from the two partial
gradients and checked against independently computed best responses (largest
discrepancy $8.9\times10^{-16}$); the evaluator never influences the updates.
Reconstruction returns the saddle point at its fixed cost, independently of
the coupling and of the accuracy, with gap at most $4.5\times10^{-29}$ and
coefficient errors at most $1.2\times10^{-15}$. The extragradient method
reaches every target down to $10^{-8}$ within the budget, but its cost grows
with the accuracy and with the coupling: raising the coupling norm from $1$ to
$10$ multiplies its median cost by a factor between $4.5$ and $5.3$ at every
pair and target, whereas the reconstruction cost is unchanged. The experiment compares the two procedures on the stated
strongly convex--strongly concave quadratic family.

A separate run checks the two recovery counts of
\cref{sp:prop-bilinear-upper,sp:prop-bilinear-compressed} on bilinear-affine
instances with a global saddle point $(x^\star,0)$, $\|x^\star\|_2\le1$, for
the block pairs $(d,1)$ and $(1,d)$ with $d\in\{8,16,32,64,128,256\}$ and for
$(8,8)$, $(8,32)$, $(32,8)$, with two instance families and twenty seeds per
pair. In all 600 runs, coefficient recovery used exactly $d_y(d_x+1)$ values
and compressed recovery exactly $d_x(d_x+1)$ values, and both returned pairs
with zero compact gap up to rounding; for $d_x=1$ the compressed rule needs
two values whatever $d_y$ is, as noted after \cref{sp:prop-bilinear-compressed}.

\begin{figure}[t]
\centering
\includegraphics[width=\linewidth]{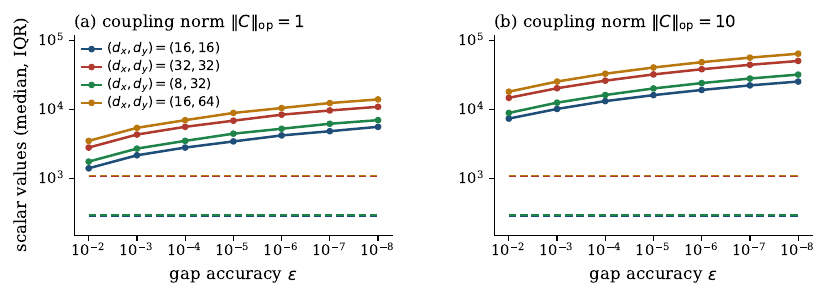}
\caption{Value-only policies on strongly convex--strongly concave quadratics
with known diagonal Hessians $\mu_x=\mu_y=1$, ten seeds per block pair, for
coupling operator norm $1$ \textbf{(a)} and $10$ \textbf{(b)}. Solid curves:
median and interquartile range of the values spent by the value-based
extragradient method before its first iterate with global gap at most $\eps$.
Dashed lines: the fixed cost $1+d_x+d_y+d_xd_y$ of the reconstruction of
\cref{sp:prop-known-diagonal-upper}, which returns the saddle point at that
cost; the lines for $(16,16)$ and $(8,32)$, and for $(32,32)$ and $(16,64)$,
almost coincide. Both policies use the same budget of $4(d_x+d_y)\cdot512$
values; no run was censored.}
\label{fig:saddle-value-queries}
\end{figure}

\section{Comparison with the concurrent work of Wu et al.}\label{app:wu}
We compare the smooth exact-value guarantees of
\citet{WuEtAl2026DeterministicExactValue} with ours. We
write $L$ for their Euclidean smoothness constant $\beta$ and express both
results in the notation of this paper.

\subsection{Oracle models and rate statements}
\begin{table}[ht]
\centering\small
\caption{Concurrent smooth exact-value results. Each column specifies
the role of $R$ in its query and localization assumptions.}
\label{tab:wu}
\begin{tabularx}{\linewidth}{@{}>{\raggedright\arraybackslash}p{.20\linewidth}
 >{\raggedright\arraybackslash}X>{\raggedright\arraybackslash}X@{}}
\toprule
& Wu et al., version 1 & This paper\\
\midrule
Policies & Deterministic adaptive & Adaptive randomized Borel; pointwise success at least $2/3$\\
Queries / output & Both in $\mathbb B_2^d(R)$ & Arbitrary finite vectors in $\R^d$\\
Localization & Unique minimizer in $\mathbb B_2^d(R/2)$ & A global minimizer in $\mathbb B_2^d(R)$\\
Lower-bound hard family & Smoothed max chain & Positive-definite quadratics for \cref{thm:main}\\
Lower polynomial order & $d\min\{\sqrt{LR^2/\eps},\allowbreak(d/\log(ed))^{1/3}\}$ & $d\min\{\sqrt{LR^2/\eps},d\}$\\
Upper order & $d\sqrt{LR^2/\eps}$, bounded queries & $d\sqrt{LR^2/\eps}$, unrestricted queries\\
Lower dimension range & Sufficiently large $d$ & Every $d\ge1$; explicit constant $1/4096$\\
\bottomrule
\end{tabularx}
\end{table}
Their lower bound assumes $d\ge d_0$ and $0<\eps\le c_\eps LR^2$ for
universal constants. Their Theorems 3.3--3.5 and Corollary 3.6 match the
upper and lower orders when
\[
 LR^2(\log(ed)/d)^{2/3}\le\eps\le c_\eps LR^2.
\]
Our full-class matching range is $LR^2/d^2\le\eps\le LR^2/32$
(\cref{cor:general-moderate}). Our quadratic matching result additionally
covers every smaller positive accuracy in the stated window
(\cref{cor:quadratic}). To compare the same localization promise, substitute
$R/2$ for our radius. This changes only absolute constants in these orders. The two models
retain their respective query regions.

\subsection{Hard instances and query access}
Their construction uses Moreau smoothing, exact prefix shielding, and delayed
rotations. The bounded-query norm enters its spherical-cap step. Our
Wishart argument instead controls the conditional law on every realized
affine slice, under a fixed prior independent of the algorithmic seed.
This conditional-law argument permits randomization and arbitrary
finite query norms, including points outside the localization ball.

Their upper guarantee keeps all queries in the prescribed ball.
Our finite-difference implementation in \cref{app:general-upper} allows
queries outside the localization set. For quadratics, coefficient
recovery matches the $d^2$ plateau. For the full smooth class, the
nonquadratic multiscale construction gives the additional accuracy
dependence in \cref{thm:near-quadratic}.

\subsection{Non-Euclidean and saddle-point extensions}
\Cref{gp:thm-norm,gp:thm-composite,gp:thm-nonsmooth} extend the
exact-value analysis to independent
regularity and localization geometries and actual sublevels of a known
regularizer. The nonsmooth transfer uses the Euclidean information
theorem of \citet{ZhangZhangQiLin2026}.
\Cref{thm:near-quadratic} gives a uniform matching estimate near $\mu=L$.
The saddle-point results cover unequal blocks and separate compact
and global gaps, as well as scalar and mixed access. \Cref{sec:matching}
identifies matching non-Euclidean regimes, and \cref{app:open} gives the
remaining questions for general geometries, coupling, and growing
condition numbers.

\end{document}